\documentclass[a4paper, 12pt]{article} 
\usepackage{amssymb,amsmath,latexsym,color,graphicx}
\usepackage{hyperref} 

\newtheorem{dref}{Definition}[section] \newtheorem{lemma}[dref]{Lemma}
\newtheorem{theo}[dref]{Theorem} \newtheorem{prop}[dref]{Proposition}
\newtheorem{remark}[dref]{Remark} 
\newtheorem{cor}[dref]{Corollary}

\newenvironment{proof}{\par\noindent{{\bf Proof.}}}{\hfill$\Box$
\medskip} 
\newenvironment{proofof}{\par\noindent{{\bf Proof} of }}{\hfill$\Box$
\medskip} 
 
\newcommand{\ekv}[2]{\begin{equation}\label{#1}#2\end{equation}}

\newcommand{\no}[1]{(\ref{#1})} 

\title{Weyl law for semi-classical resonances with randomly perturbed
  potentials} \author{Johannes Sj{\"o}strand \footnote{Ce travail a
    b\'en\'efici\'e d'une aide de l'Agence Nationale de la Recherche
    portant les r\'ef\'erences JC05-52556 et ANR-08-BLAN-0228-01 ainsi que d'une
    bourse FABER du conseil r\'egional de Bourgogne}\\\small Institut de
  Math\'ematiques de Bourgogne,
  Universit\'e de Bourgogne\\
  \small 9 avenue Alain Savary - BP 47870\\
  \small 21078 Dijon cedex\\ \footnotesize
  johannes.sjostrand@u-bourgogne.fr\\
  \footnotesize and UMR 5584 du CNRS} \date{}
\begin{document}
\maketitle
\abstract{In this work we consider semi-classical Schr\"odinger
  operators with potentials supported in a bounded strictly convex
  subset ${\cal O}$ of ${\bf R}^n$ with smooth boundary. Letting $h$
  denote the semi-classical parameter, we consider classes of
  small random perturbations and show that with probability very close
  to 1, the number of resonances in rectangles $[a,b]-i[0,ch^{2/3}[$,
  is equal to the number of eigenvalues in $[a,b]$ of the Dirichlet
  realization of the unperturbed operator in ${\cal O}$ up to a small
  remainder.

\medskip\centerline{\bf R\'esum\'e}

\medskip Dans ce travail on consid\`ere des op\'erateurs de
Schr\"odinger dont les potentiels ont leur supports dans un
ensemble strictement convexe ${\cal O}\Subset {\bf R}^n$ \`a bord
lisse. Avec $h$ d\'esignant le param\`etre semi-classique nous
consid\'erons des classes de petites perturbations al\'eatoires et
montrons qu'avec probabilit\'e tr\`es proche de 1, le nombre de
r\'esonances dans des rectangles $[a,b]-i[0,ch^{2/3}[$, est \'egal au
nombre de valeurs propres dans $[a,b]$ de la r\'ealisation de
Dirichlet de l'op\'erateur dans ${\cal O}$, \`a un petit reste pr\`es.}

\tableofcontents

\section{Introduction}\label{intro}
\setcounter{equation}{0}

There is now a very large literature about the distribution of
scattering poles (resonances) often using methods from non-self-adjoint spectral
theory and microlocal analysis, including many results about upper and
lower bounds
on the density of resonances. See for instance \cite{St06},
\cite{Chr10} and the references given there. Less is known about
actual asymptotics for the number of resonances in various domains. In
this paper we shall give such a result for the
semi-classical Schr\"odinger operator
\ekv{intro.1}{P=-h^2\Delta +V(x),}
on ${\bf R}^n$ where $V\in L^\infty ({\bf R}^n;{\bf R})$ has compact
support. 

Recall that the resonances or scattering poles of the operator
(\ref{intro.1}) can be defined as the poles of the meromorphic
extension of the resolvent $(P-z)^{-1}:C_0^\infty ({\bf R}^n)\to
H^2_\mathrm{loc} ({\bf R}^n)$ across the positive
real axis, to the logarithmic covering space of ${\bf
  C}\setminus \{ 0\}$ when $n$ is even and to the double covering when
$n$ is odd. Alternatively we can continue $(P-k^2)^{-1}$ from the
upper half-plane across ${\bf R}\setminus \{ 0\}$ which gives a
meromorphic function on {\bf C} when $n$ is odd. Using the second
definition, we can introduce the number $N(r)$ of resonances in the disc
$D(0,r)$ when $n$ is odd. 

In one dimension and for $h=1$, M.~Zworski \cite{Zw87} showed that if
$[a,b]$ is the convex hull of the support of $V$, then
\ekv{intro.2}{N(r)=\frac{2(b-a)}{\pi }r+o(r),\ r\to \infty ,} which is
2 times the asymptotic number of eigenvalues $\le r^2$ of the
Dirichlet realization of $-\Delta +V$ on $[a,b]$, the factor 2 being
explained by the fact that the resonances are symmetric around the
imaginary axis. He also showed that most of these concentrate to
narrow sectors around the real axis. This extended an earlier result
of T.~Regge \cite{Re58}. Subsequently, B.~Simon \cite{Si00} gave a
different proof, inspired by the work of R.~Froese \cite{Fr97}, who got
similar results for potentials that do not necessarily have compact
support but are very small near infinity. See also the recent works
\cite{DaPu10, DaExLi10, ExLi11} about Weyl and non-Weyl asymptotics
for graphs.

In higher odd dimensions, M.~Zworski \cite{Zw89} considered the case
of radial potentials of the form $V(x)=f(|x|)$ with support in
$\overline{B(0,a)}$ where $f\in C^2([0,a])$, $a>0$, $f(a)\ne 0$ and
obtained a Weyl type asympotics (still with $h=1$), \ekv{intro.3} {
  N(r)=K_na^nr^n+o(r^n),\ r\to +\infty , } where $ K_n>0$.  Recall
also that Zworski \cite{Zw89b} gave an upper bound in the non-radial
case with the correct power of $r$ and using his analysis, P.~Stefanov
\cite{St06}, gave an explicit formula for the constant $K_na^n$ in the
radial case and showed that the right hand side of (\ref{intro.3}) is
up to $o(r^n)$ the sum of 2 times the number of eigenvalues $\le r^2$
for the interior Dirichlet problem in the ball $B(0,a)$ and the number
of scattering poles for the exterior Dirichlet Laplacian in ${\bf
  R}^n\setminus B(0,a)$. (See also G.~Vodev \cite{Vo94}.) He also
showed (as a corollary of a more general result for operators with
black box) that if we drop the radiality assumption and only assume
that $V\in L^\infty ({\bf R}^n;{\bf R})$ has its support in
$\overline{B(0,a)}$, then we have the upper bound
\ekv{intro.4}{N(r)\le K_na^nr^n+o(r^n),\ r\to +\infty .}

\par T.~Christiansen \cite{Chr10} introduced the set $\mathfrak{M}_a$
of $L^\infty $ potentials $V$ with support in $\overline{B(0,a)}$ for
which we have (\ref{intro.3}) and gave the leading asymptotics, of the
form $Cr^n$, for the number of resonances in sectors in the lower
half-plane intersected with the disc $D(0,r)$. These formulas were
implicit in \cite{Zw89, St06} in the case of the radial potentials
considered there. In particular, when considering smaller and smaller
sectors adjacent to ${\bf R}_+$ or ${\bf R}_-$ we can see, using Lemma
3.3 of \cite{Chr10} and some wellknown formulas for the $\Gamma $
function and the volume of the unit ball, that the constant $C$
converges to the one we get in the leading Weyl asymptotics for the
number of Dirichlet eigenvalues for the Laplacian in $B(0,a)$. In the
theorems 1.2, 1.3 of the same paper the author gives interesting
extensions ``for most values of $z$'' to the case of potentials
$V(x,z)$ depending holomorphically on a parameter $z$ with
$\mathrm{supp\,}V(\cdot ,z)\subset \overline{B(0,a)}$ such that
$V(\cdot ,z_0)$ belongs to $\mathfrak{M}_a$ for at least one value of
$z_0$. Such results remain significant also after restriction to
real-valued potentials. (See also earlier results of the same author,
cited in \cite{Chr10}.) In the recent work \cite{DiVu12} (which
appeared after the submission of the present work), T.-C.~Dinh and
D.-V.~Vu obtain sharper results, namely that for holomorphic families
of potentials, if one element is in a sharpened version of the class
$\mathfrak{M}_a$, then so do all elements away from a pluri-polar set.

The main result of this paper has some relations to the above
mentioned ones. We work in the semi-classical limit ($h\to 0$) and the
ball $B(0,a)$ is replaced by a more general strictly convex set. Our
is result does not make use of any class of the type $\mathfrak{M}_a$
and the conclusion concerns the number of resonances in a thin
rectangle. Nevertheless it is very interesting to note the
similarities of the results, and there are also similarities in the
proofs at least on some ideological level.

We next proceed with a rough description of our result and leave the
precise statements to the next section. Let ${\cal O}\Subset {\bf R}^n$
be open strictly convex with smooth boundary and let $V_0\in C^\infty
(\overline{{\cal O}};{\bf R})$ vanish to the order $v_0>0$ on the
boundary. By $V_0$ we also denote the extension to all of ${\bf R}^n$
which vanishes outside ${\cal O}$ and we consider the potential
$$V(x)=V_0(x)+\delta \widetilde{q}_\omega (x)$$ where $\delta >0$ is a small
parameter $>0$ and $\widetilde{q}_\omega $ a random perturbation whose
properties will be specified in the next section. A possible choice of
$\delta $ is a high power of $h$. Our main result, Theorem \ref{re1}
then states that if $0<a<b<\infty $ and if $C>0$ is large enough so
that the exterior Dirichlet problem for $-h^2\Delta $ has no
resonances in the rectangle $[a,b]+ih^{2/3}[-C^{-1},0]$, then with
probability very close to 1, the number of resonances of $P=-h^2\Delta
+V$ in the rectangle $[a,b]+ih^{2/3}[-C^{-1},0]$ is equal to the
number $N_0([a,b])$ of eigenvalues in $[a,b]$ of the Dirichlet
realization of $h^2\Delta +V_0$ in ${\cal O}$ plus two ``errors''. The
first error is a term that can be bounded by a positive power of $h$
times $h^{-n}$. The second error is bounded by a constant times
$N_0([a-\rho ,a+\rho ])+N_0([b-\rho ,b+\rho ])$ where $\rho
=h^{\frac{2}{3}-\delta }$ for any fixed $\delta >0$. As will be stated
more explicitly in the theorems \ref{re0} and \ref{re4}, we can choose
our random perturbations to be concentrated to a ball of radius $h^N$
in the Sobolev space $H^s$ for arbitrarily large $N$ and $s$.

In the case of a deterministic potential with a potential well in an
island, one can count resonances in rectangles closer to the real
axis. Such results can be found in the appendix of \cite{NaStZw03} and
in Section 9 of \cite{HeSj86}. The phenomen is now a little different
however, due to the potential barrier, and the reference asymptotics
of eigenvalues now depends on the behaviour of the operator
near the potential well.

The motivation for this work was to apply recent results and
techniques for proving Weyl asymptotics for non-self-adjoint
differential operators with small random perturbations either in the
semi-classical limit or in the limit of large eigenvalues \cite{Sj08a,
  Sj08b, BoSj10}, to the problem of resonances. Indeed, using some
version of complex scaling or its microlocal versions, this can be
viewed as an eigenvalue problem for a non-self-adjoint operator. The
new difficulty here is however that if we want to keep a realistic
problem we should apply the random perturbation first and use complex
scaling only outside the support of the perturbation. If we let
$p(x,\xi )$ denote the leading semi-classical symbol of the scaled
operator, and we let $z$ vary in a complex domain like a thin
recatngle along the real axis, then as soon as $z$ is not real, the
set $p^{-1}(z)$ must belong to the part of phase space which
corresponds to the scaled region (since the original unscaled symbol
is real valued) and hence the support of the random perturbation is
away from the $x$-space projection of this set. This leads to a
difficulty since the method in \cite{Sj08a, Sj08b} is based on the
study of the random matrix $(\widetilde{q}_\omega e_j|\overline{e}_k)
$, where $e_1,...,e_N$ is an orthonormal family of eigenfunctions of
$(P-z)^*(P-z)$ corresponding to the small eigenvalues and where we let
$P$ denote the scaled operator. Now, the $e_j$ will be concentrated to
the projection of $p^{-1}(z)$ which sits outside the obstacle, hence
away from the support of the random perturbation. Our
random matrix will therefore tend to be small which is a serious
problem in the approach of \cite{Sj08a, Sj08b}. In order to make the
distance smaller, one could try to make the distorsion very important
already very close to the support of the perturbation, but that leads
to the use of very exotic symbols and after some attempts in that
direction we decided to follow a different less intuitive approach. In the
next section we formulate the result and in Section \ref{outl} we give
an outline of the proof.

It would be interesting to have related statements about almost sure
Weyl asymptotics of large resonances in certain parabolic
neighborhoods of the real axis in the non-semi-classical case
($h=1$). It is quite possible that such a result can be obtained from
the present paper along the same lines as the corresponding result for
large eigenvalues by W.~Bordeaux Montrieux and the author.
\cite{BoSj10}.

\paragraph{Acknowledgements} We thank J.M.~Bouclet for having pointed
out the reference \cite{Ca02} where the idea of differentiating
several times to reach trace class operators is clearly present (cf
Section \ref{gpd}). We also thank T.~Christiansen for helpful comments
about \cite{Chr10}, A.~Voros for indicating references about the
complex WKB-method and V.~Ivrii and L.~Zielinski for references and
information about Weyl laws for the eigenvalues of semi-classical
Schr\"odinger operators with potentials of limited regularity.
Discussions with M.~Zworski and M.~Hitrik around other joint works and
projects have been helpful when preparing the sections \ref{cds},
\ref{aed}. Comments by V.~Petkov and M.~Zworski led to
the correction of some errors.

We also thank the referee for his many pertinent remarks
and even for sending us some numerical calculations in dimension 1
which would deserve to be available to a larger audience.

\section{The result}\label{re}
\setcounter{equation}{0}

We start with a concrete case of our main result (Theorem
\ref{re0}). After that we give the full formulation (Theorem \ref{re1})
which inlcudes a description of the probability measures that are
involved. After that we give a simplified and partially generalized
version of the main result (Theorem \ref{re4}) which combined with a
result of V.~Ivrii \cite{Iv03} gives Theorem \ref{re0}.

Let ${\cal O}\Subset {\bf R}^n$ be open, strictly convex with smooth
boundary. Let $\kappa >0$ be the geometric constant in (\ref{re.10.7})
below and let $\zeta _1>0$ be the smallest zero of the Airy function
$\mathrm{Ai}(-t)$. The concrete version of the main result is then
\begin{theo}\label{re0}
Let $N=\min (]\frac{n-1}{2},+\infty [\cap {\bf Z})$,
$\widetilde{s}>\max(\frac{n}{2}+3,2N+\frac{n}{2})$, $s>\frac{n}{2}$
and let $\beta >0$. 
Then there exists a probability measure $\mu $ on $H^s(\overline{{\cal
  O}})$ with support in the ball $\{W\in
H^s(\overline{{\cal O}});\ \Vert W\Vert_{H^s}\le h^\beta  \}$ such
that the following holds:

Let $0<c_1<c_2<2(1/2)^{2/3}\kappa \zeta _1$. There
exists a constant $C>0$ such that if $\frac{1}{2}\le a<b\le 2$,
$c_1\le c\le c_2$, $\widetilde{\epsilon }\ge Ch(\ln 1/h)^2$ and
$V_0\in H^{\widetilde{s}}(\overline{{\cal O}})$, then for
$P=-h^2\Delta +V_0+W$, $W\in H^s(\overline{{\cal O}})$, we have with
probability (with respect to the random term $W$) \ekv{re.16concrete} { \ge
  1-{\cal O}(1)\frac{h(\ln
    1/h)^2}{h^{N_7}}e^{-\frac{\widetilde{\epsilon }} {Ch(\ln 1/h)^2}},
} that for the set $\sigma (P)$ of resonances of $P$,
counted with their algebraic multiplicity,
 \ekv{re.17concrete}
{\begin{split} |\# (\sigma (P)\cap
    ([a,b]+ih^{\frac{2}{3}}c[-1,0]))-\frac{1}{(2\pi
      h)^n}\iint_{a\le \xi ^2+V_0(x)\le b}dxd\xi |\\
    \le {\cal O}(1) h^{-\frac{2}{3}-n}\widetilde{\epsilon
    }.\end{split} }
Here we also assume that $n\ge 3$ or that neither $a$ nor $b$ is a critical value
of $V_0$.

$N_7$ is independent of the other parameters, while the constants
${\cal O}(1)$ in (\ref{re.16concrete}), (\ref{re.17concrete}) depend
on $c_1,c_2,\beta ,\widetilde{s},s$ and on an upper bound on $\Vert
V_0\Vert_{H^{\widetilde{s}}(\overline{{\cal
      O}})}$. 
\end{theo}

We now start to formulate the more complete result. Our unperturbed
operator will be \ekv{re.1} { P_0=-h^2\Delta +V_0:\ L^2({\bf R}^n)\to
  L^2({\bf R}^n), } where $V_0\in C^\infty (\overline{{\cal O}})$ and
we identify $V_0$ with its zero extension. We also assume: \ekv{re.2}
{\hbox{On }\partial {\cal O}\hbox{ we have } V_0(x)=0 \hbox{ and
  } \partial _{\nu }V_0\le 0, } where $\nu $ denotes the exterior unit
normal.

\par The result concerns the distribution of resonances of \ekv{re.3} {
  P=P_\delta  =P_0+\delta \Theta (x)q_\omega (x), } where $\Theta (x)\in C^\infty (\overline{{\cal O}})$ satisfies \ekv{re.4} { 0<\Theta (x)\asymp
  \mathrm{dist\,}(x,\partial {\cal O})^{v_0},\ x\in {\cal O}\setminus \partial
  {\cal O},\ v_0\in ]\frac{n-1}{2},+\infty [\cap {\bf N}. }
As in (\ref{re.1}) $\Theta $ also denotes the $0$-extension to all of
${\bf R}^n$. It belongs to $C_0^k ({\bf R}^n)$
if $v_0>k$. It would be interesting to be able to work with a profile in
$C_0^\infty $.

\par As in \cite{Sj08a, Sj08b}, we choose the random function
$q_\omega $ of the form 
\ekv{re.5}
{q_\omega (x)=\sum_{0<\mu _k\le L}\alpha _k(\omega )\epsilon
_k(x),\ |\alpha |_{{\bf R}^D}\le R,}
where $\epsilon _k$ is an orthonormal basis of real eigenfunctions of
$h^2\widetilde{R}$, where $\widetilde{R}$ is an $h$-independent real
positive elliptic 2nd order operator on $X$ with smooth 
coefficients. Here $X$ is a smooth compact manifold of dimension $n$
containing $\overline{{\cal O}}$ (in the sense that we have some
diffeomorphism from a neighborhood of $\overline{{\cal O}}$ onto an
open set in $X$ and we identify $\overline{{\cal O}}$ with its image).
For instance, we can let $X$ be an $n$-dimensional torus and choose
$-\widetilde{R}$ to be the Laplacian.  
Moreover, $h^2\widetilde{R}\epsilon _k=\mu
_k^2 \epsilon _k$, $\mu _k>0$. 
We choose $L=L(h)$, $R=R(h)$ in the following intervals where
$s\in ]\frac{n}{2},v_0+\frac{1}{2} [$, $\epsilon \in ]0,s-\frac{n}{2}[$, 
 $\theta
\in ]0,1/2[ $ are fixed:
\ekv{re.6}
{\begin{split}&h^{-M_\mathrm{min}}\ll L\le Ch^{-M},\quad
M\ge M_\mathrm{min}:=\frac{v_0+\frac{\frac{1}{3}+n}{1-2\theta
  }}{s-\frac{n}{2}-\epsilon },\\
&h^{-\widetilde{M}_\mathrm{min}}\le R\le
 h^{-\widetilde{M}},\ \widetilde{M}\ge
 \widetilde{M}_\mathrm{min}:=(\frac{n}{2}+\epsilon
 )M_\mathrm{min}+1+\frac{3n}{2}+v_0 ,\end{split}}
and we shall denote by $L_\mathrm{min}$ and $R_\mathrm{min}$ the lower
bounds for $L$ and $R$ in these estimates.
By Weyl's law for the large eigenvalues of elliptic
self-adjoint operators, the dimension $D$ is of the order of magnitude
$(L/h)^n$. We introduce the  small parameter  
\ekv{re.8}
{\begin{split}
&\delta =\tau _0h^\alpha /C,\quad \tau _0\in ]0,h^{\frac{5}{3}}],\\
&\alpha \ge \alpha (n,v_0,s,\epsilon ,\theta ,M,\widetilde{M} ), 
\end{split}
} 
where an explicit (and not very nice) expression for $\alpha
(n,v_0,s,\epsilon ,\theta ,M,\widetilde{M})$ can be deduced from the proof.

\par The random variables $\alpha _j(\omega )$ will have a
joint probability distribution \ekv{int.6.5}{P(d\alpha )=C(h)e^{\Phi
(\alpha ;h)}L(d\alpha ),} where for some $N_4>0$,
\ekv{re.9}{ |\nabla _\alpha \Phi |={\cal
O}(h^{-N_4}),} and $L(d\alpha )$ is the
Lebesgue measure on ${\bf R}^D$. ($C(h)$ is the norming constant.) 

\par We need the parameter \ekv{re.11}{\epsilon _0(h)=h((\ln
  \frac{1}{h})^2+\ln \frac{1}{\tau _0})} and assume that $\tau _0=\tau
_0(h)$ is not too small, so that $\epsilon _0(h)$ is small.

\par It was shown by T.~Harg\'e and G.~Lebeau \cite{HaLe94}, see also
\cite{SjZw3}, that  the exterior Dirichlet problem for
$-h^2\Delta $ on ${\bf R}^n\setminus {\cal O}$ has no resonances in
the set
\ekv{re.10.5}
{
\Im z\ge -2(h\Re z)^{\frac{2}{3}}\kappa \zeta _1+Ch,\ \frac{1}{2}\le
\Re z\le 2,
} 
if $C$ is large enough, where 
\ekv{re.10.7}{\kappa =2^{-\frac{1}{3}}\cos \frac{\pi}{6}\min _{S\partial {\cal
    O}}Q^{\frac{2}{3}},}
$Q$ is the second fundamental form on $\partial {\cal O}$ and $\zeta
_1>0$ is the smallest zero of $\mathrm{Ai}(-t)$ with $\mathrm{Ai}$
denoting the Airy function which spans the space of solutions to
$(-\partial _t^2+t)u=0$ that are exponentially subdominant on the
positive real axis. 

\par For technical reasons, we shall restrict the attention to
rectangles of the form $R=[a,b]+ih^{2/3}c[-1,0]$, $\frac{1}{2}\le
a<b\le 2$, $c>0$ with $c$ small enough so that $R$ is contained in the
domain (\ref{re.10.5}). Thus we will assume that $c<2(1/2)^{2/3}\kappa
\zeta _1$. (We could replace the bounds $1/2$ and $2$ by any other
positive bounds $0<b_1<b_2$.)

\par Let $P^0_{\mathrm{in}}$ denote the Dirichlet realization of $P_0$
in ${\cal O}$ and let $N_0(\lambda )$ denote the number of eigenvalues
of $P^0_\mathrm{in}$ in the interval $]-\infty ,\lambda ]$, counted
with their multiplicity. Similarly,
if $I\subset {\bf R}$ we let $N_0(I)$ denote the number of such eigenvalues
 in $I$.
The main result of this work is:
\begin{theo}\label{re1} 
Let $\sigma (P_\delta )$ denote the set of resonances of $P_\delta
$. Let $0<c_1<c_2<2(1/2)^{2/3}\kappa \zeta _1$, $\rho =h^{-\delta
  _0+2/3}$, where $\delta _0>0$ is arbitrarily small but fixed. Then
there exists a constant $C>0$ such that for $\frac{1}{2}\le a<b\le 2$,
$c_1\le c\le c_2$ and $\widetilde{\epsilon }\ge C\epsilon _0(h)$, we
have with probability
\ekv{re.13}
{
\ge 1-{\cal O}(1)\frac{\epsilon
  _0(h)}{h^{n+N_6+\frac{2}{3}}}e^{-\frac{\widetilde{\epsilon }}
{C\epsilon _0(h)}},
}
where the constant ${\cal O}(1)$ is independent of
$a,b,c,\widetilde{\epsilon },h$, that
\ekv{re.14}
{\begin{split}
|\# (\sigma (P_\delta )\cap
([a,b]+ih^{\frac{2}{3}}c[-1,0]))-N_0([a,b])|\\
\le {\cal
  O}(1)
(\sum_{w=a,b}N_0([w-\rho ,w+\rho ]))+h^{-\frac{2}{3}-n}\widetilde{\epsilon }).\end{split}
}
Here $N_6=\max (N_3,N_5)$, where $N_3=n(M+1)$,
$N_5=N_4+\widetilde{M}$.
\end{theo}

\begin{remark}\label{re2}
{\rm As in \cite{Sj08a, Sj08b} and in an earlier work with M.~Hager cited
there, with probability
\ekv{re.15}
{
\ge 1-{\cal O}(1)\frac{\epsilon
  _0(h)}{h^{n+N_6+\frac{4}{3}}}e^{-\frac{\widetilde{\epsilon }}{C\epsilon _0(h)}},
} we have (\ref{re.14}) simultaneously for $\frac{1}{2}\le a<b\le 2$,
$c_1\le c\le c_2$.}
\end{remark}
As we point out in Remark \ref{sv1}, for a general perturbation
$W=\delta \Theta q_\omega  $ as in Theorem \ref{re1}, we have
$$
\Vert W\Vert_{H_h^{\widetilde{s}}({\bf R}^{n})}\le {\cal O}(\delta )L^{\widetilde{s}}R,
$$
provided that $\frac{n}{2}<\widetilde{s}<v_0+\frac{1}{2}$. Here
$H_h^{\widetilde{s}}$ is the standard Sobolev space equipped with its
natural semi-classical norm (see Section \ref{al}). By playing with
the parameters, the perturbations in Theorem \ref{re1} can be chosen
to be bounded by arbitrarily high powers of $h$ in Sobolev spaces with
arbitrarily high regularity exponents.

\par We also have:
\begin{prop}\label{re3}
  The conclusion in Theorem \ref{re1} remains valid if we change
  $V_0$ by adding an $h$-independent potential  $W_0\in L^\infty
  ({\cal O})$ such that $W_0={\cal
    O}(\mathrm{dist\,}(x,\partial {\cal O})^3)$, $\partial ^\alpha
  W_0\in L^\infty $ for $|\alpha | \le 2N$ and $W_0\in H^s(\overline{{\cal O}})$. Here $N$ is the smallest
  integer in $](n-1)/2,+\infty [$ and $s>n/2$ is the
  parameter in Theorem \ref{re1}.
\end{prop}

Recall that $H^s(\overline{{\cal O}})=\{v\in H^s({\bf R}^n);\,
\mathrm{supp\,}v\subset \overline{{\cal O}} \}$.
Combining the remark and Theorem \ref{re1}, we get the following less
detailed but perhaps more transparent version of our main
result, where our unperturbed potential is $V_0=W_0$.
\begin{theo}\label{re4}
Let $N=\min (]\frac{n-1}{2},+\infty [\cap {\bf Z})$,
$\widetilde{s}>\max(\frac{n}{2}+3,2N+\frac{n}{2})$, $s>\frac{n}{2}$
and let $\beta >0$. 
Then there exists a probability measure $\mu $ on $H^s(\overline{{\cal
  O}})$ with support in the ball $\{W\in
H^s(\overline{{\cal O}});\ \Vert W\Vert_{H^s}\le h^\beta  \}$ such
that the following holds:

Let $0<c_1<c_2<2(1/2)^{2/3}\kappa \zeta _1$, $\rho =h^{-\delta
  _0+2/3}$, where $\delta _0>0$ is arbitrarily small but fixed. There
exists a constant $C>0$ such that if $\frac{1}{2}\le a<b\le 2$,
$c_1\le c\le c_2$, $\widetilde{\epsilon }\ge Ch(\ln 1/h)^2$ and
$V_0\in H^{\widetilde{s}}(\overline{{\cal O}})$, then for
$P=-h^2\Delta +V_0+W$, $W\in H^s(\overline{{\cal O}})$, we have with
probability (with respect to the random term $W$) \ekv{re.16} { \ge
  1-{\cal O}(1)\frac{h(\ln
    1/h)^2}{h^{N_7}}e^{-\frac{\widetilde{\epsilon }} {Ch(\ln 1/h)^2}},
} that for the set $\sigma (P)$ of resonances of $P$, \ekv{re.17}
{\begin{split} |\# (\sigma (P)\cap
    ([a,b]+ih^{\frac{2}{3}}c[-1,0]))-N_0([a,b])|\\
    \le {\cal O}(1) (\sum_{w=a,b}N_0([w-\rho ,w+\rho
    ]))+h^{-\frac{2}{3}-n}\widetilde{\epsilon }).\end{split} } Here
$N_7$ (equal to $n+N_6+2/3$ as in Theorem \ref{re1}, with
$M=M_\mathrm{min}$, $\widetilde{M}=\widetilde{M}_\mathrm{min}$) is
independent of the other parameters, while the constants ${\cal O}(1)$
in (\ref{re.16}), (\ref{re.17}) depend on $c_1,c_2,\beta
,\widetilde{s},s$ and on an upper bound on $\Vert
V_0\Vert_{H^{\widetilde{s}}(\overline{{\cal O}})}$.
\end{theo}

Indeed, it suffices to apply Proposition \ref{re3} with 
$V_0=W_0$ and to observe:
\begin{itemize}
\item $V_0$ is of class $C^3$ with support in $\partial
  \overline{{\cal O}}$ and therefore $V_0={\cal
    O}(\mathrm{dist\,}(x,\partial {\cal O})^3)$,
\item It suffices to choose the perturbation $W=\delta \Theta q_\omega
  $ as in (\ref{re.3})--(\ref{re.8}) with $M=M_\mathrm{min}$,
  $\widetilde{M}=\widetilde{M}_\mathrm{min}$, $\tau _0=h^{5/3}$ and the parameters $v_0$ and
  $\alpha$ sufficiently large.
\item We can choose the probability $\mu $ to be ``$P$'' in (\ref{int.6.5}), with
  $\Phi =0$ (so that $N_4=0$), but any other choice as in
  (\ref{int.6.5}), (\ref{re.9}) is OK.
\end{itemize}
We end the section by explaining how Theorem \ref{re0} follows from
Theorem \ref{re4}. It suffices to apply the following result of
V.~Ivrii \cite{Iv03}, Theorem 2.1. (See also related results by
L.~Zieli\'nski \cite{Zi04} in the case without boundary.)

Consider the semi-classical Schr\"odinger operator $P=-h^2\Delta
+V(x)$ on the open set $X\Subset {\bf R}^n$ with smooth ($C^\infty $)
boundary. We assume that $\nabla V$ is continuous with modulus of
continuity $\nu (t)={\cal O}(|\ln t|^{-1})$. We
equip $P$ with Dirichlet boundary conditions. When $n=1,2$ we assume
the micro-hyperbolicity property that $|\nabla V|\ne 0$ when $V=E$,
uniformly for $E$ in some compact interval $J$ in $]0,+\infty
[$. Then, uniformly for $E$ in $J$, $0<h\le 1$, the number of eigenvalues in
$]-\infty ,E]$ is equal to the standard Weyl term $(2\pi
h)^{-1}\mathrm{vol\,}(\{(x,\xi )\in T^*X;\, \xi ^2+V(x)\le E \}) $ plus a remainder
which is ${\cal O}(h^{1-n})$ for $n\ge 2$ and ${\cal O}(\ln 1/h)$ for
$n=1$ 

\section{Some elements of the proof}\label{outl}
\setcounter{equation}{0} We will introduce a distorsion $\Gamma
\subset {\bf C}^n$ of ${\bf R}^n$ which concides with ${\bf R}^n$
along ${\cal O}$ and with an exterior dilation of ${\bf R}^n$ outside
${\cal O}$ as in \cite{SjZw2,SjZw3,SjZw4} and
\cite{HaLe94}. Let $P=P_\Gamma $ be the corresponding dilation of
$-h^2\Delta +V$, $V=V_0+\delta \Theta (x)q_\omega (x)$. Then (see for
instance \cite{SjZw1}) $P=P_\Gamma $ has discrete spectrum in an angle
$-\theta _0<\mathrm{arg\,}z\le 0$ and the eigenvalues there coincide
with the resonances.

Let $P_\mathrm{ext}$ be the Dirichlet realization of $P$ on $\Gamma
\setminus {\cal O}$, so that the spectrum of $P_\mathrm{ext}$ in the
above angle coincides with the set of resonances for the exterior
Dirichlet problem for $-h^2\Delta $ (recalling that
$\mathrm{supp\,}V\subset \overline{{\cal O}}$). As we recalled in
Section \ref{re}, there are no such resonances in
$[\frac{1}{2},2]+ih^{2/3}c_0[-1,0]$ if we fix 
\ekv{outl.0}
{
0<c_0<2(\frac{1}{2})^{\frac{2}{3}}\kappa \zeta _1 .
}
Restricting $z$ to the
domain \ekv{outl.1} { \frac{1}{2}<\Re z<2,\quad \Im
  z>-c_0h^{\frac{2}{3}}, } we can therefore introduce the Green
operator $G_\mathrm{ext}(z):H^0(\Gamma \setminus {\cal O})\to H^2(\Gamma \setminus {\cal
  O})$ and the Poisson operator  $K_\mathrm{ext}:\, H^{3/2}(\partial {\cal O})\to
H^2(\Gamma \setminus {\cal O})$ so that the exterior Dirichlet operator
\ekv{outl.2} { {\cal
    P}_\mathrm{ext}(z)=\begin{pmatrix}P-z\\h^{\frac{1}{2}}\gamma \end{pmatrix}=H^2(\Gamma
  \setminus {\cal O})\to H^0(\Gamma )\times H^{\frac{3}{2}}(\Gamma
  \setminus \partial
  {\cal O}) } has the bounded inverse \ekv{outl.3} { {\cal
    E}_\mathrm{ext}(z)=\begin{pmatrix}G_\mathrm{ext}
    &h^{-\frac{1}{2}}K_\mathrm{ext}(z)\end{pmatrix}: H^0(\Gamma
  \setminus {\cal O})\times
  H^{\frac{3}{2}}(\partial {\cal O})\to H^2(\Gamma \setminus {\cal O}).  } Here $\gamma
$ is the operator of restriction to $\partial {\cal O}$. Let ${\cal
  N}_\mathrm{ext}=\gamma hD_\nu K_\mathrm{ext}$ denote the
exterior Dirichlet to Neumann operator, where $D_\nu =\frac{1}{i}\frac{\partial
}{\partial \nu }$ and $\nu $ denotes the exterior unit
normal. Introduce \ekv{outl.4} { B=\gamma hD_\nu -{\cal
    N}_\mathrm{ext}\gamma :\, H^2({\cal O})\to
  H^{\frac{1}{2}}(\partial {\cal O}); } \ekv{outl.5} {
  {\cal P}_\mathrm{out}(z)=\begin{pmatrix}P-z\\
    h^{\frac{1}{2}}B\end{pmatrix}:\, H^2({\cal O})\to H^0({\cal
    O})\times H^{\frac{1}{2}}(\partial {\cal O}).  } For $z$ in the
domain (\ref{outl.1}) we shall see, by considering the continuity
conditions at $\partial {\cal O}$, that $z$ is a resonance
(i.e. belongs to the spectrum of $P_\Gamma $) iff ${\cal
  P}_\mathrm{out}(z)$ is non-bijective, or equivalently if $0\in \sigma
(P_\mathrm{out}(z))$ where $P_\mathrm{out}(z)=P-z:\, H^0({\cal O})\to
H^0(0)$ is the closed unbounded operator whose domain is the
``outgoing'' space: ${\cal D}(P_\mathrm{out}(z))=\{ u\in H^2({\cal
  O});\, B(z)u=0\}$.

\par
Let 
\ekv{outl.6}
{
{\cal P}_\mathrm{in}(z)=\begin{pmatrix}P-z\\
  h^{\frac{1}{2}}\gamma \end{pmatrix}:\, H^2({\cal O})\to H^0({\cal
  O})\times H^{\frac{3}{2}}(\partial {\cal O}),
} 
which is bijective precisely when $z$ is not a (real) eigenvalue of
the Dirichlet realization of $P$ in ${\cal O}$. Away from the
Dirichlet spectrum we introduce the inverse 
$$
{\cal E}_\mathrm{in}(z)=(G_\mathrm{in}(z),h^{-\frac{1}{2}}K_\mathrm{in}(z)):\,
H^0({\cal O})\times H^{\frac{3}{2}}(\partial {\cal O})\to H^2({\cal O})
$$
and notice (cf. (\ref{ub.3}), (\ref{red.6})) that
\ekv{outl.7}
{
{\cal
    P}_\mathrm{out}(z)=\begin{pmatrix} 1 &0\\
    h^{\frac{1}{2}}BG_\mathrm{in} &{\cal N}_\mathrm{in}-{\cal
      N}_\mathrm{ext}\end{pmatrix}{\cal P}_\mathrm{in}(z).
}
Here ${\cal N}_\mathrm{in}=\gamma hD_\nu K_\mathrm{in}$ is the
interior Dirichlet to Neumann map. Thus for $z$ away from the
Dirichlet spectrum, $z$ is a resonance precisely when $0$ belongs to
the spectrum of ${\cal N}_\mathrm{in}-{\cal
  N}_\mathrm{ext}:H^{3/2}(\partial {\cal O})\to H^{1/2}(\partial {\cal
  O})$. 

In Section \ref{gpd} we show how to define -- up to some non-vanishing
factor -- $\det A(z)$ for certain holomorphic or meromorphic families of
operators that are not necessarily Schatten class perturbations of the
identity. With this extended notion of the determinant we get from
(\ref{outl.7}) that
\ekv{outl.8}
{
\det {\cal P}_\mathrm{out}(z)=\det {\cal P}_\mathrm{in}\det ({\cal N}_\mathrm{in}-{\cal N}_\mathrm{ext}).
}

\par A rather substantial part of the paper is devoted to the study of
${\cal N}_\mathrm{in}$, ${\cal N}_\mathrm{ext}$, in the regions $|\Im
z|\ge h^{2/3}/\widetilde{C}$ and $\Im z\ge -c_0h^{2/3}$ respectively,
where $\widetilde{C}$ is an arbitrarily large constant. Many such
studies have already been done (see for instance \cite{SjZw4}), but as
is often the case, we found it necessary to make a new one for the
needs of this paper. From this study we get somewhat roughly,
\ekv{outl.9}{\ln |\det ({\cal N}_\mathrm{in}-{\cal
    N}_\mathrm{ext})|\le {\cal O}(h^{1-n}).}  for \ekv{outl.10} { \Re
  z\in ]\frac{1}{2},2[,\ |\Im z|\asymp h^{2/3},\ \Im z\ge
  -h^{\frac{2}{3}}c_0.  } The exponent in (\ref{outl.9}) reflects the
fact that we have made a reduction to the $n-1$ dimensional manifold
$\partial {\cal O}$.

\par In view of (\ref{outl.8}) this gives a precise upper bound on
$\ln|\det {\cal P}_\mathrm{out}(z)|$ for $z$ in the region
(\ref{outl.10}). Combined with a rough polynomial upper bound on $\ln |\det
{\cal P}_\mathrm{out}(z)|$ in the full region 
$|\Im z|\le h^{\frac{2}{3}}/C$ and the maximum principle, we get the
upper bound
\ekv{outl.11}
{
\ln |\det {\cal P}_\mathrm{out}(z)|\le \Phi _\mathrm{in}(z)+{\cal O}(h^{1-n})
}
in the rectangle (\ref{outl.10}), where $\Phi _\mathrm{in}(z)$
coincides with $\ln |\det {\cal P}_\mathrm{in}(z)|$ for $|\Im z|\ge
h^{2/3}/\widetilde{C}$ and is extended (suitably) as a harmonic
function inside $|\Im z|<h^{2/3}/\widetilde{C}$.

A last and quite substantial part of the paper is to show (in the
spirit of \cite{Sj08a, Sj08b}) that for every $z$ with $
h^{2/3}/\widetilde{C}\le |\Im z|\le c_0h^{2/3}$, $1/2 < \Re z<2$, we
also have a lower bound on $\ln |\det ({\cal N}_\mathrm{in}-{\cal
  N}_\mathrm{ext})|$ almost as sharp as the
upper bound (\ref{outl.9}) with probability very close to 1. 

With these upper and lower bounds at our disposal, the main result
follows by applying Theorem 1.2 of \cite{Sj09} to the
holomorphic function $\det {\cal P}_\mathrm{out}(z)$, whose zeros are
the resonances.

\section{Grushin problems and determinants}\label{gpd}
The results in the  first three subsections below are not new, see
\cite{BoBrRa11, GoLe09}, but we thought that a short and
self-contained presentation can be useful. 
\setcounter{equation}{0}
\subsection{Gaussian elimination}\label{si}
We review some standard material, see for instance \cite{SjZw07}. Let
${\cal H}_j$, ${\cal G}_j$, $j=1,2$ be complex Hilbert spaces\footnote{All Hilbert spaces in this
  work are assumed to be separable.}. 
Consider a bounded linear operator
\ekv{si.1}{{\cal P}=\begin{pmatrix}P_{11} &P_{12}\\ P_{21}
    &P_{22}\end{pmatrix}:{\cal H}_1\times {\cal H}_2\to {\cal
    G}_1\times {\cal G}_2.}
When ${\cal P}$ is bijective (with bounded inverse) we denote the
inverse by
\ekv{si.1.5}
{
{\cal P}^{-1}={\cal E}=\begin{pmatrix}E_{11} &E_{12}\\ E_{21} &E_{22}\end{pmatrix}.
}
\begin{prop}\label{si1}
\par\noindent 
1) Assume that $P_{11}$ is bijective. Then
by Gaussian elimination we have the standard factorization into lower and upper triangular matrices:
\ekv{si.2}
{
{\cal P}=\begin{pmatrix}P_{11} &0\\ P_{21} &1\end{pmatrix}
\begin{pmatrix}1 &P_{11}^{-1}P_{12}\\ 0
  &P_{22}-P_{21}P_{11}^{-1}P_{12}\end{pmatrix} .
}
The first factor is bijective since $P_{11}$ is, so the bijectivity of ${\cal P}$ is equivalent to that of the second factor, which in turn is equivalent to that of
$P_{22}-P_{21}P_{11}^{-1}P_{12}$. When ${\cal P}$ is bijective,  we have the formula,
\ekv{si.3}
{
{\cal P}^{-1}=\begin{pmatrix}1 &a\\ 0
  &(P_{22}-P_{21}P_{11}^{-1}P_{12})^{-1}\end{pmatrix} \begin{pmatrix}P_{11}^{-1}
  &0\\ b
  &1\end{pmatrix}=:\begin{pmatrix}E_{11}&E_{12}\\E_{21}&E_{22}\end{pmatrix}=:{\cal
  E}, } where
$a=-P_{11}^{-1}P_{12}(P_{22}-P_{21}P_{11}^{-1}P_{12})^{-1}$,
$b=-P_{21}P_{11}^{-1}$ and in particular, 
\ekv{si.4} { E_{22}=(P_{22}-P_{21}P_{11}^{-1}P_{12})^{-1}.  }

\medskip\par\noindent 2) Now assume that ${\cal P}$ is
bijective. Then $P_{11}$ is bijective precisely
when $E_{22}$ is, and when that bijectivity holds we have 
\ekv{si.4.5}{\begin{split}E_{22}^{-1}&=P_{22}-P_{21}P_{11}^{-1}P_{12}\\
P_{11}^{-1}&=E_{11}-E_{12}E_{22}^{-1}E_{21}
\end{split}}
\end{prop}
   
The first statement is clear. The second statement is
more standard and also quite simple to verify.

\subsection{Generalized determinants for holomorphic Fredholm families}\label{mgd}
Let $\Omega \subset {\bf C}$ be open connected, let ${\cal H}_1$, ${\cal H}_2$
be two complex Hilbert spaces and let 
$$
\Omega \ni z\mapsto P(z)\in {\cal L}({\cal H}_1, {\cal H}_2)
$$
be a holomorphic family of Fredholm operators of index $0$, such that
$P(z)$ is bijective for at least one $z\in \Omega $. Then by analytic
Fredholm theory (see for instance the appendix in \cite{HeSj86}) we
know that the set $\sigma (P)\subset \Omega $ where $P(z)$ is not
bijective, is discrete. Let $z_0\in \sigma (P)$. Then we can find $N\in
{\bf N}$ and operators $R_+:{\cal H}_1\to {\bf C}^N$, $R_-:{\bf
  C}^N\to {\cal H}_2$ such that \ekv{mgd.1} { {\cal
    P}(z):=\begin{pmatrix}P(z) &R_-\\ R_+ &0\end{pmatrix} : {\cal
    H}_1\times {\bf C}^N\to {\cal H}_2\times {\bf C}^N } is bijective
for $z\in \mathrm{neigh\,}(z_0,\Omega )$ (i.e. for $z$ in some
neighborhood of $z_0$ in $\Omega $). Let \ekv{mgd.2} { {\cal
    E}(z)=\begin{pmatrix}E(z) &E_+(z)\\ E_-(z)
    &E_{-+}(z)\end{pmatrix}: {\cal H}_2\times {\bf C}^N\to {\cal
    H}_1\times {\bf C}^N } denote the inverse, depending
holomorphically on $z$.

Working in a small neighborhood of $z_0$ disjoint from $\sigma (P)\setminus \{z_0 \}$, we apply the following standard computations and arguments (\cite{MeSj, SjZw07}) where the first formula is already in (\ref{si.4.5}):
$$
P(z)^{-1}=E(z)-E_+(z)E_{-+}(z)^{-1}E_-(z),
$$
$$
P^{-1}\partial _zP=E(z) \partial _zP-E_+(z)E_{-+}(z)^{-1}E_-(z) \partial _zP,
$$
writing $\partial =\partial _z=\partial /\partial z$.
Here the first term to the right is holomorphic and the second term is
of finite rank with a finite pole at $z=z_0$. Let $\gamma$ be the
oriented boundary of the open disc
$D(z_0,\epsilon )$ with center $z_0$ and with radius $\epsilon >0$ small enough. Integrating along $\gamma $, we get
$$
\frac{1}{2\pi i}\int_\gamma P^{-1}\partial _zP dz=-\frac{1}{2\pi i}\int_\gamma E_+E_{-+}^{-1}E_-\partial_z P dz.
$$
The integrand to the right is of trace class, so the  left hand side is of trace class and we get
\ekv{mgd.3}
{
\mathrm{tr\,}\frac{1}{2\pi i}\int_\gamma P^{-1}\partial Pdz=-\frac{1}{2\pi i}\int_\gamma \mathrm{tr\,}E_+E_{-+}^{-1}E_-\partial P dz.
}
The relation ${\cal E}{\cal P}=1$ implies 
\ekv{mgd.4}
{
E_-P+E_{-+}R_+=0,\ E_-R_-=1,
}
and differentiating the relation ${\cal P}{\cal E}=1$ gives
\ekv{mgd.5}
{
(\partial P)E_++P\partial E_++R_-\partial E_{-+}=0. 
}
Combining this with the cyclicity of the trace, we have
\begin{equation*}\begin{split}
    -\mathrm{tr\,}E_+E_{-+}^{-1}E_-\partial P&=-\mathrm{tr\,}E_{-+}^{-1}E_-(\partial P)E_+\\
    &=\mathrm{tr\,}E_{-+}^{-1}E_-P\partial E_+ +\mathrm{tr\,}E_{-+}^{-1}E_-R_-\partial E_{-+}\\
    &=-\mathrm{tr\,}E_{-+}^{-1}E_{-+}R_+\partial E_+ +\mathrm{tr\,}E_{-+}^{-1}\partial E_{-+}\\
    &=-\mathrm{tr\,}R_+\partial E_++\mathrm{tr\,}E^{-1}_{-+}\partial
    E_{-+}.
\end{split}\end{equation*}

The first term in the last expression vanishes since $R_+\partial
E_+=\partial (R_+E_+)=\partial (1)=0$, so (\ref{mgd.3}) becomes
\ekv{mgd.6} {\begin{split}
    &\mathrm{tr\,}\frac{1}{2\pi i}\int_\gamma P(z)^{-1}\partial
    P(z)dz=\frac{1}{2\pi i}\int_\gamma
    \mathrm{tr\,}E_{-+}^{-1}\partial E_{-+}dz\\
    & = \frac{1}{2\pi }\mathrm{var\, arg}_\gamma (\ln\det E_{-+})=m(z_0,\det E_{-+}),
\end{split}
} 
where $m(z_0,\det E_{-+})$ denotes the multiplicity of $z_0$ as a zero of $\det E_{-+}(z)$.
\begin{remark}\label{mgd0}
  {\rm From the cyclicity of the trace in the beginning of the
  calculations we see that $\int_\gamma (\partial _zP)P^{-1}dz$ is of
  trace class and has the same trace as $\int_\gamma P^{-1}\partial
  _zPdz$.}
\end{remark}

A more elegant presentation of the above discussion could be based on
(\ref{si.2}):
$$
{\cal P}=\begin{pmatrix}P(z) &0\\ * &1\end{pmatrix}
\begin{pmatrix}1 &*\\0 &E_{-+}^{-1}\end{pmatrix}=:{\cal A}{\cal B},
$$
which at least formally leads to
\begin{equation}\begin{split}
0=\mathrm{tr\,}\int_\gamma {\cal P}^{-1}\partial {\cal P}dz&=
\mathrm{tr\,}\int_\gamma {\cal A}^{-1}\partial {\cal A}dz+
\mathrm{tr\,}\int_\gamma {\cal B}^{-1}\partial {\cal B}dz\\
&=\mathrm{tr\,}\int_\gamma P^{-1}\partial Pdz-\mathrm{tr\,}\int_\gamma
E_{-+}^{-1}\partial E_{-+}dz.
\end{split}\end{equation}
\begin{dref}\label{mgd1}
By $\det P=\det _\Omega P$ we denote any holomorphic function $f$ on $\Omega $ with $f^{-1}(0)=\sigma (P)$ for which 
\ekv{mgd.7}
{
m(z_0,f)=\mathrm{tr\,}\frac{1}{2\pi i}\int_{\partial D(z_0,r)}P(z)^{-1}\partial P(z)dz,\hbox{ for all }z_0\in \sigma (P).
}
Here $r>0$ is small enough so that $\sigma (P)\cap D(z_0,r)=\{ z_0\}$.
\end{dref}
By Mittag-Leffler's theorem such a holomorphic function exists and it is unique up to a non-vanishing holomorphic factor.
\begin{prop}\label{mgd2}
Let $\Omega \ni z\mapsto Q(z)\in {\cal L}({\cal H}_2,{\cal H}_3)$ have
the same general properties as $P(z)$. Then the determinants of $P$,
$Q$, $QP$ can be defined as above so that
\ekv{mgd.8}
{
\det (Q(z)P(z))=(\det Q(z)) (\det P(z)).
}
\end{prop}
\begin{proof}
We clearly have 
$$
\sigma (QP)=\sigma (Q)\cup \sigma (P)
$$
as sets, and we have to prove that 
\ekv{mgd.9}
{
m(z_0,\det (QP))=m(z_0,\det P)+m(z_0,\det Q),
}
for every $z_0\in \Omega $, where $m(z_0,\det P)$ is defined to be zero when $z\not\in \sigma (P)$ and otherwise as in (\ref{mgd.7}).

Let $z_0\in \sigma (P)\cup \sigma (Q)$ and let $z_0\ne z\in
\mathrm{neigh\,}(z_0)$. We have at $z$, \ekv{mgd.10} {
  (QP)^{-1}\partial (QP)=P^{-1}Q^{-1}(\partial Q)P+P^{-1}\partial P.
} Here the first term to the right needs to be transformed. For each of
the operators $A=P^{-1}$, $B=Q^{-1}(\partial Q)P$ we make a
decomposition $A=A_\mathrm{hol}+A_\mathrm{sing}$ where
$A_\mathrm{hol}$ is holomorphic in a full neighborhood of $z_0$ and
$A_\mathrm{sing}$ has a pole at $z_0$ but is of finite rank and hence
of trace class. Now write \ekv{mgd.11}
{\begin{split}AB-BA=&(A_\mathrm{hol}B_\mathrm{hol}-B_\mathrm{hol}A_\mathrm{hol})
    +(A_\mathrm{hol}B_\mathrm{sing}-B_\mathrm{sing}A_\mathrm{hol})\\
    &+(A_\mathrm{sing}B_\mathrm{hol}-B_\mathrm{hol}A_\mathrm{sing})
    +(A_\mathrm{sing}B_\mathrm{sing}-B_\mathrm{sing}A_\mathrm{sing}).\end{split}}
The first term to the right is holomorphic near $z_0$, while the other
three are of trace class with vanishing trace. Thus if $\gamma
=\partial D(z_0,r)$ with $0<r$ small enough, $\int_\gamma (AB-BA)dz$
is of trace class and with trace $0$.

\par Applying this to the first term to the right in (\ref{mgd.10}),
we see that 
$$
\int_\gamma (P^{-1}Q^{-1}(\partial Q)P-Q^{-1}\partial Q)dz
$$
is of trace class and has trace $0$. It follows that $
(2\pi i)^{-1}\int_\gamma P^{-1}Q^{-1}(\partial Q)Pdz
$ is of trace class and has the same trace as $(2\pi
i)^{-1}\int_\gamma Q^{-1}\partial Q dz$ and we get
$$
\mathrm{tr\,}\frac{1}{2\pi i}\int_\gamma (QP)^{-1}\partial (QP)dz=
\mathrm{tr\,}\frac{1}{2\pi i}\int_\gamma Q^{-1}\partial Qdz+
\mathrm{tr\,}\frac{1}{2\pi i}\int_\gamma P^{-1}\partial Pdz,
$$
which amounts to (\ref{mgd.9}).
\end{proof}

\subsection{Extension to meromorphic families}\label{mf}
In this section we essentially follow \cite{GoLe09}, see also \cite{BoBrRa11}.
Let $\Omega $ be open and connected. Let $\Omega \ni z\mapsto P(z)\in
{\cal L}({\cal H}_1,{\cal H}_2)$ be meromorphic with the poles
$z_1,z_2,...$. Here ${\cal H}_j$ are complex Hilbert spaces.
\begin{dref}\label{mf1}
We say that $P(z)$ is a meromorphic Fredholm function (or Fredholm family) if the following hold:
\begin{itemize}
\item $P(z)$ is Fredholm of index $0$ on $\Omega \setminus \{
  z_1,z_2,..\}$ and bijective for at least one $z$ in that set.
\item Let $z_0$ be any pole and write the Laurent series at $z_0$ as
$$
P(z)=\sum_1^{N_0} (z-z_0)^{-j}P_j +B(z),\ z\in \mathrm{neigh\,}(z_0),
$$
with $B(z)$ holomorphic. Then $P_j$ are of finite rank (implying that
$B(z)$ is Fredholm of index zero for $z\ne z_0$). Moreover, $B(z_0)$
is a Fredholm operator of index 0.
\end{itemize}
\end{dref}
The motivation for introducing this class is that if $Q(z)$ is a
holomorphic family of Fredholm operators on $\Omega $, bijective for
at least one $z\in \Omega $, then $P(z)=Q(z)^{-1}$ is a meromorphic
Fredholm function. 

If $P^j(z)$, $j=1,2$ are meromorphic Fredholm families on $\Omega $,
then $P^1(z)P^2(z)$ is also such a family. In fact, the first property
in the definition is easy to verify and if $z_0$ is a pole for one or
both factors, we write
$$
P^j(z)=\sum_1^{N_j}(z-z_0)^{-k}P_k^j+B^j(z)
$$
and check that 
$$
P^1(z)P^2(z)=\sum_1^{N_1+N_2}(z-z_0)^{-k}P_k+B(z)
$$
where $P_k$ are of finite rank and $B(z_0)=B^1(z_0)B^2(z_0)+K$, where
$K$ is of finite rank. 

We shall show that the class of meromorphic Fredholm functions on
$\Omega $ is closed under inversion and introduce the notion of
meromorphic determinant for such families. The key will be a well
chosen Grushin problem. 

We pause to recollect the condition for the well-posedness of a
Grushin problem
\ekv{mf.1}
{
\begin{cases}Pu+R_-u_-=v,\\
R_+u=v_+,\end{cases}
}
when $P:{\cal H}_1\to {\cal H}_2$ is a fixed Fredholm operator of
index 0 and $R_+:{\cal H}_1\to {\bf C}^N$ and $R_-:{\bf C}^N\to {\cal
  H}_2$ are of rank $N$. Since (\ref{mf.1}) defines an operator 
$$
{\cal P}=\begin{pmatrix}P &R_-\\ R_+ &0\end{pmatrix}:\, {\cal
  H}_1\times {\bf C}^N\to {\cal H}_2\times {\bf C}^N
$$
of index $0$, it is bijective precisely when it is injective, so it
suffices to review when (\ref{mf.1}) is injective. The necessary and
sufficent condition for that is 
\ekv{mf.2}
{
\begin{cases}u\in {\cal N}(R_+)\\ Pu\in {\cal
    R}(R_-)\end{cases}\Rightarrow u=0,
}
where ${\cal N}$ indicates the null space and ${\cal R}$ the range.
Now let $P(z)$ be a meromorphic Fredholm function with a pole at
$z_0$. We look for $R_\pm$ as above (independent of $z$) such that the
problem
\ekv{mf.3}
{\begin{cases}
(\sum_1^{N_0}(z-z_0)^{-j}P_j +B(z))u+R_-u_-=v\\
R_+u=v_+\end{cases}
}
is well-posed for all $z$ in a pointed neighborhood of $z_0$. 

Since the $P_j$ are finitely many operators of finite rank, we can
choose $R_+$ with $N$ large enough, so that 
$$
{{P_j}_\vert}_{{\cal N}(R_+)}=0,\ {\cal N}(R_+)\subset {\cal
  N}(B(z_0))^\perp .
$$
Then $B(z_0)({\cal N}(R_+))$ is a closed subspace of ${\cal H}_2$ of
codimension $N$, and we choose $R_-$ of rank $N$ such that $B(z_0)({\cal
  N}(R_+))\cap {\cal R}(R_-)=0$ i.e. 
\ekv{mf.3.5}{
{\cal H}_2=B(z_0)({\cal N}(R_+))\oplus {\cal R}(R_-).}

Then the problem
$$
\begin{cases}B(z_0)u+R_-u_-=v\\ R_+u=v_+\end{cases}
$$
is well-posed and we check that (\ref{mf.3}) has the same
property. Indeed, $P(z)=B(z)$ on ${\cal N}(R_+)$ and hence this
restriction is injective for $z$ close to $z_0$, and $P(z)({\cal
  N}(R_+))\oplus {\cal R}(R_-)={\cal H}_2$.

\par Let us also analyze the structure of the solution operator to the
problem (\ref{mf.3}). Let $\widetilde{E}_+$ be a right inverse of
$R_+$ so that a general $u\in {\cal H}_1$ has the direct sum
decomposition 
\ekv{mf.4}
{
u=u'+\widetilde{E}_+\widetilde{v}_+,\ u'\in {\cal N}(R_+),\
\widetilde{v}_+\in {\bf C}^N.
}
Then the second equation of (\ref{mf.3}) holds precisely when
$\widetilde{v}_+=v_+$. Let $\Pi '$, $\Pi ''$ be the projections on
the first and second summands in the direct sum decomposition (\ref{mf.3.5})
and write ${\cal H}_2\ni v=\Pi 'v+\Pi ''v=v'+v''$. 

\par Since $P_ju'=0$, the first equation in (\ref{mf.3}) becomes
$$
B(z)u'+R_-u_-=v-\sum_1^{N_0}(z-z_0)^{-j}P_j\widetilde{E}_+v_+ -B(z)\widetilde{E}_+v_+
$$
and we determine $u'$ and $u_-$ by applying $\Pi '$ and $\Pi ''$
respectively, using that ${{\Pi 'B(z)}_\vert}_{{\cal N}(R_+)}={{\Pi
    'B(z_0)}_\vert}_{{\cal N}(R_+)}+{\cal O}(z-z_0)$ is bijective:
${\cal N}(R_+)\to B(z_0)({\cal N}(R_+))$, that $\Pi ''R_-=R_-$ and that
$R_-:{\bf C}^N\to R_-({\bf C}^N)$ is bijective. If $\widetilde{E}_-$
is a left inverse of $R_-$, we get
$$
\Pi 'B(z)u'=v'-\sum_1^{N_0}(z-z_0)^{-j}\Pi 'P_j\widetilde{E}_+v_+
-\Pi 'B(z)\widetilde{E}_+v_+
,$$
$$u'=({{\Pi 'B(z)}_\vert}_{{\cal
    N}(R_+)})^{-1}\left( v'-\sum_1^{N_0}(z-z_0)^{-j}\Pi
'P_j\widetilde{E}_+v_+-\Pi 'B(z)\widetilde{E}_+v_+\right) ,$$
and
\begin{equation*}
\begin{split}
u_-=&\widetilde{E}_-\Pi
''(v-\sum_1^{N_0}(z-z_0)^{-j}P_j\widetilde{E}_+v_+-B(z)\widetilde{E}_+v_+)\\
&-\widetilde{E}_-\Pi ''(B(z)-B(z_0))u'
.
\end{split}
\end{equation*}

\par As usual, we write the solution of (\ref{mf.3}) in the form
\ekv{mf.5}{
  \begin{cases}u=Ev+E_+v_+,\\ u_-=E_-v+E_{-+}v_+,\end{cases} } where
``explicit'' expressions for $E$, $E_\cdot $ can be obtained from the
above computations. We see that \ekv{mf.6} { E(z)=({{\Pi
      'B(z)}_\vert}_{{\cal N}(R_+)})^{-1}\Pi ' } is a holomorphic
family of Fredholm operators of index $0$, while $E_+(z)$, $E_-(z)$,
$E_{-+}(z)$ are meromorphic operator valued functions with singular
terms of finite rank. In particular, $E_{-+}(z)$ is a meromorphic
function with values in the $N\times N$ matrices which is invertible
for $z\ne z_0$, so that $\det E_{-+}$ is meromorphic with a possible
pole at $z_0$, non-vanishing and holomorphic in a pointed
neighborhood of that point. Thus $E_{-+}^{-1}$ is also meromorphic and
we conclude that
$$
P(z)^{-1}=E(z)-E_+(z)E_{-+}(z)^{-1}E_-(z)
$$
is a meromorphic family of Fredholm operators near $z_0$. Thus we get
\begin{prop}\label{mf2}
If $P(z)$ is a meromorphic Fredholm function, then $P(z)^{-1}$ has the
same property.
\end{prop}
We shall next extend the discussion of determinants in Subsection
\ref{mgd}. When $R_\pm$ are independent of $z$ and 
${\cal P}=\begin{pmatrix}P(z) &R_-\\R_+ &0\end{pmatrix}={\cal
  H}_1\times {\bf C}^N\to {\cal H}_2\times {\bf C}^N$ is bijective
with inverse ${\cal E}=\begin{pmatrix}E &E_+\\ E_- &E_{-+}\end{pmatrix}$,
  we notice that 
$$
{\cal P}^{-1}\partial {\cal P}=\begin{pmatrix}E\partial P &0\\
  E_-\partial P &0\end{pmatrix}
$$
In the case of our special problem (\ref{mf.3}), $E(z)$ is given in
(\ref{mf.6}) and the non-holomorphic part of $E\partial P$ is 
$$
({{\Pi 'B(z)}_\vert}_{{\cal N}(R_+)})^{-1}\Pi '\partial _z(\sum_1^{N_0}(z-z_0)^{-j}P_j)
$$
which is of finite rank and with the same trace as 
$$
\Pi '\partial _z(\sum_1^{N_0}(z-z_0)^{-j}P_j)({{\Pi
    'B(z)}_\vert}_{{\cal N}(R_+)})^{-1} .
$$
This operator vanishes, since ${{P_j}_\vert}_{{\cal
    N}(R_+)}=0$. Thus $\int_\gamma {\cal P}^{-1}\partial {\cal P}dz$
and $\int_\gamma E\partial P dz$ are of trace class and have the trace
$0$ if $\gamma =D(z_0,r)$ for $0<r\ll 1$.
 \par As in and around (\ref{mgd.3}) we now get 
\begin{equation*}\begin{split}
\mathrm{tr\,}\frac{1}{2\pi i}\int_\gamma P^{-1}\partial P dz
&= -\mathrm{tr\,}\frac{1}{2\pi i}\int_\gamma
E_+E^{-1}_{-+}E_-\partial P dz\\
&=\mathrm{tr\,}\frac{1}{2\pi i}\int_\gamma E_{-+}^{-1}\partial E_{-+}dz,
\end{split}
\end{equation*}
leading to
\ekv{mf.7}
{
\mathrm{tr\,}\frac{1}{2\pi i}\int_\gamma P^{-1}\partial P dz=
m(z_0,\det E_{-+}),
}
where the integer $m(z_0,\det E_{-+})$ is the order of $z_0$ as a zero of $\det
E_{-+}$ when the latter function is holomorphic near $z_0$ and when
$\det E_{-+}$ has a pole at $z_0$, then
$-m(z_0,\det E_{-+})$ is the order of that pole. 

\par Note for future reference that 
\ekv{mf.8}
{
P^{-1}\partial P=a+b,
}
where $a$ is holomorphic near $z_0$ and $b$ is of finite rank and
\ekv{mf.9}
{
\mathrm{tr\,}b=\mathrm{tr\,}(E_{-+}^{-1}\partial E_{-+}).}

We emphasize that in view of (\ref{mf.7}), $(2\pi
i)^{-1}\mathrm{tr\,}\int_\gamma P^{-1}\partial P dz$ is an integer,
and we can then give the following extension to meromorphic families
of the notion of determinant:
\begin{dref}\label{mf3}
  Let $\Omega \ni z\mapsto P(z)\in {\cal L}({\cal H}_1,{\cal H}_2)$ be
  a meromorphic Fredholm function with the poles $z_1,z_2,...$. By
  $\det P=\det _\Omega P$ we denote any meromorphic function $f(z)$
  whose restriction to $\Omega \setminus \{ z_1,z_2,...\}$ is a
  determinant for $P$ in the sense of Definition \ref{mgd1}, and
  such that for every pole $z_j$ of $P$, we have
$$
\mathrm{tr\,}\frac{1}{2\pi i}\int_{\partial D(z_j,r)}P(z)^{-1}\partial P(z)dz=m(z_j,f)
$$
when $r>0$ is small enough.
\end{dref}
Observe that Proposition \ref{mgd2} and its proof extend to the
case of meromorphic Fredholm functions.

\subsection{Determinants via traces}\label{gd}
If ${\cal H}$ is a complex Hilbert space and
$P=P(z)\in {\cal L}({\cal H},{\cal H})$ is a trace class perturbation
of the identity, depending holomorphically on the complex parameter
$z$, we can define $D(z)=\ln \det P(z)$ and we have \ekv{gd.1} {
  \frac{d}{dz}D(z)=\mathrm{tr\,}P(z)^{-1}\frac{dP}{dz}, } at the
points where $P$ is bijective. Now even when $P$ is not a trace class
perturbation of the identity, it may happen that $P^{-1}\frac{dP}{dz}$
is of trace class, and we can now consider the case when $P(z)\in
{\cal L}({\cal H}_1, {\cal H}_2)$ for different complex Hilbert
spaces ${\cal H}_1$, ${\cal H}_2$. By integration of (\ref{gd.1}), we
may then say that $D(z)$ is well-defined up to a constant as a
possibly multivalued function on every connected component of the open
set where $P(z)$ is invertible. If $P^{-1}\frac{dP}{dz}$ is not of
trace class we may differentiate further and hope to reach an
expression which is of trace class. Then we would be able to define
$D(z)$ up to a polynomial. In this section we carry out such a scheme.
The idea of reaching trace class operators by means of differentiation
in connection with determinants has been used by G.~Carron \cite{Ca02}.

\par Let $\Omega \subset {\bf C}$ be open and connected, let ${\cal
  H}_j$, $j=1,2,3$ be complex Hilbert spaces. Let $\Sigma =\Sigma
(P)\subset \Omega $ be discrete and let $\Omega \setminus \Sigma \ni
z\mapsto P(z)\in {\cal L}( {\cal H}_1,{\cal H}_2)$ be holomorphic and
pointwise bijective. Let $C_p=C_p({\cal H}_1,{\cal H}_2)$ denote the
Schatten class of index $p\in [1,+\infty ]$ (see for instance
\cite{GoKr69}). Assume that for some $p\in [1,+\infty [$, \ekv{gd.2} {
\partial _z^kP(z)\in C_{\max (1,p/k)},\ 1\le k\in {\bf N},
}
locally uniformly on $\Omega $. By the Cauchy inequalities, it
suffices to check this for $k\le N$, where $N=N(p)$ is the smallest
integer $\ge p$. 

\par Recall that $C_p$ increases with $p$ and that if $C\in C_p({\cal
  H}_1,{\cal H}_2)$ and $D\in C_q({\cal H}_2,{\cal H}_3)$, then $DC\in
C_r({\cal H}_1,{\cal H}_3)$ with $\frac{1}{r}=\min
(1,\frac{1}{p}+\frac{1}{q})$. (See \cite{GoKr69}, Proposition 7.2.) In the following, we shall think of
bounded operators as being of order $= 0$ and of elements in $C_p$
as being of order $= -1/p$. In all cases we  restrict here the order to
the interval $[-1,0]$ and then orders are additive under composition:
$\mathrm{ord\,}(DC)= \max (-1,\mathrm{ord\,}(D)+\mathrm{ord\,}(C))
$. (We adopt the convention that the order is not unique; if $C$ is of
order $\alpha $ and $\alpha \le \beta \le 0$, then $C$ is also of
order $\beta $.)

\par We also notice that $P(z)^{-1}$ satisfies (\ref{gd.2}).

\par On the set $\Omega \setminus \Sigma (P)$, we check that 
\ekv{gd.3}{
\partial _z^{j-1}(P(z)^{-1}\partial _zP(z))\in
C_{\max (1,\frac{p}{j})},\ j\ge 1,
}
i.e. of order $= \max (-1,-\frac{j}{p})$. Thus, for $p\le j\in {\bf
  N}$, we can define
\ekv{gd.3.5}
{
D_{P,j}(z)=\mathrm{tr\,}(\partial _z^{j-1}(P(z)^{-1}\partial
_zP(z))),\ z\in \Omega \setminus \Sigma (P).
}
Clearly,
$$
\partial _zD_{P,j}(z)=D_{P,j+1}(z).
$$
We can now define the determinant of $P(z)$. At the end of the section
we show that this new notion coincides with the one for meromorphic
families of Fredholm operators of the preceding subsection.
\begin{dref}\label{gd0}
Let $N=N(p)$ be the smallest integer $\ge p$. We define
$D_P(z)=\ln\det P(z)$ to be any multivalued holomorphic function on $\Omega
\setminus \Sigma (P)$ which solves the equation
\ekv{gd.4}
{
\partial _z^ND_P(z)=\mathrm{tr\,}(\partial _z^{N-1}(P(z)^{-1}\partial _zP(z))).
}
Thus $D_P(z)$ is well defined (on the universal covering space of
$\Omega \setminus \Sigma (P)$) up to a polynomial of degree $N-1$.
\end{dref}

Let $\Omega \ni z\mapsto Q(z)\in {\cal L}({\cal H}_2,{\cal H}_3)$ be a
second family with the same general properties as $P(z)$ and for
simplicity with the same $p$ in (the analogue of) (\ref{gd.2}). Then
$Q(z)P(z)$ fulfills the same assumptions and
we next check the additivity property
\ekv{gd.4}
{
\ln \det PQ=\ln\det P +\ln\det Q,\hbox{ on }\Omega \setminus (\Sigma
(P)+\Sigma (Q)),
}
i.e.
\ekv{gd.5}{\left(\frac{d}{dz}\right)^N\ln \det PQ=
\left(\frac{d}{dz}\right)^N\ln\det P +
\left(\frac{d}{dz}\right)^N\ln\det Q,}
when $N$ is the smallest integer $\ge p$.

When $p=1=N$, this is straight forward:
\ekv{gd.6}
{
\begin{split}
\frac{d}{dz}\ln\det PQ&=\mathrm{tr\,}(PQ)^{-1}\frac{d}{dz}(PQ)\\
&=\mathrm{tr\,}Q^{-1}P^{-1}\frac{dP}{dz}Q+\mathrm{tr\,}Q^{-1}P^{-1}P\frac{dQ}{dz}\\&=\mathrm{tr\,}Q^{-1}P^{-1}\frac{dP}{dz}Q+\mathrm{tr\,}Q^{-1}\frac{dQ}{dz}.
\end{split}
}
Here we use the cyclicity of the trace to see that the first term in
the last expression is
equal to $\mathrm{tr\,}P^{-1}\frac{dP}{dz}$ and we thus get
(\ref{gd.5}) when $N=1$.

Recall that the cyclicity of the trace says that
$\mathrm{tr\,}(P_1P_2-P_2P_1)=0$, when $P_1\in C_{p_1}({\cal H}_1,{\cal
  H}_2)$, $P_2\in C_{p_2}({\cal H}_2,{\cal H}_1)$ and $1=1/p_1+1/p_2$.
\begin{lemma}\label{gd1} Let $P_1(z)\in {\cal L}({\cal H}_1,{\cal
  H}_2)$ and $P_2(z)\in {\cal L}({\cal H}_2,{\cal H}_1)$ depend
holomorphically on $z\in \Omega $. Then
$\frac{d}{dz}(P_1P_2-P_2P_1)$ is a sum of terms of the form
$Q_1Q_2-Q_2Q_1$.
More precisely, 
$$
(P_1P_2-P_2P_1)'
=[P_1'P_2-P_2P_1']+[P_1P_2'-P_2'P_1],
$$
where we indicate derivatives
with a prime.
\end{lemma}

Iterating the lemma we see that
$\left(\frac{d}{dz}\right)^N(P_1P_2-P_2P_1)$ is a linear combination
of terms of the form $Q_1Q_2-Q_2Q_1$, with $Q_j=\partial _z^{N_j}P_j$, $N_1+N_2=N$.

\par Now return to (\ref{gd.6}), or rather the last two equations
there that are valid without traces, and write
$$
Q^{-1}P^{-1}\frac{dP}{dz}Q=P^{-1}\frac{dP}{dz}+(P_1P_2-P_2P_1),
$$
with $P_1=Q^{-1}P^{-1}\frac{dP}{dz}$, $P_2=Q$. The lemma shows that 
\begin{multline*}
\left(\frac{d}{dz}\right)^{N-1}\left( Q^{-1}P^{-1}\frac{dP}{dz}Q\right)=
\left(\frac{d}{dz}\right)^{N-1}\left( P^{-1}\frac{dP}{dz}\right) +\\
\hbox{a linear combination of terms of the form }Q_1Q_2-Q_2Q_1\\
\hbox{with }
\mathrm{ord\,}(Q_j)\le \max \left( -1,-\frac{N_j}{p}\right) ,\ N_1+N_2=N.
\end{multline*}
The cyclicity of the trace then implies that 
$$
\mathrm{tr\,}\left(\frac{d}{dz}\right)^{N-1}\left(
  Q^{-1}P^{-1}\frac{dP}{dz}Q\right) =
\mathrm{tr\,}\left(\frac{d}{dz}\right)^{N-1}\left( P^{-1}\frac{dP}{dz}\right)
$$
and we obtain (\ref{gd.5}) for a general $N$.

\par As in the case of meromorphic families of Fredholm operators, if
$z_0\in \Sigma (P)$ and $\gamma =\partial D(z_0,r)$ with $r>0$ small
enough, $\int_\gamma P^{-1}\partial P dz$ is of trace class:
\begin{prop}\label{gd2}
With $P$, $p$, $N=N(p)$ as in Definition \ref{gd0}, let $z_0\in \Sigma
(P)$, $\gamma =\partial D(z_0,r)$ with $r>0$ small enough, so that
$D(z_0,r)\cap \Sigma (P)=\{ z_0\}$. Then $\int_\gamma P^{-1}\partial P
dz$ is of trace class and we have
\ekv{gd.7}{\begin{split}
\mathrm{tr\,}\frac{1}{2\pi i}\int_\gamma P^{-1}\partial P dz&=\mathrm{tr\,}\frac{1}{2\pi i}\int_\gamma 
\frac{(-z)^{N-1}}{(N-1)!}\partial ^{N-1}(P^{-1}\partial P)dz\\&=\frac{1}{2\pi i}\int_\gamma 
\frac{(-z)^{N-1}}{(N-1)!}D_{P,N}(z)dz,
\end{split}}
where $z^{N-1}/(N-1)!$ can be replaced by any other polynomial $p(z)$
such that $\partial ^{N-1}p(z)=1$
\end{prop}
\begin{proof}
  The second equality follows by moving the trace inside the integral
  and recalling the definition of $D_{P,N}$. The first equality and
  the fact that $\int_\gamma P^{-1}\partial P dz$ is of trace class,
  follows from the corresponding stronger equality without ``tr'' in
  front which can be obtained by integration by parts.
\end{proof}

Now, assume in addition that $\Omega $ is simply connected and that
$P$ is a meromorphic Fredholm function on $\Omega$ in the sense of
Definition \ref{mf1}. Then we know that 
\ekv{gd.8}
{
\mathrm{tr\,}\frac{1}{2\pi i}\int_\gamma P^{-1}\partial
Pdz=m(z_0,f)\in {\bf Z},
}
where $f$ denotes the meromorphic Fredholm determinant of Definition
\ref{mf3}. On the other hand, we can do integrations by
parts in the last expression in (\ref{gd.7}) and obtain
\ekv{gd.9}
{
\mathrm{tr\,}\frac{1}{2\pi i}\int_\gamma P^{-1}\partial
Pdz=\frac{1}{2\pi i}\int_\gamma \partial _zD_P(z)dz,
}
which, combined with (\ref{gd.8}), says that 
\ekv{gd.10}
{
\mathrm{var}_\gamma D_P=2\pi im(z_0,f)\in 2\pi i{\bf Z}
}
and hence $e^{D_P}$ and its logarithmic derivative $\partial D_P$ are single-valued holomorphic
functions on $\Omega \setminus \Sigma $. 

So far, this only shows that 
$$D_P=\sum_1^{\infty }(z-z_0)^{-j}a_j+m(z_0,f)\ln (z-z_0)+g(z),$$
where $g$ is holomorphic, so $e^{D_P}=e^{g+\sum
a_j(z-z_0)^{-j}}(z-z_0)^{m(z_0,f)}$ may have a bad singularity at
$z_0$. We therefore return to the Grushin problem in
Subsection \ref{mf}. The remark (\ref{mf.8}), (\ref{mf.9}) shows that 
$$
\mathrm{tr\,}\partial ^{N-1}P^{-1}\partial P=\mathrm{tr\,}(\partial
^{N-1}a)+\partial ^{N-1}\mathrm{tr\,}(E_{-+}^{-1}\partial E_{-+}),
$$
where $E_{-+}$ is a meromorphic finite matrix and
$\mathrm{tr\,}(\partial ^{N-1}a)$ is holomorphic in a full
neighborhood of $z_0$. Consequently,
$$
\partial D_P=\mathrm{tr\,}(E_{-+}^{-1}\partial
E_{-+})+\hbox{holomorphic}=\partial (\ln E_{-+})+\hbox{holomorphic},
$$ 
which rules out the bad singularity and we see that
$e^{D_P}=e^{g}(z-z_0)^{m(z_0,f)}$ near $z_0$. Globally
$e^{D_P(z)}$ is indeed a determinant in the sense of Definition \ref{mf3}.

\begin{prop}\label{gd3}
  Let $P(z)$ be a holomorphic family on $\Omega \setminus \Sigma $ as
  in the beginning of this section and assume in addition that $\Omega
  $ is simply connected and that $P$ is a meromorphic Fredholm
  function on $\Omega $. Then the determinants $\det P(z)$ in the sense
  of Definition \ref{gd0} and in the sense of Definition \ref{mf3}
  coincide up to a non-vanishing holomorphic factor.
\end{prop}

The following complement will be used in Section \ref{ub}.

\paragraph{Addendum.} Consider a Schatten class
perturbation of the identity, $Q(z)=1-K(z)$, where $K(z)\in C_p$ is
holomorphic in some domain in ${\bf C}$ and as in (\ref{gd.2}):
\ekv{ub.2new}
{\partial _z^kK(z)\in C_{\max (1,p/k)},\ 1\le k\in {\bf N}.}
This assumption remains valid if we replace $p$ by $N=[p]$, the
smallest integer $\ge p$ and then (in view of the mean value property
for holomorphic functions) takes the simpler form
\ekv{ub.3new}
{\partial _z^kK(z)\in C_{N/k},\ 1\le k\le N,}
\ekv{ub.3new.5}{K(z)\in C_N.}
Considering the Taylor expansions (and mimicking the definition of
modified determinants for Schatten class perturbations of the
identity), we get 
\ekv{ub.3New}{\begin{split}
&Q(z)=A(z)B(z),\\ & A(z)=\exp F(z),\ F(z)=K(z)+...+\frac{K(z)^{N-1}}{N-1},\\
&B(z)=(1+R_N(K)K^N),\end{split}}
where $\| R_N(K)\|\le C(\| K\| )$. Thus 
$$
\Vert R_N(K)K^N\Vert_{C_1}\le C(\| K\| )\Vert K\Vert^N_{C_N},
$$
so $\det B(z)$ can be defined as in Subsection \ref{gd}. The
definition coincides with that of determinants of trace class
perturbations of the identity and we get 
\ekv{ub.8}{
|\det B(z)|\le \exp (C(\Vert K\Vert)\Vert K\Vert^N_{C_N}). 
}

\par As for $A(z)=\exp F(z)$, we see that $F(z)$ satisfies
(\ref{ub.3new}), (\ref{ub.3new.5}). Moreover from applying $\partial
_z$ to the differential
equation $\partial _t\exp (tF(z))=F(z)\exp (tF(z))$, we have
$$
\partial _z(e^F)=\int _0^1 e^{(1-t)F(z)}(\partial _zF(z))e^{tF(z)}dt\in C_N
$$ and from similar expressions for $\partial _z^k(e^F)$ we see that
$A=e^F$ satisfies (\ref{ub.3new}), (\ref{ub.3new.5}). Now,
\[e^{-F}\partial _ze^F=\int_0^1e^{-tF}(\partial _zF)e^{tF}dt=\partial
_zF+\int_0^1[e^{-tF},(\partial _zF)e^{tF}]dt, \]
so 
$$
\mathrm{tr\,}\partial _z^{N-1}(e^{-F}\partial
_ze^F)=\mathrm{tr\,}\partial _z^NF,
$$
which is bounded in modulus by
\ekv{ub.9}{\begin{split}&{\cal O}(1)\sum_{N_1+..+N_q=N\atop N_q\ge 0,\, q\le N-1}\Vert \partial
^{N_1}K...\partial ^{N_q}K\Vert_{C_1}\le\\ &{\cal O}(1)\sum
_{N_1+..+N_q=N\atop N_q\ge 0,\, q\le N-1} \Vert \partial ^{N_1}K\Vert_{C_{N/N_1}}...\Vert \partial ^{N_q}K\Vert_{C_{N/N_q}}.\end{split}} 

Combining this with (\ref{ub.3New}), (\ref{ub.8}), we get:
\begin{prop}\label{ub1}
Under the above assumptions,
$$
\det Q(z)=\mathrm{I}(z)\mathrm{II}(z),\quad \mathrm{I}(z)=\det A(z),\ \mathrm{II}(z)=\det B(z), 
$$ where $|\mathrm{II}(z)|$ is bounded by the right hand side of
(\ref{ub.8}) and $|\partial _z^N \ln \mathrm{I}(z)|$ is bounded by the
expression (\ref{ub.9}). 
\end{prop}

\section{Complex dilations}\label{cds}
\setcounter{equation}{0}

\subsection{Complex dilations and symmetry}

We start by reviewing some easy facts for complex distortions (see
\cite{SjZw1,SjZw2,SjZw3,SjZw4,SjZw5}) and we shall pay a special
attention to symmetry with respect to the natural {\it bilinear}
form. Let $\Gamma \subset {\bf C}^n$ be a maximally totally real
(m.t.r.) simply connected smooth sub-manifold and let $P=\sum_{|\alpha
  |\le m}a_\alpha D^\alpha $, where $a_\alpha \in C^\infty (\Gamma
)$. If $u\in C^\infty (\Gamma )$, we put
$Pu={{(\widetilde{P}\widetilde{u})}_\vert}_{\Gamma }$, where
$\widetilde{P}=\sum \widetilde{a}_\alpha D^\alpha $ and
$\widetilde{a}_\alpha $, $\widetilde{u}$ are almost holomorphic
extensions of $a_\alpha $, $u$ to a neighborhood of $\Gamma $.

If $P^\mathrm{t}=\sum (-D)^\alpha \circ a_\alpha $ is the formal
transpose of $P$, we can define $P^\mathrm{t}u\in C^\infty (\Gamma )$
for $u\in C^\infty (\Gamma )$ as above and if we define 
\ekv{cds.1}{\langle u|v\rangle_\Gamma =\int_\Gamma u(x)v(x)dx_1\wedge
  ...\wedge dx_n=\int_\Gamma u(x)v(x)dx,\ u,v\in C_0^\infty (\Gamma ),}
we get from Stokes' formula that 
$$
\langle Pu|v\rangle_\Gamma =\langle u|P^\mathrm{t}v\rangle_\Gamma .
$$

Now, let $\widehat{\Gamma }\subset {\bf C}^n$ be a second maximally
totally real smooth manifold and let $\gamma :\widehat{\Gamma }\to
\Gamma $ be a smooth diffeomorphism. (For instance, $\widehat{\Gamma
}$ can be an open subset of ${\bf R}^n$ and $\gamma $ a
``parametrization'' of $\Gamma $.) We can then define
\ekv{cds.2}
{
\frac{\partial \gamma }{\partial y}=\left(\frac{\partial \widetilde{\gamma
  }_j}{\partial y_k}\right),
} 
where $\widetilde{\gamma }(y)=(\widetilde{\gamma
}_1(y),...,\widetilde{\gamma }_n(y))$ is an almost holomorphic
extension of $\gamma =(\gamma _1,...,\gamma _n)$. Let $f\in C^\infty
(\widehat{\Gamma })$ and define $U:\, C^\infty (\Gamma )\to C^\infty
(\widehat{\Gamma })$ by
\ekv{cds.3}
{
Uu(y)=f(y)u(\gamma (y)),\quad u\in C_0^\infty (\Gamma ).
}
If $u,v\in C_0^\infty (\Gamma )$, we get 
$$
\langle Uu|Uv\rangle_{\widehat{\Gamma }}=\int_{\widehat{\Gamma }}u(\gamma
(y))v(\gamma (y))f(y)^2dy,
$$
$$
\langle u|v\rangle_\Gamma =\int_\Gamma u(x)v(x)dx=
\int_{\widehat{\Gamma }}u(\gamma
(y))v(\gamma (y))\det \left( \frac{\partial \gamma }{\partial y}\right)dy.
$$

\par Choose  $f=(\det \frac{\partial \gamma }{\partial y})^{1/2}$ for some
fixed continuous branch of the square root (assuming for simplicity
that $\widehat{\Gamma }$ is simply connected). Then
\ekv{cds.4}
{
\langle Uu|Uv\rangle_{\widehat{\Gamma }}=\langle u|v\rangle_{\Gamma },
} 
so $U$ is orthogonal,
\ekv{cds.5}
{
U^\mathrm{t}=U^{-1}.
}

As usual, this imples that the operations of conjugation with $U$ and
transposition commute: If $P$ is as above and we define the pull-back 
$\widehat{P}=U\circ P\circ U^{-1}=U\circ P\circ U^\mathrm{t}$, then 
\ekv{cds.6}
{
\widehat{P}^\mathrm{t}=UP^\mathrm{t}U^\mathrm{t}.
}

Let now $\widehat{\Gamma }\subset {\bf R}^n$.
We can use $U$ to define an $L^2$-inner product on $C_0^\infty
(\Gamma )$ by putting
\ekv{cds.7}
{
(u|v)=(u|v)_\Gamma =(Uu|Uv)_{L^2(\widehat{\Gamma })}, 
}
which is the inner product that makes $U$ formally unitary. More
explicitly,
\ekv{cds.8}
{
(u|v)=\int_{\widehat{\Gamma }}u(\gamma (y))\overline{v(\gamma
  (y))}|\det \frac{\partial \gamma }{\partial y}|dy=\int_\Gamma
u(x)\overline{v(x)}\theta (x)dx,
}
where
$$
\theta (x)=\frac{|\det \frac{\partial \gamma }{\partial y}|}{\det
  \frac{\partial \gamma }{\partial y}},\quad x=\gamma (y),
$$
is the unique unimodular factor for which $\theta (x)dx$ is a positive
density on $\Gamma $ (and in particular independent of the
parametrization $\gamma $).

\par We have
\ekv{cds.9}
{
(u|v)=\langle u|Cv\rangle_\Gamma ,\quad u,v\in C_0^\infty (\Gamma ),
}
where $C$ is the antilinear involution defined by $Cv=\theta
\overline{v}$. The formal adjoint of $P$ for our scalar product on
$\Gamma $ is given by 
\ekv{cds.10}{P^*=C^{-1}P^\mathrm{t}C=CP^\mathrm{t}C.}

\subsection{Dilations and convex sets}\label{dcs}

Let \ekv{dcs.1}{P=-h^2\Delta +V(x),\quad V\in L^\infty
  _\mathrm{comp}({\bf R}^n;{\bf R}).}  Let first $f:{\bf R}^n\to {\bf
  R}$ be smooth, $=0$ near $\mathrm{supp\,}V$ and equal to $(\tan
\theta )\frac{d_0(x)^2}{2}$ for large $x$, where $d_0(x)=|x|$ and
$0<\theta <\pi /2$. Then we
consider the m.t.r. manifold $\Gamma =\Gamma _f$ of ${\bf C}^n$, given
by \ekv{dcs.2}{x=y+if'(y),\ y\in {\bf R}^n.}
(See \cite{Sj01} for a quick review in the semi-classical case.)
The bijectivity of the complex Jacobian map $\frac{\partial
  x}{\partial y}=1+if''(y)$ implies indeed that $\Gamma _f$ is
maximally totally real. $P_\Gamma $ can be computed in the
parametrization (\ref{dcs.2}) using the formal chain rule:
$$
\frac{\partial }{\partial y}=(1+if''(y))\left(\frac{\partial }{\partial
  x}\right),\ \frac{\partial }{\partial x}=(1+if''(y))^{-1}\left(\frac{\partial }{\partial y}\right),
$$
and hence away from the support of $V$ we get 
\ekv{dcs.3}
{
P_\Gamma =-h^2\det (1+if''(y))^{-1}\left(\frac{\partial }{\partial
    y}\right)^\mathrm{t}\det (1+if''(y))(1+if''(y))^{-2}\left(\frac{\partial
  }{\partial y}\right)
}
which has the semi-classical principal symbol 
\ekv{dcs.4}{((1+if''(y))^{-1}\eta )^2=\langle (1+if''(y))^{-2}\eta
  ,\eta \rangle .}
Here $\langle ,\rangle$ denotes the bilinear scalar product on ${\bf
  R}^n$ and also its bilinear extension to ${\bf C}^n$. Since $\eta $
is real in (\ref{dcs.4}), we can write this symbol as
$$
( (1+if''(y))^{-2}\eta
  |\eta ),
$$
where $(\cdot |\cdot \cdot )$ is the usual sesquilinear scalar product
on ${\bf C}^n$.

For large $y$, we have $f''(y)=(\tan \theta )1$ and here it is
convenient to use the equivalent parametrization $x=e^{i\theta
}\widetilde{y}$, where $\widetilde{y},y\in {\bf R}^n$ are related by
$y=(\cos \theta )\widetilde{y}$, and get
\ekv{dcs.5}{P_\Gamma =e^{-2i\theta }(-h^2\Delta _{\widetilde{y}}).}

\par In general we assume
\ekv{dcs.6}
{
f''(y)\ge 0,
}
and we shall study the inverse of $(1+if''(y))^2=1-f''(y)^2+2if''(y).$
If $C$ is a complex $n\times n$ matrix, define as usual 
$$\Re C=\frac{1}{2}(C+C^*),\ \Im C=\frac{1}{2i}(C-C^*).$$
\begin{prop}\label{cds1}
If $C=(1+if''(y))^{2}$ for some fixed $y\in {\bf R}^n$, then under
the assumption (\ref{dcs.6}), we have 
\begin{itemize}
\item[1)] $\Im C^{-1}\le 0$.
\item[2)] We have $\Im C^{-1}< 0$ (i.e. $C^{-1}$ is negative definite) iff $f''(y)>0$. 
\item[3)] The symbol $(C^{-1}\eta |\eta )$, $\eta \in {\bf R}^n$ is
  elliptic: $|(C^{-1}\eta |\eta )|\asymp |\eta |^2$ and takes its
  values in a sector $-\pi +\epsilon \le \mathrm{arg\,}(C\eta |\eta
  )\le 0$ for some $\epsilon >0$.
\item[4)] When $f''(y)>0$ it take its values in a sector $-\pi +\epsilon \le \mathrm{arg\,}(C\eta |\eta
  )\le -\epsilon $.
\end{itemize}
\end{prop}
\begin{proof}
We already know that $C:{\bf C}^n\to {\bf C}^n$ is bijective and a
direct calculation shows that 
\ekv{dcs.7}{\Im C^{-1}=-{C^*}^{-1}(\Im C)C^{-1}=-2{C^*}^{-1}f''(y)C^{-1},}
\ekv{dcs.8}{\Re C^{-1}={C^*}^{-1}(\Re
  C)C^{-1}={C^*}^{-1}(1-f''(y)^2)C^{-1}.}
1) and 2) follow from (\ref{dcs.7}).

\par Now look at 
\ekv{dcs.9}{
(C^{-1}\eta |\eta )=((\Re C)C^{-1}\eta |C^{-1}\eta )-i((\Im C)C^{-1}\eta |C^{-1}\eta ).}
If the imaginary part of this expression (i.e. the last term) is zero,
then since $\Im C\ge 0$, we conclude that $(\Im C)(C^{-1}\eta )=0$,
i.e. $f''(y)C^{-1}\eta =0$. For such an $\eta $ the real part of
(\ref{dcs.9}) becomes $((\Re C)C^{-1}\eta |C^{-1}\eta
)=((1-f''(y)^2)C^{-1}\eta |C^{-1}\eta )=\Vert C^{-1}\eta \Vert^2$. 3)
and 4) follow. \end{proof}

\par The proposition shows that $P_\Gamma $ is elliptic in the
classical sense. Defining the Sobolev spaces $H^s(\Gamma )$ in the
usual way and equipping $P_\Gamma $ with the domain $H^2(\Gamma )$, we
see that the essential spectrum of $P_\Gamma $ is the
half-line $e^{-2i\theta }[0,+\infty [$. As explained for instance in
\cite{SjZw1,SjZw2,SjZw3,SjZw4,SjZw5}, $P_\Gamma $ has no spectrum in the open upper half-plane
and the eigenvalues in the sector $e^{-i[0,\theta [}]0,+\infty [$ are
precisely the resonances of $P$ there. 
(For a more complete discussion and
further references, see \cite{SjZw1,SjZw2,SjZw3,SjZw4,SjZw5}.)

Let ${\cal O}\Subset {\bf R}^n$ be open with smooth boundary and
strictly convex. Then $d(x):=\mathrm{dist\,}(x,{\cal O})$ is smooth on
${\bf R}^n\setminus {\cal O}$ and we have 
\ekv{dcs.10}{\partial ^\alpha (d-d_0)={\cal O}(\langle x\rangle^{-|\alpha
    |}).}
Now assume that 
\ekv{dcs.11}{\mathrm{supp\,}V\subset \overline{{\cal O}}.}

\par Outside ${\cal O}$ we look for $f$ of the form
\ekv{dcs.12}
{
f(x)=g(d(x)),
}
where $g\in C^\infty ({\bf R};{\bf R})$ vanishes on the negative
half-axis. Then
\ekv{dcs.13}
{
f'(x)=g'(d(x))d'(x),\ f''(x)=g'(d(x))d''(x)+g''(d(x))d'(x)\otimes d'(x).
}
Here $d'(x)$ can be identified with the exterior normal
$\nu (\pi (x))$ at the projection $\pi (x)\in \partial {\cal O}$ of
$x$. When $x\notin \partial {\cal O}$ we also have $d'(x)=(x-\pi
(x))/|x-\pi (x)|$. It is further wellknown that $d''(x)$ is positive
semi-definite with null-space ${\bf R}d'(x)$. Thus we see from
(\ref{dcs.13}) that $f''(x)\ge 0$ when $g',\,g''\ge 0$ and we have
$f''(x)> 0$ when $g',\,g''> 0$. 

\par Introduce geodesic coordinates: 
Let $\Omega \ni z'\mapsto x'(z')\in \partial {\cal O}$ be a
local parametrization of the boundary, where $\Omega $ is some open
set in ${\bf R}^{n-1}$. Then we have local (geodesic) coordinates $(z',z_n)\in
\Omega \times ]-\epsilon ,+\infty [$ on ${\bf R}^n$, given by  
\ekv{dcs.14}
{
x=x(z')+z_n\nu (x(z')).
}
In these coordinates, if $f$ is as in (\ref{dcs.12}), then $\Gamma
=\Gamma _f$ is obtained by letting $z_n$ become complex:
\ekv{dcs.15}
{
z'=y',\ z_n=\gamma (y_n),\quad \gamma (y_n):=y_n+ig'(y_n).
}

We have (see \cite{SjZw3}, Section 2, also \cite{SjZw4}, Section 3 and
\cite{SjZw2}):
\ekv{dcs.16}
{
P=D_{z_n}^2+R(z,D_{z'})+a(z)\partial _{z_n},
} 
where \ekv{dcs.17}
{
R(z,D_{z'})=R(z',0,D_{z'})-z_nQ(z,D_{z'}),
}
and $R$, $Q$ are elliptic second order differential operators with
positive principal symbols:
\ekv{dcs.18}
{
r(z,\zeta '),\, q(z,\zeta ')\, >0.
}
The coefficients are analytic in $z_n$ and smooth in $z$. In the
parametrization (\ref{dcs.15}) for $\Gamma $, we get 
\ekv{dcs.19}
{\begin{split}
P_\Gamma =& (\frac{1}{\gamma '(y_n)}D_{y_n})^2+R(y',0;D_{y'})\\ &-\gamma
(y_n)Q(y',\gamma (y_n);D_{y'})+a(y',\gamma (y'))\frac{1}{\gamma
  '(y_n)}\partial _{y_n}.\end{split}
}
This formula remains valid if we make a real change of variables in
$y_n$ in order to normalize $\gamma '(y_n)$.

If we choose $g$ so that $g(d)=(\tan \theta )d^2$ for large $d\ge
r_0>0$, then as we have seen, $f''>0$ in the  corresponding
region. Let $\chi \in C_0^\infty ({\bf R}^n;[0,1])$ be equal to one in
a neighborhood of $0$ and put $\widetilde{d}=\widetilde{d}_R=\chi
(x/R)d(x)+(1-\chi (x/R))d_0(x)$. Then we still have (\ref{dcs.10}) if
we replace $d$ or $d_0$ with $\widetilde{d}$ and from this it follows
that $\widetilde{f}:=(\tan \theta )\widetilde{d}^2$ satisfies $
\widetilde{f}''(x)>0$ for $d(x)\ge r_0$, provided that $R\gg
0$. Summing up we have
\begin{prop}\label{dcs2}
Let $f(x)=g(d)$ with $g$ as above and assume that $g'(d)>0$,
$g''(d)>0$ for $d>r_0/2$ where $r_0>0$. Then we can find $f=f(x)$
smooth and real-valued such that 
\begin{itemize}
\item $f(x)=g(d)$ for $d\le r_0/2>0$,
\item $f(x)=(\tan \theta )d_0(x)^2/2$ near infinity,
\item $f''(x)>0$ for $d(x)\ge r_0/2$.
\end{itemize}
\end{prop}

\par To study the resonances for the exterior Dirichlet problem in
${\bf R}^n\setminus {\cal O}$ one may use complex scaling with
a contour
\ekv{dcs.20}{\Gamma _{\mathrm{ext},f}:\ x=y+if'(y),\ y\in {\bf
    R}^n\setminus {\cal O},}
where $f\in C^\infty ({\bf R}^n\setminus {\cal O})$ vanishes
on $\partial {\cal O}$, $f''>0$ away from $\partial {\cal O}$,
$f(x)=(\tan \theta )d_0(x)^2/2$ near infinity. One then considers the
restriction $P_\mathrm{ext}$ of $-h^2\Delta $ to this contour  with
domain $H^2\cap H_0^1(\Gamma _\mathrm{ext})$ and the exterior
Dirichlet resonances in the sector $e^{-i[0,2\theta [}$ coincide with
the eigenvalues of this operator. (See \cite{SjZw2,SjZw3,SjZw4} and references
cited there.) A convenient choice of $f$ near $\partial {\cal O}$ is
$f(x)=(\tan \theta )d(x)^2/2$ and according to \cite{HaLe94} we know
that $\theta =\pi /3$ is in some sense the optimal choice. 

In our case it will be convenient to use a Lipschitz contour: 
\ekv{dcs.21}{f(x)=\begin{cases} 0\hbox{ in }{\cal O}\\
(\tan \theta )\frac{d(x)^2}{2} \hbox{near }\partial {\cal O} \hbox{ in
} {\bf R}^n\setminus {\cal O},
\end{cases}}
and as above further away from $\overline{{\cal O}}$.
Then $f$ is of class $C^{1,1}$ and smooth away from $\partial {\cal
  O}$. Consequently, $\Gamma =\Gamma _f$ is a Lipschitz manifold,
smooth away from $\partial {\cal O}$ and is naturally decomposed into
the interior part ${\cal O}$ and and exterior part; $\Gamma
_{f,\mathrm{ext}}$. Again, we can define $P_\Gamma $ as
${{P}_\vert}_{\Gamma }$ with the appropriate continuity conditions at
$\partial {\cal O}$:
\ekv{dcs.22}{\begin{split}
{\cal D}(P_\Gamma )=\{ u=u_{\cal O}+u_\mathrm{ext};\ u_{\cal O}\in
H^2({\cal O}),\ u_\mathrm{ext}\in H^2(\Gamma _{f,\mathrm{ext}}),\\
u_{\cal O}=u_\mathrm{ext} ,\ \partial _\nu u_{\cal O}=\partial _\nu
  u_\mathrm{ext}\hbox{ on }\partial {\cal O}\},
\end{split}}
where $\nu $ is the exterior unit normal to ${\cal O}$. (On the
exterior part we identify $\partial _\nu $ with $(\partial _\nu
)_{\Gamma _\mathrm{ext}}$.) It follows
from Stokes' formula that
$P_\Gamma $ is symmetric.

Near a point $x_0\in \partial {\cal O}$, the problem
\ekv{dcs.23}
{
\begin{cases}
(P-z)u_{\cal O}=v_{\cal O},\\
(P-z)u_{\mathrm ext}=v_{\mathrm ext},\\
\gamma u_{\cal O}-\gamma u_\mathrm{ext}=v_0,\\
\gamma \partial _\nu u_{\cal O}-\gamma \partial _ \nu u_\mathrm{ext}=v_1
\end{cases}
}
can be viewed as an elliptic boundary value problem for an operator with
matrix valued symbol (after a reflexion so that, near $x_0$, we consider $u_{\cal
  O}$ and $u_\mathrm{ext}$ to live on the same side of the
boundary). Here we take $v_\cdot $ to be in $L^2$ in a neighborhood of $x_0$
and make the same starting assumption about $u_{{\cal O}}$ and
$u_\mathrm{ext
}$. Then if $v_0\in H^{3/2}$, $v_1\in H^{1/2}$, the standard theory
tells us that the traces are well-defined and that $u_{\cal O}$ and
$u_\mathrm{ext}$ actually belong to the spaces $H^2({\cal O})$,
$H^2({\bf R}^n\setminus {\cal O})$ respectively. Away from the
boundary, the usual arguments of complex scaling apply,
and we see that $P-z:{\cal D}(P)\to L^2$ is a holomorphic family of
Fredholm operators of index $0$, when $z\in {\bf C}\setminus
e^{-2i\theta }[0,+\infty [$.
\begin{prop}\label{dcs3} Let $\Gamma $ be the singular contour above.
The spectrum of $P=P_\Gamma $ in the sector $e^{-i[0,2\theta
  [}]0,+\infty [$ coincides with the set of resonances for $P$ there.
\end{prop}

We have already recalled that the proposition holds when $\Gamma $ is
a smooth contour, of the same form near infinity. We also recall from
\cite{SjZw1}, Section 3 (see also \cite{Sj01} for a semi-classical
version as well as \cite{SjZw2,SjZw3,SjZw4,SjZw5}), that one can show
directly, using a result on holomorphic extension of null solutions to
non-characteristic equations, that $P_{\Gamma _1}$ and $P_{\Gamma _2}$
have the same spectrum if $\Gamma _1$ and $\Gamma _2$ are two smooth
contours as above, which coincide near infinity.

\par The new part of the proof in the
case of singular contours will be to show how to extend null-solutions
holomorphically near the singular part of $\Gamma $, i.e. near
$\partial {\cal O}$ and in order to do so we need to study holomorphic
extensions of the resolvent kernel. Since we are not interested here
in how the estimates depend on $h$, we will take $h=1$ for
simplicity. The arguments below are related with the more abstract
method of exterior complex scaling of B.~Simon \cite{Si79}.

We first consider the free resolvent $R_0(z)=(-\Delta -z)^{-1}$ on
${\bf R}^n$ for $\Im z>0$. The distribution kernel is of the form
$R_0(z)(x,y)=R_0(z)(x-y)$, where
\ekv{dcs.24}
{
R_0(z)(x)=\frac{1}{(2\pi )^n}\int e^{ix\cdot \xi }\frac{1}{\xi
  ^2-z}d\xi .
}
As already mentioned, $R_0(z)$ extends holomorphically as an operator
$C_0^\infty ({\bf R}^n)\to C^\infty ({\bf R}^n)$
across $]0,+\infty [$ to the double and universal coverings
of ${\bf C}\setminus \{0\}$, when $n$ is odd and even
respectively. Moreover, for $x$ in any compact subset of ${\bf
  R}^n$ and for $z$ in any compact subset of the
covering space, there exists a constant $C>0$ such that
\ekv{dcs.25}{
|R_0(z)(x)|\le \begin{cases} C,\ &n=1\\ 
C (1+|\ln |x||),\ &n=2,\\
C |x|^{2-n}, &n\ge 3,
\end{cases}}
\ekv{dcs.26}{
|\nabla  _xR_0(z)(x)|\le \begin{cases} C,\ &n=1,\\
C|x|^{1-n},\ &n\ge 2.
\end{cases}}
More precise results are known of course, see for instance \cite{Vo94},
but we have a quick proof of (\ref{dcs.25}),
(\ref{dcs.26}) by noticing that we can make an $x$-dependent complex
deformation in the integral (\ref{dcs.24}) for large $x$ and obtain
\[
\begin{split}
R_0(z)(x)&={\cal O}(1)+\int_{|\xi |\ge 1}{\cal O}(1)e^{-|x||\xi |/C}|\xi
|^{-2}d\xi ,\\
\nabla R_0(z)(x)&={\cal O}(1)+\int_{|\xi |\ge 1}{\cal O}(1)e^{-|x||\xi |/C}|\xi
|^{-1}d\xi ,
\end{split}
\]
and treating the gradient estimate for $n=1$ separately.

\par Finally, $R_0(z)$ is rotation invariant; $R_0(z)(Ux)=R_0(z)(x)$ if
$U:{\bf R}^n\to {\bf R}^n$ is orthogonal. See Section 2 of \cite{Sj02}
as well as further references given there. As explained in that
reference, (\ref{dcs.24}) remains valid also for $z$ in the covering
space, we just have to make a complex deformation of the integration
contour in a region where $|\xi |$ is bounded, in order to avoid the
zeros $\xi ^2-z$ and this has no importance for the local properties of
$x\mapsto R_0(z)(x)$ while it does influence the exponential decay or
increase near infinity. 

\par We now want to extend (\ref{dcs.24}) holomorphically with respect
to $x$. The very first observation is that if $x_0\in {\bf R}^n\setminus
\{0\}$ then $R_0(z)(x)$ extends holomorphically in $x$ to small
neighborhood of $x_0$, by making the small complex deformation
of the integration contour in (\ref{dcs.24}) already alluded to. 

\par More generally, assume that $x\in {\bf C}^n$ and that $x\cdot
x\ne 0$. Write $x=(x\cdot x)^{1/2}f_1$ for some branch of the square
root. Then $f_1\cdot f_1=1$ and we
can find vectors $f_2,...,f_n\in {\bf C}^n$ such that $f_1,...,f_n$ is
an orthonormal basis for the bilinear symmetric product $x\cdot y$:
$f_j\cdot f_k=\delta _{j,k}$. Let $e_1,...,e_n$ be the canonical basis
in ${\bf R}^n$ and define the complex orthogonal map $U:{\bf C}^n\to
{\bf C}^n$ by 
\ekv{dcs.27}
{
Ue_j=f_j.
}
Let $\omega =((x\cdot x)/|x\cdot x|)^{1/2}$ with the same branch of
the square root as above. Then $x=\omega Uy$, where $y=|x\cdot
x|^{1/2}e_1\in {\bf R}^n$ and $y\cdot y=|x\cdot x|$. At least
formally, we have 
$$
R_0(z)(x)=:I(x,z)=\int e^{ix\cdot \xi }\frac{1}{\xi ^2-z}\frac{d\xi
}{(2\pi )^n}=\int e^{i\omega Uy\cdot \xi }\frac{1}{\xi ^2-z}\frac{d\xi
}{(2\pi )^n}.
$$
Choose the integration contour $\xi =\omega ^{-1}U\eta $, $\eta \in
{\bf R}^n$. Then $d\xi =\omega ^{-n}d\eta  $, $\xi ^2=\omega ^{-2}\eta
^2$ and we get 
$$
I(x,z)=\int e^{iy\cdot \eta }\frac{1}{\omega ^{-2}\eta
  ^2-z}\frac{d\eta }{\omega ^n(2\pi )^n}=\frac{1}{\omega ^{n-2}}\int
e^{iy\cdot \eta }\frac{1}{\eta ^2-\omega ^2z}\frac{d\eta }{(2\pi )^n},
$$
so at least formally, we have
\ekv{dcs.28}
{
I(x,z)=\omega ^{2-n}I(y,\omega ^2z),\ \omega =\left(\frac{x\cdot
    x}{|x\cdot x|}\right)^{\frac{1}{2}},\ y\in {\bf R}^n,\ x\cdot
x=\omega ^2y\cdot y.
}

We can use this formula together with the  initial remark
about holomorphic extentions to small neighborhoods of real points to
define the desired holomorphic extension of $I(x,z)$ from ${\bf
  R}^n_x\setminus \{0 \}$. Naturally this will give rise to a ramified
(multivalued) function and in order to get some more understanding, let $[0,1]\ni t\mapsto x_t\in {\bf C}^n$ be a continuous
map starting at a real point $x_0\in {\bf R}^n\setminus \{0 \}$ and
ending at some given point $x\in {\bf C}^n$ with $x\cdot x\ne 0$ such that $x_t\cdot x_t\ne 0$ for all $t$. Then
we can choose $U=U_t$ depending continuously on $t$ with $U_0=1$. If
we have choosen a branch of $I(y,z)$ for real $y$, then we get the
branch
$$
I(x,z)=\omega _1^{2-n}I(y,\omega _1^2z),
$$
obtained by following the curve $[0,1]\ni t\mapsto \omega _t^2z$ from
$z$ to $\omega _1^2z$. We conclude that $I(x,z)$ is a well-defined
multivalued holomorphic function of $x\in  \{w\in {\bf C}^n;\, w\cdot
w\ne 0 \}$ and $z$ in the double/universal covering space of ${\bf
  C}\setminus \{0 \}$. Moreover for $(x,z)$ in any fixed compact
subset of the above domain of definition, we still have
(\ref{dcs.25}), (\ref{dcs.26}).

Now we observe that the singular contour $\Gamma $ in Proposition
\ref{dcs3} is of the form $\Gamma =\Gamma _f$: $x=y+if'(y)$, where $f$
is real-valued of class $C^{1,1}({\bf R}^n)$ which is convex and
$f(y)=(\tan \theta )d_0(y)^2/2$ near infinity. If $x_j=y_j+if(y_j)$,
$j=0,1$ are two different points on $\Gamma _f$, then 
$$
f'(y_1)-f'(y_0)=A(y_0,y_1)(y_1-y_0),
$$
where 
$$
A(y_0,y_1)=\int_0^1 f''(ty_1+(1-t)y_0)dt\ge 0,
$$
and 
$$
(x_1-x_0)\cdot (x_1-x_0)=[(1-A(y_1,y_0)^2)+2iA(y_0,y_1)](y_1-y_0)\cdot (y_1-y_0).
$$
The same argument as for the ellipticity of $-\Delta _{\Gamma _f}$
shows that 
$$\Gamma _f\times \Gamma _f\ni (x_0,x_1)\mapsto (x_1-x_0)\cdot
(x_1-x_0)$$
takes its values in a sector $e^{i[0,\pi -\epsilon ]}[0,+\infty [$ and
that 
$$|(x_1-x_0)\cdot (x_1-x_0)|\asymp |x_1-x_0|^2,\ x_0,x_1\in \Gamma
_f.$$

Combining these facts with the deformation $[0,1]\ni t\mapsto \Gamma
_{tf}$ from ${\bf R}^n$ to $\Gamma _f$, we see that
$R_0(z)(x,y)=R_0(z)(x-y)$ is well-defined on $\Gamma _f\times \Gamma
_f$ away from the diagonal, and we can define
$$
R_{0,\Gamma }u(x)=\int_\Gamma R_0(z)(x,y)u(y)dy,\ x\in \Gamma _f,\ u\in
C_0(\Gamma ),\ \Gamma =\Gamma _f.
$$
This gives a continuous operator $C_0(\Gamma )\to C(\Gamma )$. Let
$P_0=-\Delta $. Using that
$$
(-\Delta_x-z)R_0(z)(x,y)= (-\Delta_y^{\mathrm{t}}-z)R_0(z)(x,y)=0,\
x\ne y,
$$
as well as the bound on the strength of the singularity at $x=y$ described in
(\ref{dcs.25}), (\ref{dcs.26}), we see that in the case when $f$ is
smooth, we have, 
\[\begin{split}(P_{0,\Gamma }-z)R_{0,\Gamma }(z)v(x)&=C(x,f)v(x)\\
R_{0,\Gamma }(z)(P_{0,\Gamma }-z)u(x)&=\widetilde{C}(x,f)u(x)
\end{split} \] for $x\in \Gamma $, $u,v\in C_0^\infty (\Gamma )$. It
is further clear that $C(x,f)$, $\widetilde{C}(x,f)$ only depend on
the restriction of $f$ to a small neighborhood of $\Re x$, so we can
replace $f$ be a new function $\widetilde{f}$ which is equal to $f$
near $\Re x$ with $\widetilde{f}''$ varying very little and being
constant near infinity. We can then determine the constants by letting
$v$, $u$ be suitable Gaussians and possibly after an additional
deformation argument, we get $C(x,f)=\widetilde{C}(x,f)=1$. Thus
\ekv{dcs.29} { (P_{0,\Gamma }-z)R_{0,\Gamma }(z)v=v, } \ekv{dcs.30} {
  R_{0,\Gamma }(z)(P_{0,\Gamma }-z)u=u, } when $u,v\in C_0^\infty
(\Gamma )$, $\Gamma =\Gamma _f$ and $f$ is smooth. To extend this to
the general case when $f$ is a convex $C^{1,1}$ function would require
first to define the operator $P_{0,\Gamma }$, and we prefer to avoid
that work and just consider the case of the special singular contour
in Proposition \ref{dcs3}. Then for $v\in C_0(\Gamma )$ (\ref{dcs.29})
still holds away from $\partial {\cal O}$.

\par We also remark that if $v\in C_0(\Gamma )$, then $u:=R_{0,\Gamma
}v$ is of class $C^1$ up to the boundary both on ${\cal O}$ and on
$\Gamma _\mathrm{ext}$ and we have 
\ekv{dcs.31}{
\gamma u_\Omega =\gamma u_\mathrm{ext},\ \gamma \partial _\nu u_\Omega =\gamma
\partial _\nu u_\mathrm{ext}.}
Using now that (\ref{dcs.23}) is an elliptic boundary value problem,
we see that $R_{0,\Gamma }v$ belongs locally to ${\cal D}(P_\Gamma )$ and
this holds more generally for $v\in L^2_\mathrm{comp}(\Gamma )$. 

\par If $u\in C_0(\Gamma )$ and $u_{\cal O}$ and $u_\mathrm{ext}$ are
$C^2$ up to the boundary and satisfy (\ref{dcs.31}), then we can make
integrations by parts in
\[
R_{0,\Gamma }(P_{0,\Gamma }-z)u(x)=\int R_{0}(z)(x,y)(-\Delta _\Gamma -z)u(y)dy
 \]
 after introducing a cutoff around the singularity and passing to the
 limit and get (\ref{dcs.30}) as in the case when $f$ is smooth. By
 density this extends to the case when $u\in {\cal D}(P_\Gamma )$ has
 compact support.

We can now complete the proof of Proposition \ref{dcs3}. Let $\Gamma
=\Gamma _f$ be the singular contour in that proposition and let
$\widetilde{f}$ be smooth, convex, $=0$ in ${\cal O}$ and $=f$ outside
a small neighborhood of $\overline{{\cal O}}$. Let $\widetilde{\Gamma
}=\Gamma _{\widetilde{f}}$ be the corresponding smooth contour, so
that the spectrum of $\widetilde{P}=P_{\widetilde{\Gamma }}$ in the
sector $e^{-i[0,2\theta [}]0,+\infty [$ coincides with the set of
resonances there. As in \cite{SjZw1}, it suffices to show the following two facts:
\begin{itemize}
\item[1)] If $u\in {\cal D}(P_\Gamma )$ and $(P_\Gamma -z)u=0$, then
  $u$ has a holomorphic extension to a domain containing
\ekv{dcs.32}{\{y+i(t\widetilde{f}'(y)+(1-t)f'(y));\ f(y)\ne
  \widetilde{f}(y),\ 0\le t\le 1 \} ,}
such that its restriction $\widetilde{u}$ to $\widetilde{\Gamma }$
belongs to ${\cal D}(P_{\widetilde{\Gamma }})$ and satisfies
$(P_{\widetilde{\Gamma }}-z)\widetilde{u}=0$.
\item[2)] If $\widetilde{u}\in {\cal D}(P_{\widetilde{\Gamma }} )$ and $(P_{\widetilde{\Gamma}} -z)\widetilde{u}=0$, then
  $\widetilde{u}$ has a holomorphic extension to a domain containing
  the set 
(\ref{dcs.32})
such that its restriction $u$ to $\Gamma $
belongs to ${\cal D}(P_{\Gamma })$ and satisfies
$(P_{{\Gamma }}-z){u}=0$.
\end{itemize}  

\par Let $\widehat{\chi }\in C_0^\infty ({\bf R}^n)$ be equal to one
near $\mathrm{supp\,}(f-\widetilde{f})$ and define the cutoffs $\chi $
and $\widetilde{\chi }$ on $\Gamma $ and on $\widetilde{\Gamma }$
respectively by
$$
\chi (y+if'(y))=\widetilde{\chi }(y+i\widetilde{f}'(y))=\widehat{\chi
}(y). 
$$
We first prove 1) and let $u$ be as in that statement. Then 
\ekv{dcs.33}
{
(P_\Gamma -z)\chi u=[P_\Gamma ,\chi ]u,
}
where the right hand side has its support in the region where $\Gamma
$ and $\widetilde{\Gamma }$ coincide. We can rewrite (\ref{dcs.33}) as 
\ekv{dcs.34}
{
(P_{0,\Gamma }-z)\chi u=[P_\Gamma ,\chi ]u-Vu
}
and $Vu$ also has its support where $\Gamma $ and $\widetilde{\Gamma
}$ coincide. Applying (\ref{dcs.30}) gives
\ekv{dcs.35}
{
\chi u=R_{0,\Gamma }(z)([P_\Gamma ,\chi ]u-Vu).
}
From the properties of $R_0(z)$, we see that $\chi u$ has a
holomorphic extension to a domain containing the set
(\ref{dcs.32}). Its restriction to $\widetilde{\Gamma }$ solves
$(P_{\widetilde{\Gamma }}-z)\widetilde{u}=0$ and $\widetilde{u}=u$ in
the regions where $\Gamma $ and $\widetilde{\Gamma }$ coincide. From
elliptic regularity we see that $\widetilde{u}$ is locally in $H^2$
and hence globally so $\widetilde{u}$ belongs to the domain of
$P_{\widetilde{\Gamma }}$. This proves 1). 

The proof of 2) works the same way with the small difference that
instead of invoking the ellipticity of $P_{\widetilde{\Gamma }}$ on
the smooth manifold $\widetilde{\Gamma }$, we invoke the ellipticity
of the boundary value problem (\ref{dcs.23}). \hfill{$\Box$}

\section{Semi-Classical Sobolev spaces}
\label{al}
\setcounter{equation}{0}
This section is a review of some easy facts about Sobolev spaces, see
Section 2 in
\cite{Sj08a, Sj08b} for more details about the first part.
We let $H_h^s({\bf R}^n)\subset {\cal S}'({\bf R}^n)$, $s\in {\bf R}$, 
denote the semi-classical Sobolev space of order
$s$ equipped with the norm $\Vert \langle hD\rangle^s u\Vert$ where
the norms are the ones in $L^2$, $\ell^2$ or the corresponding
operator norms if nothing else
is indicated. Here $\langle hD\rangle= (1+(hD)^2)^{1/2}$.  
\begin{prop}\label{al1}
Let $s>n/2$. Then there exists a constant $C=C(s)$ such that for all
$u,v\in H_h^s({\bf R}^n)$, we have $u\in L^\infty ({\bf R}^n) $, 
$uv\in H_h^s({\bf R}^n)$ and 
\ekv{al.1}
{
\Vert u\Vert_{L^\infty }\le Ch^{-n/2}\Vert u\Vert_{H_h^s},
}
\ekv{al.2}
{
\Vert uv\Vert_{H_h^s} \le Ch^{-n/2} \Vert u\Vert_{H_h^s} \Vert v\Vert_{H_h^s}.
}
\end{prop}

Let $X$ be a compact smooth manifold. We cover $X$ by
finitely many coordinate neighborhoods $X_1,...,X_p$ and for
each $X_j$, we let $x_1,...,x_n$ denote the corresponding local 
coordinates on $X_j$. Let $0\le \chi _j\in C_0^\infty (X_j)$ have the
property that $\sum_1^p\chi _j >0$ on $X$. Define $H_h^s(X)$ to be the
space of all $u\in {\cal D}'(X)$ such that 
\ekv{al.4}
{
\Vert u\Vert_{H_h^s}^2:=\sum_1^p \Vert \chi _j\langle hD\rangle^s \chi
_j u\Vert ^2 <\infty .
}
It is standard to show that this definition does not depend on the
choice of the coordinate neighborhoods or on $\chi _j$. With different
choices of these quantities we get norms in \no{al.4} which are
uniformly equivalent when $h\to 0$. In fact, this follows from the
$h$-pseudodifferential calculus on manifolds with symbols in the
H\"ormander space $S^m_{1,0}$ that we quickly reviewed in the
appendix in \cite{Sj08a}.
An equivalent definition of $H_h^s(X)$ is the following: Let 
\ekv{al.5}
{
h^2\widetilde{R}=\sum (hD_{x_j})^*r_{j,k}(x)hD_{x_k}
}
be a self-adjoint non-negative elliptic operator with smooth coefficients on $X$,
where the star indicates that we take the adjoint with respect to some
fixed positive smooth density on $X$. Then $h^2\widetilde{R}$ is
essentially self-adjoint with domain $H^2(X)$, so
$(1+h^2\widetilde{R})^{s/2}:L^2\to L^2$ is a closed densely defined
operator for $s\in {\bf R}$, which is bounded precisely when $s\le
0$. Standard methods allow to show that $(1+h^2\widetilde{R})^{s/2}$
is an $h$-pseudodifferential operator with symbol in $S^s_{1,0}$ and
semi-classical principal symbol given by $(1+r(x,\xi ))^{s/2}$, where
$r(x,\xi )=\sum_{j,k}r_{j,k}(x)\xi _j\xi _k$ is the semi-classical
principal symbol of $h^2\widetilde{R}$.  See the appendix in
\cite{Sj08a}.
The
$h$-pseudodifferential calculus gives for every $s\in {\bf R}$:
\begin{prop}\label{al2}
  $H_h^s(X)$ is the space of all $u\in {\cal D}'(X)$ such that 
$(1+h^2\widetilde{R})^{s/2}u\in L^2$ and the norm $\Vert u\Vert_{H_h^s}$ is
equivalent to $\Vert (1+h^2\widetilde{R})^{s/2}u\Vert$, uniformly when $h\to 0$.
\end{prop}
\begin{remark}\label{al3}
{\rm From the first definition we see that Proposition \ref{al1} remains
valid if we replace ${\bf R}^n$ by a compact $n$-dimensional 
manifold $X$.}
\end{remark}
\begin{remark}\label{al3.5}
{\rm We will also consider the case when the manifold $X$ is the disjoint
union of a compact part and ${\bf R}^n\setminus B(0,R)$ for some
$R>0$. The definition and properties of $H_h^s(X)$ are quite clear.}
\end{remark}
\par Of course, $H_h^s(X)$ coincides with the standard Sobolev space
$H_1^s(X)$ and the norms are equivalent for each fixed value of $h$, but
not uniformly so with respect to $h$. We have the following variant
(\cite{Sj08b}, Section 2):
\begin{prop}\label{al4}
Let $s>n/2$. Then there exists a constant $C=C_s>0$ such that 
\ekv{al.6}
{
\Vert uv\Vert_{H_h^s}\le C\Vert u\Vert_{H_1^s}\Vert v\Vert_{H_h^s},\
\forall u\in H^s({\bf R}^n),\, v\in H_h^s({\bf R}^n).
}
The result remains valid if we replace ${\bf R}^n$ by $X$.
\end{prop}

Let $\Omega \Subset {\bf R}^n$ be open with smooth boundary. Let
$H^s_h(\Omega )$ denote the Banach space of restrictions to $\Omega $
of elements in $H^s_h({\bf R}^n)$. It is a standard fact that if
$s>1/2$, then the restriction operator $\gamma :u\mapsto
{{u}_\vert}_{\partial \Omega }$ is bounded: $H_1^s(\Omega )\to
H_1^{s-\frac{1}{2}}(\partial \Omega )$.  $\gamma $ has a right inverse
$\gamma ^{-1}$ which is bounded $H_1^{\widetilde{s}-1/2}(\partial \Omega )\to
H_1^{\widetilde{s}}(\Omega )$ for all $\widetilde{s}\in {\bf R}$. More generally, if $s>3/2$,
then 
$$\begin{pmatrix}\gamma \\\gamma D_\nu \end{pmatrix}:\, H_1^s(\Omega
)\to H_1^{s-1/2}(\partial \Omega )\times H_1^{s-3/2}(\partial \Omega
)$$ has a right inverse which is ${\cal O}(1):\,
H_1^{\widetilde{s}-1/2}\times H_1^{\widetilde{s}-3/2}\to
H_1^{\widetilde{s}}$ for all $\widetilde{s}\in {\bf R}$. Here $\nu $
is the exterior unit normal and $D_\nu =i^{-1}\partial /\partial \nu
$.

In the semi-classical case, we obtain from the same (standard) proofs
that 
\ekv{al.10} { \gamma ={\cal
    O}_s(h^{-\frac{1}{2}}):H_h^s(\Omega )\to
  H_h^{s-\frac{1}{2}}(\partial \Omega ),\quad s>\frac{1}{2} }
has a right inverse
such that \ekv{al.11} { \gamma ^{-1} ={\cal
    O}_{\widetilde{s}}(h^{\frac{1}{2}}):H_h^{\widetilde{s}-\frac{1}{2}}(\partial \Omega )\to
  H_h^{\widetilde{s}}( \Omega ),\quad \widetilde{s}\in {\bf R}.  }

More generally, the operator
$$
\begin{pmatrix}\gamma \\ \gamma hD_\nu \end{pmatrix}:H_h^s(\Omega )\to
H_h^{s-\frac{1}{2}}(\Omega )\times H_h^{s-\frac{3}{2}}(\partial \Omega )
$$
has a right inverse which is ${\cal O}(h^{1/2}):\,H_h^{\widetilde{s}-1/2}\times
H_h^{\widetilde{s}-3/2}\to H_h^{\widetilde{s}}$ for all
$\widetilde{s}\in {\bf R}$. 

The following observation can be turned into a proof by reduction to
the standard non-semi-classical case: The change of variables
$x=h\widetilde{x}$ transforms $hD_x$ into $D_{\widetilde{x}}$ and if
$u(x)=\widetilde{u}(\widetilde{x})$, then 
$$
\Vert u\Vert_{H_h^s(\Omega )}=h^{\frac{n}{2}}\Vert
\widetilde{u}\Vert_{H_1^s(h^{-1}\Omega )}.
$$
Similarly for functions on $\partial \Omega $, we have 
$$
\Vert u\Vert_{H_h^s(\partial \Omega )}=h^{\frac{n-1}{2}}\Vert
\widetilde{u}\Vert_{H_1^s(h^{-1}\partial \Omega ) }.
$$

\section{Reductions to ${\cal O}$ and to $\partial {\cal O}$}\label{red}
\setcounter{equation}{0}

\par In this section, we let $P=-h^2\Delta +V$ and ${\cal O}$ be as in
Subsection \ref{dcs}. We choose the contour $\Gamma $ as there, either
singular or smooth. When $\Gamma $ is smooth, the
domain of $P_\Gamma $ is the space $H^2_h(\Gamma )$, and when $\Gamma
$ has a singularity along the boundary of ${\cal O} $, it is given by
(\ref{dcs.22}). (Later we shall also need to consider the case when
$\Gamma $ is constructed as in the preceeding section but with ${\cal
  O}$ replaced by a slightly larger set $\widetilde{{\cal O}}$ with
the same properties, containing an $h$-neighborhood of ${\cal O}$.) By
abuse of notation we sometimes write $H^2(\Gamma )$ also for ${\cal
  D}(P_\Gamma )$.  

The exterior Dirichlet problem is
\ekv{red.1} { (P-z)u=v\hbox{ on
  }\Gamma _\mathrm{ext}=\Gamma \setminus {\cal O},\ {{u}_\vert}_{\partial {\cal O}}=w,  }
for given $v\in L^2(\Gamma \setminus {\cal O})$, $w\in
H^{3/2}(\partial {\cal O})$ with the solution $u$ in $H^2(\Gamma
\setminus {\cal O})$. Here, $\gamma u= {{u}_\vert}_{\partial {\cal O}}$.
The corresponding closed operator $P_{\mathrm{ext}}$ has the domain
${\cal D}(P_\mathrm{ext})=\{ u\in H^2(\Gamma \setminus {\cal O});\,
\gamma u=0\}$.   The eigenvalues are the resonances for
the exterior Dirichlet problem. We restrict the attention to the case
when $1/2\le \Re z\le 2$, $\Im z\ge -ch^{2/3}$, where
$c<2(1/2)^{2/3}\kappa \zeta _1$. (Cf Theorem \ref{re1}.) When
$z\not\in \sigma (P_{\mathrm{ext}})$, we can express the solution of
(\ref{red.1}) as \ekv{red.2} {
  u=G_\mathrm{ext}(z)v+K_\mathrm{ext}(z)w.  } Put \ekv{red.3}{{\cal
    N}_\mathrm{ext}w=\gamma hD_\nu K_\mathrm{ext}w,} where $\gamma $
is the operator of restriction to $\partial {\cal O}$ and $\nu $ is
the exterior unit normal.
\begin{dref}\label{red1}
$P_\mathrm{out}(z)$ is the operator $-h^2\Delta +V-z$ on ${\cal O}$ with
domain 
\ekv{red.4}
{
{\cal D}(P_\mathrm{out}(z))=\{ u\in H^2({\cal O});\, (\gamma hD_\nu
-{\cal N}_\mathrm{ext }(z)\gamma )u=0\} .
}
\end{dref}

Notice that the domain varies with $z$ and this is why we avoid
writing ``$P_\mathrm{out}-z$''. In the first part of this section we
shall show that $z$ is a resonance of $P$ precisely when $0\in
\sigma (P_\mathrm{out}(z))$, but for technical reasons we will prefer to work
with the full problem,
\ekv{red.5}{  P_\mathrm{out}(z)u=v,\ h^{\frac{1}{2}}Bu=w ,}
where
\ekv{lb.3}{B=\gamma hD_{\nu }-{\cal
  N}_\mathrm{ext}\gamma :H^2({\cal O})\to H^{1/2}(\partial {\cal O})}
It is easy to check that this is an elliptic boundary value problem in
the classical sense. (The semi-classical structure of ${\cal
  N}_\mathrm{ext}$ and of (\ref{red.5}) will require more work below.)
The well-posedness of (\ref{red.5}) is of course equivalent to the
bijectivity of 
\ekv{red.5.5}
{
{\cal P}_\mathrm{out}(z)=\begin{pmatrix}
  P-z\\h^{\frac{1}{2}}B\end{pmatrix}:H^2({\cal O})\to H^0({\cal O})\times
H^{\frac{1}{2}}(\partial {\cal O}).
}
Here and below we sometimes write $H^s$ instead of $H^s_h$.

\par In the following we impose the condition
\ekv{lb.1}
{
|\Im z|\le h^{2/3}c_0,\ \frac{1}{2}\le \Re z\le 2
}
with $c_0$ as in (\ref{outl.0}),
so that the exterior Dirichlet problem is well-posed. (We could here
drop the upper bound on $\Im z$.)

Under the condition (\ref{lb.1}) we shall show that ${\cal P}_\mathrm{out}(z)$
and $P_\Gamma -z $ are ``equivalent'', and to do so we shall see that
${\cal P}_\mathrm{out}(z)$ appears as the effective Hamiltonian (up to
an invertible factor) in a well-posed
Grushin problem for $P_\Gamma -z$.

\par Let $\iota :L^2({\cal O})\to L^2(\Gamma )$ be the natural zero
extension map and let $\Pi :H^2(\Gamma )\to H^2({\cal O})$ be the
restriction map. Let $\widehat{K}={\cal O}(h^{1/2}):H^{1/2}(\partial
{\cal O})\to H^2({\cal O})$ be a right inverse of $B$ (cf the last
observation in Section \ref{al}).
Put
\ekv{lb.4}
{
{\cal P}(z)=\begin{pmatrix}P_\Gamma -z &\iota &0\\
\Pi &0 &\widehat{K}
\end{pmatrix}:H^2(\Gamma )\times L^2({\cal O})\times
H^{\frac{1}{2}}(\partial {\cal O})\to L^2(\Gamma )\times H^2({\cal
  O}).  } We will view ${\cal P}(z)$ as a $2\times 2$ block matrix
with the upper left block given by $P_\Gamma -z$. We claim that ${\cal
  P}(z)$ is bijective. This amounts to finding a unique solution
$(u,u_-,u_-')\in H^2(\Gamma )\times L^2({\cal O})\times
H^{\frac{1}{2}}(\partial {\cal O})$ of the problem
\ekv{lb.5}
{
\begin{cases}
(P_\Gamma -z)u+\iota u_-&=v,\\
\Pi u+\widehat{K}u_-'&=v_+
\end{cases}
} for every given $(v,v_+)\in L^2(\Gamma )\times H^2({\cal O})$.
The exterior part (i.e. the restriction to $\Gamma
_\mathrm{ext}=\Gamma \setminus {\cal O}$) of the first equation in (\ref{lb.5}) is (with the
natural notation)
$$(P_{\Gamma _\mathrm{ext}}-z)u_\mathrm{ext}=v_\mathrm{ext},$$
which has the general solution
$$
u_\mathrm{ext}=G_\mathrm{ext}(z)v_\mathrm{ext}+K_\mathrm{ext}(z)g,
$$
where $g\in H^{3/2}(\Gamma )$ is arbitrary to start with. Notice that 
$$
Bu_\mathrm{ext}=BG_\mathrm{ext}(z)v_\mathrm{ext},
$$
since $BK_{\mathrm{ext}}(z)=0$ by the definition of ${\cal N}_\mathrm{ext}(z)$. 
Here the continuity condition on $u$ given by (\ref{dcs.22}), can
be written
\ekv{lb.2}{\gamma u_\mathrm{int}=\gamma u_\mathrm{ext},\ B
  u_\mathrm{int}=B u_\mathrm{ext}.} 

\par The interior part of (\ref{lb.5}) is 
\ekv{lb.6}
{
\begin{cases}(P-z)u_\mathrm{int}+u_-=v_\mathrm{int}\\
u_\mathrm{int}+\widehat{K}u'_-=v_+
\end{cases}\hbox{ in } {\cal O},
}
giving
$$
\begin{cases}
u_\mathrm{int}=v_+-\widehat{K}u'_-\\
u_-=v_\mathrm{int}-(P-z)u_\mathrm{int}
\end{cases}.
$$
The second condition in (\ref{lb.2}) now gives
$$
Bv_+-u_-'=BG_\mathrm{ext}v_\mathrm{ext},
$$
i.e.
\ekv{lb.7}{u_-'=Bv_+-BG_\mathrm{ext}v_\mathrm{ext}.}
The first part of (\ref{lb.2}) boils down to
\ekv{lb.8}
{
\gamma v_+-\gamma \widehat{K}u_-'=g.
}

\par Thus the unique solution of (\ref{lb.5}) is given by
$u=u_\mathrm{int}+u_\mathrm{ext}$, $u_-$, $u_-'$, where
\begin{eqnarray*}
u_-'&=&B(v_+-G_\mathrm{ext}v_\mathrm{ext})\\
u_\mathrm{int}&=&(1-\widehat{K}B)v_++\widehat{K}BG_\mathrm{ext}v_\mathrm{ext}\\u_-&=&v_\mathrm{int}-(P-z)\widehat{K}BG_\mathrm{ext}v_\mathrm{ext}
-(P-z)(1-\widehat{K}B)v_+
\\
u_\mathrm{ext}&=&(1+K_\mathrm{ext}\gamma
\widehat{K}B)G_\mathrm{ext}v_\mathrm{ext}+K_\mathrm{ext}\gamma (1-\widehat{K}B)v_+.
\end{eqnarray*}
Using the characteristic functions $1_{\cal O}$ and $1_{\Gamma
  _\mathrm{ext}}$ to indicate the projection to the interior and
exterior parts of functions on $\Gamma $, we get in matrix form:
\ekv{lb.9}
{\begin{split}
&{\cal P}(z)^{-1}=\\
&\hskip -15mm \begin{pmatrix}
1_{{\cal O}}\widehat{K}BG_\mathrm{ext}1_{\Gamma
  _\mathrm{ext}}+1_{\Gamma _\mathrm{ext}}(1+K_\mathrm{ext}\gamma
\widehat{K}B)G_\mathrm{ext}1_{\Gamma _\mathrm{ext}}
&1_{{\cal O}} (1-\widehat{K}B)+1_{\Gamma
  _\mathrm{ext}}K_\mathrm{ext}\gamma (1-\widehat{K}B)\\
1_{\cal O}-(P-z)\widehat{K}BG_\mathrm{ext}1_{\Gamma _\mathrm{ext}}
&-(P-z)(1-\widehat{K}B)\\
-BG_\mathrm{ext}1_{\Gamma _\mathrm{ext}} &B
\end{pmatrix}
\end{split}
}

As already mentioned we can use block matrix notation and write
$$
{\cal P}(z)=\begin{pmatrix}P_{11} &P_{12}\\ P_{21} &P_{22}
\end{pmatrix},
$$
where
\begin{eqnarray*}
P_{11}=P_\Gamma -z,&&\ P_{12}=\begin{pmatrix}\iota &0\end{pmatrix},\\
P_{21}=\Pi ,&&\ P_{22}=\begin{pmatrix}0 &\widehat{K}\end{pmatrix}.
\end{eqnarray*}
Then 
$$
{\cal E}(z):={\cal P}(z)^{-1}=\begin{pmatrix}E_{11} &E_{12}\\ E_{21} &E_{22}
\end{pmatrix},
$$
where 
$$
E_{22}=\begin{pmatrix}-(P-z)(1-\widehat{K}B)\\ B \end{pmatrix}=
\begin{pmatrix}-1 &h^{-\frac{1}{2}}(P-z)\widehat{K}\\ 0 &h^{-\frac{1}{2}}\end{pmatrix} {\cal P}_\mathrm{out}(z),
$$
and ${\cal P}_\mathrm{out}(z)$ was defined in
(\ref{red.5.5}). The upper triangular matrix in the last
expression is invertible, so the invertibility of $E_{22}$ is
equivalent to that of ${\cal P}_\mathrm{out}$ and using also the
second part of Proposition \ref{si1}, we get
\begin{prop}\label{red2}
For $z$ in the region (\ref{lb.1})
we have that $z\in \sigma (P_\Gamma )$ if and only if 
$0\in \sigma ({\cal P}_\mathrm{out}(z))$.
\end{prop}

\par $P_\Gamma -z$, ${\cal P}_\mathrm{out}(z)$ are holomorphic
families of Fredholm operators of index 0 and combining (\ref{si.2})
with Proposition \ref{mgd2}, we see that $\det (P_\Gamma -z)$ and
$\det {\cal P}_\mathrm{out}(z)$ have zeros of the same multiplicity at the
points of $\sigma (P_\Gamma )$.

\par We next discuss a reduction to the boundary when $z$ is not a
Dirichlet eigenvalue. Let $P_\mathrm{in}$ denote the Dirichlet realization of $P$ in
${\cal O}$, so that ${\cal D}(P_\mathrm{in})=\{ u\in H^2({\cal O});\,
\gamma u=0\}$. Let
\ekv{ub.2}
{
{\cal P}_\mathrm{in}(z)=\begin{pmatrix}P-z\\ h^{\frac{1}{2}}\gamma \end{pmatrix}:
H^{2}({\cal O})\to H^0({\cal O})\times
  H^{\frac{3}{2}}(\partial {\cal O}),}
so that ${\cal P}_\mathrm{in}(z)$ is bijective precisely when $z$ is
not a Dirichlet eigenvalue; $z\notin \sigma (P_\mathrm{in})$. Let 
$$
{\cal E}_\mathrm{in}(z)=\begin{pmatrix}
G_\mathrm{in}(z) &h^{-\frac{1}{2}}K_\mathrm{in}(z) \end{pmatrix}
$$
be the inverse which is well defined for $z$ away from the spectrum of
$P_\mathrm{in}$. Then 
$$
{\cal P}_\mathrm{out}(z){\cal E}_\mathrm{in}(z)=
\begin{pmatrix}
(P-z)G_\mathrm{in} & (P-z)h^{-\frac{1}{2}}K_\mathrm{in}\\ h^{\frac{1}{2}}BG_\mathrm{in} &BK_\mathrm{in}
\end{pmatrix}.
$$
Here $(P-z)G_\mathrm{in}=1$, $(P-z)K_\mathrm{in}=0$ and \ekv{ub.2.5}{
  BK_\mathrm{in}=\gamma hD_\nu K_\mathrm{in}-{\cal
    N}_\mathrm{ext}={\cal N}_\mathrm{in}-{\cal N}_\mathrm{ext},} where
the last equility defines ${\cal
  N}_\mathrm{in}:H^{\frac{3}{2}}(\partial {\cal O})\to
H^{\frac{1}{2}}(\partial {\cal O}) $ so 
\ekv{ub.3} { {\cal
    P}_\mathrm{out}(z){\cal E}_\mathrm{in}(z)=\begin{pmatrix} 1 &0\\
    h^{\frac{1}{2}}BG_\mathrm{in} &{\cal N}_\mathrm{in}-{\cal
      N}_\mathrm{ext}\end{pmatrix} .}
Composing with ${\cal P}_\mathrm{in}$ to the right, we get
\ekv{red.6}
{
{\cal
    P}_\mathrm{out}(z)=\begin{pmatrix} 1 &0\\
    h^{\frac{1}{2}}BG_\mathrm{in} &{\cal N}_\mathrm{in}-{\cal
      N}_\mathrm{ext}\end{pmatrix}{\cal P}_\mathrm{in}(z).
}
Notice that this factorization makes sense only when $z\notin
\sigma (P_\mathrm{in}(z))$ since ${\cal N}_\mathrm{in}$ is defined only under
that assumption. The last factor in the right hand side is of course
bijective then, and the first lower triangular factor is bijective
precisely when ${\cal N}_\mathrm{in}(z)-{\cal
  N}_\mathrm{ext}(z):\,H^{3/2}\to H^{1/2}$ is bijective, or
equivalently when $0$ is not in the spectrum of this operator,
considered as an unbounded operator $H^{1/2}\to H^{1/2}$ with domain $H^{3/2}$.
\begin{prop}\label{red3}
For $z$ in the region (\ref{lb.1}) and not in $\sigma
(P_\mathrm{in})$, we have the equivalence:
$$
0\in \sigma ({\cal P}_\mathrm{out}(z)) \Leftrightarrow 0\in \sigma ({\cal
  N}_\mathrm{in}-{\cal N}_\mathrm{ext}).
$$
\end{prop}
Again we have holomorphic families of Fredholm operators of index 0
and we have the analogue of the remark after Proposition \ref{red2}.

\par We end the section with a symmetry observation (cf (\ref{cds.1}).

\begin{prop}\label{red4}
$ P_\mathrm{out}(z)$, ${\cal N}_\mathrm{in}$ and ${\cal N}_\mathrm{ext}$
are symmetric.
\end{prop}
\begin{proof}
This follows from Green's formula. For $u,v\in H^{3/2}(\partial {\cal
  O})$, we have
\begin{equation*}\begin{split}
    &\langle {\cal N}_\mathrm{in}u|v\rangle_{\partial {\cal O}}-
    \langle u|{\cal N}_\mathrm{in}v\rangle_{\partial {\cal O}}\\
    &=\langle hD_\nu K_\mathrm{in}u|v\rangle_{\partial {\cal
        O}}-\langle
    u|hD_\nu K_\mathrm{in}v\rangle_{\partial {\cal O}}\\
    &=\frac{i}{h}(\langle -h^2\Delta
    K_\mathrm{in}u|K_\mathrm{in}v\rangle_{\cal O}-\langle
    K_\mathrm{in}u|-h^2\Delta K_\mathrm{in}v\rangle_{\cal O})\\
    &=\frac{i}{h}(\langle
    (P-z)K_\mathrm{in}u|K_\mathrm{in}v\rangle_{\cal O}-\langle
    K_\mathrm{in}u|(P-z)K_\mathrm{in}v\rangle) \\
&=0.
\end{split}\end{equation*}

\par The symmetry of ${\cal N}_\mathrm{ext}$ follows in the same way
by applying Green's formula on $\Gamma _\mathrm{ext}$.

\par
Let $u,v\in {\cal D}(P_\mathrm{out}(z))$, so that $\gamma hD_\nu
u={\cal N}_\mathrm{ext}\gamma u$ and similarly for $v$. Using again
Green's formula, we get

\begin{equation*}\begin{split}
&\langle P_\mathrm{out}(z)u|v\rangle_{\cal O}-\langle
u|P_\mathrm{out}(z)v\rangle_{\cal O}\\
&=-h^2(\langle \Delta u|v\rangle_{\cal O}-\langle
u|\Delta v\rangle_{\cal O})\\
&=\frac{h}{i}(\langle hD_\nu u|v\rangle_{\partial {\cal O}}-
\langle u|hD_\nu v\rangle_{\partial {\cal O}})\\
 &=\frac{h}{i}(\langle {\cal N}_\mathrm{ext} u|v\rangle_{\partial {\cal O}}-
\langle u|{\cal N}_\mathrm{ext} v\rangle_{\partial {\cal O}}
)=0,
\end{split}\end{equation*}
where the last equality follows from the symmetry of ${\cal N}_\mathrm{ext}$.
\end{proof}

\section{Some ODE preparations}\label{prep}
In this section we make some preparations for the study of the
interior and exterior Dirichlet to Neumann maps and some related
estimates for the exterior resolvent.
\setcounter{equation}{0}
\subsection{Nullsolutions and factorizations of 2nd order ODEs}\label{nf}
It will be convenient to factorize our equations and we make some
extremely elementary and certainly well-known remarks.
Let \ekv{nf.1} { P=\partial _t^2+a(t)\partial _t+b(t) } be a
differential operator with smooth coefficients on an interval or with
holomorphic coefficients on a simply connected open set in ${\bf
  C}$. Let $e^{-\alpha (t)}$ belong to the kernel of $P$, \ekv{nf.2} {
  P(e^{-\alpha })=0.  } This means that $P$ takes the form
$P=(\partial _t+\alpha ')^2+f(t)(\partial _t+\alpha ')+g(t)$, where
$g\equiv 0$ and we get \ekv{nf.5} {
  P=(\partial _t-\beta ')(\partial _t+\alpha '), } where $\beta
'=\alpha '-a$, \ekv{nf.6} { \beta =\alpha -\int^t ads.  }

\par Notice that $P^\mathrm{t}=(\partial _t-\alpha ')(\partial
_t+\beta ')$, so $e^\beta $ belongs to the kernel of
$P^\mathrm{t}$. When $P$ is symmetric,
$P^\mathrm{t}=P$, we have $a=0$, $\beta =\alpha $.

\subsection{Simple turning point analysis}\label{sta}
We recall some elements of the complex WKB method and refer
to \cite{Vor81, Fe87} for more extensive expositions.
 Let $V=V(x)$ be holomorphic in some simply connected open set
 $\subset {\bf C}$. We consider the equation 
\ekv{sta.1}
{
((hD_x)^2+V(x))u=0,
}
with $u$ holomorphic. The zeros of $V$ are the turning points by
definition. Away from those points we can construct formal local
solutions of the form 
\ekv{sta.2}{u(x)=a(x;h)e^{i\phi (x)/h},\ a(x;h)\sim
  a_0(x)+ha_1(x)+...,}
where $\phi (x)$ is a solution of the eiconal equation
\ekv{sta.3}{(\phi '(x))^2+V(x)=0,} and $a_0$, $a_1$, ... solve a
sequence of transport equations obtained from:
$$
((\phi '(x)+hD_x)^2+V(x))a=0,
$$ 
equivalent to 
$$
(\phi '(x)hD+hD\circ \phi '+(hD)^2)a=0:
$$
\begin{equation*}
2\phi '(x)\partial a_0+\phi ''a_0=0,\eqno{(T_0)}
\end{equation*}
and for $j\ge 1$:
\begin{equation*}
2\phi '(x)\partial a_j+\phi ''a_j=i\partial ^2a_{j-1}.\eqno{(T_j)}
\end{equation*}
We can prescribe $a_0(x_0)$, $a_1(x_0)$, ... (if $x_0$ is not a
turning point) and then the formal symbol becomes uniquely determined
in a neighborhood of $x_0$. The so called exact WKB method (see also
the appendix) tells us that if $\gamma :[0,1]\to \Omega $ is a $C^1$
curve with $\gamma (0)=x_0$, avoiding the turning points and with the
property that $-\Im \phi (\gamma (t))$ has positive
derivative\footnote{so that $e^{i\phi (x)/h}$ is exponentially growing
  with increasing $t$}, then there exists an exact holomorphic
solution of (\ref{sta.1}) of the form (\ref{sta.2}) in a neighborhood
of $\gamma (]0,1])$ where $a_0(x_0)$, $a_1(x_0)$, ... can be
arbitrarily prescribed (in the sense that $a(x;h)$ is holomorphic in
$x$ with the asymptotic expansion of (\ref{sta.2}) in the space of
holomorphic functions in a neighborhood of the range of $\gamma
$). Moreover, the solution is unique up to a term ${\cal O}(h^\infty
)e^{-\Im \phi /h}$.

Actually the formal expansion can be improved by using the Ansatz
$(\Phi ')^{-1/2}e^{i\Phi /h}$, and then determining $\Phi (x;h)\sim \phi
(x)+h^2\phi _2(x)+h^4\phi _4(x)+...$ from a Riccati type
equation. Notice that the solution of ($T_0$) is of the form 
$a_0(x)=C(\phi ')^{-1/2}=\widetilde{C}V(x)^{-1/4}$.

\par We can consider multivalued solutions of (\ref{sta.3}) away from
the turning points. A $C^1$ curve in $\Omega $ is called a Stokes line
if $\Im \phi $ is constant on $\gamma $ and it is called an
anti-Stokes line if $\Re \phi $ is constant. (Sometimes the
terminology is reversed.) Locally away from the turning points the
Stokes and anti-Stokes lines intersect each other perpendicularly. The
curve $\gamma $ in the above exact WKB remark necessarily intersects
the Stokes lines transversally.

\par A turning point $x_0\in \Omega $ is called a simple turning point
if it is a simple zero of $V$, so that 
\ekv{sta.4} { V'(x_0)\ne 0.  }
We next consider the singularity of the solution of the eiconal
equation near a simple turning point that we assume to be $x_0=0$ for
simplicity. If the Taylor expansion of $-V$ at $x=0$ is
$-V(x)=a^2x+{\cal O}(x^2)$, then $\phi '(x)$ is a double-valued
holomorphic function of the form
$$
\phi '(x)=ax^{\frac{1}{2}}(1+{\cal O}(x)),
$$ 
where the last factor is holomorphic in a full neighborhood of
$x=0$. By integration it is clear that $\phi $ is also double-valued
and of the form
$$
\phi (x)=\frac{2}{3}ax^{\frac{3}{2}}(1+{\cal O}(x)),
$$
where again the last factor is holomorphic near $0$.

\par The union of the Stokes and anti-Stokes curves reaching the
turning point $x=0$ is contained in
\ekv{sta.5}{\{x\in \mathrm{neigh\,}(0);\, \Im \phi =0\hbox{ or }\Re \phi =0\}=
\{x\in \mathrm{neigh\,}(0);\, \Im (\phi ^2) =0\} ,}
which is also the set of points $x$ solving
$$
a^2x^3(1+{\cal O}(x))=t^3,\ t\in \mathrm{neigh\,}(0,{\bf R}),
$$
i.e.
$$
a^{\frac{2}{3}}x(1+{\cal O}(x))=t,
$$
or equivalently
$$
x=f(a^{-\frac{2}{3}}t),
$$
where $f$ is analytic and $f(0)=0$, $f'(0)=1$. Since there are three
branches of the cubic root of $a$ we see that the set (\ref{sta.5}) is
the union of three smooth curves, $\gamma _j$, $j=0,1,2$ that pass
through $0$ and intersect there at angles $2\pi /3$.

With a suitable orientation, each $\gamma _j$ is first a Stokes line
$\gamma _j^-$ until it hits $0$ and then becomes an anti-Stokes line
$\gamma _j^+$ on the other side. It will be convenient to let $\gamma
_j^-$ be open in the sense that $0\notin \gamma _j^-$, $0\in \gamma
_j^+$. The three Stokes lines divide a pointed neighborhood into three
``Stokes sectors'' $\Sigma _j$, $j=0,1,2$, as indicated on the figure:

\centerline{
\includegraphics[height=200pt]{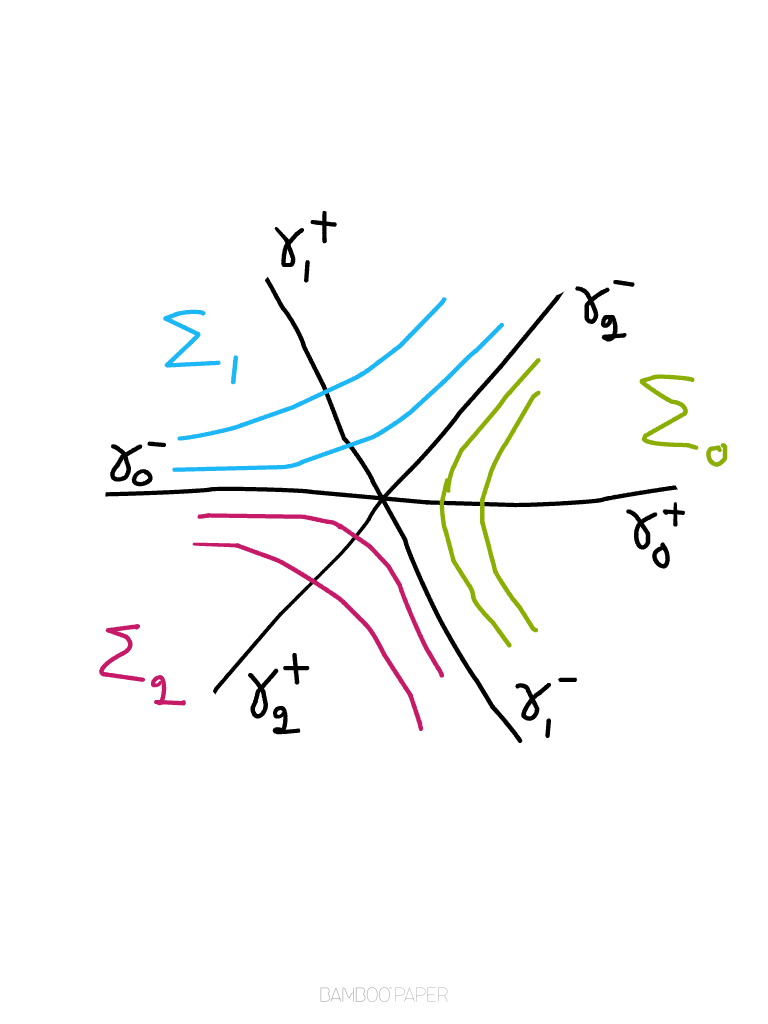}}

\par\noindent Each Stokes sector is the union of Stokes lines in
addition to the two Stokes lines that make up the boundary. In the
figure we draw two such additional lines in each sector.

For each $j\in {\bf Z}/3{\bf Z}$, we choose the branch $\phi =\phi ^j$
of the solution of the eiconal equation tending to $0$ when $x\to 0$
which has positive imaginary part on the interior of $\Sigma _j$ and
we can extend $\phi $ holomorphically to $\Omega \setminus
\overline{\gamma _j^-}$, so that $\phi ^j=-\phi ^{j\pm 1}$ in $\Sigma
_{j\pm 1}$. Here $\Omega $ is a fixed small open disc centered at
$0$. The exact WKB method tells us that (\ref{sta.1}) has a solution
$u=u_j$ in $\Omega $ of the following asymptotic form in $\Omega
\setminus \Gamma _j^-$, where $\Gamma _j^-$ is any fixed neighborhood
of $\overline{\gamma _j^-}$: \ekv{sta.6} { u_j=a^j(x;h)e^{i\phi
    _j(x)/h},\quad a^j\sim a_0^j+ha_1^j+...,\ a_0(x)\ne 0.  }

\par The Wronskian $W(u_j,u_k):=(hDu_j)u_k-u_jhDu_k$ is constant, and
can be computed asymptotically for $j\ne k$ at any point on $\gamma _\ell^-$ where
$\ell$ is the index different both from $j$ and $k$. Since $\phi
_j=-\phi _k$ there, we get 
\ekv{sta.7}
{
W(u_j,u_k)=2\phi _j'a_0^ja_0^k+{\cal O}(h).
}
Also recall that $W(u,u)=0$. 

\par This can be used to study $u_j$ near $\gamma _j^-$. Since the
space of solutions of (\ref{sta.1}) is of dimension 2, we have
\ekv{sta.8}
{
u_j=\sum_{k;\, k\ne j}c_{j,k}u_k,\quad c_{j,k}=c_{j,k}(h)\in {\bf C},
}
and if $k\ne j$, we let $\ell =\ell (j,k)$ be the index different both
from $j$ and $k$ and get
$$
W(u_j,u_\ell)=c_{j,k}W(u_k,u_\ell ),
$$
\ekv{sta.9}
{
c_{j,k}=\frac{W(u_j,u_\ell)}{W(u_k,u_\ell)}\sim
c_{j,k}^0+hc_{j,k}^1+...,\quad c_{j,k}^0\ne 0.
}

We shall next show that the presentation (\ref{sta.6}) extends to
$\Omega \setminus (\Gamma _j^-\cup D(0,Ch^{2/3}))$ where now $\Gamma
_j^-$ is a conic neighborhood of $\gamma _j^-$ and $C\gg 1$, in the
sense that the asymptotic expansion for $a^j$ is in powers of
$h/x^{3/2}$. Letting $j$ be fixed for a while, we suppress ``$j$'' from
the notation. Recall that $a_0$, $a_1$ are determined by the sequence
of transport equations ($T_0$), ($T_1$), ... above. Using the eiconal
equation for $\phi $ we get \ekv{sta.10} {
\partial (V^\frac{1}{4}a_0)=0,\ \partial
(V^\frac{1}{4}a_k)=\frac{1}{2}V^{-\frac{1}{4}}\partial ^2a_{k-1}.
}
Starting with $a_0=\mathrm{Const.\,}V^{-\frac{1}{4}}={\cal
  O}(x^{-\frac{1}{4}})$ and using (\ref{sta.10}) and the Cauchy
inequalities, we get iteratively that
\ekv{sta.11}{a_k(x)={\cal O}(x^{-\frac{1}{4}-k\frac{3}{2}}),\ x\to 0. }
Thus, we can give a meaning to  
$$
\sum _0^\infty a_kh^k=\sum_0^\infty (x^{k\frac{3}{2}}a_k)\left(\frac{h}{x^{\frac{3}{2}}}\right)^k,
$$
in the region $|x|\gg h^{2/3}$ as an asymptotic sum in powers of the
small parameter $h/x^{3/2}$.

\par  In the appendix, we show that the holomorphic function $a$ has this
asymptotic expansion in the region $|x|\gg h^{2/3}$. 
\begin{prop}\label{sta0}
  Fix $j\in {\bf Z}/3{\bf Z}$ and let $u=u_j$ be a solution of
  (\ref{sta.1}), which has the structure (\ref{sta.6}) in a
  neighborhood of a point $x_0^+\in \gamma _j^+\setminus \{ 0\}$. Then
  for $r>0$ small enough, $u$ remains of the form (\ref{sta.6}) in
  $D(0,r)\setminus (\Gamma _j^-\cup D(0,Ch^{2/3}))$, $\Gamma _j$ is
  any neighborhood of $\gamma _j^-$ of the form $\cup_{x\in \gamma
    _j^-}D(x,\epsilon |x|)$ and where $C=C_\epsilon >0$ is large
  enough. The coefficients $a_k^j$ in (\ref{sta.6}) satisfy
  (\ref{sta.11}) and the precise meaning of the asymptotics in
  (\ref{sta.6}) is that
  \ekv{sta.14}{a^j-\sum_{k=0}^{N-1}a_k^jh^k={\cal
      O}\left(x^{-\frac{1}{4}}\left( \frac{h}{x^{\frac{3}{2}}}
      \right)^N\right).}
\end{prop}

We shall next estimate the region where $u=u_0$ may have its zeros
and take $j=0$ in order to fix the ideas. From Proposition
\ref{sta0} it is clear that the zeros have to be close to $\overline{\gamma
_0^-}$ and in particular we need to study what happens in an $h^{2/3}$
neighborhood of $0$, where we have no asymptotics. If $[a,b]\ni t\to
\gamma (t)\in {\bf C}$ is a smooth curve and $v$, $w$ are holomorphic
functions defined near $\gamma $, then
$$
\int_\gamma vwdx=\int_a^bv_\gamma w_\gamma dt,
$$
where we define 
$$u_\gamma (t)=\dot{\gamma }^{\frac{1}{2}}u(\gamma (t)).$$ This means
that the passage $u\mapsto u_\gamma $ conserves symmetry of
differential operators, and more precisely, we check that 
$$
(Du)_\gamma =\dot{\gamma }^{-\frac{1}{2}}D_t\dot{\gamma
}^{-\frac{1}{2}}u_\gamma ,
$$
and the equation (\ref{sta.1}) restricted to $\gamma $ reads
\ekv{sta.17}
{
[(\dot{\gamma }^{-\frac{1}{2}}hD_t\dot{\gamma
}^{-\frac{1}{2}})^2+V(\gamma (t))]u_\gamma =0
}
Here we can rework the first term and put the two $D_t$ together in
the center. We get
\ekv{sta.20}
{
(-(h\partial _t)^2+\dot{\gamma }^2\widetilde{V})\dot{\gamma
}^{-1}u_\gamma =0,\ \dot{\gamma }^{-1}u_\gamma =\dot{\gamma
}^{-\frac{1}{2}}u\circ \gamma ,
}
where \ekv{sta.19}
{
\widetilde{V}=V(\gamma (t))+(\frac{h}{\dot{\gamma }})^2[\frac{1}{4}(\frac{\ddot{\gamma }}{\dot{\gamma
  }})^2-\frac{1}{2}\partial _t(\frac{\ddot{\gamma
  }}{\dot{\gamma }})]=V\circ \gamma +{\cal O}(h^2).
}
\begin{prop}\label{sta1}
If $\gamma $ is a Stokes curve or an anti-Stokes curve, we have $\Im
(\dot{\gamma }^2V\circ \gamma )=0$. More precisely, $\dot{\gamma
}^2V\circ \gamma $ is $<0$ in the first case and $>0$ in the second case.
\end{prop}
\begin{proof}
  Stokes and anti-Stokes curves are characterized by the property that
  $\Im \dot{\gamma }\phi '=0$ and $\Re \dot{\gamma }\phi '=0$
  respectively, where $\phi $ solves the eiconal equation
  (\ref{sta.3}). For both types of curves, we have $\Im (\dot{\gamma
  }\phi ')^2=0$ which means that $\Im
(\dot{\gamma }^2V\circ \gamma )=0$. On a Stokes curve we have
$(\dot{\gamma }\phi ')^2>0$, so $\dot{\gamma }^2V\circ \gamma <0$ and
on an anti-Stokes curve we have
$(\dot{\gamma }\phi ')^2<0$, so $\dot{\gamma }^2V\circ \gamma >0$.
\end{proof}

\par Now complete $\gamma _0$ into a smooth family of curves $\gamma _s$,
$s\in \mathrm{neigh\,}(0,{\bf R})$, so that $x=\gamma _s(t)$ defines
local coordinates $s,t$ and
 the smooth function
$$f(s,t)=\Im [(\partial _t\gamma _s)^2V(\gamma _s(t))]$$
vanishes for $s=0$. Assuming, as we may, that $\gamma _0(0)=0$, $\gamma
_0(t)=\gamma _0^\pm(t)$, for $\pm t>0$, we get for $s=0$:
$$
(\partial _sf)(0,0)=\Im (\dot{\gamma }_0^2V'(0)\partial _s\gamma _s(0)).
$$
This is $\ne 0$ since $V'(0)\ne 0$ and $\partial _s\gamma _s(0)_{s=0}$ is
not colinear with $\dot{\gamma }_0$. It follows that
$\pm f(s,t)\asymp s$ and we may assume that the plus sign is valid;
\ekv{sta.22}
{
\Im [(\partial _t\gamma _s(t))^2V(\gamma _s(t))]\asymp s,\ (s,t)\in \mathrm{neigh\,}(0).
}

Now let $u=u_0$ be a solution of (\ref{sta.1}) as in (\ref{sta.10}) which is exponentially
decaying in the Stokes sector $\Sigma _0 $ containing the anti-Stokes
line $\gamma _0^+$. 
\begin{prop}\label{sta2}
The zeros of $u_0$ are within a distance ${\cal O}(h^2)$ from $\gamma
_0^-$ and away from a disc $D(0,h^{2/3}/C)$ if $C>0$ is large enough.
\end{prop}
\begin{proof}
We first prove that the zeros are within a distance ${\cal O}(h^2)$
from $\gamma _0$. From the WKB structure we already know that they have
to be inside a small neighborhood of $\{0\}\cup \gamma _0^-$. Let
$x_0$ be a zero of $u$ and let $s=s_0$ be determined by the property
that $x_0$ belongs to $\gamma _{s_0}$, so that $x_0=\gamma
_{s_0}(t_0)$ for $-1/{\cal O}(1)\le t_0\le o(1)$. Take $\gamma
=\gamma _{s_0}$ in (\ref{sta.20}): Multiplying by
$\overline{\dot{\gamma }^{-1/2}u\circ \gamma }$, we get
$$
\int_{t_0}^1[((-h\partial _t)^2+\dot{\gamma
}^2\widetilde{V})\dot{\gamma }^{-\frac{1}{2}}u\circ \gamma )
]\overline{\dot{\gamma }^{-1/2}u\circ \gamma }dt=0.
$$
Here $u\circ \gamma $ is exponentially decaying for $t\ge 1/{\cal
  O}(1)$ and vanishes at $t_0$ so we can integrate by parts and get
\ekv{sta.23}
{
\int_{t_0}^1 [|h\partial _t(\dot{\gamma }^{-\frac{1}{2}}u\circ \gamma
)|^2
+\dot{\gamma }^2\widetilde{V}|\dot{\gamma }^{-\frac{1}{2}}u\circ
\gamma|^2] dt={\cal O}(e^{-\frac{1}{Ch}}).
}
Now $\Im \dot{\gamma}^2\widetilde{V}=\Im (\dot{\gamma }^2V\circ
\gamma )+{\cal O}(h^2) $ and $\Im (\dot{\gamma }^2V\circ \gamma
)\asymp s_0$, so taking the imaginary part of (\ref{sta.23}), we get 
$$
(|s_0|-{\cal O}(h^2))\int_{t_0}^1|\dot{\gamma }^{-\frac{1}{2}}u\circ
\gamma |^2dt \le {\cal O}(e^{-\frac{1}{Ch}}).
$$
Consequently, $s_0={\cal O}(h^2)$ so the zero is at a distance $\le
{\cal O}(h^2)$ from $\gamma _0$.

\par It remains to prove that the zeros stay away from
$D(0,h^{2/3}/C)$ and belong to a $h^2$-neighborhood of $\gamma
_0^-$. Let $x_0=\gamma _{s_0}(t_0)$ be a zero so that $s_0={\cal
  O}(h^2)$. Then, with $\gamma =\gamma _{s_0}$ we have $\Re
\dot{\gamma }^2{V}\asymp t$. Let $v=\dot{\gamma }^{-1/2}u\circ \gamma
$ and take the real part of (\ref{sta.23}):
\ekv{sta.24}
{
\int_{t_0}^1 (|h\partial _tv|^2+\Re (\dot{\gamma
}^2\widetilde{V})|v|^2)dt = {\cal O}(e^{-\frac{1}{Ch}}).
}
Now, 
$$\Re (\dot{\gamma }^2\widetilde{V})\ge \frac{t-t_0}{C}-C(|t_0|+h^2)$$
and we get 
$$
\int_{t_0}^1(|h\partial _tv|^2+\frac{t-t_0}{C}|v|^2)dt\le {\cal
  O}(e^{-\frac{1}{Ch}})+C(|t_0|+h^2)\Vert v\Vert^2,
$$
where the norm is the one in $L^2([t_0,1])$. Here, we can drop the
first term to the right since $\|v\|$ is bounded from below by a power
of $h$. On the other hand, we know (either by using well-known facts
about the Dirichlet problem for the Airy operator or by more direct
arguments) that the left hand side is bounded from below by
$C^{-1}h^{2/3}\Vert v\Vert^2$ (using also that $v(1)$ is exponentially
small). Hence, 
$$
\frac{h^{\frac{2}{3}}}{C}\le C(|t_0|+h^2),
$$
leading to 
$$
|t_0|\ge \frac{h^{\frac{2}{3}}}{\widetilde{C}}.
$$
Now a second look at (\ref{sta.24}) shows that we cannot have $t_0\ge
h^{2/3}/\widetilde{C}$, and the proof is complete.
\end{proof}
\begin{remark}\label{sta3}
{\rm By pushing the argument slightly further we see that every zero of
$u_0$ in any fixed disc $D(0,Ch^{2/3})$ is of the form 
\ekv{sta.25}
{
-h^{\frac{2}{3}}V'(0)^{-\frac{1}{3}}\zeta _j+{\cal O}(h^{\frac{4}{3}}),
}
for some $j$, where $0<\zeta _1<\zeta _2<...$ are the zeros of
$\mathrm{Ai}(-t)$.}
\end{remark}

\par In fact, let $x_1$ be such a zero and consider the equation
(\ref{sta.20}) along the curve $\gamma =\gamma _s$ that contains
$x_1$. Assume that the parametrization is chosen with $\gamma (0)=x_1$ and
such that $\gamma $ is oriented in the direction of $\Sigma _0$ for
increasing $t$. Choose a similar parametrization of $\gamma _0$ so
that $\gamma (t)-\gamma _0(t)={\cal O}(h^2)$. Pulling $\dot{\gamma
}^{-\frac{1}{2}}u\circ \gamma $ to $\gamma _0$ by means of $\gamma
\circ \gamma _0^{-1}$, we get a quasi-mode $\widetilde{u}(t)$
satisfying
\ekv{sta.26}
{
(-(h\partial _t)^2+\dot{\gamma }_0^2V(\gamma
_0(t)))\widetilde{u}(t)={\cal O}(h^2)\Vert \widetilde{u}\Vert\hbox{ in
}L^2([0,\frac{1}{C_0}]),
}
which is exponentially decaying for $t\gg h^{2/3}$ and satisfies the
Dirichlet condition $\widetilde{u}(0)=0$. This means that the
self-adjoint Dirichlet realization on $[0,1/C_0]$ of the operator to
the left in (\ref{sta.26}) has an eigenvalue $={\cal O}(h^2)$. Now it
is a routine exercise in self-adjoint semi-classical analysis to see
that the eigenvalues of this operator in any interval $]-\infty
,Ch^{2/3}]$ are of the form
\ekv{sta.27}
{
U(0)+h^{\frac{2}{3}}U'(0)^{\frac{2}{3}}\zeta _j+{\cal O}(h^{\frac{4}{3}}),
}
where $U(t)=\dot{\gamma }_0^2V(\gamma _0(t))$ is the potential in
(\ref{sta.26}). Thus for some $j$,
$$
\dot{\gamma }_0(0)^2V(\gamma _0(0))+h^{\frac{2}{3}}(\dot{\gamma
}_0(0)^3 V'(\gamma _0(0)))^{\frac{2}{3}}\zeta _j={\cal O}(h^{\frac{4}{3}}),
$$
which simplifies to
$$
V(x_1)+h^{\frac{2}{3}}V'(0)^\frac{2}{3}\zeta _j+{\cal O}(h^{\frac{4}{3}})=0,
$$
leading to (\ref{sta.25}). 

\par The remark \ref{sta3} allows us to control the exterior Dirichlet
problem for $\Im z \ge -c_0h^{2/3}$ for $c_0$ as in (\ref{outl.0}).

\subsection{The exterior ODE}\label{cw}
We are concerned with the operator
\ekv{cw.1}
{
P=-(h\partial _x)^2-xQ(x)+ha(x)h\partial _x,
}
where $Q$, $a$ are holomorphic on $\mathrm{neigh\,}(0,{\bf C})$ and
$Q>0$ on the real domain. 

Let $\gamma _\delta $ be the contour $x=\gamma _\delta (s)$, $0\le
s\le s_0$, $0<s_0\ll 1$,
\ekv{cw.4}{\begin{split}\gamma _\delta (s)&=s,\hbox{ for } 0\le s\le
    \delta ,\\ \gamma _\delta
    (s)&=\delta +e^{\frac{i\pi }{3}}(s-\delta )\hbox{ for }\delta \le
    s\le s_0,\end{split}}
and let $b=\gamma _\delta (s_0)$ be the second end point. Here $\delta \ge 0$ is a small parameter that eventually will take the
values $0$ and $Ch$.

Consider the Dirichlet problem
\ekv{cw.2}{(P-z )u=v\hbox{ on }\gamma _\delta ,\quad u(0)=0,\ u(b)=0,}
where 
\ekv{cw.3}
{
z =\lambda +h^{2/3}w,\ \lambda \in {\bf R},\ |w|\le \frac{1}{{\cal
  O}(1)}.
}

\par We start by discussing the case $\delta =0$ and later we indicate
the additional arguments in order to treat the case $\delta >0$. When
$\delta =0$, the operator reduces to the rotated Airy operator with a
perturbation, \ekv{cw.5} { e^{-\frac{2\pi i}{3}}(-(h\partial
  _s)^2+sQ(e^{\frac{\pi i}{3}}s))+e^{-\frac{\pi i}{3}}ha(e^{\frac{\pi
      i}{3}}s)h\partial _s, } which as in \cite{HaLe94,
  SjZw2,SjZw3,SjZw4} can be treated by ressorting to the spectral
theory for the Dirichlet problem for the Airy operator. When $\delta
>0$ this appeared as more difficult and in order to cover that case
also we chose to use the complex WKB method. The last term
$ha(x)h\partial _x$ will have no real importance and can be eliminated
by writing
$$
P=-(h\partial _x-\frac{h}{2}a(x))^2-xQ(x)+{\cal
  O}(h^2)=e^{\frac{A}{2}}
[-(h\partial _x)^2-xQ(x)+{\cal O}(h^2)]e^{-\frac{A}{2}},
$$ 
where $A={\cal O}(1)$ is a primitive of $a$. Since the perturbation
${\cal O}(h^2)$ can be absorbed in the estimates below, we will assume
from now on that $a=0$. We will also concentrate on the most
interesting case when $|\lambda |\le 1/C$ and indicate later how to
treat the easier cases when $\lambda $ is positive and bounded from
above as well as the case when $\lambda $ is negative and arbitrarily
large. 

\par Assuming that $|\lambda |\le 1/C$, we see that the equation
(\ref{cw.2}) has a turning point $x_0(z )$, given by \ekv{cw.6} {
  x_0Q(x_0)+z =0.  } If $x_1\in {\bf R}$ is the real turning
point, given by $x_1Q(x_1)+\lambda =0$, then
\ekv{cw.6.5}{x_0=x_1-\frac{1}{\partial V(x_1)}h^{\frac{2}{3}}w+{\cal
  O}(h^{\frac{4}{3}}),\hbox{ where }V(x)=xQ(x).}  We have the following
picture:\\ 
\centerline{
\includegraphics[height=180pt]{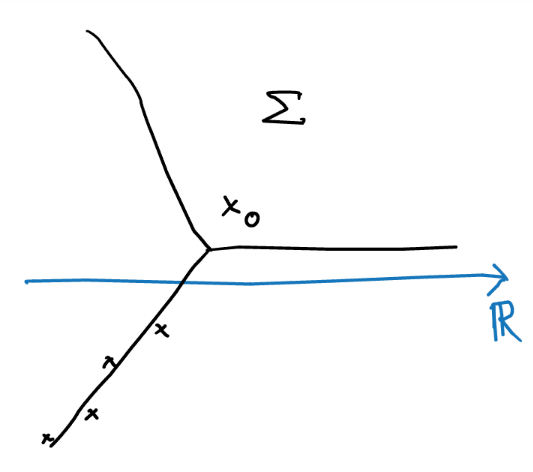}}\\
 where we
draw the three Stokes lines through $x_0$, the Stokes sector $\Sigma$,
and notice that the zeros of the corresponding subdominant solution
are very close to the Stokes line $\gamma _0^-$ opposite to $\Sigma $
and separated from the turning point by a distance $\ge h^{2/3}/C$. A
direct calculation from (\ref{cw.6}), (\ref{sta.25}) shows that the
imaginary parts of these zeros are $\le -h^{2/3}/{\cal O}(1)$ when
$|\lambda |\ll 1$ and $\Im w\ge -Q(0)^{2/3}\zeta _1\cos \frac{\pi
}{6}+1/{\cal O}(1)$.

\par From Proposition \ref{sta0}, we see that the equation $(P-z
)u=0$ has a solution which is subdominant in $\Sigma $, of the form
\ekv{cw.7}
{
e^{-\phi (x;h)/h}
}
in $(\mathrm{neigh\,}(x_0,{\bf C})\setminus V_0^-)\cup
D(x_0,h^{2/3}/C)$ where $V_0^-$ is a any small ``conic'' neighborhood of
$\gamma _0^-$ as in Proposition \ref{sta0}, such that 
\ekv{cw.8}{
\phi '(x;h)=\phi _0'(x)+\frac{{\cal O}(h)}{x-x_0}} 
and $\phi _0$ solves the eiconal equation, $ (\phi _0')^2=xQ(x)+z
$. (Compared to Proposition \ref{sta0}, we have found it convenient to
drop the prefactor ``$i$''.) Notice that the first term in the right
hand side of (\ref{cw.8}) dominates when $|x-x_0|\gg h^{2/3}$.  

\par Moreover, in any set of the form $D(x_0,h^{3/2}/C)\cup (D(x_0,Ch^{3/2})\setminus V_0^-)$,
we have
\ekv{cw.9}
{
\phi '={\cal O}(h^{1/3}).
}
In fact, writing $x-x_0=h^{2/3}y$ leads to the equation $-(\partial _y^2+W(y))u=0$ in a fixed
$h$-independent domain where $W$ is holomorphic and bounded. Rewriting
this as a first order system, we see that $|u(y)|+|\partial _yu(y)|$
is of constant order of magnitude, say $\asymp 1$ and the equation
tells us that $\partial _y^2u={\cal O}(1)$. We also know that $u$ is
non-vanishing and after shrinking the domain by a fixed rate arbitrarily
close to 1, we conclude that $|u(y)|\ge 1/{\cal O}(1)$. Indeed, if
$|u(y_0)|=\epsilon \ll 1$, then $|u'(y_0)|\asymp 1$ and from the Taylor
expansion, $u(y)=u(y_0)+u'(y_0)(y-y_0)+{\cal O}((y-y_0)^2)$, we see
that $u$ must have a zero in the disc $D(y_0,r)$ if $\epsilon \ll r\ll
1$. Thus $|u(y)|\asymp 1$, $u'(y)={\cal O}(1)$ and hence $\partial
_y\ln u={\cal O}(1)$. Hence $h^{2/3}\partial _x\ln u={\cal O}(1)$ and
$\partial _x\phi =h\partial _x\ln u={\cal O}(h^{1/3})$ as claimed.

\par As in Section \ref{nf} we factor $P-z $ as 
\ekv{cw.10}
{
P-z =(\phi '-h\partial _x)(\phi '+h\partial _x)
}
and we shall use this to find a solution $u$ of the equation $(P-z
)u=v$.
 First invert $\phi '-h\partial _x$
by integration from $b$ to get 
\ekv{cw.11}
{
(\phi '+h\partial _x)u=-\frac{1}{h}\int_{b}^x e^{(\phi (x)-\phi (y))/h}v(y)dy=:Kv(x).
}

\par In order to estimate the ${\cal L}(L^2)$-norm of this integral
operator and of similar ones, we collect some useful properties.
\begin{lemma}\label{cw0}
Assume that $0\le \delta \le Ch$ and orient $\gamma _\delta $ from $0$
to $b$. Write $y\prec x$ for $y,x\in \gamma _\delta $ if  $y$ precedes
$x$. For $x,y,w\in \gamma _\delta $ with $0\prec y\prec w\prec x\prec
b$ we have with a new constant $C>0$:
\ekv{cw.a}
{
\frac{1}{C}\int_y^x|\phi '(z)||dz|-Ch\le \Re \phi (x)-\Re \phi (y)\le
\int_y^x|\phi '(z)||dz|,
}
\ekv{cw.b}
{
\frac{1}{C}|\phi '(w)||x-y|-Ch\le \int _y^x|\phi '(z)||dz|\le C(|\phi '(x)||x-y|+h),
}

\ekv{cw.c}
{
\frac{1}{C_\epsilon }e^{-\frac{\epsilon }{h}\int_y^x|\phi '(z)||dz|}
\le \frac{h^{\frac{1}{3}}+|\phi '(x)|}{h^{\frac{1}{3}}+|\phi '(y)|}\le 
C_\epsilon e^{\frac{\epsilon }{h}\int_y^x|\phi '(z)||dz|},
}
for every $\epsilon >0$. Here $C_\epsilon >0$ is independent of $h$.
\end{lemma}
\begin{proof}
  The second inequality in (\ref{cw.a}) is obvious. By additivity it
  suffices to show the first inequality in each of the following three
  cases (where the second case may be void):
\begin{itemize}
\item[1)] $x,y$ belong to the horizontal segment $[0,\delta ]$, 
\item[2)] $x,y$ belong to $\gamma _\delta \cap D(x_0,Ch^{2/3})$,
\item[3)] $x,y$ are both beyond the cases 1 and 2.
\end{itemize}
In case 1) both $\int_y^x|\phi '(z)||dz|$ and $\Re (\phi (x)-\phi
(y))$ are ${\cal O}(h)$ since $\delta ={\cal O}(h)$. In the second case
this remains true since $|x-y|={\cal O}(h^{2/3})$ and $\phi '(z)={\cal
  O}(h^{1/3})$ for $y\prec z\prec x$. In the third case the first
inequality in (\ref{cw.a}) follows from the fact that $\gamma _\delta
$ is here transversal to the Stokes lines and more precisely that 
$$
\frac{d}{dt}\Re \phi (\gamma _\delta (t))\asymp |\phi '(\gamma _\delta
(t))|, \hbox{ for }y\prec \gamma _\delta (t)\prec x.
$$

Now consider (\ref{cw.b}). If $x$ is as in case 1 or 2 then
$\int_y^x|\phi '(z)||dz|$ and $|\phi '(w)||x-y|$ are ${\cal
  O}(h)$. If $x$ is as in case 3, then $|\phi '(x)|\ge
\frac{1}{C}|\phi '(w)|$ for $w\prec x$ and we get the desired inequalities.

We finally show (\ref{cw.c}).
Let $\mathrm{I}$ denote the modulus of the logarithmic derivative of
$h^{1/3}+|\phi '(x)|$ along $\gamma _\delta $. Then  $$\mathrm{I}\le \frac{|\phi
  ''|}{h^{\frac{1}{3}}+|\phi '(x)|}$$ which is ${\cal O}(h^{-2/3})$ on
$\gamma _\delta \cap D(x_0,Ch^{2/3})$ for every $C>0$, and on $\gamma
_\delta \setminus D(x_0,Ch^{2/3})$:
$$
\mathrm{I}={\cal
  O}(1)\frac{|x-x_0|^{-\frac{1}{2}}}{h^{\frac{1}{3}}+|x-x_0|^{\frac{1}{2}}}=\frac{{\cal
  O}(1)}{|x-x_0|}.
$$
Summing up the estimates in both regions, we have
$$
\mathrm{I}=\frac{{\cal O}(1)}{h^{\frac{2}{3}}+|x-x_0|}.
$$

\par
The modulus $\mathrm{II}$ of the logarithmic derivative with respect
to $x$ of $e^{\int_y^x|\phi ' (z)||dz|/h}$ is bounded by $|\phi
'(x)|/h$ which is ${\cal O}(h^{-2/3})$ in the first region and
$\asymp |x-x_0|^{1/2}/h$ in the second region,
provided that $C$ is large enough.

\par It follows that $\mathrm{I}\le \epsilon \mathrm{II}$, except in
the intersection of $\gamma _\delta $ with the disc $|x-x_0|\le
(h/\epsilon )^{2/3}$. The integrals of both $\mathrm{I}$ and
$\mathrm{II}$ over this exceptional region are ${\cal O}_\epsilon (1)$
and (\ref{cw.c}) follows.
\end{proof}

\begin{lemma}\label{cw0.5}
The ${\cal L}(L^2)$-norms of $(h^\frac{1}{3}+|\phi '|)\circ K$ and of
$(h^\frac{1}{3}+|\phi '|)^2\circ K\circ (h^\frac{1}{3}+|\phi '|)^{-1}$
are ${\cal O}(1)$. 
\end{lemma}

\begin{proof}
We first notice that we can replace $|\phi '(w)|$ to the left in
(\ref{cw.b}) by $|\phi '(w)|+h^{\frac{1}{3}}$. 

\par By Schur's lemma, the ${\cal L}(L^2)$-norm of $(h^\frac{1}{3}+|\phi
'|)\circ K$ is bounded by the geometric mean of the following two
quantities:
$$
\mathrm{I}=\frac{1}{h}\sup_{x\in \gamma _\delta }\int_b^x
(h^{\frac{1}{3}}+|\phi '(x)|)e^{\frac{1}{h}\Re (\phi (x)-\phi (y))}|dy|
$$ and
$$
\mathrm{II}=\frac{1}{h}\sup_{y\in \gamma _\delta }\int_0^y
(h^{\frac{1}{3}}+|\phi '(x)|)e^{\frac{1}{h}\Re (\phi (x)-\phi (y))}|dx|.
$$ 
Combining (\ref{cw.a}) and (\ref{cw.b}) with $|\phi '(w)|$ replaced by
$h^{\frac{1}{3}}+|\phi '(w)|$, we see that for $x\prec y$,
$$
(h^{\frac{1}{3}}+|\phi '(x)|)e^{\frac{1}{h}\Re (\phi (x)-\phi (y))}\le
(h^{\frac{1}{3}}+|\phi '(x)|)e^{C-\frac{1}{Ch}(h^{\frac{1}{3}}+|\phi '(x)|)|x-y|},
$$
implying that $\mathrm{I}={\cal O}(1)$.

\par In order to estimate $\mathrm{II}$, we also use (\ref{cw.c}) to
get
\begin{multline*}\frac{1}{h}(h^{\frac{1}{3}}+|\phi
  '(x)|)e^{\frac{1}{h}\Re (\phi (x)-\phi (y))}\\
\le \frac{1}{h}(h^{\frac{1}{3}}+|\phi
  '(x)|)e^{-\frac{1}{Ch}\int_y^x|\phi '(z)||dz|}\\
\le \frac{\widehat{C}}{h}(h^{\frac{1}{3}}+|\phi
  '(y)|)e^{-\frac{1}{2Ch}\int_y^x|\phi '(z)||dz|}\\
\le \frac{1}{h}(h^{\frac{1}{3}}+|\phi
  '(y)|)e^{\widetilde{C}-\frac{1}{\widetilde{C}h}(h^{\frac{1}{3}}+|\phi
    '(y)|)|x-y|},
\end{multline*}
and it follows that $\mathrm{II}$ is ${\cal O}(1)$. Thus the ${\cal
  L}(L^2)$-norm of $(h^{\frac{1}{3}}+|\phi '|)\circ K$ is ${\cal
  O}(1)$ as claimed.

\par The estimate of the norm of $(h^{\frac{1}{3}}+|\phi '|)^2\circ
K\circ (h^{\frac{1}{3}}+|\phi '|)^{-1}$ is just a slight variation of
the above arguments, using (\ref{cw.c}) from the start.
\end{proof}

From the definition of $K$ in (\ref{cw.11}) we get
\ekv{cw.13}
{
-h\partial _xKv=v-\phi '\circ Kv,
}
and we conclude that 
\ekv{cw.14}
{
h\partial _x\circ K,\ (h^{\frac{1}{3}}+|\phi '|)h\partial _x\circ
K\circ (h^{\frac{1}{3}}+|\phi '|)^{-1}\hbox{ are }{\cal O}(1)\hbox{ in
}{\cal L}(L^2).
}

\par Now, recall that we can get $u$ from $(\phi '+h\partial _x)u=:w$ by
integration outwards from $x=0$:
\ekv{cw.14.5}
{
u(x)=\frac{1}{h}\int_0^x e^{-(\phi (x)-\phi (y))/h}w(y)dy=:Lw.
}
The same estimates apply to $L$ and for the solution $u=LKv$ of the
equation $(P-z )u=v$, we get
\ekv{cw.15}
{\Vert (h^{\frac{1}{3}}+|\phi '|)^2u\Vert+\Vert (h^{\frac{1}{3}}+|\phi
  '|)h\partial _xu\Vert\le {\cal O}(1)\Vert v\Vert .
}
Recalling that 
$$
(P-z )=(\phi '-h\partial )(\phi '+h\partial )=(\phi ')^2-h\phi
''-(h\partial )^2,
$$
and that $\phi ''={\cal O}(h^{-\frac{1}{3}})$, we also get $\Vert
(h\partial )^2u\Vert\le {\cal O}(1)\Vert v\Vert$ and thus for $u=LKv$:
\ekv{cw.16}
{|\hskip -1pt |\hskip -1pt |u|\hskip -1pt |\hskip -1pt |:=\Vert (h^{\frac{1}{3}}+|\phi '|)^2u\Vert+\Vert (h^{\frac{1}{3}}+|\phi
  '|)h\partial _xu\Vert +\Vert (h\partial _x )^2u\Vert\le {\cal O}(1)\Vert v\Vert .
} 

\par By construction, $u(0)=0$, but the Dirichlet condition at $x=b$
is not necessarily fulfilled. Now, for instance by using a different
factorization $(P-z )=(\widetilde{\phi }'+h\partial
)(\widetilde{\phi }'-h\partial )$ and some easy iterations, we see
that the problem
\ekv{cw.17}
{(P-z )e_b=0,\ e_b(0)=0,\ e_b(b)=1 }
has a solution on $\gamma _\delta $ which decays exponentially away
from $b$ and satisfies $|\hskip -1pt |\hskip -1pt |e_b|\hskip -1pt
|\hskip -1pt |={\cal O}(h^{\frac{1}{2}})$. 

\par Moreover, we have
$u(b)={\cal O}(h^{-1/2})\Vert v\Vert$. In fact, (\ref{cw.16}) shows
that $\Vert u\Vert_{H_h^2}\le {\cal O}(1)\Vert v\Vert$, if we take the
$H^2_h$ norm over $\{ x\in \gamma _\delta ;\, a\prec x\prec b \}$, where
$a\in \gamma _\delta $ is close to $b$, and as in (\ref{al.10}), we
have $|u(b)|\le {\cal O}(h^{-1/2})\Vert u\Vert_{H_h^2}$. Thus the
function $\widetilde{u}=u-u(b)e_b$ solves $(P-z )\widetilde{u}=v$,
$\widetilde{u}(0)=\widetilde{u}(b)=0$ and (\ref{cw.16}) remains valid
with $u$ replaced by $\widetilde{u}$. Since our Dirichlet problem is
Fredholm of index zero, we also know that $\widetilde{u}$ is the unique
solution. Dropping the tildes we get:

\begin{prop}\label{cw1}
Consider the problem (\ref{cw.2}) for $z $ as in (\ref{cw.3}) with
$\lambda =1/{\cal O}(1)$ and let $u$ be the unique solution
constructed above. Then,
\ekv{cw.22}
{
\Vert (h^{\frac{1}{3}}+|\phi '|)^2u\Vert + \Vert (h\partial
_x)^2u\Vert
+\Vert (h^{\frac{1}{3}}+|\phi '|)h\partial _xu\Vert \le {\cal O}(1)
\Vert v\Vert ,
}
where the $L^2$ norms are taken over $\gamma _\delta $.
\end{prop}

\par We make a few remarks about extensions and variants. The first is
that we can replace $\phi $ in (\ref{cw.22}) with $\phi _0$, the
solution of the eiconal equation, $(\phi _0')^2=xQ(x)+z $. Indeed,
when $|x-x_0|\le {\cal O}(h^{2/3})$ we have $\phi ',\, \phi _0'={\cal
  O}(h^{1/3})$ and when $|x-x_0|\ge Ch^{2/3}$, then $|\phi '|\asymp
|\phi _0'|$.

\par The second observation is that along $\gamma _\delta $, if we let $x_1$
denote the real turning point (given by $x_1Q(x_1)+\lambda =0$,
$x_1\asymp -\lambda $, then
\begin{multline*}
h^{\frac{1}{3}}+|\phi _0'|\asymp
h^{\frac{1}{3}}+|x-x_0|^{\frac{1}{2}}\\
\asymp (|x-x_0|+h^{\frac{2}{3}})^{\frac{1}{2}}\asymp
(|x-x_1|+h^{\frac{2}{3}})^\frac{1}{2}\\
\asymp (s+|\lambda |+h^{\frac{2}{3}})^{\frac{1}{2}},
\end{multline*}
where we write $x=\gamma _\delta (s)$. Thus (\ref{cw.22}) can be written
\ekv{cw.23}
{
\Vert (h^{\frac{2}{3}}+|\lambda |+s)u\Vert+\Vert (h\partial
_x)^2u\Vert +\Vert (h^{\frac{2}{3}}+|\lambda
|+s)^{\frac{1}{2}}h\partial _xu\Vert \le {\cal O}(1)\Vert v\Vert .
}

\section{Parametrix for the exterior Dirichlet problem}\label{aed}
\setcounter{equation}{0}

Choose geodesic coordinates $(x',x_n)$ with $x'$ being local
coordinates on $\partial {\cal O}$, so that the exterior of ${\cal O}$
is locally given by $x_n>0$ and $P=-h^2\Delta $ in ${\bf R}^n\setminus
{\cal O}$ becomes (locally near a boundary point):
\ekv{aed.1}
{
P=(hD_{x_n})^2+R(x',hD_{x'})-x_nQ(x,hD_{x'})+ha(x)hD_{x_n}.
}
(Cf. (\ref{dcs.16}), (\ref{dcs.17}), (\ref{dcs.18}).)
Here $R$ is an elliptic second order differential operator with
principal symbol $r(x',\xi ')=|\xi '|^2$. Similarly, $Q$ is
elliptic in the $x'$ variables with principal symbol $q(x,\xi ')\asymp
|\xi '|^2$. For $z =\lambda +h^{2/3}w$ with $\lambda \in {\bf R}$,
$\lambda \sim 1$, $|w|\le 1/{\cal O}(1)$, we consider
\ekv{aed.2}
{\begin{split}
&P(x',\xi ')-z =P(x',x_n,\xi ',hD_{x_n})-z \\
&=
(hD_{x_n})^2+R(x',\xi ')-x_nQ(x,\xi ')+ha(x)hD_{x_n}-z 
\end{split}
}
as an ODO-valued symbol. We let $x_n$ vary in $\gamma _\delta $, $0\le
\delta \le Ch$. 

\par We investigate 3 different regions in $T^*\partial
{\cal O}$.

\par 1) $(x',\xi ')$ belongs to a small neighborhood of the glancing
hypersurface ${\cal G}$: $r(x',\xi ')=\lambda $. Then the estimates in Subsection
\ref{cw} apply with $\lambda $ there replaced by $\lambda -r(x',\xi
')$ and from (\ref{cw.23}) we get 
\label{aed.3}
\begin{multline}\label{aed.3}
\Vert (h^{\frac{2}{3}}+|\lambda -r(x',\xi ')|+s)u\Vert +\Vert (h\partial_{x_n})^2u \Vert
+\\ \Vert (h^{\frac{2}{3}}+|\lambda -r(x',\xi
')|+s)^{\frac{1}{2}}h\partial_{x_n}u\Vert\le {\cal O}(1)\Vert v\Vert ,
\end{multline}
when $(P(x',\xi ')-z )u=v$ along $\gamma _\delta $, $u(0)=u(b)=0$.

\par 2) $(x',\xi ')$ belongs to the hyperbolic region $r(x',\xi ')\le
\lambda -1/{\cal O}(1)$. Then the turning point $x_0$ is away from $0$
and hence also from $\gamma _\delta $ and the estimates of Section
\ref{cw} still apply and give (\ref{aed.3}), where we notice that
$h^{\frac{2}{3}}+|\lambda -r(x',\xi ')|+s\asymp 1$:
\ekv{aed.4}
{
\Vert u\Vert + \Vert h\partial_{x_n}u\Vert + \Vert (h\partial_{x_n})^2u\Vert
\le {\cal O}(1)\Vert v\Vert .
}
Notice that $q$ may be very small in this region but the estimates now
work without any reference to a turning point.

\par 3) $(x',\xi ')$ belongs to the elliptic region $r(x',\xi ')\ge
\lambda +1/{\cal O}(1)$. When in addition $r(x',\xi ')\le {\cal O}(1)$
we get (\ref{aed.4}) again. When $r(x',\xi ')\gg 1$ we multiply with
$|\xi '|^{-2}$ and get
$$
|\xi '|^{-2}(P(x',\xi ')-z
)=(\widetilde{h}D_{x_n})^2+\widetilde{R}-x_n\widetilde{Q}+\widetilde{h}a(x_n)\widetilde{h}D_{x_n}-\widetilde{z
}=\widetilde{P}-\widetilde{z },
$$
where $\widetilde{R}=|\xi '|^{-2}R(x',\xi ')\asymp 1$,
$\widetilde{Q}=|\xi '|^{-2}Q\asymp 1$, $\widetilde{h}=h/|\xi '|\ll 1$,
$\widetilde{z }=z /|\xi '|^2$, $|\widetilde{z }|\ll
1$. For the rescaled problem the turning point is well off to the
right and $\gamma _\delta $ intersects the Stokes lines
transversally. We still get (\ref{aed.4}), now for
$(\widetilde{P}-\widetilde{z })u=v$ and $h$ replaced by
$\widetilde{h}$ and after scaling back, we get
\ekv{aed.5} { \langle \xi '\rangle ^2\Vert u\Vert+ \langle \xi
  '\rangle \Vert h\partial_{x_n}u\Vert +\Vert (h\partial_{x_n})^2u\Vert \le
  {\cal O}(1)\Vert v\Vert }
for solutions of (\ref{cw.2}). 

\par For  a fixed $\delta \in \{ 0,Ch\}$, let ${\cal B}(x',\xi ')$ be
the space of functions on $\gamma _\delta $ vanishing at both end points and
equipped with the norm given be the left hand side of (\ref{aed.3}),
(\ref{aed.4}), (\ref{aed.5}) respectively when $(x',\xi ')$ is as in
the three cases. 

\par Then $P(x',\xi ')-z ={\cal O}(1):{\cal B}(x',\xi ')\to
L^2(\gamma _\delta )$
and has an inverse $E(x',\xi ')$ which is ${\cal
  O}(1):L^2(\gamma _\delta )\to {\cal B}(x',\xi ')$. 

\par Outside a fixed neighborhood of the glancing hypersurface, we
have the nice symbol properties
\ekv{aed.6}
{
\partial _{x'}^\alpha \partial _{\xi '}^\beta P={\cal O}_{\alpha
  ,\beta }(\langle \xi '\rangle^{-|\beta |}):\, {\cal B}(x',\xi ')\to
L^2(\gamma _{\delta }).
}
Near the glancing hypersurface we have a poblem when derivatives fall
on $R$ and we get the weaker estimate
\ekv{aed.7}
{
\partial _{x'}^\alpha \partial _{\xi '}^\beta P={\cal O}_{\alpha
  ,\beta }(1)(h^{\frac{2}{3}}+|\lambda -r(x',\xi ')|)^{-(|\alpha
  |+|\beta |)}.
}
This is the reason why traditionally (as in \cite{SjZw4,SjZw5} and
other works cited there) one uses some form of second
microlocalization. If $(x_0,\xi _0)$ is a point on the glancing
hypersurface, we conjugate $P(x,hD)$ with a microlocally defined
elliptic Fourier integral operator acting in the tangential variables
and get a new operator of the form (\ref{aed.1}) where now $R$, $Q$
are tangential classical $h$-pseudodifferential operators and $a$ is
replaced by $a(x,hD_{x'};h)$, a classical pseudodifferential operator
of order 0 in $h$, and where 
\ekv{aed.8}
{
R(x',\xi ')=\xi _1.
}
(See Sections 4 and 5 in \cite{SjZw4} and \cite{SjZw5} respectively.)
Then the problem appears only when we differentiate with respect to
$\xi _1$:
\ekv{aed.9}
{
\partial _{x'}^\alpha \partial _{\xi '}^\beta P={\cal O}_{\alpha
  ,\beta }(1)(h^{\frac{2}{3}}+|\lambda -r(x',\xi ')|)^{-\beta _1}.
}
Differentiating the identity $(P-z )E=1$, we get with $\partial
^\alpha =\partial _{x',\xi '}^\alpha $:
$$
(P-z )\partial ^{\alpha }E=\sum_{\alpha '+\alpha ''=\alpha \atop
  \alpha '\ne 0} c_{\alpha ',\alpha ''}(\partial
^{\alpha '}P) (\partial ^{\alpha ''}E),
$$
and after applying $E$ to the right and using that $E(P-z )=1$,
$$
\partial ^\alpha E=\sum_{\alpha '+\alpha ''=\alpha \atop
  \alpha '\ne 0} c_{\alpha ',\alpha ''}E(\partial
^{\alpha '}P)(\partial ^{\alpha ''}E) .
$$
By induction we then get
\ekv{aed.10}
{
\partial _{x'}^\alpha \partial _{\xi '}^\beta E={\cal O}_{\alpha
  ,\beta }(\langle \xi '\rangle^{-|\beta |}):\, 
L^2(\gamma _{\delta })\to {\cal B}(x',\xi '),
}
outside any fixed neighborhood of the glancing hypersurface ${\cal G}$. 
Near any fixed point of ${\cal G}$, we get
\ekv{aed.11}
{
\partial _{x'}^\alpha \partial _{\xi '}^\beta E={\cal O}_{\alpha
  ,\beta }(1)(h^{\frac{2}{3}}+|\lambda -r(x',\xi ')|)^{-\beta _1},
}
after conjugation  with an elliptic tangential Fourier integral operator,
that reduces $R$ to $\xi _1$.

We now turn to the $n$-dimensional situation and recall the definition
of the singular contour $\Gamma _f$ in (\ref{dcs.2}) and its exterior
part $\Gamma _{\mathrm{ext},f}$, where $f$ satisfies
(\ref{dcs.21}). We take $\theta =\pi /3$ there and put $\Gamma
_0=\Gamma _f$. For $\delta >0$, let ${\cal O} _{-\delta }={\cal
  O}+B(0,\delta )$. Then $\mathrm{dist\,}(x,{\cal O}_{-\delta })=\max
(d(x)-\delta ,0)$. Let $f_\delta $ be as in (\ref{dcs.21}) with $d(x)$
replaced by $\mathrm{dist\,}(\cdot ,{\cal O}_{-\delta })$, still with
$\theta =\pi /3$. Put $\Gamma _\delta =\Gamma _{f_\delta }$. In this
section we only work on the exterior parts $\Gamma
_{\mathrm{ext},\delta }$ and for simplicity we drop the subscript
``ext''. Using geodesic coordinates we have
\ekv{aed.12}
{
\Gamma _{\delta ,b}:=\{x;\, x'\in \partial {\cal O},\ x_n\in \gamma
_\delta  \}\subset \Gamma _\delta .
}
(Later on we will also include ${\cal O}$ into the contour $\Gamma
_\delta $ and the $\Gamma _\delta $ above will then be renamed
$\Gamma _{\delta ,\mathrm{ext}}$.) 

Let $ {\cal B}_b$ be the space of functions $u=u(x',x_n)$ on $\Gamma
_{\delta ,b}$ with $u(x',0)=u(x',b)=0$ for which the norm 
\ekv{aed.14}{
\Vert u\Vert_{{\cal B}}=h^{\frac{2}{3}}\Vert u\Vert +\Vert
(R(x',hD_{x'})-\lambda )u\Vert+\Vert su\Vert +\Vert (h\partial
_{x_n})^2u\Vert 
}
is finite. 

Continuing to treat $P$ as a pseudodifferential operator on $\partial
{\cal O}$ with operator valued symbol, we obtain a right parametrix of
$P-z $ in the following way (cf \cite{SjZw4,SjZw5}):

Let $\chi _1,..,\chi _N\in C_0^\infty (T^*\partial {\cal O})$ have
their supports in small neighborhoods of the points $\rho _1,...,\rho
_N\in {\cal G}$ that we assume are ``evenly distributed'' on ${\cal G}$
with $N$ sufficiently large and so that $\sum_1^N\chi _j=1$ near
${\cal G}$. Put $\chi _0=1-\sum_1^N\chi _j$. Define corresponding
tangential pseudodifferential operators $\chi _j(x',hD_{x'})$ on
$\partial {\cal O}$ in the standard way, so that $\sum_1^N\chi
_j(x',hD_{x'})=1$ microlocally near ${\cal G}$. With suitable choices
of the above quantities, there exist semi-classical elliptic Fourier
integral operators of order $0$, defined microlocally near $\rho _j$,
such that $R(x',hD_{x'})=U_jhD_{x_1}U_j^{-1}$ microlocally near $\mathrm{supp\,}\chi _j$
where $U_j^{-1}$ denotes a microlocal inverse of $U_j$. Then our
parametrix of $P-z $ is an operator $E={\cal O}(1):L^2(\Gamma
_{\delta ,b})\to {\cal B}_b$ of the form
\ekv{aed.15}
{
E=E_0\chi _0(x',hD_{x'})+\sum_1^NU_jE_j(x',hD_{x'})U_j^{-1}\chi _j(x',hD_{x'}).
}
Here the symbol $E_0(x',\xi ')$ belongs to the space $S^0(T^*\partial
{\cal O};{\cal L}(L^2,{\cal B}_b))$ of symbols that satisfy
(\ref{aed.10}) and has an asymptotic expansion,
\ekv{aed.16}{E_0\sim E_{0,0}+hE_{0,1}+h^2E_{0,2}+...,}
with $E_{0,k}\in S^{-k}$, the space of symbols $F$ satisfying 
$$
\partial _{x'}^\alpha \partial _{\xi '}^\beta F={\cal O}_{\alpha
  ,\beta }(\langle \xi '\rangle ^{-k-|\beta |}):\, L^2(\gamma _\delta
)\to {\cal B}(x',\xi ').
$$ 
Moreover, $E_{0,0}=(P(x',\xi ')-z )^{-1}$.

\par For $j=1,...,N$, $E_j$ has the property (\ref{aed.11}) with
$r=\xi _1$ and we have an asymptotic expansion
\ekv{aed.17}
{
E_j\sim E_{j,0}+h^{\frac{1}{3}}E_{j,1}+...,
}
with $E_{j,k}$ satisfying (\ref{aed.11}) and with $E_{j,0}=(P(x',\xi
')-z )^{-1}$ where it is understood that $P(x',\xi ')$ is now
simplified with the conjugation by $U_j$ so that $R(x',hD_{x'})$ has
become $hD_{x_1}$. The main property of $E$ is that 
\ekv{aed.18}
{
(P(x,hD)-z )E=1+{\cal O}(h^\infty )\hbox{ in }{\cal L}(L^2,L^2).
}
We can also construct a left parametrix $\widetilde{E}$  with an
expression similar to (\ref{aed.15}) but with the cutoff operators to
the left, and by a standard argument we see that
$\widetilde{E}=E+{\cal O}(h^\infty )$ in ${\cal L}(L^2,{\cal B}_b)$.

\par Summing up the discssion so far, we have
\begin{prop}\label{aed1}
we can construct an operator $E={\cal O}(1):L^2(\Gamma _{\delta
  ,b})\to {\cal B}_b$ as above, so that 
\ekv{aed.19}{
\begin{cases}(P(x,hD)-z )E=1+{\cal O}(h^\infty )\hbox{ in }{\cal
    L}(L^2,L^2),\\
E(P(x,hD)-z )=1+{\cal O}(h^\infty )\hbox{ in }{\cal
    L}({\cal B}_b,{\cal B}_b).
\end{cases}
}
\end{prop}

\par We now consider $P=-h^2\Delta $ on all of $\Gamma _\delta $ and
notice that $P-z $ is semi-classically elliptic away from any
fixed neighborhood of $\partial {\cal O}$, so we have a
pseudodifferential parametrix $Q(x,hD;h)$ in that region with symbol
$Q(x,\xi ;h)$ satisfying $\partial _x^\alpha \partial _{\xi }^\beta
Q={\cal O}(\langle \xi \rangle^{-2-|\beta |})$ such that if $\chi \in
C^\infty (\Gamma _\delta )$ is a standard cutoff to a small
neighborhood of $\partial {\cal O}$, then 
\[\begin{split}(P-z )Q(1-\chi )&=(1-\chi )+K_1\\
(1-\chi )Q(P-z )&=(1-\chi )+K_2,
\end{split}
\]
where $K_1$, $K_2$ are negligible operators ${\cal O}(h^\infty
):H^{-s}_h\to H_h^s$ for every $s\ge 0$. Further, we may arrange so
that the distribution kernel $K_Q(x,y)$ of $Q$ vanishes when
$|x-y|>\epsilon $, for any fixed given $\epsilon >0$.

\par Assuming that $\mathrm{supp\,}\chi \subset \Gamma _{\delta ,b}$,
we choose $\epsilon >0$ small enough and put
\ekv{aed.20}
{
F=\chi E\chi +Q(1-\chi )-Q[P,\chi ]E\chi .
}
Then, $F={\cal O}(1):L^2(\Gamma _\delta )\to {\cal B}(\Gamma _\delta
)$ and
$$
(P-z )F=1+K_3,
$$
where $K_3={\cal O}(h^\infty ):L^2\to L^2$. Here ${\cal B}(\Gamma
_\delta )$ denotes the space of distributions $u$ such that $\chi u\in
{\cal B}(\Gamma _{\delta ,b})$, $(1-\chi )u\in H_h^2(\Gamma _\delta
)$. The construction of a left parametrix is similar, and by a standard
argument we see that $F$ is also a left parametrix. Summing up, we get
\begin{prop}\label{aed2}
  The operator $F$ in (\ref{aed.20}) is ${\cal O}(1):L^2(\Gamma
  _\delta )\to {\cal B}(\Gamma _\delta )$ and satisfies \ekv{aed.21} {
    (P-z )F=1+K_3,\ F(P-z )=1+K_4, } where $K_3={\cal
    O}(h^\infty ):\, L^2(\Gamma _\delta )\to L^2(\Gamma _\delta )$,
  $K_4={\cal O}(h^\infty ):\, {\cal B}(\Gamma _\delta )\to {\cal
    B}(\Gamma _\delta )$.
\end{prop}

\section{Exterior Poisson operator and DN map}\label{pdn}
\setcounter{equation}{0}
 We need some more estimates in the one dimensional case. Recall that if
 $u\in C_0^\infty ([0,\infty [)$, then 
\ekv{pdn.1}
{
|u(0)|^2\le 2\Vert u\Vert \Vert \partial u\Vert.
}
If $u\in C^\infty ([0,\infty [)$, let $\chi \in C_0^\infty ([0,\infty
[)$, $\chi (0)=0$ and put $\chi _L(x)=\chi (x/L)$. Applying
(\ref{pdn.1}) to $\chi _Lu$ gives
\ekv{pdn.2}
{
|u(0)|^2\le C(\frac{1}{L}\Vert u\Vert^2_{[0,L]}+\Vert
u\Vert_{[0,L]}\Vert \partial u\Vert_{[0,L]}).
}
If $\Lambda >0$ is a continuous function on $[0,\infty [$ of increasing
order of magnitude ($\Lambda (x)\ge \frac{1}{C}\Lambda (y)$ when $x\ge
y$) we get
$$
|u(0)|^2\le C(\frac{1}{L\Lambda (0)^2}\Vert \Lambda
u\Vert^2_{[0,L]}+\frac{1}{h\Lambda (0)}\Vert \Lambda  u\Vert_{[0,L]}\Vert
h\partial u\Vert_{[0,L]}).
$$ Choose $L$ so that $L\Lambda (0)^2=h\Lambda (0)$, $L=h/\Lambda
(0)$. Then,
$$
|u(0)|^2\le \frac{C}{h\Lambda (0)}(\Vert \Lambda
u\Vert^2_{[0,\frac{h}{\Lambda (0)}]}+\Vert
h\partial u\Vert^2_{[0,\frac{h}{\Lambda (0)}]}),
$$
\ekv{pdn.3}
{
\sqrt{h\Lambda (0)}|u(0)|\le C(\Vert \Lambda
u\Vert_{[0,\frac{h}{\Lambda (0)}]}+\Vert
h\partial u\Vert_{[0,\frac{h}{\Lambda (0)}]}).
}

\par Recall that for $(x',\xi ')$ near a point on the glancing
hypersurface, $r=\lambda $,
\begin{equation}\label{pdn.3.5}
\Vert u\Vert_{{\cal B}(x',\xi ')}=\Vert \Lambda ^2u\Vert +\Vert
\Lambda h\partial_{x_n} u\Vert + \Vert (h\partial_{x_n} )^2u\Vert ,
\end{equation}
where $\Lambda ^2=(h^{2/3}+|r-\lambda |+s)$, $r=r(x',\xi ')$,
$x_n=\gamma _\delta (s)$, $0\prec x\prec b$. Since $\Lambda $ is
increasing, we can apply (\ref{pdn.3}) and estimate $|u(0)|$ with the
first two terms in the ${\cal B}$-norm and $|h\partial_{x_n}u(0)|$ using
the last two terms: \ekv{pdn.4} { h^{\frac{1}{2}}\Lambda
  (0)^{\frac{3}{2}}|u(0)|\le C\Vert u\Vert_{{\cal B}}, } \ekv{pdn.5} {
  h^{\frac{1}{2}}\Lambda (0)^{\frac{1}{2}}|h\partial_{x_n}u(0)|\le C\Vert
  u\Vert_{{\cal B}}, } or more explicitly, \ekv{pdn.6} {
  h^{\frac{1}{2}}(h^{\frac{2}{3}}+|r-\lambda |)^{\frac{3}{4}}|u(0)|\le
  C\Vert u\Vert_{{\cal B}}, } \ekv{pdn.7} {
  h^{\frac{1}{2}}(h^{\frac{2}{3}}+|r-\lambda
  |)^{\frac{1}{4}}|h\partial_{x_n}u(0)|\le C\Vert u\Vert_{{\cal B}}.}

\par We next estimate the ${\cal B}(x',\xi ')$-norm of the null-solution in (\ref{cw.7}),
$$u=e_{x',\xi '}=e^{-\frac{1}{h}\phi (x_n;h)},\ \phi (x_n;h)=\phi
_{x',\xi '}(x_n;h),\ \phi (0)=0,$$
of $(P(x',\xi ')-z )u=0$ along $\gamma _\delta $. We know that 
$$
(h^{\frac{1}{3}}+|\phi '|)^2\asymp h^{\frac{2}{3}}+|r-\lambda |+s,\
(x_n=\gamma _\delta (s)),
$$
and that
$$\Re \partial _s\phi \asymp |\phi '|\ge \frac{1}{C}(h^{\frac{2}{3}}+|r-\lambda
|+s)^{\frac{1}{2}}$$
when $s+|r-\lambda |\gg h^{\frac{2}{3}}$. Thus with $b=\gamma _\delta (s_0)$,
$$
\Vert e_{x',\xi '}\Vert^2=\int_0^{s_0} e^{-\frac{2}{h}\Re \phi
  (x_n(s))}ds\le \int_0^\infty
e^{-\frac{1}{Ch}(h^{\frac{2}{3}}+|r-\lambda |) ^{\frac{1}{2}} s}ds,
$$
which leads to 
$$
\Vert e_{x',\xi '}\Vert \le \frac{{\cal
    O}(1)h^{\frac{1}{2}}}{(h^{\frac{2}{3}}+|r-\lambda |)^{\frac{1}{4}}}.
$$
We will also use that the same estimate holds for $\Vert e_{x',\xi
  '}^{\frac{1}{2}}\Vert$. 

\par Next look at 
$$
\Vert (h^{\frac{2}{3}}+|r-\lambda |+s)e_{x',\xi
  '}\Vert=(h^{\frac{2}{3}}+|r-\lambda |)\Vert
\frac{h^{\frac{2}{3}}+|r-\lambda |+s}{h^{\frac{2}{3}}+|r-\lambda
  |}e_{x',\xi '}\Vert .
$$
From Lemma \ref{cw0} we see that 
$$
\frac{h^{\frac{2}{3}}+|r-\lambda |+s}{h^{\frac{2}{3}}+|r-\lambda
  |}e^{\frac{1}{2}}_{x',\xi '}
$$
is bounded, so 
$$
\Vert
\frac{h^{\frac{2}{3}}+|r-\lambda |+s}{h^{\frac{2}{3}}+|r-\lambda
  |}e_{x',\xi '}\Vert \le {\cal O}(1)\Vert e^{\frac{1}{2}}_{x',\xi
  '}\Vert \le \frac{{\cal
    O}(1)h^{\frac{1}{2}}}{(h^{\frac{2}{3}}+|r-\lambda |)^{\frac{1}{4}}}.
$$
Thus,
$$
\Vert (h^{\frac{2}{3}}+|r-\lambda |+s)e_{x',\xi
  '}\Vert\le {\cal O}(1)h^{\frac{1}{2}}(h^{\frac{2}{3}}+|r-\lambda |)^{\frac{3}{4}}.
$$
The other terms in the ${\cal B}$ norm of $u$ satisfy the same
estimates and we get \ekv{pdn.8} { \Vert e_{x',\xi '}\Vert_{{\cal
      B}}\le {\cal O}(1)h^{\frac{1}{2}}(h^{\frac{2}{3}}+|r-\lambda
  |)^{\frac{3}{4}}.  }
Since $e_{x',\xi '}(0)=1$, we see that this is the reverse inequality
to (\ref{pdn.6}) up to a bounded factor, so
\ekv{pdn.9}
{
\Vert e_{x',\xi '}\Vert_{\cal B} \asymp h^{\frac{1}{2}}(h^{\frac{2}{3}}+|r-\lambda
  |)^{\frac{3}{4}}.
}

\begin{remark}\label{pdn1}
{\rm Using that $e_{x',\xi '}(b)={\cal O}(e^{-\frac{1}{Ch}})$, we can add
an exponentially small reflected term as in (\ref{cw.17}) to get a null
solution which vanishes at $b$ and after dividing with a factor
$1+{\cal O}(e^{-\frac{1}{Ch}})$ we get a new function $e_{x',\xi '}$
satisfying $(P_{x',\xi '}-z )e_{x',\xi '}=0$, $e_{x',\xi '}(0)=1$,
$e_{x',\xi '}(b)=0$ as well as the estimate (\ref{pdn.9}).}
\end{remark}
\par Recall that $P(x',\xi ')-z :{\cal B}(x',\xi ')\to L^2$ has a
uniformly bounded inverse $E(x',\xi ')$ and that we have the estimates
(\ref{aed.9}), (\ref{aed.11}). Differentiate the equation $(P(x',\xi
')-z )e_{x',\xi '}=0$ and notice that $\partial
_{x'}^\alpha \partial _{\xi '}^\beta e_{x',\xi '}(0)=\partial
_{x'}^\alpha \partial _{\xi '}^\beta e_{x',\xi '}(b)=0$ when $|\alpha
|+|\beta |\ne 0$, so that  $\partial
_{x'}^\alpha \partial _{\xi '}^\beta e_{x',\xi '}\in {\cal B}$. We
get
\ekv{pdn.10}
{
\partial _{x'}^\alpha \partial _{\xi '}^\beta e_{x',\xi '}
=\sum_{{\alpha '+\alpha ''=\alpha \atop
\beta '+\beta ''=\beta} \atop
|\alpha ''|+|\beta ''|<|\alpha |+|\beta |}c_{\alpha ',\alpha '',\beta
',\beta ''}E(\partial _{x'}^{\alpha '}\partial _{\xi '}^{\beta '}P)
(\partial _{x''}^{\alpha ''}\partial _{\xi ''}^{\beta ''}e_{x',\xi '}).
}
By induction, we see that 
\ekv{pdn.11}
{
\Vert \partial _{x'}^\alpha \partial _{\xi '}^\beta e_{x',\xi
  '}\Vert_{{\cal B}}={\cal
  O}(1)h^{\frac{1}{2}}(h^{\frac{2}{3}}+|r-\lambda
|)^{\frac{3}{4}-\beta _1}.
}

\par As a first approximation to the Poisson operator on $\Gamma
_{\delta ,b}$, we take
\ekv{pdn.12}{K^0w=\mathrm{Op}_h(e_{x',\xi '})}
where $\mathrm{Op}_h$ denotes the classical $h$-quantization in ${\bf
  R}^{n-1}$ also in the case of vector and operator valued symbols, so
that our $K^0$ is microlocally defined in $T^*(\partial {\cal O})$
and maps functions of $x'$ to functions of $x$. (Here it is tacitly
assumed that we have reduced $R$ to $hD_{x_1}$ as in (\ref{aed.11}).) Then
\ekv{pdn.13}
{
\gamma K^0=1
}
and 
\ekv{pdn.14}
{
(P-z )K^0=\mathrm{Op}_h(f_{x',\xi '}),
}
where
\ekv{pdn.15}
{
f_{x',\xi '}\sim\sum_{\alpha \ne 0}\frac{h^{|\alpha |}}{\alpha !}\partial
_{\xi '}^\alpha P(x',\xi ')D_{x'}^\alpha e_{x',\xi '}
}
and we have used that $(P(x',\xi ')-z )e_{x',\xi '}=0$. From
(\ref{aed.9}), (\ref{pdn.11}), we see that 
$$
\Vert \partial _{x'}^\alpha \partial _{\xi '}^\beta f_{x',\xi '}\Vert
_{L^2}={\cal O}(1)h^{\frac{3}{2}}(h^{\frac{2}{3}}+|r-\lambda
|)^{-\frac{1}{4}-\beta _1}.
$$
We get the microlocal Poisson operator to all orders in $h$
by putting 
$$
\widetilde{K}=K^0-E\circ (P-z )K^0.
$$
Here 
$$
E(P-z )K^0w=\mathrm{Op}_h(\widetilde{r}),
$$
where 
\ekv{pdn.16}
{
\Vert \partial _{x'}^\alpha \partial _{\xi '}^\beta \widetilde{r}\Vert_{{\cal
    B}_{x',\xi '}} ={\cal
  O}(1)h^{\frac{3}{2}}(h^{\frac{2}{3}}+|r-\lambda
|)^{-\frac{1}{4}-\beta _1}.
}
This bound is ``better'' than (\ref{pdn.11}) by a factor
$$
h(h^{\frac{2}{3}}+|r-\lambda |)^{-1}\le h^{\frac{1}{3}},
$$
thus we get 
\ekv{pdn.17}
{
\widetilde{K}w=\mathrm{Op}_h(e_{x',\xi '}+\widetilde{r}_{x',\xi '}),
}
solving 
\ekv{pdn.18}
{
\gamma \widetilde{K}=1,\ (P-z )\widetilde{K}={\cal O}(h^\infty
):\, L^2\to {\cal B}.
}
As in Proposition \ref{aed2} 
it is now routine to show that the exact exterior Poisson operator is
microlocally given by (\ref{pdn.17}) near any fixed point of the
glancing hypersurface ${\cal G}$. 

Away from ${\cal G}$ the construction of a
Poisson operator on $\Gamma _{\delta ,b}$ and on $\Gamma _\delta $ is
more routine and we omit the details. Using a truncation as in the
preceding section, we can carry over the construction from $\Gamma
_{\delta ,b}$ to $\Gamma _\delta $. The preceding section gives an
approximate Green operator for the exterior problem while the present
section does the same for the Poisson operator. By simple Neumann
series we can replace approximate solution operators by the exact ones
and get the following result that summarizes the constructions of this
and the preceding sections where we start to use the notation $\Gamma
_\delta ^\mathrm{ext}$ to emphasize that ${\cal O}$ is not part of
this contour.

\begin{prop}\label{psd2}
The exterior Dirichlet problem
\ekv{pdn.19}
{
(P-z )u=v,\ \gamma u=w,\hbox{ on }\Gamma _\delta^{\mathrm{ext}}, 
}
where $\gamma $ is the operator of restriction to the boundary, has a
unique solution $u\in H_h^2(\Gamma _\delta ^\mathrm{ext})$ for every
$(v,w)\in L^2(\Gamma _\delta ^\mathrm{ext})\times
H_h^{\frac{3}{2}}(\partial {\cal O})$, of the form
\ekv{pdn.20}
{
u=G_\mathrm{ext}v+K_\mathrm{ext}w.
}

\par If $\chi \in C^\infty (\Gamma _\delta ^\mathrm{ext})$ has its
support away from a fixed distance to $\partial {\cal O}$ and is equal
to one near infinity (and satisfies uniform estimates  with all its
derivatives when $h\to 0$), then
\ekv{pdn.25}
{
\chi G_\mathrm{ext},\ G_\mathrm{ext}\chi ={\cal O}(1): L^2\to H^2_h,
}
\ekv{pdn.26}
{
\chi K_\mathrm{ext}={\cal O}(h^\infty ):H_h^{\frac{3}{2}}(\partial
{\cal O})\to H_h^2.
}

\par If we choose local geodesic coordinates $x',x_n$ near a boundary
point, then near that point $G_\mathrm{ext}$ is a pseudodifferential
operator with operator valued symbol,
\ekv{pdn.21}{G_\mathrm{ext}=E(x',hD_{x'};h),}
where $E$ fulfills (\ref{aed.10}), (\ref{aed.11}) (and for the latter
estimate it is assumed that $P$ has been conjugated by a tangential Fourier
integral operator in order to straighten out $R-\lambda $). 

\par In the same coordinates 
\ekv{pdn.22}
{
\chi K_\mathrm{ext}=K(x',hD_{x'};h),
}
where \ekv{pdn.23}
{
\Vert \partial _{x'}^\alpha \partial _{\xi '}^\beta K(x',\xi
';h)\Vert_{{\cal B}_{x',\xi '}}={\cal
  O}(1)h^{\frac{1}{2}}(h^{\frac{2}{3}}+|r-\lambda |)^{\frac{3}{4}-\beta
  _1}
}
near ${\cal G}$ (after straightening of $R-\lambda $), while away from
${\cal G}$:
\ekv{pdn.24}
{
\Vert \partial _{x'}^\alpha \partial _{\xi '}^\beta K(x',\xi
';h)\Vert_{{\cal B}_{x',\xi '}}={\cal O}(1)h^{\frac{1}{2}}\langle \xi
'\rangle^{-\frac{3}{2}-|\beta |}.
}
\end{prop}

\par By construction, $G_\mathrm{ext}={\cal O}(1):L^2\to {\cal B}$
near $\partial {\cal O}$ and (cf.~(\ref{pdn.3.5})) we get the first
part of
\begin{cor}\label{pdn3}
We have
\ekv{pdn.27}
{
G_\mathrm{ext}={\cal O}(h^{-\frac{2}{3}}):L^2\to H_h^2,
}
\ekv{pdn.28}
{
K_\mathrm{ext}={\cal O}(h^{-\frac{1}{6}}):H_h^{\frac{3}{2}}(\partial
{\cal O})\to H_h^2.
}
\end{cor}
For the second part, we combine (\ref{pdn.3.5}) and (\ref{pdn.24}).

\par
Finally, we consider the exterior Dirichlet to Neumann (DN) map
\ekv{pdn.29} { {\cal N}_\mathrm{ext}=hD_{\nu }K_\mathrm{ext}, } where
$\nu $ denotes the exterior unit normal. From
(\ref{pdn.22}), (\ref{pdn.23}), (\ref{pdn.7}), we see that this is a
pseudodifferential operator with symbol \[\gamma hD_{x_n}(K(x',\xi
';h))=:n_\mathrm{ext}(x',\xi ';h)\] satisfying
\ekv{pdn.30}
{
\partial _{x'}^\alpha \partial _{\xi '}^\beta n_\mathrm{ext}(x',\xi
';h)={\cal O}(\langle \xi '\rangle ^{1-|\beta |})
}
away from ${\cal G}$ and 
\ekv{pdn.31}
{
\partial _{x'}^\alpha \partial _{\xi '}^\beta n_\mathrm{ext}(x',\xi
';h)=
{\cal O}(1)(h^{\frac{2}{3}}+|r-\lambda |)^{\frac{1}{2}-\beta _1},
}
near ${\cal G}$ after the usual straightening. In particular, we have
\begin{cor}\label{pdn4}
For every $s\in {\bf R}$ we have that ${\cal N}_\mathrm{ext}={\cal
  O}(1):H_h^{s+1}\to H_h^s.
$
\end{cor}

\section{The interior DN map}\label{idn}
\setcounter{equation}{0}

In this section we work inside ${\cal O}$ and assume that \ekv{idn.1}
{ P=-h^2\Delta +V(x), } where we will first assume only that 
$V\in L^\infty ({\cal O};{\bf R})$ and soon make stronger
assumptions. The results will be applied to $V_0$ in (\ref{re.1}),
but for simplicity we drop the subscript 0 in this section.
 
We study the interior Poisson operator $K_\mathrm{in}(z
)=H^{3/2}(\partial {\cal O})\to H^2({\cal O})$ associated to $P-z
$ and the interior DN-map \ekv{idn.2} { {\cal N}_\mathrm{in}=\gamma
  hD_{\nu } K_\mathrm{in}:H^{3/2}(\partial {\cal O})\to H^{1/2}(\partial
    {\cal O}) }
  under the assumption that,
\ekv{idn.3}
{
 \Re z =\lambda \asymp 1,\quad\frac{h^{2/3}}{{\cal O}(1)}\le |\Im z |\le {\cal O}(1)h^{2/3}.
}
Using the right inverse of $\gamma $ in (\ref{al.11}), we can write 
$$K_\mathrm{in}=\gamma ^{-1}-(P_\mathrm{in}-z )^{-1}\gamma ^{-1}$$ and
see that 
\ekv{idn.3.2}{
\Vert K_\mathrm{in}(z )\Vert_{{\cal L}(H^{3/2},H^2)}={\cal
  O}(1)(h^{\frac{1}{2}}+h^{-\frac{2}{3}+\frac{1}{2}})={\cal O}(1)h^{-\frac{1}{6}}}
where $P_\mathrm{in}$ is the Dirichlet realization of
$P$. Consequently,
\ekv{idn.3.3}{
\Vert {\cal N}_\mathrm{in}(z )\Vert_{{\cal L}(H^{\frac{3}{2}},H^{\frac{1}{2}})}
\le {\cal O}(h^{-\frac{1}{2}})\Vert K_\mathrm{in}(z )\Vert_{{\cal
    L}(H^{\frac{3}{2}},H^2)}={\cal O}(h^{-\frac{2}{3}}).}
We now assume that 
\ekv{idn.3.7}
{
V\in C^\infty (\overline{{\cal O}};{\bf R}),\ \gamma V=0,\ \gamma \partial _\nu
V\le 0,
}
where the last two assumptions can be somewhat weakened.
Using parametrix constructions, we shall improve the estimate (\ref{idn.3.3})
to
\begin{prop}\label{idn1}
Under the assumption (\ref{idn.3}), we have
\ekv{idn.4}
{
\Vert {\cal N}_\mathrm{in}(z )\Vert_{{\cal
    L}(H^{\frac{3}{2}},H^{\frac{1}{2}})} ={\cal O}(1).
}
\end{prop}
\begin{proof}
We make parametrix constructions in different regions of $T^*\partial
{\cal O}$ and start with the hyperbolic region 
$$
{\cal H}=\{ (x',\xi ')\in T^*\partial {\cal O};\, r(x',\xi ')<\lambda \},
$$
where we write the operator in geodesic coordinates (with ${\cal O}$
given by $x_n\le 0$) as in (\ref{aed.1}). Near a point $(x_0',\xi
_0')\in {\cal H}$ we construct a microlocal approximation to the
Poisson operator of the form
\ekv{idn.5}
{
\widetilde{K}_\mathrm{in}(z )w(x)=\frac{1}{(2\pi h)^{n-1}}\iint 
e^{\frac{i}{h}(\phi (x,\eta ')-y'\eta ')}a(x,\eta ';h)w(y')dy'd\eta '.
}

\par
We write $P$ as in (\ref{aed.1}):
\ekv{idn.5.5}
{\begin{split}
P=(hD_{x_n})^2+R(x,hD_{x'})+ha(x)hD_{x_n},\\
R(x,hD_{x'})=R(x',hD_{x'})-x_nQ(x,hD_{x'}),\end{split}
}
where we recall that $V$ is incorporated in $P$ and hence in the term
$-x_nQ$ and the condition (\ref{idn.3.7}) together with the strict convexity
of ${\cal O}$ assures that $q>0$ for $\xi '\ne 0$. Recall that $a$ can be eliminated
and assume for simplicity that $a=0$. As before $p$ denotes the
semi-classical principal symbol of $P$

Now consider the eiconal equation 
\ekv{idn.6}
{
p(x,\phi ')-z =0\hbox{ for }x\in \mathrm{neigh\,}(x_0',0)\cap
{\cal O},\quad \phi (x',0,\eta ')=x'\eta '.
}
With $r(x,\xi ')=r(x',\xi ')-x_nq(x,\xi ')$ it becomes
$$
\partial _{x_n}\phi =\pm (\lambda +h^{\frac{2}{3}}w-r(x,\phi
'_{x'}))^{\frac{1}{2}},\ \mp \Im w>0.
$$
Using the principal branch of the square root we choose the sign as
indicated. If $\phi _0$ is the real solution of the corresponding
eiconal problem when $w=0$, we can solve (\ref{idn.6}) to all orders
in $h$ by the asymptotic expansion,
$$
\phi (x,\eta ')=\phi _0(x,\eta ')+h^{\frac{2}{3}}\phi _1(x,\eta
')+h^{\frac{4}{3}}\phi _2(x,\eta ';h),
$$
where $\phi _1,\phi _2,...={\cal O}(x_n)$,
$$
\partial _{x_n}\phi _1=\pm \frac{1}{2}(\lambda -r(x,\partial _{x'}\phi _0))^{-\frac{1}{2}}w,
$$
so that
\ekv{idn.7}
{
\Im \phi \asymp h^{\frac{2}{3}}\Im \phi _1\asymp |x_n\Im w|h^{\frac{2}{3}} 
}

\par By solving the transport equations in the usual way, we get the
amplitude $a$ as a symbol of order 0 and if $\chi \in C_0^\infty
({\cal H}) $ has its support in a small neighborhood of $(x_0',\xi
_0')$ we get a Fourier integral operator
$\widetilde{K}_\mathrm{in}(z ):C^\infty (\partial {\cal O})\to
C^\infty (\overline{{\cal O}})$ solving
\ekv{idn.8}
{
(P-z )\widetilde{K}_\mathrm{in}(z )={\cal O}(h^\infty ):{\cal
  D}'(\partial {\cal O})\to C^\infty (\overline{{\cal O}}),
} 
\ekv{idn.9}
{
\gamma \widetilde{K}_\mathrm{in}(z )=\chi (x',hD_{x'}).
}
Here (\ref{idn.7}) is important, since it assures that the
distribution kernel $\widetilde{K}_{\mathrm{in}}(x,y',z )$ of
$\widetilde{K}_\mathrm{in}(z )$ is ${\cal O}(h^\infty )$ with all
its derivatives when $\mathrm{dist\,}(x,\partial {\cal O})\ge
h^{\frac{1}{3}-\delta }$ for any fixed $\delta >0$. (Another standard
fact, implicitly used here, is that the distribution kernel is ${\cal
  O}(h^\infty )$ with all its derivatives as soon as $(x',y')$ is
outside any fixed neighborhood of the diagonal.)

\par From (\ref{idn.7}) we get additional damping, leading to
\ekv{idn.9.5}
{
\widetilde{K}={\cal O}(h^{\frac{1}{6}}):H_h^{\frac{3}{2}}\to H_h^2.
}

\par It also follows that
\ekv{idn.10}
{
\gamma hD_\nu \widetilde{K}_\mathrm{in}(z )=\widetilde{\chi }(x',hD_{x'};h)
}
where $\widetilde{\chi }(x',\xi ';h)$ is a classical symbol of order
$0$ in $h$ and of order $-\infty $ in $\xi '$ which is ${\cal
  O}(h^\infty )$ with all its derivatives outside any fixed
neighborhood of the support of $\chi $.

A similar even more standard construction works in the elliptic region
$$
{\cal E}=\{ (x',\xi ')\in T^*\partial {\cal O};\, r(x',\xi ')>\lambda \}.
$$
We get an operator $\widehat{K}={\cal
  O}(h^{\frac{1}{2}}):H^{\frac{3}{2}}_h\to H^2_h$ such that
\ekv{idn.10.2}{(P-z )\widehat{K}={\cal O}(h^\infty ),}
\ekv{idn.10.4}{\gamma \widehat{K}=1-\chi (x',hD_{x'}),}
\ekv{idn.10.6}{\gamma hD_\nu \widehat{K}=n_\chi (x',hD_{x'};h),}
where $\chi \in C_0^\infty (T^*\partial {\cal O})$ is any function
equal to one in a neighborhood of ${\cal G}\cup {\cal H}$.
$\widetilde{\chi }$ has the same properties as $\chi $ and $n_\chi \in
S^1(T^*\partial \Omega )$ is equal to ${\cal O}(h^\infty )$ with all
its derivatives away from $\mathrm{supp\,}(1-\chi ) $.

\par We next turn to the more difficult study near the glancing
hypersurface
$$
{\cal G}=\{ (x',\xi ')\in T^*\partial {\cal O};\, r(x',\xi ')=\lambda \},
$$
and we shall start by pushing the construction in ${\cal H}$ closer to
${\cal G}$ and almost up to a distance $\gg h^{\frac{2}{3}}$ from
that set. We write the operator in geodesic coordinates as in
(\ref{aed.1}). Let $\rho _0=(x_0',\xi _0')\in {\cal G}$ and assume,
after conjugation with an elliptic tangential Fourier integral
operator that microlocally,
\ekv{idn.11}{
R(x',hD_{x'})-\lambda =hD_{x_1},\ (x_0',\xi _0')=(0,0).
}

\par Let $\eta '\in{\bf R}^{n-1}$ satisfy 
$$
(\eta _2,...,\eta _{n-1})=\frac{1}{{\cal O}(1)},\ \eta _1=-\epsilon ,\
h^{\frac{2}{3}}\ll \epsilon \ll 1.
$$
We shall construct an asymptotic solution to the problem
\ekv{idn.12}{(P-z )u=0,\ u(x',0)=a(x')e^{\frac{i}{h}x'\eta '},}
or equivalently with $u=e^{ix'\eta '/h}\widetilde{u}$,
\ekv{idn.13}
{
e^{-\frac{i}{h}x'\eta '}(P-z )e^{\frac{i}{h}x'\eta
  '}\widetilde{u}=0,\ \widetilde{u}(x',0)=a(x').
}
The conjugated operator to the left can be written
\ekv{idn.14}
{
(hD_{x_n})^2+hD_{x_1}-x_nQ(x,\eta 
'+hD_{x'})-(\epsilon +h^{\frac{2}{3}}w).
}

\par From looking at the eiconal equation $p(x,\phi ')-z =0$ with
boundary condition $\phi '_{x'}(x',0)=\eta '$, it is natural to make
the dilation in $x_n$, \ekv{idn.15} { x_n=\epsilon \widetilde{x}_n,\
  x'=\widetilde{x}'. } Then $hD_{x_n}=\frac{h}{\epsilon
}D_{\widetilde{x}_n}$, $hD_{x'}=hD_{\widetilde{x}'}$ and a direct
calculation shows that \ekv{idn.16} { e^{-\frac{i}{h}x'\eta '}(P-z
  )e^{\frac{i}{h}x'\eta '}=\epsilon
  (\widetilde{P}-(1+\widetilde{h}^{\frac{2}{3}}w)), } where
$\widetilde{h}=h\epsilon ^{-\frac{3}{2}}\ll 1$ and \ekv{idn.17} {
  \widetilde{P}=(\widetilde{h}D_{\widetilde{x}_n})^2+\epsilon
  ^{\frac{1}{2}}\widetilde{h}D_{\widetilde{x}_1}-\widetilde{x}_nQ(\widetilde{x}',\epsilon
  \widetilde{x}_n,\eta '+\epsilon^{\frac{3}{2}}
  \widetilde{h}D_{\widetilde{x}'}). } Thus after dilation, we are in a
``uniformly hyperbolic'' situation and we get a solution
$$
\widetilde{u}=b(\widetilde{x};\widetilde{h})e^{\frac{i}{\widetilde{h}}\widetilde{\phi
  }(\widetilde{x})},\quad \widetilde{x}=(x',\frac{x_n}{\epsilon }),
$$
of the problem
\ekv{idn.18}
{
(\widetilde{P}-(1+\widetilde{h}^{\frac{2}{3}}w))\widetilde{u}={\cal
  O}(\widetilde{h}^\infty ),\
\widetilde{u}(\widetilde{x}',0)=a(\widetilde{x}'),
}
defined in a region 
$$
|\widetilde{x}'|\le {\cal O}(1),\ 0\le
-\widetilde{x}_n<\frac{1}{{\cal O}(1)},
$$
where $b$ is a classical symbol of order $0$ and $\widetilde{\phi
}(\widetilde{x})$ is uniformly bounded with all its derivatives in the
same region. $\widetilde{\phi }$ is here the solution of the eiconal
equation,
\ekv{idn.18.5}
{
\widetilde{p}(\widetilde{x},\widetilde{\phi
}'_{\widetilde{x}})-(1+\widetilde{h}^{\frac{2}{3}}w)=0,\
{{\widetilde{\phi }}}_{\vert\widetilde{x}_n=0}=0,
}
which satisfies
\ekv{idn.19}
{
\Im \widetilde{\phi }\asymp |\widetilde{x}_n|\widetilde{h}^{\frac{2}{3}}.
}
Thus, 
$$
|\widetilde{u}|={\cal O}(1)e^{-|\widetilde{x}_n|/(C\widetilde{h}^{\frac{1}{3}})},
$$
which is ${\cal O}(\widetilde{h}^\infty )$ in any region
$-\widetilde{x}_n\ge \widetilde{h}^{\frac{1}{3}-\delta }$ for any fixed
$\delta >0$.

\par In the original coordinates, we get the asymptotic solution of
(\ref{idn.12})
\ekv{idn.20}
{
u(x;\eta ';h)=b(\frac{x_n}{\epsilon },x',\eta
';\widetilde{h})e^{\frac{i}{h}(x'\eta '+\epsilon
  ^{\frac{3}{2}}\widetilde{\phi }(\frac{x_n}{\epsilon
  },x',\eta '))}.
}
These solutions can be superposed to build a microlocal Poisson
operator, if we take $a=1$, and we get 
$\check{K}={\cal O}(\widetilde{h}^{1/6}):H_h^{3/2}\to H_h^2$, where
we use the modified norm
$$\sum_{|\alpha |\le 2}\| (hD_{x'})^{\alpha
  '}(\widetilde{h}D_{\widetilde{x}_n})^{\alpha _n}v\|$$ on $H_h^2$
with $L^2(dx'd\widetilde{x}_n)$ as the underlying $L^2$-norm. This
gives in the original coordinates, \ekv{idn.19.3} { \sum_{|\alpha |\le
    2}\Vert (hD_{x'})^{\alpha '}(h\epsilon
  ^{-\frac{1}{2}}D_{x_n})^{\alpha _n}\check{K}u\Vert_{L^2(dx)}\le {\cal
    O}(1)h^{\frac{1}{6}}\epsilon ^{\frac{1}{4}}\Vert
  u\Vert_{H^{\frac{3}{2}}_h}.  } In particular,
\ekv{idn.19.7}{\check{K}={\cal O}(1)h^{\frac{1}{6}}\epsilon
    ^{\frac{1}{4}}:H_h^{\frac{3}{2}}\to H_h^2,} with the ordinary
$H^2$ norm.

\par We get the approximation to the DN map:
\ekv{idn.21}
{
{\cal N}^\mathrm{approx}_\mathrm{in}=
\mathrm{Op}_h(\epsilon ^{\frac{1}{2}}\partial
_{\widetilde{x}_n}\widetilde{\phi }(x',0,\xi
')+\frac{h}{i\epsilon }(\partial
_{\widetilde{x}_n}b)(x',0,\xi ';\widetilde{h})).
}

\par Here we must recall that $\epsilon =-\xi _1$, so the symbol of 
${\cal N}^\mathrm{approx}_\mathrm{in}$ is singular in that
variable but good enough for our 2-microlocal calculus, in view of the
fact that $\epsilon \gg h^{2/3}$ and it is a uniformly bounded operator: $H^{3/2}_h\to H^{1/2}_h$.

\par It remains to study the region 
\ekv{idn.23}
{
-h^{\frac{2}{3}-\delta }\le r(x',\xi ')-\lambda \le \widetilde{\delta },
}
where $\delta ,\widetilde{\delta }>0$ are small and independent of
$h$. Again, we reduce $R$ to the form (\ref{idn.11}) and restrict $\xi
'$ to a set 
$$
(\xi _2,...,\xi _{n-1})=\frac{1}{{\cal O}(1)},\ -h^{\frac{2}{3}-\delta }
\le \xi _1\le \widetilde{\delta }.
$$
We consider (cf (\ref{idn.14}))
\ekv{idn.24}
{
P(x,\xi ',hD_{x_n})-z =(hD_{x_n})^2+\xi _1-x_nQ(x,\xi ')-h^{\frac{2}{3}}w,
}
and we follow the approach for the exterior problem started in Subsection
\ref{cw}, with two not very essential differences: \begin{itemize}
\item $x_n$ remains real and we study the
  Dirichlet problem on an interval $[-b,0]$ for $0<b\ll 1$ independent
  of $h$.
\item There will be a slight degeneration when $\xi _1\ll
  -h^{\frac{2}{3}}$. 
\end{itemize}

We review the one-dimensional analysis with $x',\xi '$ as parameters,
writing $x$ instead for $x_n$ and $Q(x)$ instead of $Q(x',x_n,\xi
')$. {\it We first assume that $Q$ is analytic.} Let $x_0$ be the  complex turning point, 
given by
$$
x_0Q(x_0)=\xi _1-h^{\frac{2}{3}}w,
$$
and we let $x_1\asymp \xi _1$ be the corresponding real turning point
given by
$$
x_1Q(x_1)=\xi _1.
$$
Then
$$
x_0=x_1-\frac{h^{\frac{2}{3}}w}{V'(x_1)}+{\cal
  O}(h^{\frac{4}{3}}),\hbox{ where }V(x)=xQ(x).
$$

\par As in the exterior case we take a null solution of the form
$u=e^{-\phi (x;h)/h}$ which is subdominant in the direction of
negative $x$ and increasing in order of magnitude when $x$
increases. More precisely, for $x-x_1\ll -h^{2/3}$ we have
\ekv{idn.24.5}{-\partial _x (\Re \phi )\asymp |\partial _x\phi |\asymp
  |x-x_1|^{1/2}} and for $|x-x_1|\le {\cal O}(h^{2/3})$ we have
$\partial _x\phi ={\cal O}(h^{1/3})$.

\par For $x-x_1\gg h^{2/3}$ (as well as for $x-x_1\ll -h^{2/3}$) we
have (\ref{cw.8}), where 
$$-\phi _0'=(\xi _1-xQ(x)-h^{2/3}w)^{1/2},$$
and we choose the principal branch of the square root with a cut along
$\overline{{\bf R}}_-$, which has positive real
part.  Then for $x-x_1\gg h^{2/3}$ we get when $\pm
\Im w>0$:
\begin{eqnarray*}
  -\phi _0'&=&\mp i (xQ(x)-\xi _1+h^{2/3}w)^{\frac{1}{2}}\\
  &=&\mp i (xQ(x)-\xi
  _1)^{\frac{1}{2}}(1+\frac{h^{\frac{2}{3}}w}{xQ(x)-\xi
    _1})^{\frac{1}{2}}\\
  &=&\mp i (xQ(x)-\xi _1)^{\frac{1}{2}}\mp
  \frac{i h^{\frac{2}{3}}w}{2(xQ(x)-\xi _1)^{\frac{1}{2}}}+\frac{{\cal
      O}(h^{\frac{4}{3}})}{(xQ(x)-\xi _1)^{\frac{3}{2}}}.
\end{eqnarray*}
It follows that 
\ekv{idn.25}{
-\Re \phi '_0\asymp \frac{h^{\frac{2}{3}}}{|x-x_1|^{\frac{1}{2}}}
\hbox{ when }x-x_1\gg h^{\frac{2}{3}}.
}
This quantity dominates over the remainder ${\cal O}(h)|x-x_0|^{-1}$ in
(\ref{cw.8}) when $|x-x_0|\gg h^{2/3}$, 
$$
\frac{h^{\frac{2}{3}}}{|x-x_0|^{\frac{1}{2}}}\gg \frac{h}{|x-x_0|}
$$
and hence 
\ekv{idn.26}{
-\Re \phi '\asymp \frac{h^{\frac{2}{3}}}{|x-x_1|^{\frac{1}{2}}}
\hbox{ when }x-x_1\gg h^{\frac{2}{3}}.
}
This is slightly worse than (\ref{idn.24.5}) and if that estimate had
been valid also for $x-x_1\gg h^{2/3}$, then we would get exactly the
same estimates as in the case of the exterior problem. 

\par It is natural to ask how much worse (\ref{idn.26}) is than
(\ref{idn.24.5}). Recall that we work on an interval $[-b,0]$ and that
$x_1\asymp \xi _1\ge -h^{\frac{2}{3}-\delta }$, so $x-x_1\le -x_1\le
h^{\frac{2}{3}-\delta }$. Thus we get
\ekv{idn.27}
{
\frac{\mathrm{RHS}(\ref{idn.24.5})}{\mathrm{RHS}(\ref{idn.26})}=\frac{|x-x_1|}{h^{\frac{2}{3}}}\le
h^{-\delta }.
}
For $-b\le y\le w\le x\le 0$ we have
\ekv{idn.28}
{
\frac{1}{C}h^\delta \int_y^x|\phi '(t)|dt-Ch\le -\Re \phi (x)+\Re \phi
(y)\le \int_y^x|\phi '(t)|dt,
}
\ekv{idn.28.3}
{
\frac{1}{C}|\phi '(w)||x-y|-Ch\le \int_y^x|\phi '(t)|dt\le C(|\phi
'(\widetilde{z}(x,y))||x-y|+h),
}
where $\widetilde{z}$ is the point in $\{ x,y\}$ maximizing
$|\widetilde{z}-x_1|$. 

\par (\ref{cw.c}) remains valid and we even have
\ekv{idn.28.7}
{
\frac{1}{C_\epsilon }e^{-\frac{\epsilon }{h}(-\Re \phi (x)+\Re \phi
  (y))}
\le \frac{h^{\frac{1}{3}}+|\phi '(x)|}{h^{\frac{1}{3}}+|\phi '(y)|}\le
C_\epsilon e^{\frac{\epsilon }{h}(-\Re \phi (x)+\Re \phi
  (y))},
}
as can be seen by comparing the logarithmic derivative of
$h^{1/3}+|\phi '(x)|$ with $-\Re \phi '/h$ in the region $x-x_1\gg
h^{2/3}$, where $\phi ''(x)={\cal O}(|x-x_0|^{-1/2})$ and
(\ref{idn.25}) holds.

\par The factor $h^\delta $ in (\ref{idn.28}) gives slight losses in
the estimates of Subsection \ref{cw} and we get
\begin{lemma}\label{idn2}
If $(P(x',\xi ')-z )u=0$ on $[-b,0]$, $u(0)=u(-b)=0$, then 
\ekv{idn.29}
{
\Vert (h^{\frac{1}{3}}+|\phi '|)^2u\Vert + \Vert (h\partial
_x)^2u\Vert
+\Vert (h^{\frac{1}{3}}+|\phi '|)h\partial _xu\Vert \le {\cal
  O}(h^{-2\delta })
\Vert v\Vert ,
}
when $\xi _1\ge -h^{2/3-\delta }$.
\end{lemma}
\begin{proof}
We solve the Dirichlet problem on $[-b,0]$ as in Subsection \ref{cw}
and start with applying the natural modification of the operator
$K$:
\ekv{idn.29.5}
{
Kv(x)=-\frac{1}{h}\int_{-b}^x e^{§\phi (x)-\phi (y))/h}v(y)dy
}
and Lemma \ref{cw0.5} deteriorates slightly to 
\begin{lemma}\label{idn3}
The ${\cal L}(L^2)$-norms of $(h^{\frac{1}{3}}+|\phi '|)\circ K$,
$(h^{\frac{1}{3}}+|\phi '|)^2\circ K\circ (h^{\frac{1}{3}}+|\phi
'|)^{-1}$, $ K \circ (h^{\frac{1}{3}}+|\phi '|)$ are ${\cal
  O}(1)h^{-\delta }$.
\end{lemma}
\begin{proof}
We use Schur's lemma as in the proof of Lemma \ref{cw0.5}. Thus for
instance, the $L^2$-norm of $(h^{\frac{1}{3}}+|\phi '|)\circ K$ is
bounded by the geometric mean of 
\begin{eqnarray*}
\mathrm{I}&=&\frac{1}{h}\sup_{-b\le x\le 0}\int_{-b}^x (h^{\frac{1}{3}}+|\phi
'(x)|)e^{\frac{1}{h}(\Re (\phi (x)-\phi (y))} dy,\\
\mathrm{II}&=&\frac{1}{h}\sup_{-b\le y\le 0}\int_{y}^0 (h^{\frac{1}{3}}+|\phi
'(x)|)e^{\frac{1}{h}(\Re (\phi (x)-\phi (y))} dx.
\end{eqnarray*}
Here, by (\ref{idn.28}), (\ref{idn.28.3}),
\ekv{idn.29.7}{e^{\frac{1}{h}\Re (\phi (x)-\phi (y))}\le
  Ce^{-\frac{1}{Ch^{1-\delta }}\int_y^x|\phi '(t)|dt}\le
  \widetilde{C}e^{-\frac{1}{\widetilde{C}h^{1-\delta
      }}(h^{\frac{1}{3}}+|\phi '(x)|)|x-y|},}
and we get $\mathrm{I}={\cal O}(h^{-\delta })$. 

\par To get the same estimate for $\mathrm{II}$ we also use
(\ref{idn.28.7}). The other $L^2$-norms are estimated similarly.
\end{proof}

\par The proof of Lemma \ref{idn2} can now be finished as in Subsection
\ref{cw}.
\end{proof}

We next eliminate the analyticity assumption in Lemma \ref{idn2}. Let
$x_1$ be the real turning point determined by $x_1Q(x_1)=\xi _1$, so
that $x_1\le {\cal O}(1)h^{\frac{2}{3}-\delta }$. Let
$x_2=x_1-h^{\frac{2}{3}-\delta }$. For a large but fixed $N$, put
\[
\widetilde{Q}(x)=\begin{cases}
Q(x),\ x\le x_2,\\
\sum_0^{N-1}\frac{1}{\alpha !}Q^{(\alpha )}(x_2)(x-x_2)^\alpha ,\ &x\ge x_2.
\end{cases}
\]

Since $\widetilde{Q}$ is holomorphic in a $h^{\frac{2}{3}-\delta
}$-neighborhood of $x_1$, we see that if $\widetilde{P}$ is the
corresponding operator then we have a null solution
$e^{-\widetilde{\phi }/h}$ of $P-z$ with the same properties as
$e^{-\phi /h}$ in the analytic case above and such that Lemma 11.2
applies. Now $\widetilde{Q}-Q={\cal O}(1)h^{(\frac{2}{3}-\delta )N}$
and if we choose $N$ large enough, it follows that $P-z$ has a null
solution $e^{-\phi /h}$, where $$\widetilde{Q}-Q,\, \phi
-\widetilde{\phi },\, \phi '-\widetilde{\phi }',\, \phi
''-\widetilde{\phi }''={\cal O}(h).$$ Another perturbation argument
shows that Lemma \ref{idn2} holds for $P-z$.

Let $x_{n,1}(x',\xi ')$ be the real turning point determined
by
$$
-x_{n,1}Q(x',x_{n,1},\xi ')+\xi _1=0
$$
where we recall that $\xi _1=r(x',\xi ')-\lambda $. In analogy with
(\ref{aed.3}), we can reformulate (\ref{idn.29}) as
\begin{eqnarray}
&\Vert (h^{\frac{2}{3}}+|x_n-x_{n,1}|)u\Vert + \Vert
(h\partial _{x_n})^2u\Vert + \Vert
(h^{\frac{2}{3}}+|x_n-x_{n,1}|)^{\frac{1}{2}}(h\partial _{x_n})u\Vert&\nonumber\\
&\le {\cal O}(h^{-2\delta })\Vert (P(x',\xi ')-z )u\|&\label{idn.30}
\end{eqnarray}
for smooth functions $u$ on $[-b,0]$, vanishing at the end
points. Notice here that 
$$
(h^{\frac{1}{3}}+|\phi '|)^2\asymp h^{\frac{2}{3}}+|x-x_{n,1}|.
$$

Define the ${\cal B}(x',\xi ')$ norm to be the left hand side in
(\ref{idn.30}) and let ${\cal B}$ be the space of functions on $[-b,0]$
with finite ${\cal B}$-norm that vanish at the end points. Then we
still have the symbol property (\ref{aed.9}) for $P(x',\xi '):{\cal
  B}(x',\xi ')\to L^2$ and we get (\ref{aed.11}) for $E=(P(x',\xi
')-z )^{-1}$ with a slight loss:
\ekv{idn.31}
{
\partial _{x'}^\alpha \partial _{\xi '}^\beta E={\cal O}_{\alpha
  ,\beta }(h^{-2\delta (1+|\alpha |+|\beta |)})(h^{\frac{2}{3}}+|\lambda
-r(x',\xi ')|)^{-\beta _1},\ L^2\to {\cal B}.
}

\par
We get (\ref{pdn.6}),
(\ref{pdn.7}) with loss (due to the non-monotonicity of $\Lambda
=(h^{\frac{2}{3}}+|\lambda -r(x,\xi ')|)^{\frac{1}{2}}$ as a function
of $x_n$ between $x_{n,1}$ and $0$ when $x_{n,1}<0$):
\ekv{idn.32} {
  h^{\frac{1}{2}}(h^{\frac{2}{3}}+|r-\lambda |)^{\frac{3}{4}}|u(0)|\le
  Ch^{-3\delta/4 }\Vert u\Vert_{{\cal B}}, } 
\ekv{idn.33} {
  h^{\frac{1}{2}}(h^{\frac{2}{3}}+|r-\lambda
  |)^{\frac{1}{4}}|h\partial_{x_n}u(0)|\le Ch^{-\delta/4 }\Vert u\Vert_{{\cal B}}.}

\par Normalize $\phi $ by imposing the condition $\phi (0)=0$ and let
$e_{x',\xi '}=e^{-\frac{1}{h}\phi }$ be the null solution of $P(x',\xi
')-z $ so that $e_{x',\xi '}(0)=1$ and $e_{x',\xi '}(-b)$ is
exponentially small. Using (\ref{idn.28.7}), (\ref{idn.29.7}), we get (\ref{pdn.8}) with a $\delta
$ loss:
\ekv{idn.34} { \Vert e_{x',\xi '}\Vert_{{\cal B}}\le {\cal
    O}(1)h^{\frac{1-\delta }{2} }(h^{\frac{2}{3}}+|r-\lambda
  |)^{\frac{3}{4}}.  } 

Adding an exponentially small reflected null solution to $e_{x',\xi
  '}$ and renormalizing, we get a new null solution, that we denote
by $e_{x',\xi '}$ instead of the earlier one, which satisfies the
boundary conditions $e_{x',\xi '}(0)=1$, $e_{x',\xi '}(-b)=0$ and
which also satisfies (\ref{idn.34}).
Then we get the weakened version of
(\ref{pdn.11}): 
\ekv{idn.35} {
  \Vert \partial _{x'}^\alpha \partial _{\xi '}^\beta e_{x',\xi
    '}\Vert_{{\cal B}}={\cal O}(1)h^{\frac{1-\delta }{2} -2\delta
    (|\alpha |+|\beta |)}(h^{\frac{2}{3}}+|r-\lambda
  |)^{\frac{3}{4}-\beta _1}.  }

\par As a first approximation to the microlocal interior Poisson
operator on $\{ x;\, -b\le x_n\le 0,\, |x'|\le {\cal O}(1)\}$ we take
(cf (\ref{pdn.12}))
\ekv{idn.36}
{
K^0w=\mathrm{Op}_h(e_{x',\xi '}).
}
Then $\gamma K^0=1$, $(P-z )K^0=\mathrm{Op}_h(f_{x',\xi '})$,
where,
$$
f_{x',\xi '}=\sum_{\alpha \ne 0}\frac{h^{|\alpha |}}{\alpha !}\partial
_{\xi '}^\alpha P(x',\xi ')D_{x'}^\alpha e_{x',\xi '},
$$
and by (\ref{aed.9}), (\ref{idn.35}),
$$
\Vert \partial _{x'}^\alpha \partial _{\xi '}^\beta f_{x',\xi
  '}\Vert_{L^2}
={\cal O}(1)h^{\frac{3}{2}-\frac{5\delta }{2}-2\delta (|\alpha |+|\beta
  |)}(h^{\frac{2}{3}}+|r-\lambda |)^{-\frac{1}{4}-\beta _1}.
$$
Using $E$ as a first approximation, we can construct an
operator-valued symbol $\widetilde{E}(x',\xi ';h)$ such that
$\widetilde{E}(x',hD_{x'};h)$ inverts
$P(x',hD_{x'})-z $ to all orders in in $h$. We get a microlocal
Poisson operator to all orders in $h$ by putting
$$
\widetilde{K}=K^0-\widetilde{E}\circ (P-z )K^0=K^0+\mathrm{Op}_h(\widetilde{r}),
$$
and $\widetilde{r}$ fulfills the slightly deteriorated version of
(\ref{pdn.16}):
$$
\Vert \partial _{x'}^\alpha \partial _{\xi '}^\beta
\widetilde{r}\Vert_{{\cal B}}={\cal O}(1)h^{\frac{3}{2}-\frac{5\delta
  }{2}-2\delta (1+|\alpha |+|\beta |)}(h^{\frac{2}{3}}+|r-\lambda
|)^{-\frac{1}{4}-\beta _1}.
$$
Now $\widetilde{K}$ can be written as in (\ref{pdn.17}) and we have
(\ref{pdn.18}). The symbol $e_{x',\xi '}+\widetilde{r}_{x',\xi '}$
there satisfies
\[
\Vert \partial _{x'}^\alpha \partial _{\xi '}^\beta (e_{x',\xi
  '}+\widetilde{r}_{x',\xi '})\Vert_{{\cal B}}={\cal
  O}(1)h^{\frac{1}{2}-\frac{\delta }{2}-2\delta (|\alpha |+|\beta
  |)}(h^{\frac{2}{3}}+|r-\lambda |)^{\frac{3}{4}-\beta _1},
\]
when $\delta >0$ is small enough.
From this estimate and the similar ones in the other regions we get 
\ekv{idn.36.5}{
\widetilde{K}={\cal O}(h^{\frac{1}{6}}):H_h^{\frac{3}{2}}\to H_h^2,
}
and this also holds for the exact Poisson operator
$K_\mathrm{in}=K^V_\mathrm{in}.$

The corresponding DN-map is a pseudodifferential operator with symbol 
$$
n(x',\xi ';h)=\gamma hD_{x_n}(e+\widetilde{r}),
$$
and combing the above estimate with (\ref{idn.33}), we get the
estimate
\ekv{idn.37}{
\partial _{x'}^\alpha \partial _{\xi '}^\beta n=
{\cal O}(1)h^{-\frac{3\delta }{4}-2\delta (|\alpha |+|\beta
  |)}(h^{\frac{2}{3}}+|r-\lambda |)^{\frac{1}{2}-\beta _1}.}
This is a bounded symbol in the region where 
$h^{-3\delta /4}|r-\lambda |^{1/2}={\cal O}(1)$, i.e. where
$|r-\lambda |={\cal O}(1)h^{3\delta /2}$ and to get an better
conclusion, we take a closer look:

First, we see that 
$$
\gamma hD_{x_n}e_{x',\xi '}=i\partial _{x_n}\phi (0)={\cal
  O}(1)(h^{\frac{2}{3}}+|r-\lambda |)^{\frac{1}{2}}
$$
is bounded. Secondly, from the above estimate on the ${\cal B}$ norm
of $\widetilde{r}$ and (\ref{idn.33}), we conclude that
$$
\gamma hD_{x_n}\widetilde{r}={\cal
  O}(1)h^{1-(\frac{5}{2}+\frac{1}{4})\delta
}(h^{\frac{2}{3}}+|r-\lambda |)^{-\frac{1}{2}}
$$
which is also bounded. Thus we have an improvement of (\ref{idn.37})
when $\alpha =\beta =0$, and we conclude that $n$ is in a sufficiently
good symbol class to conclude that its quantization is $L^2$ bounded. 

Patching together the different microlocal Poisson operators, we get
an approximation mod ${\cal O}(h^\infty )$ in ${\cal
  L}(H_h^{\frac{3}{2}},H_h^2)$ of $K_{\mathrm{in}}$ and also the
conclusion of Proposition \ref{idn1} from the boundedness of the
corresponding microlocal DN-maps.
\end{proof}

Let $V$ be as in Proposition \ref{idn1} and let $K_\mathrm{in}^V$ and
${\cal N}_\mathrm{in}^V$ denote the corresponding Poisson and
Dirichlet to Neumann operators. Let $W\in L^\infty (\Omega ;{\bf
  R})$. Then
$$
K_\mathrm{in}^{V+W}=K_\mathrm{in}^V-(P_\mathrm{in}^{V+W}-z )^{-1}WK_\mathrm{in}^V=:
K_\mathrm{in}^V+A,$$
where in view of (\ref{idn.36.5}):
$$
\Vert A\Vert_{{\cal L}(H_h^{3/2},H_h^2)}\le {\cal
  O}(1)h^{-\frac{2}{3}+\frac{1}{6}}\Vert W\Vert_{L^\infty }={\cal
  O}(1)h^{-\frac{1}{2}}\Vert W\Vert_{L^\infty }.
$$
Thus ${\cal N}_\mathrm{in}^{V+W}={\cal N}_\mathrm{in}^V+B$, $B=\gamma
hD_\nu A$, and we get
$$
\Vert B\Vert_{{\cal L}(H_h^{3/2},H_h^2)}={\cal
  O}(1)h^{-\frac{1}{2}-\frac{1}{2}}\Vert W\Vert_{L^\infty }={\cal
  O}(1)h^{-1}\Vert W\Vert_{L^\infty }.
$$
This implies,
\begin{prop}\label{idn4}
The conclusion of Proposition \ref{idn1} remains valid if we replace
$V$ in there with $V+W$, where $W\in L^\infty (\Omega ;{\bf R})$
satisfies
\ekv{idn.38}
{
\Vert W\Vert_{L^\infty }\le {\cal O}(h).
}
\end{prop}
 When $W=\delta \Theta q_\omega $ is as in Theorem \ref{re1}, we have
 (\ref{idn.38}), provided $\alpha $ is large enough. See Remark \ref{sv1}.

For a greater generality of our results it is of interest to have a
the following variant of the last proposition, where the perturbation
$W$ can be independent of $h$. We start with some simple exponentially
weighted estimates. Let $\phi \in C^\infty (\overline{{\cal O}};{\bf
  R})$ and consider
$$
P^{V,\epsilon }=e^{\frac{\epsilon \phi }{h}}P^Ve^{-\frac{\epsilon \phi
  }{h}}=P^V+F,
$$
where $$F=i\epsilon (\phi '\cdot hD_x+hD\cdot \phi ')-\epsilon ^2(\phi
')^2={\cal O}(\epsilon ):H_h^1\to H_h^0.$$ Since $(P_\mathrm{in}^V-z
)^{-1}={\cal O}(h^{-2/3}):H_h^0\to H_h^2$ when $1/2<\Re z <2$, $|\Im
z |\asymp h^{2/3}$, we get the same conclusion for
$(P_\mathrm{in}^{V,\epsilon }-z )^{-1}=e^{\epsilon \phi
  /h}(P_\mathrm{in}^V-z )^{-1}e^{-\epsilon \phi /h}$, provided
that $\epsilon \ll h^{2/3}$.

\par Now, let ${{\phi }_\vert}_{\partial {\cal O}}=0$. Then
$K^{V,\epsilon }=e^{\epsilon \phi /h}K^V$ is the Poisson operator for
$P^{V,\epsilon }-z $. We can also get $K^{V,\epsilon
}$ by a perturbative argument, writing
\[\begin{split}
K^{V,\epsilon }=K^V-(P_\mathrm{in}^{V,\epsilon }-z )^{-1}FK^V&\\
=K^V+{\cal O}(h^{-\frac{2}{3}}\epsilon h^{\frac{1}{6}})={\cal
  O}(h^{\frac{1}{6}}):&\, H_h^{\frac{3}{2}}\to H_h^2.
\end{split}
\]
Thus $e^{\epsilon \phi /h}K^V(z )={\cal O}(h^{1/6}):H_h^{3/2}\to
H_h^2$.
Now assume that
$$
W(x)={\cal O}(\mathrm{dist\,}(x,\partial {\cal O})^{N_0}),
$$
for some $N_0>0$, to be determined. Then $WK^V=We^{-\epsilon \phi
  /h}e^{\epsilon \phi /h}K^V$ and taking $\phi \asymp
\mathrm{dist\,}(\cdot ,\partial {\cal O})$, $\epsilon \ge
h^{2/3}/{\cal O}(1)$, we see that 
$We^{-\epsilon \phi /h}={\cal
  O}(\mathrm{dist}^{N_0}e^{-\mathrm{dist}/(Ch^{1/3})})={\cal
  O}(h^{N_0/3})$. Then as in the discussion prior to Proposition
\ref{idn4}, we have $K_\mathrm{in}^{V+W}=K_\mathrm{in}^{V}+A$, where
$$
A=(P_\mathrm{in}^{V+W}-z )^{-1}WK_\mathrm{in}^V={\cal
  O}(1)h^{-\frac{2}{3}+\frac{N_0}{3}+\frac{1}{6}}:H_h^{\frac{3}{2}}\to H_h^2.
$$
The choice $N_0=3$ gives $A={\cal O}(h^{1/2}):H_h^{3/2}\to H_h^2$ and
we get the following variant and extension of Proposition \ref{idn4}:
\begin{prop}\label{idn5}
The conclusion of Proposition \ref{idn1} remains valid if we replace
$V$ there with $V+W$, where $W\in L^\infty (\Omega ;{\bf R})$
satisfies
\ekv{idn.39}
{W(x)={\cal O}(\mathrm{dist\,}(x,\partial {\cal O})^{3}).
}

\par More generally, we can take $W=W_1+W_2$, where $W_1$ and $W_2$
fulfill (\ref{idn.38}) and (\ref{idn.39}) respectively.
\end{prop}

\section{Some determinants}\label{sbd}
\setcounter{equation}{0}

Let 
\ekv{sbd.1}
{
V=V_0+W,
}
where $V_0$ is as in (\ref{idn.3.7}) and the real-valued term $W$ is
${\cal O}(1)$ in $L^\infty $.
We let 
\ekv{sbd.2}
{
P=-h^2\Delta +V=:P^V,\quad P_0=-h^2\Delta +V_0.
}
Recall the definitions of $P_\mathrm{out}$, ${\cal P}_\mathrm{out}$,
${\cal P}_\mathrm{in}$, $P_\mathrm{in}$ in Section \ref{red}, with the
potential $V$ as above. 

\par Our first task is to define the determinants of the factors in
(\ref{red.6}). In the following, $H^s$ denotes $H^s_h$ if nothing
else is indicated.
\begin{prop}\label{sbd1}
  The three factors in (\ref{red.6}) are meromorphic families of
  Fredholm operators in the region $\frac{1}{2}<\Re z<\frac{3}{2}$,
  $\Im z>-h^{2/3}c_0$, where $c_0$ is as in (\ref{outl.0}). More precisely, 
  $${\cal P}_\mathrm{in}(z):H^2({\cal O})\to H^0({\cal O})\times H^{3/2}(\partial {\cal O}),$$
$${\cal P}_\mathrm{out}(z):H^2({\cal O})\to H^0({\cal O})\times H^{1/2}(\partial {\cal O})$$
are holomorphic Fredholm families, while
$$\begin{pmatrix}1 &0\\ h^{\frac{1}{2}}BG_\mathrm{in} & {\cal N}_\mathrm{in}-{\cal
    N}_\mathrm{out}\end{pmatrix}:H^0({\cal O})\times H^{3/2}(\partial
{\cal O})\to H^0({\cal O})\times H^{1/2}(\partial
{\cal O})$$
is a meromorphic Fredholm family.
\end{prop}
\begin{proof}
This is clear for ${\cal P}_\mathrm{in}$, ${\cal P}_\mathrm{out}$, and
the factorization (\ref{red.6}) then implies that the remaining factor
is a meromorphic Fredholm family.
\end{proof}

\par From (\ref{red.6}) and the last proposition, we get
\ekv{sbd.3}
{
\det {\cal P}_\mathrm{out}(z)=\det ({\cal N}_\mathrm{in}-{\cal
  N}_\mathrm{ext})\det {\cal P}_\mathrm{in}(z).}

\par The next result will permit us to do some analysis.
\begin{prop}\label{sbd2}
The determinants of the factors in (\ref{red.6}) can also be defined as
in Subsection \ref{gd}.
\end{prop}

\begin{proof}
We have 
\ekv{sbd.4}
{
\partial _z{\cal
  P}_\mathrm{in}(z)=\begin{pmatrix}-1\\0\end{pmatrix},\ \partial
_z^2{\cal P}_\mathrm{in}(z)=0.
}
Thus the $C_p$-norm of $\partial _z{\cal P}_\mathrm{in}(z):H^2\to
H^0\times H^{3/2}$ can be bounded by that of the inclusion map $\iota
:H^2({\cal O})\to H^0({\cal O})$. Here we can
consider ${\cal O}$ as a bounded subset with smooth boundary of a
torus $T$ and choose a uniformly bounded Seeley extension $\sigma
:H^2({\cal O})\to H^2(T)$ so that $\iota =\rho \iota _T\sigma
$, where $\iota _T:H^2(T)\to H^0(T)$ is the inclusion map and $\rho
:H^0(T)\to H^0({\cal O})$ is the restriction map. $\rho $ and $\sigma
$ being uniformly bounded, it suffices to study the Schatten class
norm of $\iota _T$. Here $H^2(T)=(1-h^2\Delta )^{-1}(H^0(T))$ so the
problem is that of the $C_p$-norm of $(1-h^2\Delta )^{-1}:H^0(T)\to
H^0(T)$.

By Weyl's law we get for $p>n/2$,
\begin{eqnarray*}
&&\Vert (1-h^2\Delta )^{-1}\Vert_{C_p}^p=\int_0^\infty (1+h^2\lambda
)^{-p}d{\cal O}(\lambda ^{\frac{n}{2}})\\
&&={\cal O}(h^2)\int_0^\infty \frac{\lambda
  ^{\frac{n}{2}}}{(1+h^2\lambda )^{p+1}}d\lambda ={\cal
  O}(h^{-n})\int_0^\infty \frac{t^{\frac{n}{2}}}{(1+t)^{p+1}}dt
\end{eqnarray*}
and then
$$\Vert \iota_T \Vert_{C_p}^p={\cal O}(h^{-n}),$$
so 
\ekv{sbd.5}
{
\Vert \partial _z{\cal P}_\mathrm{in}\Vert_{C_p}={\cal
  O}(h^{-\frac{n}{p}}),\ p>\frac{n}{2}.
}
This implies that ${\cal P}_\mathrm{in}(z)$ satisfies (\ref{gd.2}) for
any $p>n/2$, so its determinant can be defined as in Subsection \ref{gd}.

In order to treat the other two operators, we need to collect some
more information about ${\cal N}_\mathrm{ext}$. 
\begin{lemma}\label{sbd3}
For $z$ as in Proposition \ref{sbd1}, we have for all $s\in {\bf R}$,
$k\in {\bf N}$:
\ekv{sbd.8}
{
\partial _z^k{\cal N}_\mathrm{ext}(z)={\cal O}((\Im
z+c_0h^{\frac{2}{3}})^{-k}):H^s\to H^{s-1+2k}.
}
\end{lemma}
\begin{proof}
Microlocally near the glancing hypersurface and in the hyperbolic
region, this follows from Corollary \ref{pdn4} and the Cauchy
inequalities. The extra regularization comes from the elliptic region
and here $K_\mathrm{ext}(z)$ is the Poisson operator of an elliptic
boundary value problem and satisfies
$$
\partial _z^kK_\mathrm{ext}(z)=C_k(P_\mathrm{ext}-z)^{-k}K_\mathrm{ext}(z).
$$
\end{proof}

\par Applying the lemma to $B=B(z)$ in (\ref{lb.3}), we get 
\ekv{sbd.6}
{
  \partial _z^kB(z)={\cal O}(1)h^{-\frac{1}{2}}(\Im
  z+c_0h^{\frac{2}{3}})^{-k}:H^2({\cal O})\to
  H^{\frac{1}{2}+2k}(\partial {\cal O}). }

\par The $C_p$-norm of the inclusion map $H^{\frac{1}{2}+2k}\to
H^{\frac{1}{2}}$ is bounded by a constant times the $C_p$-norm of
$(1-h^2\Delta _{\partial {\cal O}})^{-k}$ which by Weyl asymptotics is
finite and ${\cal O}(h^{(1-n)/p})$ when $p\ge 1$ and
$p>(n-1)/(2k)$. Thus for each such $p$,
$$\partial _z^kB\in C_p(H^2,H^{\frac{1}{2}}),\ \Vert \partial
_z^kB\Vert_{C_p}= {\cal O}(h^{-\frac{1}{2}+\frac{1-n}{p}}(\Im z+c_0h^{\frac{2}{3}})^{-k}).
$$
It then follows as in the proof of (\ref{sbd.5}) that when $p\ge 1$
and $p>n/(2k)$,
\ekv{sbd.7}
{\begin{split}
  &\partial _{z}^k{\cal P}_\mathrm{out}(z)\in C_p(H^2,H^0\times
  H^{\frac{1}{2}}),\\ &\Vert \partial _z^k{\cal
    P}_\mathrm{out}(z)\Vert_{C_p}={\cal O}(h^{-\max
    (\frac{n}{p},\frac{1}{2}+\frac{n-1}{p}+\frac{2}{3}k)}). \end{split}  }  Thus
we have verified (\ref{gd.2}) with $p=(n+\epsilon )/2$ and $\det {\cal
  P}_\mathrm{out}(z)$ can indeed be defined as in Subsection \ref{gd}.

\par In that subsection we have seen that if $P(z)$ fulfills
(\ref{gd.2}), then so does $P(z)^{-1}$ on the open subset of
bijectivity. We also saw that if $P_1(z)\in {\cal L}({\cal
  H}_1,{\cal H}_2)$, $P_2(z)\in {\cal L}({\cal
  H}_2,{\cal H}_3)$ satisfy (\ref{gd.2}), then so does
$P_1(z)P_2(z)$. Having checked that ${\cal P}_\mathrm{in}(z)$ and
${\cal P}_\mathrm{out}(z)$ satisfy (\ref{gd.2}), we conclude from
(\ref{red.6}) that $\begin{pmatrix}1 &0\\h^{\frac{1}{2}}BG_{\mathrm{in}} &{\cal
    N}_\mathrm{in}-{\cal N}_\mathrm{ext}\end{pmatrix}$ also satisfies
(\ref{gd.2}) and the proposition follows from Subsection \ref{gd}.
\end{proof}

\section{Upper bounds on the basic determinant}\label{ub}
\setcounter{equation}{0}
The first task will be to get an upper bound on $\ln |\det {\cal P_\mathrm{out}}|$
in the whole region 
\ekv{ub.1}{|\Im z|<c_0 h^{\frac{2}{3}},\ \frac{1}{2}<\Re
  z<2}
by some negative power of $h$. 
 
\par Using the addendum at the end of Subsection \ref{gd}, we shall derive a rough upper bound
on $\ln |\det {\cal P}_\mathrm{out}(z)|$. Let
$\widetilde{P}=P+i1_{{\cal O}}$, $\widetilde{{\cal
    P}}_\mathrm{out}(z)=\begin{pmatrix}\widetilde{P}-z \\
  B(z)\end{pmatrix}$.
Assume first that $W=0$ so that $V=V_0$ is smooth.
Thanks to the perturbation $i1_{\cal O}$, \ekv{ub.9.2} {
  \widetilde{{\cal P}}_\mathrm{in}:=\begin{pmatrix}\widetilde{P}-z\\ h^\frac{1}{2}\gamma \end{pmatrix}:\, H^{s+2}({\cal O})\to H^s({\cal
    O})\times H^{s+\frac{3}{2}}(\partial {\cal O}) } is bijective with
an inverse
$\widetilde{E}_\mathrm{in}(z)=\begin{pmatrix}\widetilde{G}_\mathrm{in}(z)
  & h^{-\frac{1}{2}}\widetilde{K}_\mathrm{in}(z)\end{pmatrix}$, where
$\widetilde{G}_\mathrm{in}={\cal O}_s(1):H^s\to H^{s+2}$,
$\widetilde{K}_\mathrm{in}={\cal O}_s(h^{1/2}):H^{s+3/2}\to H^s$, for
$0<h\le h(s)$, $0\le s <\infty $. This is the inverse of an elliptic
boundary value problem and we see that $\widetilde{{\cal
    N}}_\mathrm{in}$, defined as in (\ref{ub.2.5}),
is a nice $h$-pseudodifferential operator on $\partial {\cal O}$ of
order 0 in $h$ and of order 1 in $\xi '$, with leading symbol
$-i(i+(\xi ')^2-z)^{1/2}$, where we use the principal branch of the
square root with a cut along the negative real axis. This symbol takes
its values in the interior of the fourth quadrant.  Then in analogy
with (\ref{red.6}), we have \ekv{ub.9.1} {
  \widetilde{{\cal P}}_\mathrm{out}(z)=\begin{pmatrix} 1 &0\\
    h^{\frac{1}{2}}B\widetilde{G}_\mathrm{in} &\widetilde{{\cal N}}_\mathrm{in}-{\cal
      N}_\mathrm{ext}\end{pmatrix}\widetilde{{\cal P}}_\mathrm{in}(z),
} where $B$ was given in (\ref{lb.3}).

We have already investigated ${\cal N}_\mathrm{ext}$ and found that it
is an $h$-pseudodifferential operator whose symbol is nice away from
${\cal G}$ where it becomes exotic but  small. Away from that set it
is of order $(0,1)$ in $(h,\xi ')$ with leading part $i((\xi
')^2-z)^{1/2}$. When $\Im z\ge 0$ its values are confined to the
first quadrant. 

\par From this it follows that $\widetilde{{\cal N}}_\mathrm{in}-{\cal
  N}_\mathrm{ext}$ is an elliptic $h$-pseudodifferential operator of
order (0,1) whose symbol has a small exotic part near ${\cal G}$. 
Consequently, for every $s\in {\bf R}$;
\ekv{ub.9.2.5}
{
\widetilde{{\cal N}}_\mathrm{in}-{\cal N}_\mathrm{ext}:\,
H^{s+\frac{3}{2}}\to H^{s+\frac{1}{2}}
}
is bijective with a uniformly bounded inverse for $0<h\le h(s)\ll 1$.

\par It now follows from (\ref{ub.9.1}) and from the fact that
$B={\cal O}_s(h^{-1/2}):H^{s+2}\to H^{s+3/2}$ for every $s\ge 0$, that 
\ekv{ub.9.3}
{
\begin{split}
\widetilde{{\cal P}}_\mathrm{out}^{-1}=\widetilde{{\cal
    P}}_\mathrm{in}(z)^{-1}\begin{pmatrix} 1 &0\\
-(\widetilde{{\cal N}}_\mathrm{in}-{\cal N}_\mathrm{ext})^{-1}h^{\frac{1}{2}}B\widetilde{G}_\mathrm{in} &
(\widetilde{{\cal N}}_\mathrm{in}-{\cal N}_\mathrm{ext})^{-1}\end{pmatrix}\\
=\begin{pmatrix}\widetilde{G}_\mathrm{in}-\widetilde{K}_\mathrm{in}(\widetilde{{\cal
      N}}_\mathrm{in}-{\cal N}_\mathrm{ext})^{-1}B\widetilde{G}_\mathrm{in} &h^{-\frac{1}{2}}\widetilde{K}_\mathrm{in}(\widetilde{{\cal N}}_\mathrm{in}-{\cal N}_\mathrm{ext})^{-1}
\end{pmatrix}.
\end{split}
}
We conclude that for every $s\in [0,+\infty [$,
\ekv{ub.9.4}
{\begin{split}
&\widetilde{{\cal P}}_\mathrm{out}(z)=H^{s+2}\to H^s\times H^{s+\frac{1}{2}}
\hbox{ has an inverse}\\ &\widetilde{{\cal
    E}}_\mathrm{out}(z)=\begin{pmatrix}\widetilde{G}_\mathrm{out}
& h^{-\frac{1}{2}}\widetilde{K}_\mathrm{out}\end{pmatrix} \hbox{ with }
\widetilde{G}_\mathrm{out}={\cal O}_s(1):H^s\to H^{s+2},\\
&\widetilde{K}_\mathrm{out}={\cal O}_s(h^{1/2}):H^{s+1/2}\to H^{s+2},
\hbox{ for }0<h\le h(s).
\end{split}}

\par Now drop the assumption that $W=0$ and take again $V=V_0+W$ where
we assume that $\Vert W\Vert_{L^\infty }\le 1/C$ with $C$ large
enough. Then from (\ref{ub.9.4}) (where we had $V=V_0$) and a simùple
perturbation argument we see that
\ekv{ub.9.5}
{
\hbox{(\ref{ub.9.4})  remains
  valid for }s=0.
}

\par Write
\ekv{ub.9.6}{
{\cal P}_\mathrm{out}(z)=(1+{\cal
  K}(z))\widetilde{{\cal P}}_\mathrm{out}(z),
}
where 
$$
{\cal K}(z)=\begin{pmatrix}P-\widetilde{P} \\ 0
\end{pmatrix}\widetilde{\cal E}_\mathrm{out}(z).
$$

Now $\widetilde{\cal P}_\mathrm{out}(z)$ satisfies (\ref{sbd.7})
when $p\ge 1$ and $p>n/(2k)$ and hence also (\ref{gd.2}) with $p$
there equal to $(n+\epsilon )/2$. Moreover, as in the case of ${\cal
  P}_\mathrm{out}$, the corresponding Schatten class norm of $\partial
_z^k\widetilde{{\cal P}}_\mathrm{out}$ is bounded by some negative
power of $h$. Using the bounds on the norm $\widetilde{{\cal
    E}}_\mathrm{out}$, we see that this operator has the same
property. Consequently we have the same properties for ${\cal K}(z)$
and Proposition \ref{ub1} applies and shows that $\det (1+{\cal
  K}(z))$ can be defined as in Subsection \ref{gd} and satisfies the
upper bound
\ekv{ub.9.7}{
\ln |\det (1+{\cal K}(z))|\le {\cal O}(h^{-N})
}
for some $N\ge 0$. Similarly, $\det \widetilde{{\cal P}}_\mathrm{out}(z)$ is
well-defined and can be realized so that 
\ekv{ub.9.8}
{
|\ln |\det \widetilde{{\cal P}}_\mathrm{out}||\le {\cal O}(h^{-N}).
}
Combining this with (\ref{ub.9.6}), we get
\begin{prop}\label{ub2}
$\exists$ $N_0>0$ such that 
\ekv{ub.10}
{
\ln |\det {\cal P}_\mathrm{out}(z)|\le {\cal O}(1)h^{-N_0}.
}
\end{prop}
 
We next start a more precise study of $\det {\cal P}_\mathrm{out}$ in
the region
\ekv{ub'.10.5}
{
\frac{1}{2}<\Re z <2,\ c{h^{\frac{2}{3}}}<|\Im z|<c_0{h^{\frac{2}{3}}},
}
where $c>0$ can be chosen
arbitrarily small. For that we shall use Proposition \ref{sbd2} and
study the two factors to the right in (\ref{sbd.3}).

\par We start with $\det ({\cal N}_\mathrm{in}-{\cal N}_\mathrm{ext})$
and the aim is to write this function as a product of two factors, one
being holomorphic and non-vanishing in the whole rectangle
$]1/2,2[+i]-h^{2/3}c_0,h^{2/3}c_0[$ and the other being of the form
$\det (1+T(z))$, where $T$ is a meromorphic family of trace class
operators on $\partial {\cal O}$ with poles at $\sigma (P_\mathrm{in})$
and whose trace class norm is ${\cal O}(h^{1-n})$ when $|\Im
z|>h^{2/3}c$.

Let $P=P^V=-h^2\Delta +V$, $P_0=P^{V_0}=-h^2\Delta +V_0$, $V=V_0+W$
with $V_0$ as before, $W={\cal O}(h)$ in $L^\infty $ and we shall have
to strengthen the assumptions on $W$. In geodesic coordinates,
\ekv{ub'.11} { P=(hD_{x_n})^2+R(x,hD_{x'}),\
  P_0=(hD_{x_n})^2+R_0(x,hD_{x'}).  } Let $S:C^\infty (\overline{{\cal
    O}})\to C^\infty (\overline{{\cal {\cal O}}})$ be of the form
$S=S(x,hD_{x'})$ near $\partial {\cal O}$ in geodesic coordinates,
where $S\ge 0$ has compact support in $\xi '$. In the interior of
${\cal O}$ we arrange by cutting and pasting so that $S$ is a
pseudodifferential operator in all the variables of order $0$ in $h$
and with symbol of compact support in $\xi $. Put
\ekv{ub'.12}{\widetilde{P}_0=P_0+S,\ \widetilde{P}=P+S.}  Let $\chi
=\chi (x',\xi ')\in C_0^\infty (T^*\partial {\cal O})$ be equal to 1
near ${\cal H}\cup {\cal G}$. Let ${\cal N}={\cal N}_\mathrm{in}$ be
the Dirichlet to Neumann map associated to $P-z$ (and we will write
$P=P_\mathrm{in}$ when we wish to emphasize that we take the Dirichlet
realization). We start with the trivial decomposition \ekv{ub'.12.5}{
  {\cal N}={\cal N}\chi (x',hD_{x'})+\ {\cal N}(1-\chi (x',hD_{x'})).
} By Proposition \ref{idn4} the first term to the right is of trace
class $C_1(H^{3/2},H^{1/2})$ and the corresponding trace class norm is
${\cal O}(h^{1-n})$ when $|\Im z|\ge h^{2/3}c$.

Now $S$ can be chosen so that 
$$\begin{pmatrix}\widetilde{P}_0-z\\ h^\frac{1}{2}\gamma \end{pmatrix}:H^2\to
H^0\times H^{\frac{3}{2}}$$ is bijective with a uniformly bounded
inverse $\begin{pmatrix}\widetilde{G}_0
  &h^{-\frac{1}{2}}\widetilde{K}_0\end{pmatrix}$. Since $\Vert
W\Vert_{L^\infty }={\cal O}(h)\ll 1$, we have the same fact for 
$$\begin{pmatrix}\widetilde{P}-z\\ h^\frac{1}{2}\gamma \end{pmatrix}:H^2\to
H^0\times H^{\frac{3}{2}}$$ and we let $\begin{pmatrix}\widetilde{G}
  &h^{-\frac{1}{2}}\widetilde{K}\end{pmatrix}$ be the inverse. 

\par $K=K_\mathrm{in}$ satisfies
\ekv{ub'.13}
{
K(1-\chi )=\widetilde{K}(1-\chi
)+(P_\mathrm{in}-z)^{-1}S\widetilde{K}(1-\chi ).
}
Hence 
\ekv{ub'.14}
{\begin{split}
&{\cal N}(1-\chi )=\mathrm{I}+\mathrm{II},\ \mathrm{I}=\widetilde{{\cal
  N}}(1-\chi ),\\
&\mathrm{II}=\gamma hD_\nu (P-z)^{-1}S\widetilde{K}(1-\chi ).
\end{split}} Here
$\widetilde{K}=\widetilde{K}_0-(\widetilde{P}-z)^{-1}W\widetilde{K}_0=\widetilde{K}_0+{\cal
  O}(h^{1/2})\Vert W\Vert_{L^\infty }:\, H^{3/2}\to H^2$, so \ekv{ub'.15} {
  \widetilde{{\cal N}}=\widetilde{{\cal N}}_0+{\cal O}(1)\Vert
  W\Vert_{L^\infty }:\, H^{3/2}\to H^{1/2}.  } Now, as we saw earlier
in a slightly different situation, $\widetilde{{\cal N}}_0$ is a nice
$h$-pseudodifferential operator of order (0,1) in $(h,\xi '))$ with
leading symbol $-i(s(x',\xi ')+(\xi ')^2-z)^{1/2}$ and as in
(\ref{ub.9.2.5}) $\widetilde{{\cal N}}_0-{\cal
  N}_\mathrm{ext}=H^{s+\frac{3}{2}}\to H^{s+\frac{1}{2}}$ is bijective
with a uniformly bounded inverse for $0<h<h(s)\ll 1$. From
(\ref{ub'.15}) we get the same conclusion for $\widetilde{{\cal N}}-{\cal
  N}_\mathrm{ext}:H^{\frac{3}{2}}\to H^{\frac{1}{2}}$.

\par We shall next estimate the norm of $S\widetilde{K}(1-\chi
):H^{3/2}\to H^0$ and for that we try to ``commute'' $1-\chi $ and $K$
and exploit that $S(1-\chi )={\cal O}(h^\infty )$. From $\gamma
[\widetilde{K},\chi ]=0$, $(\widetilde{P}-z)[\widetilde{K},\chi
]=-[\widetilde{P},\chi ]\widetilde{K}$, we get
\ekv{ub'.16}
{
[\widetilde{K},\chi
]=-(\widetilde{P}_\mathrm{in}-z)^{-1}[\widetilde{P},\chi ]\widetilde{K}.
}

Moreover,
\ekv{ub'.17}
{
S\widetilde{K}(1-\chi )=S(1-\chi )\widetilde{K}-S[\widetilde{K},\chi ],
}
where the first term to the right is ${\cal O}(h^\infty ):H^{3/2}\to
H^0$ and we shall see that $[\widetilde{K},\chi ]={\cal
  O}(h^{3/2}):H^{3/2}\to H^0$, provided that $\nabla W={\cal O}(1)$ in
$L^\infty $: Assume
\ekv{ub'.20}
{
\partial ^\alpha W={\cal O}(1)\hbox{ in }L^\infty ,\hbox{ for }|\alpha
|\le 1 ,
}
in addition to the previous assumption that $\Vert W\Vert ={\cal O}(h)$.
As in the remark after Proposition \ref{idn4}, this will hold for
$W=\delta \Theta q_\omega $ as in Theorem \ref{re1}.
\begin{lemma}\label{ub'3}
Under the assumption (\ref{ub'.20}), we have
\ekv{ub'.21}{
[\widetilde{K},\chi ]={\cal O}(h^{3/2}):\, H^\frac{3}{2}\to H^2.
}
\end{lemma}
\begin{proof}
If $Q\in C_0^\infty ({\bf R}^{2n})$ we have the following
representation of the $h$-pseudodifferential operator $Q(x,hD_{x})$ in
the classical quantization, obtained in \cite{AnSj04}:
\ekv{ub'.23}
{
\begin{split}
Q(x,hD)&=(-\frac{1}{\pi })^{2n}\int ...\int
(z_1-x_1)^{-1}..(z_n-x_n)^{-1} (\zeta _1-hD_{x_1})^{-1}..\\ & (\zeta 
_n-hD_{x_n})^{-1}
\partial _{\overline{z}_1}..\partial _{\overline{z}_n}\partial
_{\overline{\zeta }_1}..\partial _{\overline{\zeta
  }_n}\widetilde{Q}(z_1,..,z_n,\zeta _1,..,\zeta _n)L(dz)L(d\zeta ),
\end{split}
}
where $\widetilde{Q}\in C_0^\infty $ is an almost holomorphic
extension satisfying
$$
\partial _{(\overline{z},\overline{\zeta })}\widetilde{Q}={\cal
  O}((|\Im z_1|..|\Im z_n||\Im \zeta _1|..|\Im \zeta _n|)^\infty ).
$$
From this representation we recover the wellknown fact that $Q={\cal
  O}(1):L^2\to L^2$ and for $[Q,W]$ we get a similar formula with $2n$
terms, obtained by replacing one of $(z_j-x_j)^{-1}$ or $(\zeta
_j-hD_{x_j})^{-1}$ by $(z_j-x_j)^{-1}[x_j,W](z_j-x_j)^{-1}$ or $(\zeta
_j-hD_{x_j})^{-1}[hD_{x_j},W](\zeta
_j-hD_{x_j})^{-1}$ respectively. Then from the boundedness of $W$ and
$\nabla W$ we see that 
\ekv{ub'.25}
{
[Q(x,hD_x),W]={\cal O}(h):\, L^2\to L^2.
}
The lemma now follows from (\ref{ub'.25}) and (\ref{ub'.16}).
\end{proof}

\par Returning to (\ref{ub'.17}), we see that 
\ekv{ub'.27}
{
S\widetilde{K}(1-\chi )={\cal O}(h^{\frac{3}{2}}):\, H^{\frac{3}{2}}\to H^0.
}
We use this in the expression for $\mathrm{II}$ in (\ref{ub'.14})
together with the telescopic formula
\ekv{ub'.28}
{
(P-z)^{-1}=(\widetilde{P}-z)^{-1}\sum_0^{N-1} (S(\widetilde{P}-z)^{-1})^k+(P-z)^{-1}(S(\widetilde{P}-z)^{-1})^N
}
to see that 
\ekv{ub'.29}
{
\mathrm{II}(z)=\mathrm{III}(z)+\mathrm{IV}(z),
}
where 
\ekv{ub'.30}
{\mathrm{III}(z)=\gamma hD_\nu (\widetilde{P}-z)^{-1}\sum_0^{N-1}
  (S(\widetilde{P}-z)^{-1})^kS\widetilde{K}(1-\chi )
}
is holomorphic and ${\cal O}(h):H^{3/2}\to H^{1/2}$ in the whole
rectangle
$]1/2,2[+i]-h^{2/3}c_0,h^{2/3}c_0[$ and
\ekv{ub'.31}
{
\mathrm{IV}(z)=\gamma hD_\nu
(P-z)^{-1}(S(\widetilde{P}-z)^{-1})^NS\widetilde{K}(1-\chi ).
}

\par Let $N$ be the smallest integer with
\ekv{ub'.31.5}
{
N>\frac{n-1}{2}
}
and assume that 
\ekv{ub'.32}
{
\partial ^\alpha W={\cal O}(1)\hbox{ in }L^\infty \hbox{ for }|\alpha |
\le 2N.
}
Again this will hold for $W=\delta \Theta q_\omega $ as in Theorem
\ref{re1} if $\alpha (...)$ there is large enough.
Then $\mathrm{IV}(z)$ is locally
uniformly bounded $H^{3/2}\to H^{2(N+1)-3/2}=H^{2N+1/2}$ away from
$\sigma (P_\mathrm{in})$ and when
$|\Im z|\ge h^{2/3}c$ the norm is uniformly $\le {\cal
  O}(h^{\frac{3}{2}-\frac{1}{2}-\frac{2}{3}})={\cal O}(h^{\frac{1}{3}})$. Since
$2N>n-1$, we see that $\mathrm{IV}(z)\in C_1(H^{3/2},H^{1/2})$ and
that when $|\Im z|\ge h^{2/3}c$ the corresponding trace
class norm is $\le {\cal O}(h^{\frac{1}{3}+1-n})$. Summing up the
discussion so far, we have
\begin{prop}\label{ub'4}
  ${\cal N}={\cal N}_\mathrm{in}$ can be decomposed as \ekv{ub'.33} {
    {\cal N}=\widetilde{{\cal N}}+\mathrm{III}+({\cal
      N}-\widetilde{{\cal N}})\chi +\mathrm{IV}, } where
  $\widetilde{{\cal N}}=\widetilde{{\cal N}}_0+{\cal O}(1)\Vert
  W\Vert_{L^\infty }={\cal O}(1):H^{3/2}\to H^{1/2}$ and $\mathrm{III}={\cal
    O}(h):H^{3/2}\to H^{1/2}$ are holomorphic in the whole
  rectangle $]1/2,2[+i]-h^{2/3}c_0,h^{2/3}c_0[$, while $({\cal
    N}-\widetilde{{\cal N}})\chi $ and $\mathrm{IV}(z)$ are
  holomorphic away from $\sigma (P_\mathrm{in})$ with values in
  $C_1(H^{3/2},H^{1/2})$ and \ekv{ub'.34} { \Vert ({\cal
      N}-\widetilde{{\cal N}})\chi \Vert_{C_1}+\Vert
    \mathrm{IV}\Vert_{C_1}={\cal O}(h^{1-n}),\ |\Im z|\ge
    h^{2/3}c.  }
\end{prop}

Now write 
\ekv{ub'.35}
{
{\cal N}_\mathrm{in}-{\cal N}_\mathrm{ext}=\widehat{A}(z)+({\cal N}-\widetilde{{\cal
    N}})\chi +\mathrm{IV},
}
where 
\ekv{ub'.37}{\widehat{A}(z):=\widetilde{{\cal
    N}}+\mathrm{III}-{\cal N}_\mathrm{ext}:H^{3/2}\to H^{1/2},}
 is
holomorphic, uniformly bounded and uniformly invertible in the whole
rectangle, and factorize,
\ekv{ub'.36}{{\cal N}_\mathrm{in}-{\cal N}_\mathrm{ext}=\widehat{A}(z)\widehat{B}(z),}
\ekv{ub'.38}{\widehat{B}(z)=1+\widehat{A}(z)^{-1}\left(  ({\cal
      N}-\widetilde{{\cal N}})\chi +\mathrm{IV}\right) =:1+\widehat{C}(z),}
where $\widehat{C}(z)$ belongs to $C_1(H^{3/2},H^{3/2})$ away from $\sigma
(P_\mathrm{in})$ and the corresponding trace class norm is ${\cal
  O}(h^{1-n})$ when $|\Im z|\ge h^{2/3}c$.

\par We conclude that
\ekv{ub'.40}
{
\ln |\det \widehat{B}(z)|\le {\cal O}(h^{1-n}), \hbox{ when }|\Im z|\ge h^{2/3}c.
}
$\widehat{A}(z)$ in (\ref{ub'.37}) is holomorphic in the whole
rectangle. It follows from Lemma \ref{sbd3} and the discussion after
(\ref{sbd.6}) that the $C_p$-norm of $\partial _z^k{\cal
  N}_\mathrm{ext}:H^{3/2}\to H^{1/2}$ is bounded by a negative power
of $h$ when $p$ is  $\ge 1$ and $>(n-1)/(2k)$.

\par As in the proof of that lemma, we write $\partial
_z^k\widetilde{{\cal N}}(z)=C_k\gamma hD_\nu
(\widetilde{P}_\mathrm{in}-z)^{-k}\widetilde{K}_\mathrm{in}$ and using
(\ref{ub'.32}) we see that $\partial _z^k\widetilde{{\cal N}}(z)={\cal
  O}(1):H^{3/2}\to H^{1/2+2k}$ for $2k\le 2N+2$ and hence the
$C_p$-norm of $\partial _z^k\widetilde{{\cal N}}:\, H^{3/2}\to
H^{1/2}$ is bounded by some negative power of $h$ when $p$ is $\ge 1$
and $>(n-1)/(2k)$, for $k\le N+1$. For $k=N+1$ we have $k>(n-1)/2$, so
$n/(2k)<1$. From (\ref{ub'.30}) we get the same estimates for
$\partial _z^k\mathrm{III}$. Thus the $C_p$-norm of $\partial
_z^k\widehat{A}(z):H^{3/2}\to H^{1/2}$ is bounded by some negative
power of $h$ when $p$ is $\ge 1$ and $>(n-1)/(2k)$, $k\le N+1$.

\par In conclusion, $\det \widehat{A}(z)$ and its inverse $\det
\widehat{A}(z)^{-1}$ can be defined in the whole rectangle as in
Subsection \ref{gd}, such that
$$
\ln |\det \widehat{A}(z)|={\cal O}(h^{-N_0}),
$$ 
for some $N_0$.

\par The desired factorization of $\det ({\cal N}_\mathrm{in}-{\cal
  N}_\mathrm{ext})$ is now 
\ekv{ub'.39}
{
\det ({\cal N}_\mathrm{in}-{\cal N}_\mathrm{ext})=\det \widehat{A}(z) \det \widehat{B}(z),
}
where $\det \widehat{A}(z)$ and its inverse are holomorphic in the whole
rectangle and bounded from above by $C \exp (Ch^{-N_0})$ for some
$C,N_0>0$.

Before continuing, we sum up and compare the two main results so
far. Proposition \ref{ub1}, applied to $1+{\cal K}(z)$ in
(\ref{ub.9.6}), gives
\ekv{ub'.41}
{
1+{\cal K}(z)=A(z)B(z),
}
where in the rectangle (\ref{ub.1}),
\ekv{ub'.42}{\ln |\det A(z)|={\cal O}(h^{-N}),}
\ekv{ub'.43}{\ln |\det B(z)|\le {\cal O}(h^{-N}).}
More precisely, $B(z)=1+R_N({\cal K}){\cal K}^N=:1+C(z)$, where $C(z)$
is holomorphic with values in the trace class operators and
\ekv{ub'.44}{\Vert C(z)\Vert_{C_1}\le {\cal O}(h^{-N}).}
Here, the exponent $N$ may take a new value at each
appearance. Further (see (\ref{ub.9.6}))
\ekv{ub'.45}{\det {\cal P}_\mathrm{out}=\det \widetilde{{\cal P}}_\mathrm{out}\det
A(z)\det B(z),}
where $\det \widetilde{{\cal P}}_\mathrm{out}$ can be defined as in Subsection
\ref{gd} such that
\ekv{ub'.46}
{
|\ln |\det \widetilde{{\cal P}}_\mathrm{out}||={\cal O}(h^{-N}).
}

\par On the other hand we have (\ref{red.6}), (\ref{sbd.3}):
\ekv{ub'.47}
{
\det {\cal P}_\mathrm{out}(z)=\det ({\cal P}_\mathrm{in}(z))\det ({\cal
  N}_\mathrm{in}-{\cal N}_\mathrm{ext}),
}
where 
\ekv{ub'.48}
{
\det ({\cal N}_\mathrm{in}-{\cal N}_\mathrm{ext})=\det
\widehat{A}(z)\det \widehat{B}(z),\ \widehat{B}(z)=1+\widehat{C}(z).
}
Here, $\det \widehat{A}(z)$ is holomorphic and 
\ekv{ub'.49}
{
\ln |\det \widehat{A}(z)|={\cal O}(h^{-N})
}
in the whole rectangle, while $\widehat{C}(z)$ is meromorphic with
values in $C_1(H^{3/2},H^{3/2})$ with the poles at the (real)
eigenvalues of $P_\mathrm{in}$. Moreover, for $|\Im z|\ge
h^{2/3}c$ we have $\Vert \widehat{C}(z)\Vert_{C_1}\le
{\cal O}(h^{1-n})$, so
\ekv{ub'.50}
{
\ln |\det (1+\widehat{C}(z))|\le {\cal O}(h^{1-n})
}
in that region.

\par We shall now compare the expressions (\ref{ub'.45}) and
(\ref{ub'.47}).

\par In (\ref{ub'.45}) the first two factors to the left are well
defined up to factors of the form $\exp p(z)$ where $p$ is a polynomial
of degree $\le N$ and as we have seen, we can choose realizations
satisfying (\ref{ub'.45}), (\ref{ub'.42}). As for $\det B(z)$, defined
as a determinant of a trace class perturbation of $1$ (which is a
special case of the definition in Subsection \ref{gd}), we only have
the upper bound (\ref{ub'.43}).

\par In (\ref{ub'.47}), $\det {\cal P}_\mathrm{in}(z)=\det
(P_\mathrm{in}-z)$ can be defined as in Subsection \ref{gd} up to a
factor $\exp p(z)$ as before, in such a way that $\ln |\det {\cal
  P}_\mathrm{in}|\le {\cal O}(h^{-N})$ and when $|\Im z|\ge
h^{2/3}/\widetilde{C}$, we even have $\ln |\det {\cal
  P}_\mathrm{in}(z)|={\cal O}(h^{-N})$. This factor will be further
studied below. Similarly, we have (\ref{ub'.48}), (\ref{ub'.49}) and
again we define $\det \widehat{B}$ as the determinant of a trace class
perturbation of the identity. 

\par When writing the identity
\ekv{ub'.51}
{
\begin{split}
\det {\cal P}_\mathrm{out}(z)&=\det \widetilde{{\cal
    P}}_\mathrm{out}\det A(z) \det B(z)\\
&= \det {\cal P}_\mathrm{in}\det \widehat{A}(z)\det \widehat{B}(z),
\end{split}
}
it is not apriori clear that we can choose $\det
\widetilde{P}_\mathrm{out}$, $\det A(z)$, $\det \widehat{A}(z)$, $\det
{\cal P}_\mathrm{in}$ all satifying the above bounds simultaneously,
since we have made definite choices of $\det B(z)$ and $\det
\widehat{B}(z)$. However, if we restrict the attention to the region
$|\Im z|\ge h^{2/3}c$ we know that $B(z)^{-1}$ and
$\widehat{B}(z)^{-1}$ are bounded in operator norm by some negative
power of $h$, and this additional information implies that
$B(z)^{-1}=1+D(z)$, $\widehat{B}(z)^{-1}=1+\widehat{D}(z)$, where
$D(z)$ and $\widehat{D}(z)$ are bounded in trace class norm by
negative powers of $h$, so in that region we also get 
$$
\ln |\det B(z)|,\, \ln |\det \widehat{B}(z)|={\cal O}(h^{-N}).
$$
Then if we choose the other factors with moduli that have polynomially
bounded logarithms, we can modify one of them by a factor $\exp
p(z)$, where $p(z)$ is a polynomial of degree $\le N$ with real part
$={\cal O}(h^{-N})$ and achieve (\ref{ub'.51}) in such a way that 
\begin{itemize}
\item $\ln |x|={\cal O}(h^{-N})$ when $x=$ $\det A$, $\det
  \widehat{A}$, $\det \widetilde{{\cal P}}_\mathrm{out}$ in the
  whole rectangle,
\item $\ln |x|={\cal O}(h^{-N})$ for $|\ln z|\ge
  h^{2/3}c$, when $x=$ $\det B(z)$, $\det
  \widehat{B}(z)$, $\det {\cal P}_\mathrm{in}$,
\item $\ln |x|\le {\cal O}(h^{-N})$ in the whole rectangle, when
  $x=\det B(z)$, $\det {\cal P}_\mathrm{in}$.
\end{itemize}
Moreover, as we have seen, 
\ekv{ub'.52}
{
\ln |\det \widehat{B}(z)|\le {\cal O}(h^{1-n}),\hbox{ when }|\Im z|\ge h^{2/3}c.
}

\par The aim is to study the zeros of $\det {\cal P}_\mathrm{out}(z)$
in the rectangle (\ref{ub.1}), using the upper bound (\ref{ub.10}) and
the more precise upper bound for $|\Im z|\ge h^{2/3}c$
resulting from the last expression in (\ref{ub'.51}) together with
(\ref{ub'.52}) and the fact that $\ln |\det \widehat{A}|={\cal
  O}(h^{-N})$. After division with $\det \widehat{A}(z)$ we can
concentrate on the function
\ekv{ub'.53}
{
f(z)=\det {\cal P}_\mathrm{in} \det \widehat{B}(z),
}
for which 
\ekv{ub'.54}
{
\ln |f(z)|\le {\cal O}(h^{-N}).
}

Next, look at $\det
{\cal P}_\mathrm{in}(z)$. Let $\widetilde{K}={\cal O}(h^{1/2}):H^s\to
H^{s+1/2}$, $s\in {\bf R}$ be a right inverse of $\gamma $. Then,
$$
\begin{pmatrix}1 & \widetilde{K}\end{pmatrix}:\, {\cal
  D}(P_\mathrm{in})\times H^{\frac{3}{2}}\to H^2
$$
is a bijection with a bounded inverse and
$$
{\cal P}_\mathrm{in}(z) \begin{pmatrix}1 &h^{-\frac{1}{2}}\widetilde{K}\end{pmatrix}
=\begin{pmatrix} P_\mathrm{in}-z &h^{-\frac{1}{2}}(P-z)\widetilde{K}\\
0 & 1\end{pmatrix},
$$
so $$
\det {\cal P}_\mathrm{in}(z)\det \begin{pmatrix}1 &h^{-\frac{1}{2}}\widetilde{K}\end{pmatrix}
=\det (P_\mathrm{in}-z)
$$
and since $\widetilde{K}$ is independent of $z$, we can take
$\det \begin{pmatrix}1 &h^{-\frac{1}{2}}\widetilde{K}\end{pmatrix}$ to be an arbitrary
non-vanishing constant, say $1$ and get
\ekv{ub.11}
{
\det {\cal P}_\mathrm{in}(z)=\det (P_\mathrm{in}-z).
}

The method in Subsection \ref{gd} shows that 
\ekv{ub.12}
{
\partial _z^N\ln \det (P_\mathrm{in}-z)=-(N-1)!\, \mathrm{tr\,}(P_\mathrm{in}-z)^{-N},
}
for $N>n/2$, so that $(P_\mathrm{in}-z)^{-N}$ is of trace
class. 

\par Let $\chi \in C_0^\infty (]1/4,4 [;[0,1])$
be equal to 1 in a neighborhood of $[1/3,3]$. If $N(\lambda )=\# (\sigma
(P_\mathrm{in})\cap ]-\infty ,\lambda ])$, we get
\ekv{ub.14}{\begin{split}
\partial _z^N\ln \det (P_\mathrm{in}-z)=&-(N-1)!\int (\lambda
-z)^{-N}dN(\lambda )\\
=&-(N-1)!\int (\lambda -z)^{-N}\chi (\lambda )dN(\lambda )\\
&-(N-1)!\int (\lambda -z)^{-N}(1-\chi (\lambda ))dN(\lambda ).
\end{split}}
Thus,
\ekv{ub.14.5}{
\ln \det (P_\mathrm{in}-z)=\mathrm{I}(z)+\mathrm{II}(z),
}
where
\ekv{ub.15}
{
-\partial _z^N\mathrm{I}(z)=(N-1)!\int (\lambda -z)^{-N}\chi (\lambda )dN(\lambda )
}
\ekv{ub.16}
{
-\partial _z^N\mathrm{II}(z)=(N-1)!\int (\lambda -z)^{-N}
(1-\chi (\lambda ))dN(\lambda ).
}
Up to a polynomial, we have for $\Im z\ne 0$:
\ekv{ub.17}
{
\mathrm{I}(z)=\int \ln (\lambda -z)\chi (\lambda )dN(\lambda ),
}
where we use the standard branch of $\ln$ with a cut along $]-\infty
,0[$. In particular,
\ekv{ub.18}
{
\Re \mathrm{I}(z)=\int \ln |\lambda -z|\chi (\lambda )dN(\lambda ).
}

\par In order to estimate $\mathrm{II}(z)$, we shall use the rough
estimate
\ekv{ub.19}
{
N(\lambda )={\cal O}(h^{-n}\lambda ^{n/2}),
}
which is valid uniformly for $0<h\ll 1$, $\lambda \ge 1$. It follows
from (\ref{ub.19}) and an integration by parts in (\ref{ub.16}), that
\ekv{ub.20}
{
  \partial _z^N\mathrm{II}(z)={\cal O}(h^{-n}) } in the domain
(\ref{ub.1}). By integration, we see that we can choose
$\mathrm{II}(z)$ holomorphic in this domain such that \ekv{ub.21} {
  \mathrm{II}(z)={\cal O}(h^{-n}).  } This will allow us to replace
$\det {\cal P}_\mathrm{in}$ by $\exp \mathrm{I}(z)$ in the definition
of $f(z)$ in (\ref{ub'.53}), without affecting the validity of
(\ref{ub'.54}).

Before that we will discuss some harmonic majorants of
$\Re\mathrm{I}(z)$. Recall that if $\Omega \Subset {\bf C} $ has
piecewise smooth boundary and if $G=G_\Omega $, $K=K_\Omega $ are the
corresponding Dirichlet and Poisson kernels for the Dirichlet problem
for the Laplacien, then by Green's formula, we have 
$$K(x,y)=\partial _{\nu _y}G(x,y),$$
where $\nu $ is the exterior unit normal. This still holds when
$\Omega =\Omega _r$ is the infinite strip $\{ x\in {\bf C};\, |\Im
x|<r\}$ and we consider the solutions to the Dirichlet problem that
are bounded when the data are bounded. In the case $\Omega =\Omega _1$
we have (see for instance \cite{Sj09}) that $G(x,y)$ is of class
$C^\infty (\overline{\Omega }\times \overline{\Omega })$ away from the
diagonal and there exists $C_0>0$ such that for every $r>0$ and all
$\alpha ,\beta \in {\bf N}$, there exists a constant $C=C_{\alpha
  ,\beta ,r}$ such that \ekv{ub'.55} { |\nabla _x^\alpha \nabla
  _y^\beta G(x,y)|\le C \exp -\frac{1}{C_0}|\Re x-\Re y|,\hbox{ when
  }|x-y|>r>0.  } Moreover, \ekv{ub'.56}
{\begin{split}G_{r\Omega }(x,y)&=G_\Omega (\frac{x}{r},\frac{y}{r})\\
    K_{r\Omega }(x,y)&=\frac{1}{r}K_\Omega (\frac{x}{r},\frac{y}{r}).
\end{split}}
  
\par Consider first the subharmonic function $\ln |x|$ on $\Omega _r$ and its
smallest harmonic majorant there, given by
$$
\Delta h_r=0,\quad {{h _r}_\vert}_{\partial \Omega _r}=\ln |x|.
$$
Then, $\psi _r:=h_r-\ln |x|\ge 0$ is equal to $-2\pi G_{\Omega
  _r}(x,0)$ and we are interested in
$$
f_r:=-\partial _{\nu} \psi _r=2\pi \partial _\nu G_{\Omega
  _r}(x,0)=2\pi K_{\Omega _r}(0,x)=\frac{2\pi }{r}K_{\Omega _1
}(0,\frac{x}{r})=\frac{1}{r}f_1(\frac{x}{r}),
$$ 
which is a non-negative function defined on the boundary and satisfies
\ekv{ub'.57}
{
\partial _x^\alpha f_r={\cal O}_\alpha (1)r^{-1-|\alpha
  |}e^{-\frac{1}{C_0r}|\Re x|}.
}
Also,
\ekv{ub'.58}
{
\int_{\partial \Omega _r}f_r|dx|=2\pi ,\ f_r(\overline{x})=f_r(x).
}

The smallest harmonic majorant in $\Omega _r$ of 
\ekv{ub'.59}
{
\phi _\mathrm{in}:=\Re \mathrm{I}(x)=\sum \chi (\lambda
_j)\ln |z-\lambda _j|
}
is
\ekv{ub'.60}
{
h_{r,\mathrm{in}}(x)=\sum \chi (\lambda _j)h_r (x-\lambda
_j).
}
The function
\ekv{ub'.61}
{
\Phi _r=\begin{cases}\phi _\mathrm{in}\hbox{ outside }\Omega _r\\
h_\mathrm{in} \hbox{ in }\Omega _r
\end{cases}
} is subharmonic, $\Delta \Phi _r$ is supported in $\partial \Omega
_r$ and equal to \ekv{ub'.62} { \sum \chi (\lambda _j)(f_r(x-\lambda
  _j)\delta (\Im x-r)+f_r(x-\lambda _j)\delta (\Im x+r)).  } If
$\frac{1}{2}\le a<b\le 2$, we get with \ekv{ub'.63} {
  g_r(t)=\frac{1}{2\pi
  }(f_r(t+ir)+f_r(t-ir))=:\frac{1}{r}g_1(\frac{t}{r})\ge 0, } that
\ekv{ub'.64} { \int_{a\le \Re x \le b}\Delta \Phi _r(x)L(dx)=2\pi
  \int_a^b g_r*(\chi dN)(t)dt.  } Notice that
$g_r(t)=\frac{1}{r}g_1(\frac{t}{r})$ is an approximation of $\delta $
and we will use (\ref{ub'.64}) with $r=h^{2/3}c$.

\par Returning to (\ref{ub'.53}), (\ref{ub'.54}), we see that the
zeros of $f$ in the rectangle (\ref{ub.1}) will not change if we
replace $\det {\cal P}_\mathrm{in}$ in (\ref{ub'.53}) by $\exp
\mathrm{I}(z)$, so we now redefine $f$ to be \ekv{ub'.65} {
  f(z)=e^{\mathrm{I}(z)} \det \widehat{B}(z), } and notice that
(\ref{ub'.54}) still holds because of (\ref{ub.21}). Moreover,
\ekv{ub'.66}
{
\ln |f(z)|=\phi _\mathrm{in}(z)+\ln |\det \widehat{B}(z)|,
}
and (\ref{ub'.54}) tells us that 
\ekv{ub'.67}{\ln |f(z)|\le {\cal O}(h^{-N})}
in the whole rectangle, while (\ref{ub'.52}) shows that
\ekv{ub'.68}
{
\ln |f(z)|\le \phi _\mathrm{in}(z)+{\cal O}(h^{1-n}),
}
in the part of the rectangle where $|\Im z|\ge h^{2/3}c$.

\par Clearly, the whole discussion so far remains valid if we enlarge
the rectangle (\ref{ub.1}) by replacing $1/2$ by a slightly smaller
constant and the bound $2$ by a slightly larger constant. We can find
$\alpha$, $\beta $ with $\frac{1}{2}-\alpha \asymp 1/{\cal O}(1)$,
$\beta -2\asymp 1/{\cal O}(1)$ such that $\phi _\mathrm{in}\ge -{\cal
  O}(h^{-N})$ for $\Re z=\alpha ,\beta $, and (\ref{ub'.54}) tells us
that \ekv{ub'.69} { \ln |f(z)|\le h_r(z)+{\cal O}(h^{-N}), } on the
same vertical segments, while (\ref{ub'.68}) tells us that
\ekv{ub'.70} { \ln |f(z)|\le h_r(z)+{\cal O}(h^{1-n}) } on the
horizontal parts of the boundary of $[\alpha ,\beta ]+ir[-1,1]$. By
the maximum principle, we get in the latter rectangle
$$
\ln |f(z)|\le \widetilde{h}(z)+{\cal O}(h^{1-n}),
$$
where $\widetilde{h}$ is the harmonic function on $[\alpha ,\beta
]+ir[-1,1]$ which is equal to a constant= ${\cal O}(h^{-N})$ on the
vertical parts of the boundary and equal to $h_r(z)$ on the horizontal
parts. Using that $r$ is of the order of $h^{2/3}$ together with
simple estimates on the Poisson kernel in thin rectangles (see
\cite{Sj09}, Section 2), we see that
$$
\widetilde{h}(z)\le {\cal O}(1)h^{-N}\exp (-\frac{1}{{\cal
    O}(1)r})+h_r(z)
\le h_r(z)+{\cal O}(h^{1-n})
$$
on $[\frac{1}{2},2]+ir[-1,1]$ and we get the estimate
$$
\ln |f(z)|\le h_r(z)+{\cal O}(h^{1-n})
$$
on the latter rectangle, leading to 
\ekv{ub'.71}
{
\ln |f(z)|\le \Phi _r(z)+{\cal O}(h^{1-n}) \hbox{ in the rectangle (\ref{ub.1}).}
}
This estimate together with (\ref{ub'.64}) form the main conclusion of
this section.

\section{Some estimates for $P_\mathrm{out}$}\label{out}
\setcounter{equation}{0}

\par In this and the next two sections we shall construct a
suitable perturbation $W$ as in Theorem \ref{re1} such that we get a
lower bound for $f(z)$ in (\ref{ub'.53}) that matches
(\ref{ub'.71}). Here $z$ is any given point in the set (\ref{ub'.10.5})
and the perturbation will depend on that point. As we shall see, this
amounts to getting a good bound on the smallest singular value on
$\widehat{B}$ (cf (\ref{ub'.48})) or equivalently on that of ${\cal
  P}_\mathrm{out}$, or of ${\cal N}_\mathrm{in}(z)-{\cal
  N}_\mathrm{out}(z)$.
 
\par For $\mu >0$, let $E(\mu )\subset L^2({\cal O})$ be the spectral
subspace associated to all eigenvalues $<\mu ^2$ of
$P_\mathrm{out}(z)^*P_\mathrm{out}(z)$. We shall show that if $\mu $
is small enough (to be specified below) and $u\in E(\mu )$ is
normalized, then $\Vert u\Vert_{L^2({\cal O}_h\setminus {\cal
    O}_{2h})}$ cannot be too small. Here we define
$$
{\cal O}_c=\{x\in {\cal O};\, \mathrm{dist\,}(x,\partial {\cal O})>c \},
$$
when $c\ge 0$.

\par If $u\in E(\mu )$, we have $u=\sum_1^N u_je_j$,
where $e_1,...,e_N$ is an orthonormal basis of eigenfunctions in
$E(\mu )$, $P_\mathrm{out}(z)^*P_\mathrm{out}(z)e_j=t_j^2e_j$, $0\le
t_j<\mu $, and
$$
\Vert
P_\mathrm{out}(z)u\Vert^2=(P_\mathrm{out}(z)^*P_\mathrm{out}(z)u|u)=\sum_1^N|u_j|^2t_j^2\le
\mu ^2\sum |u_j|^2=\mu ^2\Vert u\Vert^2,
$$
where all norms are in $L^2$ if nothing else is specified. Thus,
if $u\in E(\mu )$, and $\Vert u\Vert =1$,
\ekv{lob.9}
{
P_\mathrm{out}(z)u=v,\ \Vert v\Vert <\mu .
}
By standard elliptic estimates, combined with the dilation
$x=hy$, $hD_{x_j}=D_{y_j}$, we have 
\ekv{lob.10}
{\begin{split}
\Vert u\Vert_{H_h^2({\cal O}_{(1+\theta )h}\setminus {\cal O}_{2h/(1+\theta
    )})}\le & C_\theta (\Vert v\Vert +\Vert u\Vert_{L^2({\cal O}_h\setminus {\cal
    O}_{2h})})\\&\le C_\theta (\mu +\Vert u\Vert_{L^2({\cal O}_h\setminus {\cal O}_{2h})}),\end{split}
}
for every fixed $\theta $ with $0<\theta \ll 1$.

Let $\chi \in C_0^\infty ({\cal O}_{(1+\theta )h};[0,1])$ be equal to
1 on ${\cal O}_{3h/2}$ and satisfy $\partial ^\alpha \chi ={\cal
  O}(h^{-|\alpha |})$, $\alpha \in {\bf N}^n$. Let $\Gamma =\Gamma _f$
be a Lipschitz contour as in and around (\ref{dcs.21}) with $\theta
=\pi /3$. Let $P_\mathrm{ext}$ be the
Dirichlet realization of $P$ on $\Gamma \setminus {\cal O}_{2h}$. Then
\ekv{lob.11}
{
(P_\mathrm{ext}-z)(1-\chi )u=(1-\chi )v-[P,\chi ]u,
}
where we let $u$ also denote the outgoing extension of $u$ which is
well-defined since $u\in {\cal D}(P_\mathrm{out}(z))$ and where $v$ also
denotes the 0 extension. Similarly,
\ekv{lob.12}
{
(P_\mathrm{in}-z)\chi u=\chi v+[P,\chi ]u.
}

\par If $V$ vanishes outside ${\cal O}_{2h}$, we know from Section
\ref{aed} (with ${\cal O}$ there replaced by ${\cal O}_{2h}$) that  $\Vert (P_\mathrm{ext}-z)^{-1}\Vert_{{\cal
    L}(L^2,L^2)}={\cal O}(h^{-2/3})$. More generally, we shall assume
that 
\ekv{lob.12.5}
{
\Vert V\Vert_{L^\infty ({\cal O}\setminus {\cal O}_{2h}})\ll h^{2/3},
}
and we notice that this holds for $V=V_0+\delta \Theta q_\omega $ in
Theorem \ref{re1} if $\alpha $ is large enough.
Then by a simple perturbation argument, the preceding estimate on the
exterior resolvent remains valid
and we get from (\ref{lob.10}),
(\ref{lob.11}),
\ekv{lob.13}
{
h^{\frac{2}{3}}\Vert (1-\chi )u\Vert_{L^2({\cal O})}\le {\cal
  O}(1)(\mu +\Vert u\Vert_{L^2({\cal O}_h\setminus {\cal O}_{2h})}).
}
Similarly, by using that $\Vert (P_\mathrm{in}-z)^{-1}\Vert_{{\cal
    L}(L^2,L^2)}={\cal O}(h^{-2/3})$, we get
\ekv{lob.14}
{
h^{\frac{2}{3}}\Vert \chi u\Vert_{L^2({\cal O})}\le {\cal
  O}(1)(\mu +\Vert u\Vert_{L^2({\cal O}_h\setminus {\cal O}_{2h})}).
}
Combining the two estimates and recalling that $\Vert u\Vert=1$, we
get
\ekv{lob.15}
{
h^{\frac{2}{3}}\le {\cal O}(1)(\mu +\Vert u\Vert_{L^2({\cal
    O}_h\setminus {\cal O}_{2h})}),
}
and if $\mu \ll h^{2/3}$,
\ekv{lob.16}
{
\Vert u\Vert_{L^2({\cal
    O}_h\setminus {\cal O}_{2h})}\ge \frac{h^{\frac{2}{3}}}{{\cal O}(1)},
}
for all $u\in E(\mu )$ with $\Vert u\Vert_{L^2({\cal O})}=1$.

Next we make a remark about the $H^s$ regularity of of
elements in $E(\mu )$. Assume that for some fixed $s>\frac{n}{2}$, we
have $V=V_1+V_2$
\ekv{lob.17}
{
\Vert V_1\Vert_{H^s_1}+h^{-\frac{n}{2}}\Vert V_2\Vert_{H^s_h}\le {\cal O}(1).
}
When $V=V_0+W=V_0+\delta \Theta q_\omega $ is a potential as in
Theorem \ref{re1}, we take $V_1=V_0$, $V_2=W$ and get (\ref{lob.17}),
provided $\alpha (n,v_0,s,\epsilon ,\theta ,M,\widetilde{M})$ in
(\ref{re.8}) is large enough (cf Remark \ref{sv1}).
So far we have systematically used the semi-classical Sobolev spaces
$H^s=H^s_h$ but in (\ref{lob.17}) we also use the standard Sobolev
space $H^s=H^s_1$ (with $h=1$). Following standard conventions, we let 
$$H_\cdot ^\sigma ({\cal O})={{H^\sigma _\cdot }({\bf R}^n)_\vert}_{{\cal O}},$$ 
$$
H_\cdot ^\sigma (\overline{{\cal O}})=\{ u\in H^\sigma _\cdot ({\bf
  R}^n);\, \mathrm{supp\,}u\subset \overline{{\cal O}}\}.
$$

\par If $u=\sum_1^N u_je_j\in E(\mu )$, we have
$(P_\mathrm{out}^*P_\mathrm{out})^ku=\sum_1^Nt_j^{2k}u_je_j $, so 
\ekv{lob.19}
{
\Vert (P_\mathrm{out}^*P_\mathrm{out})^ku\Vert \le \mu ^{2k}\Vert u\Vert,\ k\in
{\bf N}.
}

We will assume that $\mu ={\cal O}(1)$ and limit the attention to $k$
in a bounded interval, so the right hand side of (\ref{lob.19}) will
be ${\cal O}(\Vert u\Vert )$.
We study apriori estimates
in the interior. Let $\Omega _2\subset \Omega _1\subset {\cal O}$ be
open with $\mathrm{dist\,}(\Omega _2,\complement \Omega _1)\ge
h/C$. If $P_\mathrm{out}u=v$,\ \ $u,v\in H^\sigma _h(\Omega _1)$, $0\le
\sigma \le s$, we can write $-h^2\Delta u=v+(z-V)u=:w$, where
$$
\Vert w\Vert_{H^\sigma _h(\Omega _1)}\le {\cal O}(1)(\Vert
v\Vert_{H_h^\sigma (\Omega _1)}+\Vert u\Vert _{H_h^\sigma (\Omega _1)})
$$
and standard apriori estimates for $-\Delta $ (after the dilation
$x=hy$) give
\ekv{lob.20}
{
\Vert u\Vert_{H_h^{\sigma +2}(\Omega _2)}\le {\cal O}(1)(\Vert
v\Vert_{H_h^\sigma (\Omega _1)}+\Vert u\Vert_{H_h^\sigma (\Omega _1)}).
}
If $s<\sigma <s+2$, we only get
\ekv{lob.21}
{
\Vert u\Vert_{H_h^{s+2}(\Omega _2)}\le {\cal O}(1)(\Vert
v\Vert_{H_h^\sigma (\Omega _1)}+\Vert u\Vert_{H_h^\sigma (\Omega _1)}).
}
The same apriori estimate holds for $P_\mathrm{out}^*$. 

We shall now use these estimates to study elements of $E(\mu )$ and
first assume for simplicity that (\ref{lob.17}) holds for all
$s>0$. From the fact that $(P_\mathrm{out}^*P_\mathrm{out})^ku={\cal
  O}_k(1)\Vert u\Vert$ in $H^0({\cal O})$ for all $k\in{\bf N}$ we first infer by
integration by parts, that $P_\mathrm{out}(P_\mathrm{out}^*P_\mathrm{out})^{k-1}u={\cal O}(1)$ in $H^0({\cal
  O})$. Using the apriori estimate for $P_\mathrm{out}^*$, we get
\begin{equation*}\begin{split}
&\Vert P_\mathrm{out}(P_\mathrm{out}^*P_\mathrm{out})^{k-1}u\Vert_{H^2({\cal O}_{h/C})}\le \\
&{\cal O}(1)(\Vert (P_\mathrm{out}^*P_\mathrm{out})^ku\Vert_{H^0({\cal O})}+\Vert
P_\mathrm{out}(P_\mathrm{out}^*P_\mathrm{out})^{k-1}u\Vert_{H^0({\cal O})})\le {\cal O}(1),
\end{split}\end{equation*}
and using the one for $P_\mathrm{out}$, we get
\begin{equation*}\begin{split}
&\Vert (P_\mathrm{out}^*P_\mathrm{out})^{k-1}u\Vert_{H^2({\cal O}_{h/C})}\le \\
&{\cal O}(1)(\Vert P_\mathrm{out}(P_\mathrm{out}^*P_\mathrm{out})^ku\Vert_{H^0({\cal O})}+\Vert
(P_\mathrm{out}^*P_\mathrm{out})^{k-1}u\Vert_{H^0({\cal O})})\le {\cal O}(1).
\end{split}
\end{equation*}
Thus for all $k\in {\bf N}$,
$$
\Vert (P_\mathrm{out}^*P_\mathrm{out})^ku\Vert_{H^2({\cal O}_{h/C})}+\Vert
P_\mathrm{out}(P_\mathrm{out}^*P_\mathrm{out})^ku\Vert_{H^2({\cal O}_{h/C})}
\le {\cal O}(1).
$$
Here we use again the apriori estimates for $P_\mathrm{out}^*$ and
$P_\mathrm{out}$ and get that for every $k\in {\bf N}$,
$$
\Vert (P_\mathrm{out}^*P_\mathrm{out})^ku\Vert_{H^4({\cal O}_{2h/C})}+\Vert P_\mathrm{out}(P_\mathrm{out}^*P_\mathrm{out})^ku\Vert_{H^4({\cal O}_{2h/C})}\le {\cal O}(1).
$$
Iterating this argument, we get for every $j\in {\bf N}$ that for
every $k\in {\bf N}$,
$$
\Vert (P_\mathrm{out}^*P_\mathrm{out})^ku\Vert_{H^{2j}({\cal O}_{2jh/C})}+\Vert
P_\mathrm{out}(P_\mathrm{out}^*P_\mathrm{out})^ku\Vert_{H^{2j}({\cal O}_{2jh/C})}\le {\cal O}(1).
$$

\par Now if we make the assumption (\ref{lob.17}) for a fixed
$s>n/2$, we see that the above iteration works
as long as $2j\le s+2$, then if this last $j$ is strictly less than
$(s+2)/2$, we can make one more iteration and reach the degree of regularity
$s+2$. Hence the final conclusion is that if $\mu ={\cal O}(1)$ and we
assume (\ref{lob.17}) for a fixed $s>n/2$, then for every $C>0$,
we have \ekv{lob.22} {\Vert
  (P_\mathrm{out}^*P_\mathrm{out})^ku\Vert_{H^{s+2}({\cal
      O}_{h/C})}+\Vert
  P_\mathrm{out}(P_\mathrm{out}^*P_\mathrm{out})^ku\Vert_{H^{s+2}({\cal
      O}_{h/C})}\le {\cal O}(1).}

We end this section with some estimates relating the small singular
values of $P_\mathrm{out}(z)$ to those of ${\cal P}_\mathrm{out}$ and
when $z$ belongs to the set (\ref{ub'.10.5}), to those of ${\cal
  N}_\mathrm{in}-{\cal N}_\mathrm{out}$ and of
$\widehat{B}(z)=1+\widehat{C}(z)$ in (\ref{ub'.36}) and
(\ref{ub'.38}).

\par Recall that ${\cal P}_\mathrm{out}(z)$ is bijective precisely
when $P_\mathrm{out}(z)$ is, and when so is the case it easy to check
that 
\ekv{out.1}
{
{\cal P}_\mathrm{out}(z)^{-1}=\begin{pmatrix}P_\mathrm{out}(z)^{-1} &(1-P_\mathrm{out}(z)^{-1}(P-z))h^{-\frac{1}{2}}\widehat{K}\end{pmatrix},
}
where we recall that $\widehat{K}={\cal O}(h^{1/2}):H^{1/2}\to H^2$ is
a right inverse of $B$.

Recall that when $A:{\cal H}_1\to {\cal H}_2$ is a bounded operator
between two Hilbert spaces, then the singular values $s_1(A)\ge
s_2(A)\ge ...$ are defined by the fact that $s_j(A)^2$ is the
decreasing sequence formed first by all discrete eigenvalues of $A^*A$
above the essential spectrum and then (when ${\cal H}_1$ is infinite
dimensional only) by an infinite repetition of $\sup \sigma
_\mathrm{ess}(A^*A)$. It is well known and easy to see that the
non vanishing singular values of $A$ and of $A^*$ are the same.

We have the Ky Fan inequalities
\ekv{out.2}
{
\begin{split}
&s_{n+k-1}(A+B)\le s_n(A)+s_k(B),\\
&s_{n+k-1}(BA)\le s_n(A)s_k(B),
\end{split}
}
in the cases when $B:$ ${\cal H}_1\to {\cal H}_2$ and ${\cal H}_2\to {\cal H}_3$
respectively. 

\par Applying this to (\ref{out.1}), we get
\ekv{out.3}
{
s_j({\cal P}_\mathrm{out}(z)^{-1})\ge s_j(
P_\mathrm{out}(z)^{-1}).}
If $\Pi _1:\, H^0\times H^{1/2}\to H^0$, $\Pi _2:\, H^0\times
H^{1/2}\to H^{1/2}$ are the natural projections (of norm 1), we can
rewrite (\ref{out.1}) as 
\begin{equation*}\begin{split}{\cal P}_\mathrm{out}(z)^{-1}&=P_\mathrm{out}(z)^{-1}\Pi
_1+(1-P_\mathrm{out}(z)^{-1}(P-z))h^{-1/2}\widehat{K}\Pi _2\\
&=P_\mathrm{out}(z)^{-1}(\Pi _1-(P-z)h^{-1/2}\widehat{K}\Pi _2)+h^{-1/2}\widehat{K}\Pi _2,\end{split}
\end{equation*}
which leads to 
\ekv{out.4}
{
s_j({\cal P}_\mathrm{out}(z)^{-1})\le {\cal O}(1)(1+s_j(P_\mathrm{out}(z)^{-1}))
}

\par We now restrict $z$ to (\ref{ub'.10.5}) and consider
(\ref{red.6}) which can be written 
\ekv{out.5}
{
{\cal P}_\mathrm{out}(z)^{-1}={\cal
  P}_\mathrm{in}(z)^{-1}\begin{pmatrix}
1 &0\\ 0 &({\cal N}_\mathrm{in}-{\cal N}_\mathrm{ext})^{-1}
\end{pmatrix}\begin{pmatrix}1 &0\\ -h^{\frac{1}{2}}BG_\mathrm{in} &1\end{pmatrix}
}
and also
\ekv{out.6}
{
\begin{pmatrix}1 &0\\ 0 &({\cal N}_\mathrm{in}-{\cal
    N}_\mathrm{ext})^{-1}\end{pmatrix}={\cal P}_\mathrm{in}(z){\cal
  P}_\mathrm{out}(z)^{-1}\begin{pmatrix}1 &0\\
  h^{\frac{1}{2}}BG_\mathrm{in} &1\end{pmatrix}.
}
Here the operator norms of ${\cal P}_\mathrm{in}^{-1}$ and
$h^{1/2}BG_\mathrm{in}$ are ${\cal O}(h^{-2/3})$. From (\ref{out.5})
we get 
\ekv{out.7}
{
s_j({\cal P}_\mathrm{out}(z)^{-1})\le {\cal
  O}(h^{-\frac{4}{3}})(1+s_j(({\cal N}_\mathrm{in}-{\cal N}_\mathrm{ext})^{-1})),
}
while 
(\ref{out.6}) leads to
\ekv{out.8}
{
s_j(({\cal N}_\mathrm{in}-{\cal N}_\mathrm{ext})^{-1})\le {\cal O}(h^{-\frac{2}{3}})s_j({\cal P}_\mathrm{out}(z)^{-1}).
}

\par Finally, from (\ref{ub'.36}), (\ref{ub'.38}) and the uniform
boundedness of $\widehat{A}(z)$ and its inverse, we get
\ekv{out.9}
{
s_j(({\cal N}_\mathrm{in}-{\cal N}_\mathrm{ext})^{-1})\asymp s_j(\widehat{B}(z)^{-1})=s_j((1+\widehat{C}(z))^{-1}).
}

\par When $A:{\cal H}_1\to {\cal H}_2$ is a Fredholm operator of index 0,
we let $t_1^2\le t_2^2\le ...$ with $t_j\ge 0$ describe the lower part
of the spectrum of $A^*A$ in analogy with $s_j^2$. Again
$t_j(A)=t_j(A^*)$ and when $A$ is bijective we have
$t_j(A)=1/s_j(A^{-1})$.

\par Let $N$ be the number of singular values $0\le t_1\le ...\le t_N$
of $1+\widehat{C}(z)$ that are $\le 1/2$. If $e_1, ...,e_N$ is a
corresponding orthonormal family of eigenfunctions of
$(1+\widehat{C}(z))^*(1+\widehat{C}(z))$, then $\Vert
(1+\widehat{C}(z))u\Vert \le \frac{1}{2}\Vert u\Vert$ and consequently
$\Vert \widehat{C}(z)u\Vert \ge \frac{1}{2}\Vert u\Vert$, for all $u\in
{\bf C}e_1\oplus ...\oplus {\bf C}e_N$. By the mini-max
characterization of singular values, we get $s_N(\widehat{C}(z))\ge
1/2$ and using that the trace class norm of $\widehat{C}(z)$ is ${\cal
  O}(h^{1-n})$, we conclude that $N={\cal O}(h^{1-n})$. Combining this
with (\ref{out.9}), (\ref{out.7}), (\ref{out.3}), we see that there
exists a constant $C>0$ such that 
\ekv{out.11}
{
t_j(P_\mathrm{out}(z))\ge h^{\frac{4}{3}}/C,\hbox{ for
}j\ge Ch^{1-n}.
}

\section{Perturbation matrices and their singular values}\label{sv}
\setcounter{equation}{0}

We shall use a general estimate from \cite{Sj08a}. Let $e_1,...,e_N\in
C^0(\Omega )\cap L^2(\Omega )$, where $\Omega \subset {\bf R}^n$ is
open. Let ${\cal E}_\Omega =((e_j|e_k)_{L^2(\Omega )})_{1\le j,k\le
  N}$ be the corresponding Gramian and let $0\le \epsilon _1\le ...\le
\epsilon _N$ be its eigenvalues. Then (see \cite{Sj08a}, Proposition 5.5) $\exists a_1,...,a_N\in \Omega $
such that the singular values $s_1\ge ...\ge s_N\ge 0$ of the $N\times
N$ matrix
$M=M_{\delta _a}$, given by $$M_{j,k}=\sum_{\nu =1}^N e_j(a_\nu )e_k(a_\nu )
=\int \delta _a(x)e_j(x)e_k(x),$$ satisfy the estimates,
$$s_1\ge \frac{(E_1\cdot ..\cdot
  E_N)^{\frac{1}{N}}}{\mathrm{vol\,}(\Omega )},$$
$$
s_k\ge s_1\left(\prod_1^N \left( \frac{E_j}{s_1\mathrm{vol\,}(\Omega )}  \right) \right)^{\frac{1}{N-k+1}}.
$$
Here $E_j=\epsilon _1+...+\epsilon _{N+1-j}$, and we write $\delta
_a=\sum \delta (\cdot -a_\nu )$.

\par Let $\widehat{e} _1,...,\widehat{e}_N$ be an orthonormal
basis in $E(\mu )$, $\mu \ll h^{2/3}$, and choose $\Omega ={\cal O}_h\setminus {\cal
  O}_{2h}$, $e_j={{\widehat{e}_j}\vert}_{\Omega }$. Define ${\cal
  E}_\Omega $ as above and let $a_1,...,a_N\in \Omega $ be a
corresponding set of points. The eigenvalues
$\epsilon _j$ and the singular values $s_j=s_j(M_{\delta _a})$ remain
unchanged if we replace $\widehat{e}_1,...,\widehat{e}_j$ by another
orthonormal basis in $E(\mu )$.

\par Applying (\ref{lob.16}) to $u=\sum u_j\widehat{e}_j$, when
$\overrightarrow{u}:=(u_1,...,u_N)^\mathrm{t}$ is normalized in $\ell
^2$, we see that ${\cal E}_\Omega
(\overrightarrow{u}|\overrightarrow{u})\ge h^{4/3}/{\cal O}(1)$, so 
$E_j\ge (N-j+1)h^{4/3}/{\cal O}(1)$. Thus, for a suitable choice of
$a_1,...,a_N\in \Omega $, we get after a simple calculation:
\ekv{sv.1}{s_1\ge \frac{(N!)^{\frac{1}{N}}}{h{\cal O}(1)}h^{\frac{4}{3}},}
\ekv{sv.2}{
s_k\ge s_1^{-\frac{k-1}{N-k+1}}h^{\frac{1}{3}\frac{N}{N-k+1}}(N!)^{\frac{1}{N-k+1}}C^{-\frac{N}{N-k+1}}.
}

\par We will also need an upper bound on $s_1=s_1(M_{\delta _a})$.
Let $s>n/2$ and adopt the assumption (\ref{lob.17}).  If
$\overrightarrow{u}=(u_1,..,u_N)^{\mathrm{t}}$,
$\overrightarrow{v}=(v_1,..,v_N)^{\mathrm{t}}$ are normalized,
(\ref{lob.22}) with $k=0$ implies that $\Vert u\Vert_{H_h^s({\cal
    O}_{h/C})}$, $\Vert v\Vert_{H_h^s({\cal O}_{h/C})}$ are ${\cal
  O}(1)$ when $u=\sum u_j\widehat{e}_j$, $v=\sum v_j\widehat{e}_j$ and
also from Proposition \ref{al1} that $uv={\cal O}(h^{-n/2})$ in
$H^s_h({\cal O})$. Furthermore, we know from \cite{Sj08a} that
$\|\delta _a\|_{H_h^{-s}(\overline{{\cal O}}_{h/C})}={\cal
  O}(Nh^{-n/2})$. Hence,
\begin{equation*}
\langle M_{\delta _a}u,v\rangle=\int \delta _auvdx=
{\cal O}(1)\Vert \delta _a\Vert_{H_h^{-s}(\overline{{\cal O}}_{h/C})}\Vert
uv\Vert_{H_h^s({\cal O}_{h/C})}\le {\cal O}(1)Nh^{-n},
\end{equation*}
and varying $u,v$ we conclude that
\ekv{sv.3}
{
s_1(M_{\delta _a})=\|M_{\delta _a}\| \le {\cal O}(1)Nh^{-n}.
}
Using this in (\ref{sv.2}) gives
\ekv{sv.4}
{
s_k(M_{\delta _a})\ge C^{-\frac{N+k-1}{N-k+1}}e^{-\frac{N}{N-k+1}}Nh^{\frac{\frac{N}{3}+n(k-1)}{N-k+1}}.
}
If we restrict $k$ to the range $1\le k\le \theta N$ for some
$0<\theta <1$, we get
\ekv{sv.5}
{
s_k(M_{\delta _a})\ge C^{-\frac{1+\theta }{1-\theta
  }}e^{-\frac{1}{1-\theta }}Nh^{\frac{\frac{1}{3}+n\theta }{1-\theta }}.
}

\par Recall the form of the perturbed operator in (\ref{re.3}),
(\ref{re.4}), (\ref{re.5}), where
$\Theta \in C^\infty (\overline{{\cal O}})$
 is also described. Clearly,
$\Theta \asymp \widetilde{\Theta }(h):=h^{v_0}$ in ${\cal O}_h\setminus {\cal
  O}_{2h}$. The potential 
$\delta _a/\Theta $ satisfies
\ekv{sv.6}
{
\Vert \Theta ^{-1}\delta _a\Vert_{H_h^{-s}(\overline{{\cal O}})}\le
{\cal O}(1)\frac{N}{\widetilde{\Theta }(h)h^{\frac{n}{2}}}.
}
As in \cite{Sj08a}, (6.15)--(6.18), we get the decomposition
\ekv{sv.7}{\Theta ^{-1}\delta _a=q+r,\ q=\sum_{\mu _k\le L}\alpha
  _k\epsilon _k,}
where
\ekv{sv.8}{
\Vert q\Vert_{H^{-s}_h({\cal O})}\le \frac{CN}{\widetilde{\Theta }(h)h^{n/2}},
}
\ekv{sv.9}
{
\Vert r\Vert_{H_h^{-s}({\cal O})}\le {\cal O}(1)L^{-(s-\frac{n}{2}-\epsilon )}
\frac{N}{\widetilde{\Theta }(h)h^{n/2}},
}
\ekv{sv.10}
{
\Vert \alpha \Vert_{\ell^2}\le C\frac{L^{\frac{n}{2}+\epsilon
  }N}{\widetilde{\Theta }(h)h^{n/2}}.
}

We also denote by $\Theta $ the zero extension of $\Theta $ to all of ${\bf
  R}^n$. Under the assumption (\ref{re.4}), we have for $|\alpha |=v_0+1$,
\ekv{sv.11}
{
D^\alpha \Theta =f_\alpha +g_\alpha , 
}
where $f_\alpha \in C^\infty (\overline{{\cal O}})1_{{\cal O}} $ and
$g_{\alpha }$ is a smooth boundary layer ($\in C^\infty (\partial
{\cal O})\otimes \delta (\omega (x))$ where $\omega \in C^\infty ({\bf
  R}^n;{\bf R})$, $\omega ^{-1}(0)=\partial {\cal O}$, $d\omega \ne 0$
on $\partial {\cal O}$). Using the strict convexity and stationary
phase, we see that $\widehat{g}_\alpha (\xi )={\cal O}(\langle \xi
\rangle ^{-(n-1)/2})$ and by integration by parts, it follows that 
$$
\widehat{\Theta }(\xi )={\cal O}(1)\langle \xi \rangle ^{-v_0-1-(n-1)/2}.
$$
Here the hat indicates the ordinary ($h$-independent) Fourier
transform. In the following, we shall
assume that
\ekv{sv.12}
{
\frac{n}{2}<s<v_0+\frac{1}{2},
}
and then
\ekv{sv.13}
{
\Theta \in H_1^s(\overline{{\cal O}}).
}

From \cite{Sj08b}, we recall that if $s>n/2$, $u\in H^s({\bf R}^n)$,
$v\in H^\sigma ({\bf R}^n)$ for some $\sigma \in [-s,s]$, 
then $uv\in H^\sigma ({\bf R}^n)$ and we have
$$
\Vert uv\Vert_{H_h^\sigma }\le {\cal O}(1)\Vert u\Vert_{H^s_1}\Vert
v\Vert_{H^\sigma _h}.
$$
From (\ref{sv.7})--(\ref{sv.9}), we now deduce that 
\ekv{sv.14}
{
\delta _a=\Theta q+\widetilde{r},\quad \widetilde{r}=\Theta r,
}
where
\ekv{sv.15}
{
\Vert \widetilde{r}\Vert_{H_h^{-s}(\overline{{\cal O}})}
\le {\cal O}(1)L^{-(s-\frac{n}{2}-\epsilon
  )}\frac{N}{\widetilde{\Theta }(h)h^{n/2}},
}
\ekv{sv.16}
{
\Vert \Theta q\Vert_{H_h^{-s}(\overline{{\cal O}})}\le
\frac{CN}{\widetilde{\Theta }(h)h^{n/2}}.
}

We also need to control the $H^s_h({\cal O})$-norm of $\Theta
q$. Recall from \cite{Sj08a, Sj08b} that
$$
\Vert q\Vert_{H^s_h({\cal O})}^2\le {\cal O}(1)\sum_{\mu _k\le
  L}|\alpha _k|^2\langle \mu _k\rangle^{2s}\le {\cal O}(1)L^{2s}\Vert
\alpha \Vert_{\ell^2}^2,
$$
so 
\ekv{sv.17}{
\Vert \Theta q\Vert_{H_h^s({\cal O})}\le {\cal O}(1)\Vert
q\Vert_{H_h^s({\cal O})}\le {\cal O}(1)L^{\frac{n}{2}+s+\epsilon
}\frac{N}{\widetilde{\Theta }(h)h^{n/2}},
}
and in particular, 
\ekv{sv.18}{
\Vert \Theta q\Vert_{L^\infty ({\cal O})}\le {\cal O}(h^{-\frac{n}{2}})
\Vert \Theta q\Vert_{H_h^s({\cal O})}
\le {\cal O}(1)L^{\frac{n}{2}+s+\epsilon
}\frac{N}{\widetilde{\Theta }(h)h^{n}}.
}

\par From (\ref{sv.15}) we deduce (as above for $M_{\delta _a}$) that
\ekv{sv.19} { \Vert M_{\widetilde{r}}\Vert \le {\cal O}(1)\Vert
  \widetilde{r}\Vert_{H^{-s}_h(\overline{{\cal
        O}})}h^{-\frac{n}{2}}\le {\cal
    O}(1)L^{-(s-\frac{n}{2}-\epsilon )}\frac{N}{\widetilde{\Theta
    }(h)h^n}, }
and returning to the decomposition (\ref{sv.14}) and the lower bound
(\ref{sv.5}), we get for $1\le k\le \theta N$, $0<\theta <1$:
\ekv{sv.20}
{
s_k(M_{\Theta q})\ge C^{-\frac{1+\theta }{1-\theta
  }}e^{-\frac{1}{1-\theta }}Nh^{\frac{\frac{1}{3}+n\theta }{1-\theta
  }}
-{\cal O}(1)\frac{N}{L^{s-\frac{n}{2}-\epsilon }\widetilde{\Theta }(h)h^n}.
}
The lower bounds on $L$ will imply that the first term to the right
dominates over the second.\\

\begin{remark}\label{sv1}
{\rm For a general perturbation $W=\delta \Theta q_\omega $ as in Theorem
\ref{re1}, the discussion above shows that 
\ekv{sv.21}
{
\Vert W\Vert _{H_h^{\widetilde{s}}({\bf R}^n)}
\le {\cal O}(\delta )L^{\widetilde{s}}\Vert \alpha \Vert_{\ell^2}
\le {\cal O}(\delta )L^{\widetilde{s}}R,
}
provided that $\frac{n}{2}<\widetilde{s}<v_0+\frac{1}{2}$.}
\end{remark}

\section{End of the construction}\label{eco}
\setcounter{equation}{0}

To start with we choose $z$ in the full rectangle (\ref{ub.1}) and
later on we will restrict the attention to $ch^{2/3}<|\Im z|<c_0h^{2/3}$.
We recall that ${\cal P}_\mathrm{out}(z)$ is an elliptic boundary
value problem in the semi-classical sense in the region $|\xi '|\gg
1$. It follows that 
\ekv{eco.0}
{
\Vert u\Vert_{H^2}\le {\cal O}(1)(\Vert (P-z)u\Vert +\Vert u\Vert)
}
for $u\in {\cal D}(P_\mathrm{out}(z))$. From this estimate we see that
the small singular values $t_1(P_\mathrm{out}(z))\le
t_2(P_\mathrm{out}(z))\le ... $ are of the same order of magnitude as
the small singular values $\widetilde{t}_j$ in the $L^2$-sense defined as the square
roots of the small eigenvalues of
$P_\mathrm{out}(z)^*P_\mathrm{out}(z)$ where $P_\mathrm{out}(z)^*$ is
the adjoint of $P_\mathrm{out}(z)$ as a closed densely defined
operator: $L^2({\cal O})\to L^2({\cal O})$. This follows from
(\ref{eco.0}) and the mini-max characterizations of $t_j$ and of
$\widetilde{t}_j$. In this section it will be convenient to work with
the $\widetilde{t}_j$ and we shall drop the tildes in order to simplify
the notation.

Recall that $\widetilde{\Theta }(h)= h^{v_0}$. Let $\tau _0\in
]0,h^{4/3}/{\cal O}(1)]$ and let $N$ be determined by \ekv{eco.1}
{0\le t_1(P_\mathrm{out})\le ... \le t_N(P_\mathrm{out})<\tau _0\le
  t_{N+1}(P_\mathrm{out}),}
so that $N\le {\cal O}(h^{1-n})$ in view of (\ref{out.11}).
 The basic iteration step (cf Proposition
7.2 in \cite{Sj08a}) is
\begin{prop}\label{eco1}
Let $0<\theta <1/2 $ be the parameter in (\ref{re.6}), let
$\widetilde{\theta }\in ]0,\theta [$ and $\kappa >0$. If $N$ is sufficiently large,
depending on 
$\theta $, $\widetilde{\theta }$ only, there exists an admissible potential $q$ as in
(\ref{re.5}) with $L=L_\mathrm{min}$ and $R=R_\mathrm{min}$ (as
introduced in and after (\ref{re.6})),
 such that if
\ekv{eco.2}
{
P_\delta =P-\delta \Theta q,\ \delta =C^{-1}h^\alpha \tau _0,
}
$C\gg 1$, $\alpha \ge \alpha (n,v_0,s,\epsilon ,\theta
,\widetilde{\theta }, \kappa  )$ large enough,
then
\ekv{eco.3}
{
t_\nu (P_{\delta ,\mathrm{out}})\ge t_\nu (P_\mathrm{out})-{\cal O}(1)\delta
Nh^{-(\frac{n}{2}+s+\epsilon )M_\mathrm{min}-v_0-n},\ \nu \ge N+1,
}
\ekv{eco.4}
{
t_\nu (P_{\delta ,\mathrm{out} })\ge \tau _0h^{N_2},\ [(1-\widetilde{\theta}
)N]+1\le \nu \le N.
}
Here we put $N_2=\alpha +(\frac{1}{3}+2n\theta )/(1-2\theta
)+\kappa $ and let $[a]=\max({\bf Z}\cap ]-\infty ,a])$ denote the integer part
of $a$.

\par When $N={\cal O}(1)$ we have the same result, provided that we
replace (\ref{eco.4}) by the estimate $t_{N}(P_{\delta ,\mathrm{out}})\ge
\tau _0h^{N_2}$.
\end{prop} 

\begin{proof}
The estimate (\ref{eco.3}) follows from the mini-max characterization
of singular values, which gives 

\ekv{eco.4.5}{t_\nu
(P_{\delta ,\mathrm{out}})\ge t_\nu (P_\mathrm{out})-\delta \Vert
\Theta q\Vert_{L^\infty },}
to which we can apply (\ref{sv.18}).

\par
Let $e_1,...,e_N\in L^2({\cal O})$ be an orthonormal family of
eigenfunctions of $P_\mathrm{out}^*P_\mathrm{out}$, corresponding to
the eigenvalues $t_1^2,...,t_N^2$. Using the symmetry of $P_\mathrm{out}$,
established in Proposition \ref{red4} we see as in \cite{Sj08a} that a
corresponding family of eigenfunctions of
$P_\mathrm{out}P_\mathrm{out}^*$ is given by 
$$
f_j=\Gamma e_j,
$$
where $\Gamma $ denotes the antilinear operator of complex
conjugation. The $f_j$ form an orthonormal family corresponding to 
$$\sigma (P_\mathrm{out}P_\mathrm{out}^*)\cap [0,\tau _0^2[=\{
t_1^2,...,t_N^2\} .$$ 

Let $E_N=\bigoplus_1^N {\bf C}e_j$, $F_N=\bigoplus_1^N {\bf
  C}f_j$. Then $P_\mathrm{out}:E_N\to F_N$ and
$P_\mathrm{out}^*:F_N\to E_N$ have the same singular values
$t_1,...,t_N$. Define $R_+:L^2({\cal O})\to {\bf C}^N$, $R_-:{\bf
  C}^N\to L^2({\cal O})$, by
$$
R_+u(j)=(u|e_j),\quad R_-u_-=\sum_1^Nu_-(j)f_j.
$$
Then 
\ekv{eco.5}
{
{\cal P}=\begin{pmatrix}P_\mathrm{out} &R_-\\ R_+ & 0\end{pmatrix}:
{\cal D}(P_\mathrm{out})\times {\bf C}^N\to L^2\times {\bf C}^N
}
has the bounded inverse 
\ekv{eco.6}
{
{\cal E}=\begin{pmatrix}E &E_+\\ E_- &E_{-+}\end{pmatrix},
}
where 
\ekv{eco.7}
{
\begin{split}\Vert E\Vert\le \frac{1}{t_{N+1}}\le \frac{1}{\tau _0},\ \
E_+v_+=\sum_1^N v_+(j)e_j,\ \
E_-v(j)=(v|f_j),
\end{split}
}
and $E_{-+}$ has the singular values $t_j(E_{-+})=t_j(P_\mathrm{out})$
or equivalently, $s_j(E_{-+})=t_{N+1-j}(P_\mathrm{out})$.

\par When $N$ is large, we consider two cases:

\paragraph{Case 1.} $s_j(E_{-+})\ge \tau _0h^{N_2}$ for $1\le j\le
N-[(1-\widetilde{\theta } )N]$. We get the proposition with $q=0$, $P_\delta =P$.

\paragraph{Case 2.} $s_j(E_{-+})< \tau _0h^{N_2}$ for some $j\le
N-[(1-\widetilde{\theta } )N]$. Put $P_\delta =P+\delta \Theta q$ with $q$ as in
Section \ref{sv}. 
From (\ref{eco.2}) we deduce that
\ekv{eco.8}
{
\delta \frac{CN}{\widetilde{\Theta }(h)h^n}L^{\frac{n}{2}+s+\epsilon
}\le \frac{\tau _0}{2},
}
and then by (\ref{sv.18}) that $\delta \Vert
\Theta q\Vert_{L^\infty }\le \tau _0/2$.
We can therefore replace $P_\mathrm{out}$ by $P_{\delta ,\mathrm{out}}$ in
(\ref{eco.5}) and still get a bijective operator
$$
{\cal P}_\delta =\begin{pmatrix}P_{\delta ,\mathrm{out}} &R_-\\ R_+ & 0\end{pmatrix}
$$
with the inverse
$$
{\cal E}_\delta =\begin{pmatrix}E^\delta  &E_+^\delta \\ E_-^\delta
  &E_{-+}^\delta \end{pmatrix}.
$$
As in \cite{Sj08a}, we have
\ekv{eco.9}
{
\begin{split}
E_{-+}^\delta &=E_{-+}+\delta E_-\Theta qE_++\delta ^2E_-\Theta
qE\Theta qE_++...,\\
E^\delta &=E+\sum_1^\infty \delta ^kE(\Theta qE)^k,\\
E_+^\delta &=E_++\sum_1^\infty \delta ^k(E\Theta q)^kE_+,\\
E_-^\delta &= E_-+\sum_1^\infty \delta ^kE_-(\Theta qE)^k.
\end{split}
}
Here $\Vert E_\pm\Vert \le 1$, $\Vert E\Vert\le 1/\tau _0$ and 
in view of (\ref{eco.8}), we have $\delta \Vert \Theta
q\Vert_{L^\infty }\le \tau _0/2$, leading to:
\ekv{eco.10}{
\begin{split}
  E^\delta &=E+{\cal O}(\frac{1}{\tau _0}\frac{\delta \Vert \Theta
    q\Vert_{L^\infty }}{\tau _0}),\\
  E_+^\delta &=E_++{\cal O}(\frac{\delta \Vert \Theta
    q\Vert_{L^\infty }}{\tau _0}),\\
  E_-^\delta &= E_-+{\cal O}(\frac{\delta \Vert \Theta
    q\Vert_{L^\infty }}{\tau _0}),\\
  E_{-+}^\delta &=E_{-+}+\delta E_-\Theta qE_++{\cal O}(\frac{(\delta
    \Vert \Theta q\Vert_{L^\infty })^2}{\tau _0}).\\
\end{split}
}
The leading perturbation in $E_{-+}^\delta $ is $\delta M=\delta
E_-\Theta qE_+$, where $M=M_{\Theta q}:{\bf C}^N\to {\bf C}^N$ has
the matrix
\ekv{eco.11}
{
M_{j,k}=(\Theta qe_k|f_j)=\int \Theta qe_ke_j dx.
}

\par From the Ky Fan inequalities, we get
$$
\delta s_{k+\ell -1} (M_{\Theta q})\le s_k(E_{-+}^\delta )+s_\ell
(E_{-+})+{\cal O}( \frac{(\delta \Vert \Theta q\Vert_{L^\infty
  })^2}{\tau _0}),
$$
which we write
\ekv{eco.12}
{
s_k(E_{-+}^\delta )\ge \delta s_{k+\ell -1}(M_{\Theta q})-s_\ell
(E_{-+})-
{\cal O}( \frac{(\delta \Vert \Theta q\Vert_{L^\infty
  })^2}{\tau _0}).
}
Let $\ell=N-[(1-\widetilde{\theta } )N]$ so that $s_\ell(E_{-+})<\tau _0h^{N_2}$
and let $k\le N-[(1-\widetilde{\theta })N]$ so that
$$
k+\ell -1 \le 2(N-[(1-\widetilde{\theta } )N])-1\le 2\theta N,
$$
for $N$ large enough.
Here, $2\theta <1$, so we can apply (\ref{sv.20}) with $\theta $ there
replaced by $2\theta $ and get a $q$ as in the proposition such that
\ekv{eco.13}
{
s_{k+\ell-1}(M_{\Theta q})\ge \frac{N}{C(\theta
  )}h^{\frac{\frac{1}{3}+2n\theta }{1-2\theta }}-{\cal
  O}(1)\frac{N}{L^{s-\frac{n}{2}-\epsilon }\widetilde{\Theta }(h)h^n}.
}
Then (\ref{eco.12}) gives
\ekv{eco.14}
{
s_k(E_{-+}^\delta )\ge \delta N\left(\frac{h^{\frac{\frac{1}{3}+2n\theta
    }{1-2\theta }}}{C(\theta )}-\frac{{\cal
    O}(1)}{L^{s-\frac{n}{2}-\epsilon }\widetilde{\Theta
  }(h)h^n}\right) -\tau
_0h^{N_2}-{\cal O}(\frac{(\delta \Vert \Theta q\Vert_{L^\infty
  })^2}{\tau _0}).
}
Here we notice that with our choice of $L=L_\mathrm{min}$ large
enough, we have
$$
\frac{{\cal
    O}(1)}{L^{s-\frac{n}{2}-\epsilon }\widetilde{\Theta }(h)h^n}
\le \frac{h^{\frac{\frac{1}{3}+2n\theta
    }{1-2\theta }}}{2C(\theta ).}
$$
Thus for $k\le N-[(1-\widetilde{\theta })N]$:
$$
s_k(E_{-+}^\delta )\ge \frac{\delta N}{2C(\theta
  )}h^{\frac{\frac{1}{3}+2n\theta }{1-2\theta }}-\tau _0h^{N_2}-{\cal
  O}( \frac{(\delta \Vert \Theta q\Vert_{L^\infty })^2}{\tau _0}),
$$
and using (\ref{sv.18}): 
\ekv{eco.17} {\begin{split} &s_k(E_{-+}^\delta
    )\ge \delta N(\frac{1}{2C(\theta )}h^{\frac{\frac{1}{3}+2n\theta
      }{1-2\theta }}-{\cal O}(1)\frac{\delta }{N\tau _0}\Vert \Theta q
    \Vert_{L^\infty }^2)-
\tau _0h^{N_2}\\
&\ge \delta N(\frac{1}{2C(\theta )}h^{\frac{\frac{1}{3}+2n\theta
      }{1-2\theta }}-\frac{{\cal O}(1)\delta N}{\tau
      _0}h^{-2(\frac{n}{2}+s+\epsilon )M-2v_0-2n})-\tau _0h^{N_2}\\
&\ge \delta N(\frac{1}{2C(\theta )}h^{\frac{\frac{1}{3}+2n\theta
      }{1-2\theta }}-\frac{{\cal O}(1)\delta }{\tau
      _0}h^{1-3n-2v_0-2(\frac{n}{2}+s+\epsilon )M})-\tau
    _0h^{N_2}\\
&\ge \frac{\delta N}{4C(\theta )}h^{\frac{\frac{1}{3}+2n\theta
      }{1-2\theta }}-\tau _0h^{N_2},
  \end{split} } where the last estimate follows from the choice of
$\delta $ in (\ref{eco.2}) and we recall that $\alpha $ is large enough.

\par Here by the choice of $N_2$ the last term is subdominant when
$h>0$ is small enough and we get
\ekv{eco.18}{s_k(E_{-+}^\delta )\ge
\tau _0h^{N_2},\hbox{ for }1\le k\le N-[(1-\widetilde{\theta })N].}

\par After an arbitrarily small abstract perturbation of $P_{\delta
  ,\mathrm{out}}$, we may assume that this operator is bijective, and
we can then write the standard identity
$$
P_{\delta ,\mathrm{out}}^{-1}=E^\delta -E_+^\delta (E_{-+}^\delta
)^{-1}E_-^\delta 
$$
and apply the Ky Fan inequalities to get for $1+[(1-\widetilde{\theta})N]\le \nu \le
N$:
\begin{equation*}\begin{split}
    s_\nu (P_{\delta ,\mathrm{out}}^{-1})&\le s_1(E^\delta )+\Vert
    E_+^\delta \Vert \Vert E_-^\delta \Vert s_\nu ((E_{-+}^\delta
    )^{-1})\\
&\le {\cal O}(1)\frac{1}{h^{N_2}\tau _0},
\end{split}\end{equation*}
since $
s_\nu ((E_{-+}^\delta )^{-1})=1/s_{N+1-\nu }(E_{-+}^\delta )
$ and $1\le N+1-\nu \le N-[(1-\widetilde{\theta })N]$,
or in other terms,
$$
t_\nu (P_{\delta ,\mathrm{out}})\ge \frac{\tau _0h^{N_2}}{{\cal O}(1)}.
$$
This is (\ref{eco.4}) apart from the factor $1/{\cal O}(1)$, which can
be eliminated by increasing $N_2$ slightly.

\par When $N={\cal O}(1)$ we consider the two cases $s_1(E_{-+})\ge
\tau _0h^{N_2}$ and $s_1(E_{-+})<\tau _0h^{N_2}$. In the first case we
take the perturbation $0$ as before. In the second case, we repeat the
proof above with $k=\ell=1$ and reach first (\ref{eco.18}) with $k=1$
and finally (\ref{eco.4}) with $\nu =N$.
\end{proof}
\begin{remark}\label{eco2}
{\rm \begin{itemize}
\item[1)] In the proof we have seen that $\delta \Vert \Theta
  q\Vert_{L^\infty }\le \tau _0/2$ and (\ref{eco.4.5}) shows that
$$t_\nu (P_{\delta ,\mathrm{out}})\ge t_\nu
(P_\mathrm{out})-\frac{\tau _0}{2}\ge \frac{\tau _0}{2},\ \nu \ge
N+1.$$
\item[2)] From (\ref{eco.8}), (\ref{sv.17}), we get 
$$
\Vert \delta \Theta q\Vert_{H^s_h}\le {\cal O}(1)\tau _0h^{\frac{n}{2}}.
$$ 
\item[3)] Let $\widetilde{s}>\frac{n}{2}+2N$, where $N$ is the
  smallest integer in $]\frac{n-1}{2},+\infty [$. If we choose $\alpha
  $ in (\ref{re.8}) sufficiently large, then
$$\Vert \delta \Theta q\Vert_{H^{\widetilde{s}}_h}\le {\cal O}(h^{\frac{n}{2}}).$$ 
\end{itemize}
We see that the perturbed operator $P_\delta $ satisfies the
general assumptions of our discussion, including (\ref{idn.38}),
(\ref{ub'.32}), (\ref{lob.17}) for $W=\delta \Theta q$.}
\end{remark}

\par The last remark shows that we can apply Proposition \ref{eco1} to
$P_{\delta ,\mathrm{out}}$ with $\tau _0$ replaced by $\tau _0h^{N_2}$
and $N$ replaced by an $N_\mathrm{new}\le [(1-\widetilde{\theta } )N]$. The
procedure can be iterated at most ${\cal O}(1)\ln \frac{1}{h}$
times until we get a perturbation $P_{\mathrm{final\,},\delta
  ,\mathrm{out}}$ with $t_1(P_{\mathrm{final\,},\delta
  ,\mathrm{out}})\ge \tau _0h^{{\cal O}(1)\ln
  \frac{1}{h}}$. Thus in the end we get
\begin{prop}\label{eco3}
  Let $0<\theta <1/2 $ be the parameter in (\ref{re.6}) and let $\tau
  _0\in ]0,h^{4/3}]$. Then there exists
  an admissible potential $q$ as in (\ref{re.5}) with
  $L=L_\mathrm{min}$ and $R=R_\mathrm{min}$ (as introduced in and
  after (\ref{re.6}))  such that if 
\ekv{eco.19} { P_\delta =P+\delta
    \Theta q,\ \delta =C^{-1}h^\alpha \tau _0, } $C\gg 1$, $\alpha \ge
  \alpha (n,v_0,s,\epsilon ,\theta )$ large enough, then 
\ekv{eco.20}
{
t_1(P_{\delta
  ,\mathrm{out}})\ge \tau _0h^{{\cal O}(1)\ln
  \frac{1}{h}}.
}
\end{prop}

From (\ref{out.8}) we get for the special perturbation above
\ekv{eco.21} {s_1(({\cal N}_\mathrm{in}-{\cal
    N}_\mathrm{ext})^{-1})\le \frac{{\cal
      O}(1)}{h^{\frac{2}{3}}t_1({\cal P}_\mathrm{out})}\le \frac{{\cal
      O}(1)}{\tau _0 h^{{\cal O}(1)\ln \frac{1}{h}}},} and
(\ref{out.9}) then gives \ekv{eco.22} {
  s_1((1+\widehat{C}(z))^{-1})\le \frac{{\cal O}(1)}{\tau _0h^{{\cal
        O}(1)\ln \frac{1}{h}}}.  }

\par
Recall from Proposition \ref{ub'4} and
(\ref{ub'.36})--(\ref{ub'.38}) that
\ekv{eco.23}
{
\widehat{C}(z)={\cal O}(1):\, H^{\frac{3}{2}}\to H^{\frac{3}{2}},
\ |\Im z|\ge h^{2/3}c,
}
in addition to the fact that the trace class norm of the same operator
is ${\cal O}(h^{1-n})$. We now work with $H^{3/2}(\partial {\cal O})$
as the underlying Hilbert space and let $\widehat{C}^*$ denote the
adjoint of $\widehat{C}$. Consider,
\ekv{eco.24}
{
|\det (1+\widehat{C})|^2=\det (1+\widehat{C}^*)(1+\widehat{C})=\det (1+D),
}
where $D=\widehat{C}+\widehat{C}^*+\widehat{C}^*\widehat{C}$ is
self-adjoint, ${\cal O}(1)$ in operator norm and ${\cal O}(h^{1-n})$
in trace class norm. Let $\lambda _1,\,\lambda _2,...$ denote the
non-vanishing eigenvalues of $D$, so that 
\ekv{eco.25}
{
1+\lambda _j\ge \frac{\tau _0^2}{{\cal O}(1)}h^{2{\cal O}(1)\ln \frac{1}{n}}
}
by (\ref{eco.22}) (which is a bound on the norm of
$(1+\widehat{C})^{-1}$). We also know that $\sum |\lambda _j|={\cal
  O}(h^{1-n})$, so there are at most ${\cal O}(h^{1-n})$ values $j$
for which $|\lambda _j|\ge 1/2$.  Thus we get from (\ref{eco.24}):
\begin{multline*}
|\det (1+\widehat{C})|^2=\prod (1+\lambda _j)\\
=\prod_{j;\, |\lambda _j|\ge \frac{1}{2}} (1+\lambda _j)\prod_{j;\, |\lambda
  _j|< \frac{1}{2}}(1+\lambda _j)\\
\ge \left( \frac{\tau _0^2}{{\cal O}(1)}h^{2{\cal O}(1)\ln
    \frac{1}{h}}\right)^{{\cal O}(h^{1-n})}
\prod_{j;\, |\lambda _j|\le \frac{1}{2}}e^{-{\cal O}(1)|\lambda _j|}.
\end{multline*}
Since $\sum |\lambda _j|={\cal O}(h^{1-n})$, we get
\ekv{eco.26}
{
\ln |\det (1+\widehat{C})|\ge -{\cal O}(h^{1-n})((\ln
\frac{1}{h})^2+\ln \frac{1}{\tau _0}).
}

\par Now return to the function $f(z)$ that was (re)defined in
(\ref{ub'.65}). From (\ref{ub'.66}), (\ref{eco.26}) and (\ref{ub'.68})
we get for our special perturbation $V=V_0+W$ (where $W$ depends on
$z$ with $ch^{2/3}\le |\Im z|\le c_0h^{2/3}$): \ekv{eco.27} { \phi
  _\mathrm{in}(z)-{\cal O}(h^{1-n})((\ln \frac{1}{h})^2+\ln
  \frac{1}{\tau _0})\le \ln |f(z)|\le \phi _\mathrm{in}(z)+{\cal
    O}(h^{1-n}) } Here the upper bound is valid for all perturbations
$V$ of $V_0$ in our class independently of $z$ with $|\Im z|\asymp
h^{2/3}/C$, while the lower bound is valid for our special
$z$-dependent perturbation.

$\phi _\mathrm{in}$ (cf. (\ref{ub'.59})) is defined in terms of the
interior Dirichlet problem for the perturbed potential $V_0+W$ where
$W$ also depends on $z$, and we would like to replace this function by
one which is independent of the perturbation $W$. To emphasize the
presence of the perturbation we write
$$
\phi _\mathrm{in}^\delta (z)=\sum \chi (\lambda _j^\delta )\ln
|z-\lambda _j^\delta |
$$
for the function in (\ref{eco.27}), and
$$
\phi _\mathrm{in}^0(z)=\sum \chi (\lambda _j^0)\ln
|z-\lambda _j^0|
$$
for the corresponding function, associated to the unperturbed
operator $P_0^\mathrm{in}$. 

\par From the mini-max principle, we get
$$
|\lambda _j^\delta -\lambda _j^0|\le \Vert W\Vert_\infty .
$$
For $|\Im z|\ge r$, $0<r\le 1$, we see that
$$
|\frac{\partial }{\partial \lambda }(\chi (\lambda )\ln |z-\lambda
|)|\le {\cal O}(\frac{1}{r}),
$$
so 
$$
|\chi (\lambda _j^\delta )\ln |z-\lambda _j^\delta |-\chi (\lambda
_j^0)\ln |z-\lambda _j^0 ||\le {\cal O}(1)\frac{\Vert W\Vert_\infty }{r}.
$$
The number of eigenvalues of $P_\mathrm{in}^\delta $ and of
$P_\mathrm{in}^0 $ in $\mathrm{supp\,}\chi $ is ${\cal O}(h^{-n})$ and
it follows that
$$
|\phi _\mathrm{in}^\delta (z) -\phi _\mathrm{in}^0 (z)|\le {\cal
  O}(1)\frac{\Vert W\Vert_\infty }{rh^n}.
$$
Here we take $r\asymp h^{2/3}$ as in (\ref{eco.27}). From the second part
of Remark \ref{eco2} we know that $W=\delta \Theta q$ satisfies
$$
\Vert W\Vert_\infty \le {\cal O}(1)h^{-\frac{n}{2}}\Vert
W\Vert_{H_h^s}\le {\cal O}(1)\tau _0
$$
and thus
$$
|\phi _\mathrm{in}^\delta (z)-\phi _\mathrm{in}^0(z)|\le {\cal
  O}(1)\tau _0h^{-\frac{2}{3}-n}.
$$
In Proposition \ref{eco3} we have assumed that $0<\tau _0\le
h^{4/3}$. We now strengthen that assumption to 
\ekv{eco.28}
{
\tau _0\in ]0,h^{\frac{5}{3}}].
}
Then,
\ekv{eco.29}
{
|\phi _\mathrm{in}^\delta (z)-\phi _\mathrm{in}^0(z)|\le {\cal
  O}(1)h^{1-n}
}
and we obtain
\begin{prop}\label{eco4}
In (\ref{eco.27}) we can replace $\phi _\mathrm{in}=\phi
_\mathrm{in}^\delta $ by the function $\phi _\mathrm{in}^0$, defined
for the unperturbed operator $P_\mathrm{in}^0$ as in (\ref{ub'.59}).
\end{prop}

\section{End of the proof of Theorem \ref{re1} and proof of
  Proposition \ref{re3}}\label{ep}
\setcounter{equation}{0}

Let $\phi _\mathrm{in}^0$ be defined in (\ref{ub'.59}) with respect to
the unperturbated operator $P_\mathrm{in}^0$. With $r=h^{2/3}c/4$,
let $h^0=h^0_r$ be the harmonic majorant in $\Omega _r$ and define
$\Phi _r^0=\Phi ^0$ as in (\ref{ub'.61}). Recall that $f$
 is defined in (\ref{ub'.65}) (for the perturbed operator $P_\delta
 $). Since $\phi ^\delta _\mathrm{in}-\phi ^0_\mathrm{in}={\cal
   O}(h^{1-n})$ by (\ref{eco.29}), we have the same estimate for
 $h_r-h_r^0$ and hence for $\Phi _r-\Phi _r^0$. Then by (\ref{ub'.71})
 we conclude that 
\ekv{ep.1}
{
\ln |f(z)|\le \Phi ^0_r(z)+{\cal O}(h^{1-n})\hbox{ in the rectangle (\ref{ub.1}).}
}

\par For each $z$ as in (\ref{ub'.10.5}) we have constructed a
perturbation $W=\delta \Theta q$ as in and after (\ref{re.6}) with $L=L_\mathrm{min}$,
$R=R_\mathrm{min}$ such that (cf Proposition \ref{eco4})
\ekv{ep.2}
{
\Phi _r^0-{\cal O}(h^{1-n})((\ln \frac{1}{h})^2+\ln \frac{1}{\tau
  _0})\le \ln |f(z)|.
}
Let 
\ekv{ep.3}
{
\epsilon _0(h)=Ch((\ln \frac{1}{h})^2+\ln \frac{1}{\tau _0})
}
so that 
\ekv{ep.4}
{
\ln |f(z)|\le \Phi _r^0(z)+h^{-n}\epsilon _0(h)
}
for all $z$ in the rectangle (\ref{ub.1}) and so that for every $z$
as in (\ref{ub'.10.5}), there is a perturbation as in
(\ref{ep.2}) such that
\ekv{ep.5}
{
\ln |f(z)|\ge \Phi _r^0-h^{-n}\epsilon _0(h). 
}
If we fix such a value of $z$ and work in the $\alpha $-variables, we
are in the same situation as in Section 8 in \cite{Sj08a} and we can
apply Proposition 8.2 and Remark 8.3
 of that paper to obtain
\begin{prop}\label{ep1}
Let $\epsilon >0$ be small enough so that $\epsilon \exp ({\cal
  O}(\epsilon _0)h^{-n})\le 1$. For each $z$ as in (\ref{ub'.10.5}), we have
\ekv{ep.6}
{
P(|f(z)|\le e^{\Phi _r^0}\epsilon )\le {\cal O}(1)\frac{\epsilon
  _0(h)}{h^{n+N_6}}\exp \left( \frac{h^n}{{\cal O}(1)\epsilon
    _0(h)}\ln \epsilon \right).
}
Here $N_6=\max (N_3,N_5)$, where $N_3=n(M+1)$,
$N_5=N_4+\widetilde{M}$. (Cf (\ref{re.9}).)
\end{prop}

If we write $\epsilon =e^{-\widetilde{\epsilon }/h^n}$, then the
condition on $\epsilon $ is fulfilled when
\ekv{ep.7}
{
\widetilde{\epsilon }\ge \mathrm{Const.\,}\epsilon _0
}
and (\ref{ep.6}) becomes
\ekv{ep.8}
{
P(|f(z)|\le e^{\Phi ^0_r(z)-\frac{\widetilde{\epsilon }}{h^n}})
\le {\cal O}(1)\frac{\epsilon _0(h)}{h^{n+N_6}}\exp \left(
  -\frac{\widetilde{\epsilon }}{{\cal O}(1)\epsilon _0(h)} \right).
}

Let $\frac{1}{2}\le a<b\le 2$ and put $\Gamma
=[a,b]+i{h^{\frac{2}{3}}}c[-1,1]$, $r=h^{2/3}c/4$. We shall
apply Theorem 1.2 in \cite{Sj09} to the function $u=f$, with $h$ there
replaced by $h^n$ and with $\phi =h^n\Phi _r$. Let 
$$
\rho (t)=\max (4ch^{\frac{2}{3}}-\frac{1}{2}(t-a),
{h^{\frac{2}{3}}}c/2,4ch^{\frac{2}{3}}-\frac{1}{2}(b-t)),\ a\le
t\le b,
$$
and define the function $\widetilde{r}:\partial \Gamma \to ]0,\infty
[$ by 
$$
\widetilde{r}(z)=\rho (\Re t).
$$
Then $\widetilde{r}$ has Lipschitz modulus $\le \frac{1}{2}$ and this
will be our function ``$r$'' in \cite{Sj09}. Choose points
$z_1^0,...,z_N^0\in \partial \Gamma $ as in the introduction of
\cite{Sj09}. This can be done in a such a way that $|\Im
z_j^0|=h^{2/3}c$ for all $j$. Moreover, we see that $N\asymp
h^{-2/3}$ and further $\Delta \Phi _r=0$ in $D(z_j^0,r(z_j^0))$ except
for at most ${\cal O}(1)$ values of $j$ . Let $\widetilde{z}_j\in D(z_j^0,r(z_j^0)/(2C_1))$ be as in
Theorem 1.2 in \cite{Sj09}, where we recall that these points depend
on $\Phi _r, \Gamma , \widetilde{r}$ but not on the function
$f$. Moreover we notice that $C_1$ can be chosen arbitrarily
large. Then according to (\ref{ep.8}) we have 
\ekv{ep.9}
{
|f(\widetilde{z}_j)|\ge e^{\Phi
  _r(\widetilde{z}_j)-\frac{\widetilde{\epsilon }}{h^n}},\ j=1,2,...,N
}
with probability
\ekv{ep.10}
{
\ge 1-{\cal O}(1)\frac{N\epsilon
  _0(h)}{h^{n+N_6}}e^{-\frac{\widetilde{\epsilon }}{{\cal
      O}(1)\epsilon _0(h)}}=
1-{\cal O}(1)\frac{\epsilon
  _0(h)}{h^{n+N_6+\frac{2}{3}}}e^{-\frac{\widetilde{\epsilon }}{{\cal
      O}(1)\epsilon _0(h)}}
}
Here we recall that (\ref{ep.7}) holds and that $|f|\le e^{\Phi
  _r+\widetilde{\epsilon }/h^n}$ in a neighborhood of $\Gamma
$. Theorem 1.2 in \cite{Sj09} then shows that with $\sigma (P_\delta
)$ denoting the set of resonances of $P_\delta $,
\ekv{ep.11}
{
\begin{split}
&|
\# (\sigma (P_\delta )\cap ([a,b]+i{h^{\frac{2}{3}}}c[-1,0])
-\frac{1}{2\pi }\int_{[a,b]+{ih^{\frac{2}{3}}}c[-1,1]} \Delta \Phi
_r^0 L(dz)
|\\
&\le 
C_2(\sum_{w=a,b}\int_{[w-Ch^{\frac{2}{3}},w+Ch^{\frac{2}{3}}]+{ih^{\frac{2}{3}}}c[-1,1]}
\Delta \Phi _r^0 L(dz)+h^{-n}\sum_1^N \widetilde{\epsilon }),
\end{split}
}
with a probability as in (\ref{ep.10}). Here we assume for simplicity
that $c\ll c_0$, otherwise we have to slightly modify the choice of
$\rho ,r,z_j^0$ above.

\par Now recall (\ref{ub'.64}) where $g_r(t)=r^{-1}g_1(t/r)$, $0\le g_1\in
{\cal S}({\bf R})$, $\int g_1 dt=1$. With $N_0$ denoting the
eigenvalue counting function for $P_\mathrm{in}^0$, we get with
probability as in (\ref{ep.10}),
\ekv{ep.12}
{
\begin{split}
&|
\# (\sigma (P_\delta )\cap ([a,b]+i{h^{\frac{2}{3}}}c[-1,0])
-\int_a^b g_r*(\chi dN_0)(t)dt
|\\
&\le 
C_2(\sum_{w=a,b}\int_{w-Ch^{\frac{2}{3}}}^{w+Ch^{\frac{2}{3}}}
g_r*(\chi dN_0)(t)dt+{\cal O}(h^{-\frac{2}{3}-n}\widetilde{\epsilon })).
\end{split}
} 

This is a slightly stronger version of the main result (\ref{re.14})
as we shall see next. Consider
$$
J:=\int_a^b g_r*(\chi dN_0)(t)dt=\int_a^b\int_{\bf R}g_r(t-s)\chi (s)dN_0(s)dt,
$$
where we recall that $r=h^{2/3}c/4$. We split the integral into
$\mathrm{I}+\mathrm{II}$, where $\mathrm{I}$ is obtained by retricting
the $s$ integration to the interval $[a-\rho ,b+\rho ]$ and
$\mathrm{II}$ is obtained from integration in $s$ over ${\bf
  R}\setminus [a-\rho ,b+\rho ]$. Here we take $\rho =h^{-\delta
+{2/3}}$, where $\delta >0$ can be arbitrarily small but independent of
$h$. 

\par Carrying out first the $t$ integration, we see that 
$$\mathrm{I}\le \int_{[a-\rho ,b+\rho ]}\chi (s)dN_0(s)=N_0(b+\rho
)-N_0(a-\rho ) .$$

\par As for $\mathrm{II}$, we have uniformly for $t\in [a,b]$ that 
$$
\int_{{\bf R}\setminus [a-\rho ,b+\rho ]}g_r(t-s)\chi (s)dN(s)\le
\int_{|t-s|\ge \rho }\frac{1}{r}g_1(\frac{t-s}{r})\chi (s)dN(s)={\cal
  O}(h^\infty ),
$$
since $\rho /r\ge h^{-\delta }c/4$ so that $g_1((t-s)/r)/r={\cal
  O}(h^\infty )$ and $\int \chi (s)dN(s)={\cal O}(h^{-n})$.
Thus,
$$J\le N_0(b+\rho )-N_0(a-\rho )+{\cal O}(h^\infty ).$$ 

To get a corresponding lower bound, assume $b-a\ge 2\rho $ (in order to
exclude a trivial case), and write
$$
J\ge \int_a^b\int_{a+\rho }^{b-\rho }g_r(t-s)\chi (s)dN_0(s) dt.
$$
For $a+\rho \le s\le b-\rho $, we have 
$$
1\ge \int_a^b g_r(t-s)dt\ge 1-{\cal O}(h^\infty ),
$$
so 
\begin{multline*}
J\ge \int_{a+\rho }^{b-\rho }(1-{\cal O}(h^\infty ))dN_0(s)\\\ge
(1-{\cal O}(h^\infty ))(N_0(b-\rho )-N_0(a+\rho ))\\\ge N_0(b-\rho
)-N_0(a+\rho )-{\cal O}(h^\infty ). 
\end{multline*}

\par In conclusion, for $r=h^{2/3}c/4$, $\rho =h^{-\delta +2/3}$, we
get from (\ref{ep.12}), \ekv{ep.13}  {\begin{split} N_0(b-\rho )&-N_0(a+\rho )-{\cal O}(h^\infty )\\&\le
  \int_a^b g_r*(\chi dN_0)(t)dt\\&\hskip 2cm \le N_0(b+\rho )-N_0(a-\rho )+{\cal
    O}(h^\infty ) .\end{split}}

\par Applying this to (\ref{ep.12}), we get with a probability as in
(\ref{ep.10})
\ekv{ep.14}
{\begin{split}
|\# (\sigma (P_\delta )\cap
([a,b]+i{h^{\frac{2}{3}}}c[-1,0])-(N_0(b)-N_0(a))|\\
\le {\cal
  O}(1)
(\sum_{w=a,b}(N_0(w+\rho )-N_0(w-\rho
))+h^{-\frac{2}{3}-n}\widetilde{\epsilon }).\end{split}
}
This concludes the proof of Theorem \ref{re1}.

\medskip
\begin{proofof} Proposition \ref{re3}.
Let $V_0$ be as in Theorem \ref{re1} and let $W_0$ satisfy the
assumptions of the proposition. Our unperturbed operator is now
\ekv{ep.15}
{
P_0=-h^2\Delta +V_0+W_0=P^{V_0+W_0}.
}
rather than the right hand side of (\ref{re.1}) that we now denote by
$P_0^0$. The proof will consist in checking the proof of Theorem
\ref{re1} with this new operator $P_0$. 

\par Nothing changes until Section \ref{idn}. Here Proposition
\ref{idn5} can be used instead of Proposition \ref{idn4} to see that
the conclusion of Proposition \ref{idn1} is valid for (the new)
unperturbed operator $P_0$ as well as for the perturbed operator $P^V$
in (\ref{sbd.2}), where now $V=V_0+W_0+W$ and as before $W={\cal O}(h)$ in
$L^\infty $.

\par The discussion in Section \ref{sbd} remains valid.

\par In Section \ref{ub} the first change appears after
(\ref{ub.9.4}), where we now take $V=V_0+W_0+W$ with $\Vert
W\Vert_{L^\infty }={\cal O}(1)$. Then we still have  (\ref{ub.9.5})
provided that we modify the definition of $\widetilde{P}$ prior to
(\ref{ub.9.2}) by taking $\widetilde{P}=P+Ci1_{\cal O}$ with $C$ large
enough. We obtain Proposition \ref{ub2} as before.

In the subsequent discusson, $P_0$ is the same operator but with the
new notation $P_0^0=P^{V_0}$, while $P=P^V$ with $V=V_0+W_0+W$ with
the initial assumption that $W={\cal O}(h)$ in $L^\infty $. After
(\ref{ub'.12.5}) we just have to invoke Proposition \ref{idn5} instead of 
Proposition \ref{idn4}.

\par In the expression for $\widetilde{K}$ after (\ref{ub'.14}) we
have to replace $W$ by $W_0+W$ and as in the proof of Proposition
\ref{idn5}, we have $(\widetilde{P}-z)^{-1}W_0\widetilde{K}_0={\cal
  O}(h^2):$ $H^{3/2}\to H^2$. Thus instead of (\ref{ub'.15}) we get
\ekv{ep.16}
{
\widetilde{{\cal N}}=\widetilde{{\cal N}}_0+{\cal O}(1)\Vert
W\Vert_{L^\infty }+{\cal O}(h^2):\, H^{3/2}\to H^{1/2}.
}

\par Lemma \ref{ub'3} remains valid since $W_0$ also satisfies
(\ref{ub'.20}). Since $W_0$ satisfies (\ref{ub'.32}), the following
discussion goes through without any changes until Proposition
\ref{ub'4}, where we just have to add a term ${\cal O}(h^2)$ to the
estimate of $\widetilde{{\cal N}}-\widetilde{{\cal N}}_0$ after
(\ref{ub'.33}). The remainder of Section \ref{ub} goes through without
any changes. 

\par After that, there are no changes. $P^0_\mathrm{in}$ in
Proposition \ref{eco4} is the Dirichlet realization of (the new) $P_0=P^{V_0+W_0}$.
\end{proofof}
\appendix\section{WKB estimates on an interval}\label{a}
\setcounter{equation}{0}
We follow \cite{Fe87, Vor81}. See also \cite{BeMo72}. Let $V\in C^2([a,b])$, $-\infty
<a<b<+\infty $ and assume that $V(x)\ne 0$ for all $x\in
[a,b]$. Choose a branch of $\ln V(x)$ and put $V(x)^\theta =\exp
(\theta \ln V(x))$. Put 
\[
\begin{split}
&y_\pm(x)=V(x)^{-\frac{1}{4}}e^{\pm \phi (x)/h}=e^{\psi _\pm (x)/h},\\
&\psi _\pm =\pm \phi -\frac{h}{4}\ln V(x),\ \phi '(x)=V(x)^{\frac{1}{2}}.
\end{split}
\]
Then 
\[\begin{split}
e^{-\psi _\pm /h}\circ (V(x)-(h\partial )^2)\circ e^{\psi _\pm
  /h}&=-(h\partial )^2-2\psi _\pm '\circ h\partial +h^2r,\\
r&=\frac{1}{4}\frac{V''}{V}-\frac{5}{16}\left( \frac{V'}{V}\right)^2,
\end{split}
\]
so 
$$
(V-(h\partial )^2)y_\pm =h^2ry_\pm .$$

\par The equation $(V-(h\partial )^2)y=0$ can be written 
\ekv{a.1}
{
\left(h\partial -\begin{pmatrix}0 &1\\V &0\end{pmatrix}
\right)\begin{pmatrix}y\\h\partial y\end{pmatrix}=0.
}
Put 
$$
e_\pm =\begin{pmatrix}1\\\frac{h\partial
    y_\pm}{y_\pm}\end{pmatrix}=\begin{pmatrix}1\\ \partial \psi _\pm\end{pmatrix}.
$$
From the identity
$$
\left(h\partial -\begin{pmatrix}0 &1\\V &0\end{pmatrix}
\right) \begin{pmatrix}y_\pm \\h\partial y_\pm\end{pmatrix}+h^2ry_\pm \begin{pmatrix}0\\1\end{pmatrix}=0
$$
we get
\ekv{a.2}
{
\left(h\partial +\psi '_\pm -\begin{pmatrix}0 &1\\V &0\end{pmatrix} \right)
e_\pm +h^2r\begin{pmatrix}0\\1\end{pmatrix}=0.
}
If $u_\pm$ is a scalar $C^1$-function, we get
\ekv{a.3}
{
\left(h\partial -\begin{pmatrix}0 &1\\ V &0\end{pmatrix}
\right)u_{\pm}e_\pm =h\partial (u_\pm )e_\pm -u_\pm \psi '_\pm e_\pm
-u_\pm h^2r\begin{pmatrix}0\\1\end{pmatrix}.
}
Here,
$$
\begin{pmatrix}0\\1\end{pmatrix}=\frac{1}{2}V^{-\frac{1}{2}}(e_+-e_-)
$$
and with the substitution
\ekv{a.23}{
\begin{pmatrix}y\\h\partial y\end{pmatrix}=u_+e_++u_-e_-
\Leftrightarrow \begin{cases}y=u_++u_-\\ h\partial y=u_+\partial \psi
  _++u_-\partial \psi _-\end{cases},}
we find after some calculation that (\ref{a.1}) is equivalent to
\ekv{a.4}
{
\left(h\partial -\begin{pmatrix}\psi _+' &0\\0&\psi _-'\end{pmatrix}
-h^2r\frac{1}{2}V^{-\frac{1}{2}}\begin{pmatrix}1 &1\\ -1 &-1\end{pmatrix}
 \right)\begin{pmatrix}u_+\\u_-\end{pmatrix}=0.
}
Here,
\ekv{a.5}
{
\frac{r}{V^{\frac{1}{2}}}=\frac{1}{4}\frac{V''}{V^{\frac{3}{2}}}-\frac{5}{16}\frac{(V')^2}{V^{\frac{5}{2}}} .
}

\par Let $E(x,y)$ be the forward fundamental solution of the
differential operator in (\ref{a.4}), i.e. the one which vanishes for
$x<y$. Then for $a\le y\le x\le b$:
\ekv{a.7}
{
\|E(x,y)\|\le \frac{1}{h}\exp \frac{1}{h}\int_y^x \left(\max (\Re \psi
_+',\Re \psi _-')(t)+Ch^2|rV^{-\frac{1}{2}}|(t) \right)dt.
}

\par Assume from now on that
\ekv{a.8}
{
\Re V(x)^{\frac{1}{2}}\ge 0,\ x\in [a,b].
}
Then (\ref{a.7}) simplifies to
\ekv{a.9}
{
\|E(x,y)\| \le \frac{1}{h}e^{\frac{1}{h}(\Re \psi _+(x)-\Re \psi
  _+(y))}e^{Ch\int_y^x|rV^{-\frac{1}{2}}|(t)dt}.
}

\par Let us consider the situation of a simple turning point:
\ekv{a.10}
{\begin{split}
&|V(x)|\asymp |x-z_0|,\ V',\, V''={\cal O}(1),\\
&|x-z_0|\ge \frac{h^{\frac{2}{3}}}{C}\hbox{ for }x\in [a,b],
\end{split}
}
where $z_0\in {\bf C}$. Then from (\ref{a.5}) we have
$\int_y^x|r/V^{1/2}|dz={\cal O}(1/h)$ and the last exponential in
(\ref{a.9}) is ${\cal O}(1)$. We get
\ekv{a.11}
{
\|E(x,y)\| \le {\cal O}(\frac{1}{h})e^{\frac{1}{h}(\Re \psi _+(x)-\Re \psi _+(y))},\
a\le y\le x\le b.
}

\par Apply the operator in (\ref{a.4}) to 
$$
u^0=\begin{pmatrix}u_+^0\\ u_-^0\end{pmatrix}=\begin{pmatrix}y_+\\0\end{pmatrix}.
$$
We get
$$
\left(h\partial -\begin{pmatrix}\psi _+'&0\\0 &\psi_-' \end{pmatrix} 
-h^2\frac{r}{2V^{\frac{1}{2}}}\begin{pmatrix}1 &1\\ -1&-1\end{pmatrix}\right)u^0=
-h^2\frac{r}{2V^{\frac{1}{2}}}\begin{pmatrix}y_+\\-y_+\end{pmatrix},
$$
and we have the solution
$$
\begin{pmatrix}u_+\\u_-\end{pmatrix}=u^0+\begin{pmatrix}f_+\\f_-\end{pmatrix}
$$
of (\ref{a.4}), where 
$$
\begin{pmatrix} f_+\\f_- \end{pmatrix}=\int_a^x E(x,y)h^2\frac{r}{2V^{\frac{1}{2}}}(y)\begin{pmatrix}y_+\\-y_+\end{pmatrix}(y)dy.
$$
Here 
$$
\frac{r}{V^{\frac{1}{2}}}(y)=\frac{{\cal O}(1)}{|y-z_0|^{\frac{5}{2}}}
$$
and using (\ref{a.11}), we get
\ekv{a.12}
{
\|\begin{pmatrix}f_+\\f_-\end{pmatrix}\| 
\le Che^{\frac{\psi _+(x)}{h}}\int_a^x\frac{1}{|y-z_0|^{\frac{5}{2}}}dy\le
{\cal O}(1)e^{\frac{\psi _+(x)}{h}}.
}
Thus we have the exact solution of (\ref{a.4}):
\ekv{a.13}
{
\begin{pmatrix}u_+\\u_-\end{pmatrix}= e^{\frac{\psi
    _+}{h}}{\cal O}(1).
}
If we make the substitution (\ref{a.23}), 
we see that $y$ is an exact solution of 
\ekv{a.14}
{
(V-(h\partial )^2)y=0,
}
which satisfies
\ekv{a.15}
{
y={\cal O}(1)e^{\frac{\psi _+}{h}},
}
\ekv{a.16}{h\partial y={\cal O}(1)e^{\frac{\psi _+}{h}}.}
Using this with (\ref{a.14}), we get similar approximations for the
higher derivatives of $y$.

\par The inhomogeneous equation 
\ekv{a.17}
{
(V-(h\partial )^2)y=z,
}
can be transformed into a system
\ekv{a.18}
{
\left(h\partial -\begin{pmatrix}0 &1\\V&0\end{pmatrix}
\right)\begin{pmatrix}y\\ h\partial y\end{pmatrix}=\begin{pmatrix}0\\ -z\end{pmatrix},
}
where the right hand side can be written $z_+e_++z_-e_-$,
$z_+=-z_-=-z/(2V^{1/2})$. The substitution (\ref{a.23}) gives 
\ekv{a.19}
{
\left(h\partial -\begin{pmatrix}\psi _+' &0\\0&\psi _-'\end{pmatrix}
-h^2r\frac{1}{2}V^{-\frac{1}{2}}\begin{pmatrix}1 &1\\ -1 &-1\end{pmatrix}
 \right)\begin{pmatrix}u_+\\u_-\end{pmatrix}=-\frac{z}{2V^{1/2}}\begin{pmatrix}1\\-1\end{pmatrix},
}
which has the solution
\ekv{a.20}
{
\begin{pmatrix}u_+\\u_-\end{pmatrix}=-\int_a^x
E(x,y)\frac{z(y)}{2V(y)^{1/2}}dy\begin{pmatrix}1\\-1\end{pmatrix}.
}
Writing 
$$
E(x,y)=\begin{pmatrix}E_{++}&E_{+-}\\E_{-+}&E_{--}\end{pmatrix},
$$
we get
\ekv{a.21}
{
\begin{split}
u_+(x)=\int_a^x (-E_{++}(x,y)+E_{+-}(x,y))\frac{z(y)}{2V(y)^{1/2}}dy\\
u_-(x)=\int_a^x (-E_{-+}(x,y)+E_{--}(x,y))\frac{z(y)}{2V(y)^{1/2}}dy
\end{split}
}
cf (\ref{a.23}).

\par Now we add the assumption that $V\in C^\infty ([a,b])$. Assume
for simplicity that $\Re z_0=0$ and assume that $b\le 0$. It is
standard that we have exact solutions to 
\ekv{a.21.5}{
(V-(h\partial )^2)(a(x;h)e^{\psi (x)/h})=0,\ \psi =\psi _+
}
for which $a$ has a complete asymptotic expansion in $C^\infty
([a,c])$ of the form
\ekv{a.22}
{
a\sim \sum_{j=0}^\infty a_j(x)h^j,
}
where $c$ is any fixed number in $]a,b-1/{\cal O}(1)[$. 

\par By solving the usual sequence of transport equations, we have
a unique continuation of the $a_j$ to the full interval $[a,b]$ so
that $e^{\psi /h}\sum_0^\infty a_jh^j$ is a formal asymptotic solution
of (\ref{a.21.5}) and as we have seen in Subsection \ref{sta}, we have 
\ekv{a.22.5}
{
\partial ^\alpha a_j(x)={\cal O}(|x-z_0|^{-\frac{3j}{2}-\alpha }).
}
The power $|x-z_0|^{-1/4}$ in Subsection \ref{sta} corresponds
to the factor $V(x)^{-1/4}$ which is no longer counted in $a$ but in
the exponential factor $e^{\psi /h}=V^{-1/4}e^{\phi /h}$.

On the other hand $ae^{\psi /h}$ has a unique extension to the full
interval $[a,b]$ as a solution of (\ref{a.21}) that we can still write
on the same form and we shall show that the asymptotic expansion
(\ref{a.22}) still holds in sup norm and with the natural remainder
estimates. Write $a=\sum_0^Na_jh^j+r_N=a^N+r_N$, so that 
$$
(V-(h\partial )^2)(r_Ne^{\psi /h})=((h\partial )^2-V)(a^Ne^{\psi /h}).
$$ 
We know that $r_N={\cal O}(h^{N+1})$ with all its derivatives on
$[a,c]$. 

\par Let $\chi \in C^\infty ([a,b];[0,1])$ vanish near $a$ and be
equal to one in a neighborhood of $[c,b]$. Write
\ekv{a.24}
{
(V-(h\partial )^2)(\chi r_Ne^{\psi /h})=((h\partial )^2-V)(a^Ne^{\psi
  /h})+((h\partial )^2-V)((1-\chi )r_Ne^{\psi /h}).
}
Here $((h\partial )^2-V)((1-\chi )r_Ne^{\psi /h})=b_Ne^{\psi /h}$,
where $b_N={\cal O}(h^{N+2})$ with all its derivatives. On the other
hand, using that $e^{\psi /h}\sum_0^\infty a_jh^j$ is a formal
asymptotic solution, we get 
$$
e^{-\psi /h}((h\partial )^2-V)(a^Ne^{\psi /h})=h^{N+2}c_N,
$$
where $\partial ^\alpha c_N={\cal O}(|x-z_0|^{-\frac{3N}{2}-2-\alpha
})$, so 
$$
(V-(h\partial )^2)(\chi r_Ne^{\psi /h})=h^{N+2}d_Ne^{\psi /h},
$$
where $\partial ^\alpha d_N={\cal O}(|x-z_0|^{-\frac{3N}{2}-2-\alpha
})$.

\par We conclude
that 
\ekv{a.25}
{
\chi r_N={\cal O}\left( \frac{1}{h}\right)\int_a^x
\frac{h^{N+2}}{|y-z_0|^{\frac{3N}{2}+2+\frac{1}{2}}}dy={\cal O}(1)\frac{h^{N+1}}{|x-z_0|^{\frac{3}{2}(N+1)}}.
}
thus $r_N$ satisfies the same estimate. 

In principle we could also show that $\partial ^\alpha r_N={\cal
  O}(1)h^{N+1}/|x-z_0|^{\frac{3}{2}(N+1)+\alpha }$, but content
ourselves with the observation that this is the case in the
situation of Subsection \ref{sta}, since the holomorphy then allows us
to use the Cauchy inequalities.

\end{document}